\documentclass[a4paper,english,10pt]{article}
\usepackage[left=.9in, right=.9in, bottom=1.2in, top=1in,dvips]{geometry}
\makeatletter

\AtBeginDocument{\setlength\abovedisplayskip{5pt plus 2pt minus 1pt}
\setlength\belowdisplayskip{5pt plus 2pt minus 1pt}}

\usepackage[utf8]{inputenc}
\usepackage{lmodern}
\usepackage{indentfirst}
\usepackage[affil-it]{authblk}
\usepackage{rotating}
\usepackage{tocloft}
\usepackage{titlesec}
 \titleformat{\subparagraph}[hang]{\normalfont}{\thesubparagraph}{0pt}{\underline}
 \titleformat{\paragraph}[hang]{\normalfont}{\theparagraph}{0pt}{\myuline}
\usepackage{titletoc}
\usepackage{color, soul}
 \usepackage{xcolor}
 \usepackage[citebordercolor={green}]{hyperref}
\usepackage{array}
\usepackage{tabularx}
\usepackage{amssymb}
\usepackage{amsmath}

\usepackage{amsthm,amsfonts,amssymb,euscript}
\usepackage{latexsym, multicol, fancybox}
\usepackage{amsmath, amsthm, amssymb, bm}
\usepackage{caption}
\usepackage{psfrag}
\usepackage{setspace}
\usepackage{mathrsfs}
\usepackage{epsfig}
\usepackage{verbatim}
\usepackage[affil-it]{authblk}
\usepackage{ipa}

\newtheorem{theorem}{Theorem}[section]
\newtheorem{lemma}[theorem]{Lemma}
\newtheorem{proposition}[theorem]{Proposition}

\newtheorem{definition}[theorem]{Definition}
\newtheorem{remark}[theorem]{Remark}

\newcommand{\bea}{\begin{eqnarray}}
\newcommand{\eea}{\end{eqnarray}}
\def\beaa{\begin{eqnarray*}}
\def\eeaa{\end{eqnarray*}}

\newcommand{\MM}{\mathcal{M}}
\def\DD{{\mathcal D}}
\def\TT{{\mathcal T}}
\def\PP{\mathcal{P}}

\def\A{{\bf A}}
\def\D{{\bf D}}
\def\F{{\bf F}}
\def\R{{\bf R}}
\def\W{{\bf W}}
\def\g{{\bf g}}
\def\M{{\bf M}}
\def\CCC{{\Bbb C}}
\def\Ddot{\dot{\D}}

\def\a{{\alpha}}
\def\al{{\alpha}}
\def\b{{\beta}}
\def\be{{\beta}}
\def\ga{\gamma}

\def\de{\delta}

\def\ep{\epsilon}

\def\la{\lambda}

\def\om{\omega}
\def\vphi{\varphi}
\def\si{\sigma}
\def\th{\theta}
\def\ze{\zeta}
\def\nab{\nabla}
\def\mub{\underline{\mu}}
\def\trch{{\mbox tr}\, \chi}
\def\chih{{\hat \chi}}
\def\chib{{\underline \chi}}
\def\chibh{{\underline{\chih}}}
\def\etab{{\underline \eta}}
\def\omb{{\underline{\om}}}
\def\bb{{\underline{\b}}}
\def\aa{{\underline{\a}}}
\def\xib{{\underline \xi}}
\def\Xib{\underline{\Xi}}
\def\lap{{\Delta}}
\def\atr{\,^{(a)}\mbox{tr}}
\def\trchb{{\tr \,\chib}}
\def\atrch{\atr\chi}
\def\atrchb{\atr\chib}
\def\rhod{\,\dual\rho}
\def\nabc{\,^{(c)}\nab}

\def\varo{{\varrho}}
\def\varoc{\check{\varrho}}

\def\bff{\mathfrak{b}}
\def\qf{\frak{q}}
\def\pf{\mathfrak{p}}
\def\xf{\mathfrak{x}}
\def\ff{\frak{f}}
\def\Ffr{\mathfrak{F}}
\def\Bfr{\mathfrak{B}}
\def\Xfr{\mathfrak{X}}
\def\sk{\mathfrak{s}}
\def\Pfr{\mathfrak{P}}
\def\Qfr{\mathfrak{Q}}

\def\bF{\,^{(\bold{F})} \hspace{-2.2pt}\b}
\def\bbF{\,^{(\bold{F})} \hspace{-2.2pt}\bb}
\def\rhoF{\,^{(\bold{F})} \hspace{-2.2pt}\rho}

\def\BF{\,^{(\bold{F})} \hspace{-2.2pt} B}
\def\BBF{\,^{(\bold{F})} \hspace{-2.2pt}\underline{B}}
\def\PF{\,^{(\bold{F})} \hspace{-2.2pt} P}

\def\Lb{{\,\underline{L}}}
\def\Xho{\,^{(h)}X}
\def\pr{{\partial}}

\def\c{\cdot}
\def\dual{{\,\,^*}}
\def\div{{\mbox div\,}}
\def\curl{{\mbox curl\,}}
\def\lot{\mbox{ l.o.t.}}
\def\Hb{\,\underline{H}}
\def\Ab{\underline{A}}
\def\Bb{\underline{B}}
\def\Xb{\underline{X}}
\def\Xh{\widehat{X}}
\def\Xbh{\widehat{\Xb}}
\def\tr{\mbox{tr}}
\def\hot{\widehat{\otimes}}

\def\squared{\dot{\square}}

\def\DDc{\,^{(c)} \DD}
\def\DDb{\ov{\DD}}
\def\DDbc{\ov{\DDc}}
\def\DDov{\ov{\DD}}

\def\Us{{\,^{(s)}U}}
\def\Ua{{\,^{(a)}U}}
\def\nn{\nonumber}
\def\ov{\overline}
\def\Mb{\underline{M}}
\def\DDs{ \, \DD \hspace{-2.4pt}\dual    \mkern-20mu /}

\begin{document}

 \title{\LARGE \textbf{Electromagnetic-gravitational perturbations \\ of Kerr-Newman spacetime:\\ the Teukolsky and Regge-Wheeler equations}}
 
 \author[1]{{\Large Elena Giorgi\footnote{egiorgi@princeton.edu}}}

\affil[1]{\small  Princeton University, Gravity Initiative, Princeton, 
United States \vspace{0.2cm} \ }

\maketitle

\begin{abstract} We derive the equations governing the linear stability of Kerr-Newman spacetime to coupled electromagnetic-gravitational perturbations. The equations generalize the celebrated Teukolsky equation for curvature perturbations of Kerr, and the Regge-Wheeler equation for metric perturbations of Reissner-Nordstr\"om.  Because of the ``apparent indissolubility of the coupling between the spin-1 and spin-2 fields", as put by Chandrasekhar \cite{Chandra}, the stability of Kerr-Newman spacetime can not be obtained through standard decomposition in modes. Due to the impossibility to decouple the modes of the gravitational and electromagnetic fields, the equations governing the linear stability of Kerr-Newman have not been previously derived.
Using a tensorial approach that was applied to Kerr \cite{GKS}, we produce a set of generalized Regge-Wheeler equations for perturbations of Kerr-Newman, which are suitable for the study of linearized stability by physical space methods. The physical space analysis overcomes the issue of coupling of spin-1 and spin-2 fields and represents the first step towards an analytical proof of the stability of the Kerr-Newman black hole.

\end{abstract}

\bigskip\bigskip

\tableofcontents

\section{Introduction }

One of the fundamental problems in General Relativity is to understand the final state of evolution of initial data for the Einstein equation. Through gravitational collapse and dispersion of gravitational waves,  the geometry to which solutions to the Einstein equation are expected to relax outside the event horizon of a black hole is the one given by the known stationary and axisymmetric explicit solutions: the Kerr and the Kerr-Newman black hole.

According to General Relativity, the interaction between gravitational and electromagnetic fields in a spacetime is governed by the \textbf{Einstein-Maxwell equation}:
\bea\label{Einstein-Maxwell-intro}
\operatorname{Ric}_{\mu\nu}(g)=2 F_{\mu \lambda} {F^{\lambda}}_\nu - \frac 1 2 g_{\mu\nu} F^{\alpha\beta} F_{\alpha\beta}, \qquad   D_{[\mu} F_{\nu \lambda]}=0, \qquad D^\mu F_{\mu\nu}=0
\eea
where $\operatorname{Ric}$ denotes the Ricci curvature of the metric, $D$ its covariant derivative, and $F$ is an antisymmetric 2-tensor, called the electromagnetic tensor, which satisfies the Maxwell equations.

The \textbf{Kerr-Newman metric} \cite{Newman} is the most general known explicit black hole solution to the Einstein-Maxwell equation \eqref{Einstein-Maxwell-intro}, and it is a $3$-parameter family which describes the gravitational field around an isolated rotating charged black hole of mass $M$, angular momentum $Ma$ and electric charge $Q$, within the subextremal range $\sqrt{a^2+Q^2} < M$.  Its expression in Boyer-Lindquist coordinates $(t, r, \th, \vphi)$ is given by
\beaa
\g_{M, a, Q}=-\frac{\Delta}{|q|^2}\left( dt- a \sin^2\th d\vphi\right)^2+\frac{|q|^2}{\Delta}dr^2+|q|^2 d\th^2+\frac{\sin^2\th}{|q|^2}\left(a dt-(r^2+a^2) d\vphi \right)^2,
\eeaa
where 
\beaa
\Delta &=& r^2-2Mr+a^2+Q^2, \qquad |q|^2 = r^2+a^2\cos^2\theta.
\eeaa
The Kerr-Newman metric generalizes the Reissner-Nordstr\"om solution (for $a=0$), and also the Kerr (for $Q=0$) and Schwarzschild metric (for $Q=a=0$), which are solutions to the Einstein vacuum equation. As such, the Kerr-Newman spacetime plays a fundamental role in describing the final state of evolution in General Relativity. 

As part of the resolution of the description of the final state, we focus on the issue of stability of the Kerr-Newman black hole, which consists in showing that solutions to the Einstein equation which are given as small perturbations of the initial data of such a black hole asymptotically converge in time to a member of the Kerr-Newman family. 
The stability of the Kerr-Newman family can be analyzed at different levels:
\begin{enumerate}
\item the {\textit{linear stability}} consists in the analysis of the linearized Einstein-Maxwell equation around the background metric $\g_{M, a, Q}$. It can be further divided into (a) mode stability, and (b) full linear stability. 
\item the {\textit{non-linear stability}} consists in the analysis of the full Einstein-Maxwell equation for a perturbation of a member of the Kerr-Newman family. 
\end{enumerate}

The mode analysis (a) of the Einstein equation consists in analyzing only special solutions, the so-called mode solutions. In the simplified case of the linear wave equation
\bea\label{wave-equation}
\square_{\g_{M, a, Q}} \psi=0,
\eea
where $\square_{\g_{M, a, Q}}$ is the D'Alembertian associated to the Kerr-Newman metric, mode solutions are solutions of the  separated form 
\bea\label{mode-solution}
\psi(r, t, \theta, \phi)= e^{- i \omega t} e^{i m\phi} R(r) S(\th)
\eea
where $\omega \in \mathbb{C}$ is the time frequency, and $m$ is the azimuthal mode. 
Because of the integrability of the geodesic flow in the Kerr-Newman metric, functions of the form \eqref{mode-solution} are solutions to the wave equation \eqref{wave-equation}, as long as$R(r)$ satisfies a radial ODE and $S(\th)$ satisfies an angular ODE (which defines spheroidal harmonics $S_{\om m \ell}$). 
 The mode stability consists in proving that solutions of the form \eqref{mode-solution} with finite initial energy do not have imaginary part of $\omega$ which is positive, i.e. do not exponentially grow in time. 
 The mode stability of Schwarzschild, Reissner-Nordstr\"om and Kerr black hole was obtained as a combination of many results in black hole perturbation theory by the physics community in the 70s and 80s, see \cite{Regge-Wheeler}, \cite{Zerilli}, \cite{Bardeen-Press}, \cite{Chandra}, \cite{Teukolsky}, \cite{Whiting}. 
 
 Particularly relevant are the case of axial metric perturbations of Schwarzschild, which are governed by the so-called \textbf{Regge-Wheeler equation} \cite{Regge-Wheeler}, of the form
 \bea\label{RW-intro}
 \square_{\g_{M}} \psi= \frac{4}{r^2} \left( 1-\frac{2M}{r} \right) \psi.
 \eea
 Observe that the potential on the right hand side of \eqref{RW-intro} is positive in the exterior of the black hole. Inspired by \eqref{RW-intro}, we denote by \textit{Regge-Wheeler equation} any equation of the form $\square_g \psi - V\psi=0$ for a positive real potential $V$.

 In the case of gravitational perturbations of Kerr, in order to obtain an equation decoupled from any other component, one needs to consider perturbations at the level of curvature. The extreme null Weyl scalars then satisfy the \textbf{Teukolsky equation} of spin $s=\pm2$ and $s=\pm1$ for gravitational and electromagnetic perturbations of Kerr respectively \cite{Teukolsky}, of the form 
 \bea\label{Teukolsky-eq-intro}
 \begin{split}
\TT^{[s]}(\psi)&:= \Box_{\g_{M,a}} \psi^{[s]} + \frac{2s}{|q|^2}(r-M) \partial_r \psi^{[s]} +\frac{2s}{|q|^2} \left( \frac{a(r-M)}{\Delta} + i \frac{\cos\th}{\sin^2\th} \right) \partial_\vphi \psi^{[s]} \\
&+\frac{2s}{|q|^2} \left(\frac{M(r^2-a^2)}{\Delta} - r - i a \cos\th \right) \partial_t \psi^{[s]} -\frac{s}{|q|^2} (s \cot^2 \th-1) \psi^{[s]}=0,
\end{split}
\eea
 which is also a separable equation. 
 In \cite{Regge-Wheeler} and  \cite{Whiting} (see also \cite{Yakov} \cite{Rita}), the mode analysis of the Regge-Wheeler and the Teukolsky equation was proved, and a transformation theory \cite{Chandra} (now known as \textbf{Chandrasekhar transformation}) was discovered to connect the metric perturbations approach to the curvature perturbations one. The results in the mode analysis are collected in the monumental book by Chandrasekhar \cite{Chandra}. 
 
 Nevertheless, mode stability for the equation \eqref{wave-equation} is still consistent with the unboundedness of (finite initial energy) solutions as the time increases, i.e. it does not exclude the possibility that
 \bea\label{limsup}
 \limsup_{t \to \infty}\psi(r, t, \theta, \phi)= \infty
 \eea
  This is because statements at the level of the single mode solutions do not imply boundedness statements for the infinite superposition of those modes. The full linear stability (b) consists in proving a uniform bound for a general solution $\psi$ of \eqref{wave-equation}, and therefore excluding \eqref{limsup}.

Extensive progress has been obtained in the last fifteen years which allowed to go beyond the mode analysis in Kerr spacetime, tackling the full linear stability (b) for the linear wave equation. A robust geometric interpretation of the redshift effect \cite{redshift}, a physical space analysis of the trapping region and the superradiance \cite{lectures}, a hierarchy of $r$-weighted decay \cite{rp} all contributed to a complete understanding of the boundedness of solutions to the linear wave equation, finally proving that \eqref{limsup} indeed does not happen. 
 The complete resolution of boundedness and decay statements for the linear wave equation \eqref{wave-equation} was obtained for slowly rotating Kerr solutions \cite{small-a} and then for the full subextremal range $|a|<M$  \cite{big-a} (see also \cite{And-Mor}, \cite{Tataru}).  Similarly, proofs of boundedness and decay statements for the Teukolsky equation have been obtained in Schwarzschild \cite{DHR} (see also \cite{mu-tao}, \cite{pei-ken}, \cite{pei-ken-2},  \cite{Johnson2}) in Reissner-Nordstr\"om \cite{Giorgi4}\cite{Giorgi5} and in Kerr, for slowly rotating \cite{ma2}\cite{TeukolskyDHR} and very recently in the full subextremal range \cite{Y-R}. These results have been used to obtain proof of the full linear stability of Schwarzschild \cite{DHR}, Reissner-Nordstr\"om, for small charge \cite{Giorgi6} and then in the full subextremal range \cite{Giorgi7}, and for slowly rotating Kerr \cite{Kerr-lin1}, \cite{Kerr-lin2}.   Concerning the full non-linear stability of black hole solutions to the Einstein equation,  the only known result is the proof of non-linear stability of Schwarzschild under the class of symmetry of axially symmetric polarized perturbations \cite{stabilitySchwarzschild}.  In the presence of a positive cosmological constant, the Kerr--de Sitter and the Kerr--Newman--de Sitter family  with small angular momentum have also been proved to be non-linearly stable \cite{Hintz-Vasy}\cite{Hintz-M}.
 
 \medskip 
 
Quite strikingly, the Kerr-Newman solution stands up as genuinely different from the similar cases of Kerr or Reissner-Nordstr\"om, even in the simplest possible form of stability, i.e. the mode stability as studied by the black hole perturbation theory community. 
As stated by Chandrasekhar in Section 111 of \cite{Chandra}, ``the methods that have proved to be so successful in treating the gravitational perturbations of the Kerr spacetime do not seem to be applicable (nor susceptible to easy generalizations) for treating the coupled electromagnetic-gravitational perturbations of the Kerr-Newman spacetime." 
The techniques applied in those early works, which relied on decomposition in frequency modes of perturbations of the solutions, failed to be extended to the case of Kerr-Newman spacetime, despite the manifest similarity of the metric to the Kerr case. Again as pointed out by Chandrasekhar in Section 111 of \cite{Chandra}, ``the principal obstacle is in finding separated equation" and in the ``apparent indissolubility of the coupling between the spin-1 and spin-2 fields in the perturbed spacetime". Following the same procedure as in the case of Kerr or Reissner-Nordstr\"om, one reaches a point where the equations can not be decoupled or separated any further. In page 583 of \cite{Chandra}, Chandrasekhar gives an explanation of ``why the system of equations proves intractable in contrast to apparently similar system of equations encountered in the treatment of the perturbations of the Reissner-Nordstr\"om and Kerr spacetimes". The reason has to do with the interaction of the spin-1 and spin-2 fields in a non-spherically symmetric background. We summarize his argument here and describe how we intend to overcome such difficulties towards an analytical proof of the stability of the Kerr-Newman black hole.

\subsection{Why the analytical proof of mode stability for Kerr-Newman fails}

Being the most general explicit black hole solution of the Einstein equation coupled with matter, the Kerr-Newman spacetime has been at the center of analytical and numerical research for decades. 
 Numerical works strongly support the mode stability of Kerr-Newman spacetime \cite{numerics1}, \cite{numerics2}, and the Kerr-Newman metric is expected to be stable as a solution to the fully non-linear Einstein-Maxwell equation. Nevertheless, an analytical proof of even its mode stability is missing, and the state of the art on this problem is pretty much the same as described by Chandrasekhar \cite{Chandra} in 1983. We now explain what are the main issues.
 
 The Einstein-Maxwell equation \eqref{Einstein-Maxwell-intro} governs the interaction between the gravitational radiation, encoded in the left hand side of the equation (i.e. the curvature), and the electromagnetic radiation, encoded in the right hand side (i.e. the electromagnetic tensor). From the study of perturbations of Kerr, we know that the gravitational and the electromagnetic radiation are transported by a spin-2 field $\psi^{[2]}$ and a spin-1 $\psi^{[1]}$ respectively. This is more precisely related to the fact that the extreme null component of the Weyl curvature is a 2-tensor on the sphere $(\th, \vphi)$, while the extreme null component of the electromagnetic tensor is a 1-tensor on the sphere. 
 
 When taken independently, the gravitational and electromagnetic perturbations of Kerr satisfy the Teukolsky equation \eqref{Teukolsky-eq-intro} for spin $s=\pm2$ or $s=\pm1$ respectively. On the other hand, when considering coupled electromagnetic-gravitational perturbations of Kerr-Newman, one should expect a \textbf{system of coupled Teukolsky equations}, as in the case of Reissner-Nordstr\"om \cite{Giorgi4} \cite{Giorgi5}, of the schematic form:
 \bea\label{Teukolsky-system-intro}
 \begin{split}
 \TT^{[1]}(\psi^{[1]})&= \div ( \psi^{[2]})  \\
  \TT^{[2]}(\psi^{[2]})&= \nab \hot ( \psi^{[1]}) 
  \end{split}
 \eea
 where the angular operators on the right hand side relate 1-tensors and 2-tensors. More precisely, if $\psi^{[1]}$ is a 1-tensor and $\psi^{[2]}$ is a symmetric traceless 2-tensor, then $\div(\psi^{[2]})_{a}:=\nab^{b} \psi^{[2]}_{ab}$ is a 1-tensor, and $2\nab \hot (\psi^{[1]})_{ab}=\nab_a \psi^{[1]}_b+\nab_b\psi^{[1]}_a-\de_{ab} \div \psi^{[1]}$ is a symmetric traceless 2-tensor on the sphere.
 
 The issue in the analysis of a coupled system like \eqref{Teukolsky-system-intro} comes from the decomposition in modes.
 The mode decomposition of the Teukolsky variables
\beaa
\psi^{[s[}(t, r, \th, \phi) &=&e^{-i \omega t} e^{i m \phi} R^{[s]}(r)S^{[s]}_{m\ell} ( a \om, \cos\th) 
\eeaa
involves the \textit{spin $s$-weighted spheroidal harmonics} $ S^{[s]}_{m\ell} ( a \om, \cos\th) $ which are eigenfunctions of the spin $s$-weighted laplacian
\beaa
\lap^{[s]}&=& \frac{1}{\sin\th} \partial_\th (\sin\th \partial_\th) - \frac{m^2+2ms \cos\th +s^2}{\sin^2\th}+a^2\om^2\cos^2\th-2a\om s \cos\th.
\eeaa
For $a=0$, they reduce to the spherical harmonics $ S^{[s]}_{m\ell} ( 0, \cos\th) = Y^{[s]}_{m\ell} (\cos\th) $. \textit{Spin-weighted spherical harmonics} of different spins are simply related through the angular operators $\div$ and $\nab\hot$, and have the same eigenvalues. Schematically:
\beaa
\nab\hot (Y^{[1]}_{m\ell})&=&-\lambda Y^{[2]}_{m\ell}\\
\div (Y^{[2]}_{m\ell})&=&\lambda Y^{[1]}_{m\ell}
\eeaa
On the other hand, in the general axisymmetric case, as in Kerr or Kerr-Newman, the spin-weighted spheroidal harmonics of different spins are not simply related through those angular operators. 

We are now ready to explain the ``apparent indissolubility of the coupling between the spin-1 and spin-2 fields" \cite{Chandra} for electromagnetic-gravitational perturbations of Kerr-Newman, in contrast with Reissner-Nordstr\"om or Kerr. In a spherically symmetric background, as in Reissner-Nordstr\"om, the fact that the spherical harmonics of different spins are simply related through the angular derivatives implies that the decomposition in modes of the system of Teukolsky equations \eqref{Teukolsky-system-intro} passes through. When considering the separated versions of the equations, one obtains
\beaa
 \TT^{[1]}(Y^{[1]}_{m\ell})&= \div ( Y^{[2]}_{m\ell})= \lambda Y^{[1]}_{m\ell}  \\
  \TT^{[2]}(Y^{[2]}_{m\ell})&= \nab \hot ( Y^{[1]}_{m\ell}) =-\lambda Y^{[2]}_{m\ell}
\eeaa
 giving two decoupled equations for the spin-1 and the spin-2 fields. For gravitational perturbations of Kerr one only uses the spin-2 decomposition for $\TT^{[2]}(Y^{[2]}_{m\ell})=0$, so the problem of the coupling does not arise.

 In electromagnetic-gravitational perturbations of the axially symmetric Kerr-Newman, the interaction between the spin-2 and spin-1 prevents the separability in modes. In particular, when trying to derive equations for $S^{[1]}_{m\ell}$ and for $S^{[2]}_{m\ell}$, one cannot separate them:
 \beaa
 \TT^{[1]}(S^{[1]}_{m\ell})&= \div ( S^{[2]}_{m\ell}) \\
  \TT^{[2]}(S^{[2]}_{m\ell})&= \nab \hot ( S^{[1]}_{m\ell}) 
\eeaa
as the right hand side of the first equation cannot be written in terms of $S^{[1]}_{m\ell}$, and the right hand side of the second equation cannot be written in terms of $S^{[2]}_{m\ell}$.

These are the main obstacles to separability of the equations in the case of electromagnetic-gravitational perturbations of Kerr-Newman spacetime. As Chandrasekhar ends at page 583 of \cite{Chandra}, ``one might be inclined to conclude that a decoupling of the system of equations and a separation of the variables will be possible, if at all, only by contemplating equations of order 4 or higher".

\subsection{Towards the full linear stability of Kerr-Newman}  
 
  In treating the coupled electromagnetic-gravitational perturbations of Kerr-Newman spacetime, the decomposition in modes of the equations, which had the objective of simplifying the analysis of the perturbations, actually makes them unsolvable as consequence of the discussion in the previous section.
Observe that such failure is explicitly related to the fact that the equations as analyzed in \cite{Chandra} required the decomposition in spheroidal harmonics, which yields the problem of non-separability of the decomposition. There is no reason to believe that if one does not decompose in modes nor separate the equations using the spheroidal  harmonics such problems could not be circumvented.

Our approach to solve this issue is to abandon the decomposition in modes, and perform a \textbf{physical space analysis} of the equations.
Following the road map that mathematicians have taken in the last few years in interpreting in physical space the mode analysis done by the physics community, the Kerr--Newman solution may be the case where a physical space approach could succeed where the mode analysis in physics failed.  Observe that our proof of boundedness of a general solution through a physical space analysis will in particular imply the absence of exponentially growing modes, therefore proving mode stability.

We summarize here the four main ingredients in the analysis: the formalism to study perturbations of the Kerr-Newman spacetime, the identification of the gauge-invariant quantities in the linear perturbations, the derivation of the system of coupled Teukolsky equations, and finally the derivation, through the Chandrasekhar transformation, of a system of generalized Regge-Wheeler equations. 

 \subsubsection{The GKS formalism}

 As a first step, we present the formalism which we use to treat perturbations of axially symmetric Petrov Type D spacetimes, like Kerr or Kerr-Newman.  One way to analyze the perturbations is to use the Newman-Penrose formalism, which consists in decomposing all the components in null frames, obtaining complex scalars. We instead make use of a more geometric formalism, more commonly used in the mathematical community, and first developed in the proof of non-linear stability of Minkowski space \cite{Ch-Kl}. Such formalism was extended  in \cite{GKS} for general Petrov Type D spacetime in the context of the non-linear stability of Kerr.  
 
 We recall that a Petrov Type D spacetime's Weyl curvature $\W_{\mu\nu\a\b}$ is diagonalizable by two linearly independent eigenvectors, the so-called principal null-directions. We call the outgoing null direction $e_4$ and the ingoing one $e_3$. The tangent space orthogonal to them is spanned by two orthonormal vectors, $e_a$, for $a=1,2$. Observe that in Kerr, the orthogonal structure determined by the principal null frames $e_3, e_4$ is not integrable, i.e. $e_1$ and $e_2$ do not span the tangent space of a 2-surface, like in Schwarzschild. This can be seen in the non symmetry of the 2-tensors $\chi$ and $\chib$ defined by
 \beaa
 \chi(e_a, e_b)= g(D_a e_4, e_b), \qquad  \chib(e_a, e_b)= g(D_a e_3, e_b), \qquad a,b=1,2
 \eeaa
 which in the case of an integrable horizontal structure would be the null second fundamental forms of the embedding of the sphere in the spacetime, therefore being symmetric. In the case of Kerr or Kerr-Newman, the 2-tensor $\chi$ and $\chib$ are not symmetric in $a$ and $b$. 
 
 In the GKS formalism (from the authors in \cite{GKS}), the non-integrability of the horizontal structure is allowed, and all components are decomposed in null frames, obtaining a range of complex 2-tensors, 1-tensors and scalars.  See Section \ref{section-preliminaries} for the description of the formalism, with particular attention to the comparison with the Newman-Penrose formalism in Section \ref{section:connection-NP}. We extend it here to the case of the Einstein-Maxwell equation, by deriving the equations in their full generality in Section \ref{section-equations}.

 We then apply this general formalism to the case of Kerr-Newman and its linear perturbations. More precisely, the GKS quantities which vanish in Kerr-Newman,  are considered to be $O(\ep)$, where $\ep$ is a smallness parameter, in linear perturbations of Kerr-Newman, see Section \ref{kn-section}. The next step is to identify the $O(\ep)$-quantities which govern the linear perturbations.

 \subsubsection{The identification of the gauge-invariant quantities}

 The first issue to specifically treat electromagnetic-gravitational perturbations of Kerr-Newman spacetime is to identify what are the Teukolsky variables which represent the electromagnetic and gravitational radiations respectively. Since those variables have a physical meaning, they should be independent of the choice of coordinates  to a certain extent, or more precisely being (quadratically) invariant under infinitesimal tetrad transformations. For example, the spin-2 complex Teukolsky variable given by
 \beaa
 A_{ab}&=& \W(e_4, e_a, e_4, e_b) + i \dual \W(e_4, e_a, e_4, e_b),
 \eeaa
where $\dual$ denotes the Hodge dual, is a symmetric traceless 2-tensor on the horizontal structure (which corresponds to $\Psi_0$ in Newman-Penrose formalism) and is known to be invariant under infinitesimal rotations of the frame. More precisely, if a rotation is applied to the frame $(e_3, e_4, e_1, e_2)$ into a new frame $(e_3', e_4', e_1', e_2')$ which is $\ep$-close to the previous one, the variable $A'$ computed with respect to the primed frame is $\ep^2$-close to the original one, i.e. $A'=A+O(\ep^2)$. The quantity $A$ is precisely the Teukolsky variable representing gravitational perturbations of Kerr, and satisfies the Teukolsky equation of spin $2$ \cite{Teukolsky}\cite{GKS}. 

For electromagnetic-gravitational perturbations of Kerr-Newman, the Teukolsky variable $A$ is not sufficient to describe the full perturbation. In particular, one would need a quantity satisfying a spin $1$ Teukolsky equation as the electromagnetic contribution. The Teukolsky variable of spin $1$ in electromagnetic perturbations of Kerr, i.e.
\beaa
\BF_a&=& \F(e_4, e_a) + i \dual \F(e_4, e_a)
\eeaa
where $\F$ is the electromagnetic tensor, is not invariant under infinitesimal rotations of the frame, and therefore cannot represent electromagnetic radiation. This problem appears already in the perturbations of Reissner-Nordstr\"om spacetime, where $\BF$ also fails to be gauge-invariant. In the case of Reissner-Nordstr\"om, two additional quantities $\ff$ and $\tilde{\b}$, a 2-tensor and a 1-tensor respectively, were identified to be gauge-invariant and satisfy a coupled system of Teukolsky equation \cite{Giorgi4}\cite{Giorgi5}.

Inspired by the quantities in Reissner-Nordstr\"om, in Section \ref{section-gauge} we define the symmetric traceless 2-tensor $\Ffr$ and the 1-tensor $\Bfr$ which are quadratically invariant upon infinitesimal rotation of the frame (see \eqref{definition-F-Bianchi} and \eqref{definition-mathfrak-B} for the precise definition). In addition to those, we have the gauge-invariant 1-tensor $\Xfr$ which is auxiliary in the derivation.

 \subsubsection{The system of coupled Teukolsky equations}
 
 As described above, we have defined four gauge-invariant quantities for linear electromagnetic-gravitational perturbations of Kerr-Newman, given by
 \beaa
 A, \qquad \Ffr, \qquad \qquad \qquad \Bfr, \qquad \Xfr
 \eeaa
 where $A$ and $\Ffr$ are symmetric traceless 2-tensors, and therefore good candidates to represent gravitational radiation, and $\Bfr$ and $\Xfr$ are 1-tensors, to represent electromagnetic radiation. As in the case of Reissner-Nordstr\"om, it turns out that $\Ffr$ and $\Bfr$ are the most significant quantities, while $A$ and $\Xfr$ can be thought of as auxiliary quantities.

 Observe that under a rotation of the frame given by a conformal rescaling of the null vectors $e_3$ and $e_4$, i.e. if $e_3'=\lambda e_3$ and $e_4'=\lambda^{-1}e_4$ for a real scalar $\lambda$, the quantities $A$ and $\Xfr$ change as $A'=\lambda^2A$, $\Xfr'=\lambda^2 \Xfr$, while $\Ffr$ and $\Bfr$ change as $\Ffr'=\lambda \Ffr$, $\Bfr'=\lambda \Bfr$. We say that $\Ffr$ and $\Bfr$ are of conformal type $1$, and $A$ and $\Xfr$ of conformal type $2$. 
 
In Section \ref{Teukolsky-equations-section}, we derive the wave-like equations satisfied by $A$, $\Ffr$ and $\Bfr$ as a consequence of the Einstein-Maxwell equations. We obtain the following system of coupled Teukolsky equations, see Theorem \ref{Teukolsky-equations-theorem}:
 \beaa
 \TT_1(\mathfrak{B})&=&\M_1[\mathfrak{F}, \mathfrak{X}] \\
\TT_2(\mathfrak{F})&=&\M_2[A, \mathfrak{X}, \mathfrak{B}]\\
\TT_3(A)&=&\M_3[\Ffr, \Xfr]
\eeaa
where the $\TT_s$ are Teukolsky operators, and $\M_s$ denotes the dependence on the right hand side of each equation. 

The projection to the first component of the Teukolsky operators gives the Teukolsky equation for complex scalars. It turns out that in Kerr-Newman, not only the spin of the variable is not the only parameter appearing in the Teukolsky equation, but so does its conformal type. We define for a complex scalar $\psi$ of spin $s$ and conformal type $c$ the Teukolsky operator
\beaa
\mathcal{T}^{[s,c]}(\psi)&:=&\square_{\g_{M, a,Q}} \psi + \frac{2c}{|q|^2}(r-M) \partial_r \psi +\frac{2}{|q|^2} \left( c\frac{a(r-M)}{\Delta} + s i \frac{\cos\th}{\sin^2\th} \right) \partial_\vphi \psi \\
&&+\frac{2}{|q|^2} \left(c \big(\frac{M(r^2-a^2)-Q^2 r}{\Delta} - r \big)- s i a \cos\th \right) \partial_t \psi + \frac{1}{|q|^2} (s-s^2 \cot^2 \th) \psi
\eeaa
 By comparing it to \eqref{Teukolsky-eq-intro}, one can see that for $c=s$, the Teukolsky operator $\mathcal{T}^{[s,s]}(\psi)$ reduces to the standard Teukolsky operator in Kerr. In Kerr-Newman, it is important to keep the distinction since $\Ffr$ is of spin $2$ and conformal type $1$.

Due to the non-separability in modes, our goal is to analyze the Teukolsky equations in physical space. Unfortunately, this is not possible, even in the case of Schwarzschild or Kerr. Recall that to prove boundedness of the energy for a solution of the wave equation $\square \psi=0$, one multiplies it by $\partial_t \psi$ and integrate it by parts. For example, in the case of Minkowski:
\beaa
0=\square \psi \cdot \partial_t \psi &=& \big(-\partial_t^2 \psi + \partial_x^2 \psi \big)\cdot \partial_t \psi \\
&=& -\partial_t^2 \psi \cdot \partial_t \psi - \partial_t \partial_x \psi \cdot \partial_x \psi + \partial_x( \partial_x \psi \cdot \partial_t \psi )\\
&=& -\frac 1 2 \partial_t ( |\partial_t \psi|^2+ |\partial_x \psi|^2 )+\text{boundary terms}
\eeaa
Upon integration on a causal domain, one can neglect the boundary terms obtaining conservation of the energy. 
Similarly for a Regge-Wheeler equation, the term with the potential can be written as a boundary term:
\beaa
0&=& \big(\square \psi - V \psi \big) \cdot \partial_t \psi \\
&=& -\frac 1 2\partial_t ( |\partial_t \psi|^2+ |\partial_x \psi|^2 )- \frac 1 2V \partial_t(  |\psi|^2)+\text{boundary terms}\\
&=& -\frac 1 2\partial_t ( |\partial_t \psi|^2+ |\partial_x \psi|^2 + V |\psi|^2)+\text{boundary terms}
\eeaa
If the potential $V$ is positive, one obtains the conservation of a positive definite energy. The Teukolsky equation is instead of the form $\square \psi-V \psi =c_1 \partial_r \psi + c_2 \partial_\vphi \psi +c_3 \partial_t \psi$, and so clearly one cannot obtain boundedness of the energy directly in this way, because of the presence of the first order terms. This motivates our search for a more amenable system of equations, of the Regge-Wheeler type.

 \subsubsection{The system of generalized Regge-Wheeler equations}

 One would like to transform the system of Teukolsky equations, which are intractable to physical space energy estimates, to a system of Regge-Wheeler-type equations. Such transformation is related to the passage from curvature perturbations to metric perturbations, and was referred to as ``transformation theory" in \cite{Chandra}. Chandrasekhar describes such transformation in the mode analysis as consisting in taking derivatives along the null direction of the Teukolsky variables, in order to obtain solutions to the Regge-Wheeler equation. Dafermos-Holzegel-Rodnianski crucially extended the Chandrasekhar transformation to a physical space one, first in Schwarzschild \cite{DHR} and then in Kerr \cite{TeukolskyDHR}, see also \cite{ma2}. In \cite{stabilitySchwarzschild} and \cite{GKS}, physical space non-linear analogue of the Chandrasekhar transformation have also been introduced.

 Following this idea, in Section \ref{Chandra-section} we define the Chandrasekhar-transformed of the quantities $\Bfr$ and $\Ffr$ in the case of linear perturbations of Kerr-Newman, and obtain the main result of the paper, see Theorem \ref{main-theorem-RW} for the precise statement. 
 
 \begin{theorem}\label{thm-intro} Consider a linear electromagnetic-gravitational perturbation of Kerr-Newman spacetime $\g_{M, a, Q}$, with associated gauge-invariant quantities $\mathfrak{B}$ and $\mathfrak{F}$. 
Then there exist a 1-tensor $\pf$ and a symmetric traceless 2-tensor $\qf^\F$, obtained as Chandrasekhar-transformed of $\Bfr$ and $\Ffr$ respectively, such that $\pf$ and $\qf^\F$ satisfy
 the following coupled system of generalized Regge-Wheeler equations:
\bea
 \square_{\g_{M, a, Q}} \pf-i  \frac{2a\cos\th}{|q|^2}\nab_t \pf  -V_1  \pf &=&4Q^2 \frac{(r-ia\cos\th)^3 }{|q|^5} \div  \qf^\F  + \lot \label{final-eq-1-intro}\\
\square_{\g_{M, a, Q}} \qf^\F-i  \frac{4a\cos\th}{|q|^2}\nab_t \qf^\F -V_2  \qf^\F &=&-   \frac 1 2\frac{(r+i a \cos\th)^3}{|q|^5}  \nab \hot  \pf  + \lot \label{final-eq-2-intro}
 \eea
with real positive potentials $V_1$ and $V_2$, and $\lot$ denotes lower order terms with respect to $\pf$ and $\qf^\F$.
\end{theorem}

These equations represent the main system of equations governing electromagnetic-gravitational perturbations of the Kerr-Newman spacetime.  
In particular observe that, in applying the Chandrasekhar transformation, the dependence on the auxiliary quantities $A$ and $\Xfr$ disappears. The above equations have the same structure as the generalized Regge-Wheeler equation in Kerr obtained in \cite{GKS}. Also, for $a=0$ the above system of equations reduces to the Regge-Wheeler system of Reissner-Nordstr\"om in \cite{Giorgi7}, for which boundedness of the energy was obtained in the full subextremal range. 

The above Theorem is proved through a careful (and lengthy) computation which consists in applying an ingoing null derivative to both Teukolsky equations for $\Bfr$ and $\Ffr$, and then choose a precise complex rescaling of the transformed quantities. Such rescaling is applied in order to obtain the above structure of the equations, for which boundedness of the energy can be obtained in physical space. Observe that the above equations are not precisely of the Regge-Wheeler form, but have additional terms, like the first order term $\nab_t$ and the coupling terms on the right hand side. Nevertheless, in Theorem \ref{thm-intro} we are careful to obtain precisely a structure which allows for the proof of boundedness of energy in physical space.  More precisely, upon multiplying equation \eqref{final-eq-1-intro} by $\nab_t \overline{\pf}$, and equation \eqref{final-eq-2-intro} by $\nab_t \overline{\qf^\F}$ and taking their real part, we obtain the following simplifications.
\begin{enumerate}
\item The Regge-Wheeler pieces of the equations, i.e. $\big(\square_{\g_{M, a, Q}} \pf-V_1  \pf\big) \nab_t \overline{\pf} $ and $\big(\square_{\g_{M, a, Q}} \qf^\F -V_2  \qf^\F\big) \nab_t\overline{\qf^\F} $, for real positive potentials, can be written as boundary term,
\item The first order terms, being of the form $ i \nab_t$, get cancelled in the energy estimates:
\beaa
i  \frac{2a\cos\th}{|q|^2}\nab_t \pf \nab_t \overline{\pf} +\overline{i  \frac{2a\cos\th}{|q|^2}}\nab_t \overline{\pf} \nab_t \pf=0\\
i  \frac{4a\cos\th}{|q|^2}\nab_t \qf^\F \nab_t \overline{\qf^\F} +\overline{i  \frac{4a\cos\th}{|q|^2}}\nab_t \overline{\qf^\F} \nab_t \qf^\F=0
\eeaa
\item The coupling terms, given by adjoint operators $\div$ and $\nab\hot$ multiplied by complex conjugate functions, get simplified upon summing the estimates for the two equations:
\beaa
\frac{(r-ia\cos\th)^3 }{|q|^5} \div  \qf^\F \nab_t \overline{\pf}-\frac{(r+i a \cos\th)^3}{|q|^5}  \nab \hot  \pf \nab_t \overline{\qf^\F}=\lot
\eeaa
\item The lower order terms have a favorable structure in using transport equations to be estimated. 
\end{enumerate}

As a consequence of the above theorem, the good properties of the equations obtained in Reissner-Nordstr\"om and in Kerr can be generalized to the case of Kerr-Newman. This strikingly compares with the equations in separated modes as described at the beginning of this introduction, which could not be generalized from the Kerr and Reissner-Nordstr\"om case. By avoiding the decomposition in modes, and maintaining the above equations for the 1-tensor $\pf$ and the 2-tensor $\qf^\F$, the issue of non-commutativity of the decomposition is not present and a physical space analysis of the above system is possible. 
In Section \ref{section-bdn-energy}, we sketch how to prove that solutions to the generalized Regge-Wheeler equations as obtained in Theorem \ref{thm-intro} have bounded energy. Such proof has to be combined with spacetime Morawetz estimates to obtain the complete statement of boundedness and decay for the Teukolsky system of equations. More precisely, such analysis would have to avoid decomposition in modes for the solutions, for example in the spirit of \cite{And-Mor} for small angular momentum. 
 Nevertheless, the Morawetz estimates, which will be obtained in a future work, are much less sensitive to the structure of the equation, and the procedure described in Section \ref{section-bdn-energy} to obtain bounded energy is crucial to justify the precise structure of the equations here derived.

\vspace{0.3cm}

This paper is organized as follows. In Section \ref{section-preliminaries}, we recall the general formalism introduced in \cite{GKS} and in Section \ref{section-equations}, we derive the Einstein-Maxwell equations in their full generality. In Section \ref{kn-section}, we introduce the Kerr-Newman spacetime and its linear perturbations. In Section \ref{section-gauge} we define the main gauge-invariant quantities in electromagnetic-gravitational  perturbations of Kerr-Newman spacetime. In Section \ref{Teukolsky-equations-section}, we derive the system of Teukolsky equations satisfied by the gauge-invariant quantities. Finally, in Section \ref{Chandra-section} we define the Chandrasekhar transformation in Kerr-Newman and derive the Regge-Wheeler-type equations for the perturbations, proving the main theorem of the paper. 

To facilitate the reading of the paper, we diverted most of the proofs (involving lengthy computations) to the Appendix. In Appendix \ref{appendix-section-a} we collect the explicit computations needed in the first five sections of the paper. In Appendix \ref{proof-main-thm} we derive the system of Teukolsky equations and in Appendix \ref{RWeqapp} we derive the system of generalized Regge-Wheeler equations.

\vspace{0.3cm}

\noindent\textbf{Acknowledgements.} The author is supported by the NSF grant DMS 2006741.

\section{The GKS formalism}\label{section-preliminaries}

In this section we collect the main definitions and preliminaries to the formalism introduced in \cite{GKS}. From the authors of \cite{GKS} we refer to this formalism as GKS formalism. We refer to Section 2 of \cite{GKS} for more details.

\subsection{Null pairs and horizontal structures}

Let $(\MM,\g)$ be a Lorentzian $4$-dimensional manifold. 
Consider an arbitrary null pair $e_3=\underline{L}$ and $e_4=L$, i.e. 
\beaa
\g(e_3, e_3)=\g(e_4, e_4)=0, \qquad \g(e_3,e_4)=-2.
\eeaa 
We say that a vectorfield $X$ is  horizontal  if
\beaa
\g(L,X)=\g(\Lb, X)=0. 
\eeaa
On the set of horizontal vectors, given a fixed orientation we define the induced volume form by $\in(X, Y):=\frac 1 2 \in(X, Y, \underline{L}, L)$. 

Observe that the commutator $[X,Y]$ of two horizontal vectorfields
may fail to be horizontal. We say that the pair $(L,\Lb)$ is integrable if the set of 
horizontal vectorfields  forms an integrable distribution,
i.e. $X, Y$ horizontal implies that $[X,Y]$ is horizontal.

Given an arbitrary vectorfield $X$ we denote by $\Xho$
its  horizontal projection,
\beaa
\Xho&=&X+ \frac 1 2 \g(X,\Lb)L+ \frac 1 2   \g(X,L)\Lb.
\eeaa
A  $k$-covariant tensor-field $U$ is said to be horizontal
if  for any $X_1,\ldots X_k$ we have,
\beaa
U(X_1,\ldots X_k)=U(\Xho_1,\ldots \Xho_k).
\eeaa
 \begin{definition}  For any horizontal $X,Y$ we define\footnote{In the particular case where the horizontal structure is integrable, $\ga$ is the induced metric, and $\chi$ and $\chib$ are the  null second fundamental forms.} 
  \bea
 \ga(X,Y)&=&\g(X,Y)
 \eea
 and   
\beaa
\chi(X,Y)=\g(\D_XL,Y),\qquad \chib(X,Y)=\g(\D_X\Lb,Y).
\eeaa
where $\D$ is the covariant derivative of $\g$. 
\end{definition}
Observe that    $\chi$
 and $\chib$  are  symmetric if and 
 only if   the horizontal structure is 
 integrable, as follows from 
 \beaa
  \chi(X,Y)-\chi(Y,X)&=&\g(\D_X L, Y)-\g(\D_YL,X)=-\g(L, [X,Y]).
 \eeaa

 Given $X, Y$ horizontal vectors, the covariant derivative $\D_XY$ fails in general  to be horizontal.
 We thus define,\footnote{In the integrable case, $\nab$ coincides with the Levi-Civita connection
 of the metric induced on the integral surfaces of   the horizontal distribution.}
 \beaa
 \nab_X Y&:=&^{(h)}(\D_XY)=\D_XY- \frac 1 2 \chib(X,Y)L -  \frac 1 2 \chi(X,Y) \Lb.
 \eeaa

 Given a general covariant, horizontal tensor-field  $U$
 we define its horizontal covariant derivative according to
 the formula,
 \beaa
 \nab_Z U(X_1,\ldots X_k)=Z (U(X_1,\ldots X_k))&-&U(\nab_ZX_1,\ldots X_k)-
 \\
\ldots   &-& U(X_1,\ldots \nab_ZX_k).
 \eeaa
 Given $X$ horizontal, $\D_L X$ and $\D_\Lb X$ are in general not horizontal.
We thus define
 \beaa
 \nab_4 X&:=&^{(h)}(\D_L X)=\D_L X- \g(X, \D_L\Lb) L- \g(X, \D_L L) \Lb,\\
 \nab_3 X&:=&^{(h)}(\D_\Lb X)=\D_\Lb X- \g(X, \D_\Lb\Lb) L - \g(X, \D_\Lb L) \Lb. 
 \eeaa
  We can extend the operators $\nab_4$ and $\nab_3$ to
 arbitrary  $k$-covariant, horizontal tensor-fields  $U$ as above.
 Therefore, with the above definitions $\nab$, $\nab_4$ and $\nab_3$ take horizontal tensor-fields into horizontal tensor-fields.
We can then extend the definition of horizontal covariant derivative to any $X$ in the tangent space of $M$ and Y horizontal as 
 \bea\label{eq:def-horizontal-cov-der}
 \Ddot_X Y&:=&^{(h)}(\D_X Y)
 \eea
 and we can extend it to horizontal tensor-fields as above. 
 
  Given a  horizontal structure  defined by $e_3=\Lb$,  $e_4=L$, 
  we  associate a null frame by choosing   orthonormal  horizontal  vectorfields
  $e_1, e_2$  such that $\ga(e_a, e_b)=\de_{ab}$ for $a, b=1,2$.  
For an arbitrary orthonormal horizontal frame $(e_a)_{a=1,2}$, we denote $\nab_a Y=\nab_{e_a} Y$. We write $\nab Y$ to denote the 2-tensor whose contraction with $e_a$ results in $\nab_a Y$, i.e. $\nab Y(e_a)=\nab_a Y$.

In what follows we fix  a null pair $e_3=\Lb$ and $e_4=L$ and an orientation  on the horizontal tensors.

\begin{definition}
Given a 2-covariant horizontal tensor-field $U$ and an arbitrary orthonormal horizontal frame $(e_a)_{a=1,2}$ we define the trace of $U$ as
\beaa
\tr (U):=\de^{ab}U_{ab}=\de^{ab}\Us_{ab},
\eeaa
where $\Us_{ab}=\frac 12 (U_{ab}+U_{ba})$.
We define the anti-trace of $U$ by,
\beaa
\atr (U):=\in^{ab}U_{ab}=\in^{ab}\Ua_{ab},
\eeaa
where $\Ua_{ab}=\frac 12 (U_{ab}-U_{ba})$. 
\end{definition}
A general  horizontal,  2-tensor $U$  can be decomposed according to,
\beaa
U_{ab}&=&\Us_{ab}+\Ua_{ab}=\widehat{U}_{ab}+\frac 1 2 \de_{ab}\, \tr( U)+\frac 12 \in_{ab}\atr (U). 
\eeaa

 \begin{definition}\label{definition-SS-real}
We denote by $\sk_0=\sk_0(\MM)$ the set of pairs of real scalar functions on $\MM$, $\sk_1=\sk_1(\MM)$ the  set of real horizontal $1$-forms  on $\MM$, and by $\sk_2=\sk_2(\MM)$
  the set of real symmetric traceless   horizontal $2$-tensors on $\MM$.
  \end{definition}
We define the following operators on horizontal tensors.

\begin{definition}\label{def-duals}
We define the duals of $f \in \sk_1$ and $u \in \sk_2$ by
\beaa
\dual f_{a}&=&\in_{ab}f_b,\qquad (\dual u)_{ab}=\in_{ac} u_{cb}.
\eeaa
Given  $\xi, \eta\in\sk_1 $  we denote
\beaa
\xi\c \eta&:=&\de^{ab} \xi_a\eta_b,\\
\xi\wedge\eta&:=&\in^{ab} \xi_a\eta_b=\xi\c\dual \eta,\\
(\xi\hot \eta)_{ab}&:=&\frac 1 2 \big( \xi_a \eta_b +\xi_b \eta_a-\de_{ab} \xi\c \eta\big).
\eeaa
 Given   $\xi\in \sk_1 $,  $u\in \sk_2$ we denote
\beaa
(\xi\c u)_a&:=&\de^{bc} \xi_b u_{ac}.
\eeaa
Given     $u, v \in \sk_2$ we denote
\beaa
(u\wedge v)_{ab} &:=& \in^{ab}u_{ac}v_{cb}.
\eeaa
 For $f \in \sk_1$ and $u \in \sk_2$
we  define the frame dependent   operators,   
\beaa
\div f &=&\de^{ab}\nab_b f_a,\qquad 
\curl f=\in^{ab}\nab_af_b,\\
(\nab\hot f)_{ba}&=&\frac 1 2 \big(\nab_b f_a+\nab_a  f_b-\de_{ab}( \div f)\big)\\
(\div u)_a&=& \de^{bc} \nab_b u_{ca}.
\eeaa
\end{definition}

 \subsection{Ricci, electromagnetic and curvature components}\label{section-ricci-el-curv-comp}
  In what follows we define Ricci coefficients, electromagnetic and curvature components of a general spacetime $(\MM, \g)$. 

 \begin{definition} We define the horizontal  $1$-forms,
 \beaa
 \etab(X)&:=& \frac 1 2 \g(X, \D_L\Lb),\qquad \eta(X):= \frac 1 2  \g(X, \D_\Lb L),\\
 \xib(X)&:=& \frac 1 2  \g(X, \D_\Lb \Lb),\qquad \xi(X):= \frac 1 2 \g(X, \D_L L), \\
  \ze(X)&=&\frac 1 2 \g(\D_XL,\Lb).
 \eeaa
 and the scalars
 \beaa
 \omb&:=& \frac 1  4 \g(\D_\Lb\Lb, L),\qquad\quad  \om:=\frac 1 4 \g(\D_L L, \Lb).
 \eeaa
  \end{definition}
Observe that the quantities underlined are obtained from the non-underlined ones by interchanging the null vectors $e_3=\Lb$ and $e_4=L$. 
 
 \begin{definition}
 The horizontal tensor-fields $\chi,\chib, \eta, \etab, \ze,  \xi,\xib,\om, \omb$ are  called
 the connection coefficients of the null pair $(L,\Lb)$. Given  an  arbitrary  basis of 
  horizontal vectorfields $e_1,  e_2$, we write using the short hand notation $\D_a=\D_{e_a}, a=1,2$, 
  \beaa
\chib_{ab}&=&\g(\D_a\Lb, e_b),\qquad \chi_{ab}=\g(\D_aL, e_b),\\
\xib_a&=&\frac 1 2\g(\D_\Lb\Lb, e_a),\qquad \xi_a=\frac 1 2 \g(\D_L L, e_a),\\
\omb&=&\frac 1 4 \g(\D_\Lb\Lb, L),\qquad\quad  \om=\frac 1 4\g(\D_L L, \Lb),\qquad \\
\etab_a&=&\frac 1 2\g(\D_L\Lb, e_a),\qquad \quad \eta_a=\frac 1 2 \g(\D_\Lb L, e_a),\qquad\\
 \ze_a&=&\frac 1 2 \g(\D_{a}L,\Lb).
\eeaa
 \end{definition}
We easily  derive the  Ricci formulae,
\bea\label{conn-coeff}
\D_a e_b&=&\nab_a e_b+\frac 1 2 \chi_{ab} e_3+\frac 1 2  \chib_{ab}e_4,\nn\\
\D_a e_4&=&\chi_{ab}e_b -\ze_a e_4,\nn\\
\D_a e_3&=&\chib_{ab} e_b +\ze_ae_3,\nn\\
\D_3 e_a&=&\nab_3 e_a +\eta_a e_3+\xib_a e_4,\nn\\
\D_3 e_3&=& -2\omb e_3+ 2 \xib_b e_b,\label{ricci}\\
\D_3 e_4&=&2\omb e_4+2\eta_b e_b,\nn\\
\D_4 e_a&=&\nab_4 e_a +\etab_a e_4 +\xi_a e_3,\nn\\
\D_4 e_4&=&-2 \om e_4 +2\xi_b e_b,\nn\\
\D_4 e_3&=&2 \om e_3+2\etab_b e_b.\nn
\eea 
\begin{definition}
We  introduce the  notation
\beaa
\trch:=\tr(\chi), \quad \atr\chi:=\atr(\chi),\quad \trchb:=\tr(\chib), \quad \atr\chib:=\atr(\chib).
\eeaa
The symmetric traceless part of $\chi$ and $\chib$, denoted $\chih$ and $\chibh$, are called the (outgoing and ingoing respectively) shear of the horizontal distribution, while the scalars $ \trch$ and $ \trchb$  are the (outgoing and ingoing respectively) expansion 
of the distribution. The scalars $ \atr\chi$ and $\atr\chib$ measure the
integrability defects of the distribution.

In particular we can write
\bea
\chi_{ab}&=&\chih_{ab}+\frac 1 2 \de_{ab}\, \trch+\frac 12 \in_{ab}\atrch. \label{decompU-chi}
\eea

\end{definition}

Let $\F$ be an antisymmetric 2-tensor on $(\MM, \g)$. We define the null components of $\F$ as the horizontal vectors  $\b(\F),\bb(\F)$ by the formulas
\beaa
&&\bF(X)=\b(\F)(X)=\F(X, L) \\
&&\bbF(X)=\bb(\F)(X)=\F(X, \underline{L}) \\
&&\varrho(\F)(X, Y)=\F(X, Y)
\eeaa
For a 2-form $\F$, the dual $\dual \F$ denotes the Hodge dual on $(\MM, \g)$ of $\F$, defined by $\dual \F_{\a\b}=\frac 1 2 \in_{\mu\nu\a\b} \F^{\mu\nu}$. 

It is convient to express in terms of the following two scalar quantities 
 \bea
 \label{fo5'-em}
\rhoF= \rho(\F)=\frac 1 2 \F(\underline{L},L),\qquad \dual\rhoF=\dual\rho(\F)=
\frac 1 2   \dual \F (\underline{L}, L)\label{rho-dualrho}.
 \eea
Thus, $\varrho(\F)(X,Y)=-\rhod(\F)\, \in(X,Y)$ for horizontal vectors $X, Y$. 
i.e. $\F_{ab}=- \in_{ab}\dual \rho$.

Let $\W$ be a Weyl field on $(\MM, \g)$. We define the null components of the Weyl field $\W$, horizontal 2-tensors $\al(\W),\aa(\W),\varrho(\W)$ and horizontal 1-tensors $\be(\W),\bb(\W)$ by the formulas
\beaa
&&\al(\W)(X,Y)=\W(L,X,L,Y),\\
&&\aa(\W)(X,Y)=\W(\Lb,X,\Lb,Y),\\
&&\b(\W)(X)=\frac 1 2 \W(X,L,\Lb,L),\\
&&\bb(\W)(X)=\frac 1 2 \W(X,\Lb,\Lb, L),\\
&&\varrho(\W)(X,Y)= \W(X,\Lb,Y,L).
\eeaa
Recall that if $\W$ is a Weyl field its 
Hodge dual $\dual \W$, defined by 
${}^{\ast}\W_{\al\be\mu\nu}=\frac{1}{2}{\in_{\mu\nu}}^{\rho\si}\W_{\al\be\rho\si}$,  is also a Weyl field.   It is convenient to express it in terms
 of  the following  two scalar quantities,
 \bea
 \label{fo5'}
 \rho(\W)=\frac 1 4 \W(L,\Lb,L,\Lb),\qquad \dual\rho(\W)=
\frac 1 4  \dual \W(L,\Lb,L,\Lb)\label{rho-dualrho}.
 \eea
Thus, $\varrho(X,Y)=-\rho\,\ga(X,Y)+\rhod\, \in(X,Y)$ for horizontal vectors $X, Y$. 
We have
 \beaa
 \W_{a3b4}&=&\varrho_{ab}=(-\rho\de_{ab} +\dual \rho \in_{ab}),\\
 \W_{ab34}&=& 2 \in_{ab}\dual\rho,\\
 \W_{abcd}&=&-\in_{ab}\in_{cd}\rho,\\
 \W_{abc3}&=&\in_{ab}\dual \bb_c,\\
 \W_{abc4}&=&-\in_{ab}\dual \b_c.
 \eeaa

 \subsection{Complex notations}\label{complex-notations-section}

Recall Definition \ref{definition-SS-real} of the set of horizontal tensors $\sk_k$ on $\MM$. By definition \ref{def-duals}, the duals of real horizontal tensors are real horizontal tensors of the same type. We define the complexified version of horizontal tensors on $\MM$.
\begin{definition} We denote by $\sk_k(\CCC)$ the set of complex anti-self dual $k$-tensors on $\MM$. More precisely, $a+i b\in \sk_0(\CCC)$ is a complex scalar function on $\MM$ with $(a, b) \in \sk_0$, $F=f+ i \dual f \in \sk_1(\CCC)$ is a complex anti-self dual 1-tensor on $\MM$ with $f \in \sk_1$, and $U=u+ i \dual u \in \sk_2(\CCC)$ is a complex anti-self dual symmetric traceless 2-tensor on $\MM$ with $u \in \sk_2$.
\end{definition}
Observe that $F \in \sk_1(\CCC)$ and $U \in \sk_2(\CCC)$ are indeed anti-self dual tensors, i.e.
\beaa
\dual F=-i F, \qquad \dual U=-i U.
\eeaa
In particular $F_2=-i F_1$ and $U_{12}=-i U_{11}$, $U_{22}=- U_{11}$, where we denote $F_1:=F(e_1)$ and $U_{11}:=U(e_1, e_1)$ the contraction of the tensors with the horizontal frame.

 We extend the definitions for the Ricci, electromagnetic and curvature components given in Section \ref{section-ricci-el-curv-comp} to the complex case by using the anti-self dual tensors.
 \begin{definition}\label{definition-qu} Let $(\MM, \g)$ be a Lorentzian  $4$-dimensional manifold.
 We define the following complexified versions of the Ricci components
\beaa
&& X=\chi+i\dual\chi, \quad \Xb=\chib+i\dual\chib, \\
&& H=\eta+i\dual \eta, \quad \Hb=\etab+i\dual \etab, \\ 
&& \Xi=\xi+i\dual\xi, \quad \Xib=\xib+i\dual\xib, \\
&& Z=\ze+i\dual\ze.
\eeaa    
In particular, note that 
\beaa
\tr X = \trch-i\atrch, \quad \Xh=\chih+i\dual\chih, \quad \tr\Xb = \trchb -i\atrchb, \quad \Xbh=\chibh+i\dual\chibh.
\eeaa
 We define the following complexified versions of the electromagnetic components
\beaa
&& \BF= \bF + i \dual \bF,  \quad  \BBF= \bbF + i \dual \bbF\\
&& \PF=\rhoF + i \dual \rhoF, 
\eeaa     
and of the curvature components
\beaa
&&A=\a+i\dual\a,  \quad \Ab:=\aa+i\dual\aa, \\
&& B=\b+i\dual\b, \quad \Bb=\bb+i\dual\bb,\\
&& P=\rho+i\dual\rho. 
\eeaa

\end{definition}

With the above definition, the complex scalars, one-forms and symmetric traceless 2-tensors are respectively given by
\beaa
\tr X , \tr \Xb , P, \PF \in \sk_0(\CCC), \\
H, \Hb, Z, \Xi, \Xib, \BF, \BBF, B, \Bb \in \sk_1(\CCC), \\
\Xh, \Xbh, A, \Ab \in \sk_2(\CCC). 
\eeaa

\begin{definition}\label{definition-complex-operators} 
We define the complexified version of the $\nab$ horizontal derivative as
\beaa
\DD= \nab + i \dual \nab, \qquad \DDov= \nab- i \dual \nab
\eeaa
More precisely,
\begin{itemize}
\item For $a+ib \in \sk_0(\CCC)$ 
\beaa
\DD(a+ib) &:=& (\nabla+i\dual\nabla)(a+ib), \qquad \DDov(a+ib) := (\nabla-i\dual\nabla)(a+ib)\
\eeaa

\item For $F= f+ i \dual f \in\sk_1 (\mathbb{C})$, 
\beaa
\DDov\c(f+i\dual f) &:=& (\nabla-i\dual\nabla)\c(f+i\dual f)= 2 \big(\div f + i \curl f \big), \\
\DD\hot(f+i\dual f) &:=& (\nabla+i\dual\nabla)\hot(f+i\dual f)=2\big(\nab \hot f + i \dual (\nab \hot f)\big).
\eeaa

\item For  $U= u + i \dual u \in \sk_2 (\mathbb{C})$, 
\beaa
\DDov (u+i\dual u) &:=& (\nabla- i\dual\nabla) (u+i\dual u)= 2\big( \div u + i \dual(\div u) \big)
\eeaa
\end{itemize}
\end{definition}

 For $F \in \sk_1(\CCC)$, the operator $- \DD\hot$ is formally adjoint to the operator $\DDov \c U$ applied to $U \in \sk_2(\CCC)$, as shown in the following lemma.
 \begin{lemma}\label{lemma:adjoint-operators}
 For $F=f+i\dual f \in \sk_1(\CCC)$ and $U=u+i\dual u \in \sk_2(\CCC)$, we have
   \bea
 ( \DD \hot   F) \c   \ov{U}  &=&  -F \c (\DD \c \ov{U}) -( (H+\Hb ) \hot F )\c \ov{U} +\D_\a (F \c \ov{U})
 \eea
 \end{lemma}
 \begin{proof}  We have
 \beaa
 ( \nab\hot   f) \c   u  &=&\frac 1 2 \big(\nab_a f_b+\nab_b f_a- \de_{ab} \div f \big) u_{ab} =(\nab_a f_b) u_{ab}= \nab_a (u_{ab} f_b)-(\div u) \c f 
 \eeaa
 Let $\xi \in \sk_1$. Then the difference between the spacetime and the horizontal divergence is given by
 \beaa
 \D_\a \xi^\a-  \nab_a \xi^a &=&-\frac 1 2\big( \D_3 \xi_4+\D_4 \xi_3) =(\eta+\etab)\c  \xi
 \eeaa
 which applied to $\xi= u \c f$ gives
  \beaa
 ( \nab\hot   f) \c   u  &=& \D_\a (u_{\a b} f_b)- (\eta+ \etab) \c (u \c f) -(\div u) \c f \\
 &=&  -(\div u) \c f -( (\eta+ \etab) \hot f )\c u +\D_\a (u_{\a b} f_b)
 \eeaa
By complexifying the above, we obtain the stated identity.
 \end{proof}

 \subsection{Frame transformations and conformal invariance}\label{conf-inv-section}

A general frame transformation of the null frame basis vectors $\{e_3, e_4, e_a\}$ into a transformed null frame $\{e'_3, e'_4, e'_a\}$ can be decomposed 
 into the following  three elementary   types:
 \begin{itemize}
 \item \textit{rotations of class I}, which leave the vector $e_4$ unchanged, 
 \bea
\label{SSMe:general.e4}
e_4'=  e_4 , \qquad e_3'=e_3 + \mub_a e_a +\frac 1 4 |\mub|^2  e_4,\qquad 
e_a'=e_a  + \frac 1 2 \mub_a e_4
\eea
 \item    \textit{rotations of class II}, which leave the vector $e_3$ unchanged, 
   \bea
\label{SSMe:general.e3}
e_3'=e_3, \qquad 
e_4'=  e_4+ \mu_a e_a +\frac 1 4| \mu|^2 e_3 \qquad 
e_a'= e_a +\frac 1 2 \mu_a e_3,
\eea
 \item  \textit{rotations of class III}, which leave the directions of $e_3$ and $e_4$ unchanged and rotate $e_a$,
   \bea
   e_3'= \la  e_3, \qquad e_4'=\la^{-1} e_4, \qquad e_a'= O_{ab} e^b
   \eea  
 \end{itemize}   
 where $\mu$ and $\mub$ are real one-forms, $\lambda$ is a real function, $O_{ab}$ is a orthogonal matrix, and the repeated indices indicate the sum on those.


\begin{definition} We say that a frame transformation is conformal if it is a rotation of class III with $O_{ab}=I_{ab}$ the identity matrix, i.e. such that    
\bea\label{eq:conf-type-trans}
   e_3'= \la  e_3, \qquad e_4'=\la^{-1} e_4, \qquad e_a'= e_a
   \eea 
   for a real function $\lambda$. 
\end{definition}

Note that  under  a conformal  frame transformation, the Ricci, electromagnetic, curvature components get modified in the following way:
\beaa
 \trchb'&=&\la^{-1} \trchb, \quad \atrchb'=\la^{-1} \atrchb,  \quad   \trch'=\la\trch, \quad \atrch'=\la \atrch \\
  \xi'&=& \la^2 \xi, \quad   \eta'=\eta, \quad \etab'=\etab,  \quad \xib'=\la^{-2}\xib\\
  \bF'&=& \la \bF, \quad \rhoF'=\rhoF, \quad \dual\rhoF'=\dual\rhoF, \quad \bbF'=\la^{-1} \bbF\\
   \a'&=&\la^2 \a,\quad \b'=\la \b,    \quad \rho'=\rho,    \quad \rhod'=\rhod,  \quad \bb'=\la^{-1} \bb,\quad  \aa'=\la^{-2} \aa
   \eeaa
and similarly for their complex counterparts, and
   \beaa
 \omb'&=& \la^{-1}\left(\omb +\frac{1}{2} e_3(\log \la)\right), \quad \om'= \la\left(\om -\frac{1}{2} e_4(\log \la)\right), \quad
 \ze'= \ze - \nab (\log \la).\\
\eeaa

\begin{definition}We say that  a horizontal tensor $f$ is conformal invariant of type $s$  if, under the  conformal  frame transformation \eqref{eq:conf-type-trans}, it changes as $f'=\la^s f $. 
\end{definition}
\begin{remark} 
Note that $s$ is precisely what   in \cite{Ch-Kl} is called       the  signature of the tensor\footnote{For a horizontal tensor $f$ defined in terms of the null frames $e_3$ and $e_4$, its signature is the number of $e_4=L$ used in its definition minus the number of $e_3=\Lb$. For example, the signature of $\a=\W(L, e_a, L, e_b)$ is $2-0=2$, while the signature of $\eta=\frac 1 2 \g( \D_{\Lb} L, e_a)$ is $1-1=0$.}. 
\end{remark}

Observe that if $f$  is conformal invariant of type $s$, then  $\nab_3 f, \nab_4 f, \nab_a f$ are not conformal invariant.
In GKS formalism, we correct the lacking of being conformal invariant by making the following  definition.  
\begin{lemma}
If $f$  is conformal invariant of type $s$, then 
 \begin{enumerate}
 \item $\nabc_3 f:= \nab_3f-2 s \omb f$ is conformal invariant of type $(s-1)$.
 \item $\nabc_4 f:= \nab_4f+2 s \om f$ is conformal invariant of type $(s+1)$.
 \item $\nabc_a f:= \nab_a f+ s \ze_a  f$ is conformal invariant of type $s$. 
  \end{enumerate}
Also, $\DDc f:= \nabc f + i \dual \nabc f= \DD f + s Z f $ is conformal invariant of type $s$.
\end{lemma}



 \subsection{Comparison with the Newman-Penrose formalism}\label{section:connection-NP}
The GKS formalism here recalled is strongly connected with the more familiar Newman-Penrose (NP) formalism. In NP formalism, one chooses a basis of null vectors  $(n, l, m, \ov{m})$
 with $n$ and $l$ real and $m$ complex, scaled such that $\g(n, l)=-1$. They are related to the null frame $(e_3, e_4, e_1, e_2)$ here presented for example by 
 \beaa
n=\frac 1 2 e_3, \qquad l=e_4,  \qquad m=\frac{1}{\sqrt{2}}(e_1+i e_2), \qquad \ov{m}= \frac{1}{\sqrt{2}}(e_1-i e_2).
 \eeaa
 In NP  formalism, the connection coefficients, electromagnetic and curvature components are all complex scalar functions obtained  by contracting the tensors with the null frame. For example, the extreme null curvature component $\Psi_0$ is the spin-2 complex scalar defined as 
 \beaa
 \Psi_0=-\W_{\mu\nu\gamma\sigma} l^\mu m^\nu l^\gamma m^\sigma=-\W(l, m, l, m).
 \eeaa
 In GKS formalism, the extreme null curvature component is a complex horizontal 2-tensor, as in
 \beaa
 A_{ab}=  \W(e_4, e_a, e_4, e_b)+ i \dual  \W(e_4, e_a, e_4, e_b)=\a_{ab}+ i \dual \a_{ab}.
 \eeaa
 The relation between $\Psi_0$ and $A_{ab}$ is the following: the projection of the horizontal 2-tensor $A$ into its first component gives, up to a scalar, precisely the complex scalar $\Psi_0$, i.e.
 \beaa
A(e_1, e_1)= \a(e_1, e_1)+ i \dual \a(e_1, e_1)=-\Psi_0.
 \eeaa
 Such relation also explains why the complex scalar $\Psi_0$ is of spin-2, as it can be realized as a projection of a 2-tensor.
 Similarly, in NP formalism the extreme electromagnetic component $\phi_0$ is the spin-1 complex scalar defined as
 \beaa
 \phi_0=-\F_{\mu\nu} l^\mu m^\nu =-\F(l, m).
 \eeaa
 In GKS formalism, the extreme electromagnetic component is a complex horizontal 1-tensor, as in
 \beaa
 \BF_a= \F(L, e_a) + i \dual \F(L, e_a)=\bF_a + i \dual \bF_a
 \eeaa
 and the projection of the horizontal 1-tensor $\BF$ into its first component gives, up to a scalar, precisely the complex scalar $\phi_0$, i.e.
 \beaa
\BF(e_1)= \bF(e_1)+ i \dual \bF(e_1)=-\phi_0.
 \eeaa
  The information about the spin of the complex scalars in NP formalism is encoded in the tensors in GKS formalism: a spin 2-scalar is substituted by a horizontal 2-tensor, and a spin 1-scalar by a horizontal 1-tensor.

For future reference, we collect here a table of conversion from NP and GKS formalism, where it is understood that the correspondence between the curvature, electromagnetic and Ricci components holds up to projection on the first component.
\begin{table}[h!]
\centering
 \begin{tabular}{||c c ||} 
 \hline
 NP formalism & GKS formalism \\ [0.5ex] 
 \hline\hline
        $D, \Delta$ & $\nab_4, \nab_3$  \\  
          $\delta, \ov{\delta}$ & $\DD \hot , \ov{\DD}\c $  \\ 
 \hline
$ \Psi_0, \Psi_4$ &  $A, \Ab$   \\ 
 $\Psi_1, \Psi_3$ & $B, \Bb$  \\
 $\Psi_2$ & $P$  \\ [1ex] 
 \hline
 $\phi_0$, $\phi_2$ & $\BF, \BBF$  \\
 $\phi_1$ & $\PF$  \\  [1ex] 
 \hline
 $\sigma, \lambda$ & $\Xh, \Xbh$  \\ 
 $\tau, \pi$ & $H, \Hb$  \\ 
   $\kappa, \nu$ & $\Xi, \Xib$  \\ 
     $\a, \b$ & $Z$  \\ 
       $\rho, \mu$ & $\tr X, \tr \Xb$  \\  
          $\ep, \gamma$ & $\om, \omb$  \\  [1ex] 
 \hline
 \end{tabular}
\end{table}

The conformal derivatives in GKS formalism are the equivalent of the spin and boost weight operators defined in GHP formalism. Just as in GHP formalism, the derivatives $\text{\thorn}$, $\text{\thorn}'$, $\text{\eth}$ , $\text{\eth}' $ absorb in their definitions the Ricci coefficients $\epsilon$, $\gamma$, $\alpha$ and $\beta$, similarly our $\nabc_3$, $\nabc_4$ and $\nabc_a$ absorb $\om$, $\omb$ and $Z$. 
\begin{table}[h!]
\centering
 \begin{tabular}{||c c ||} 
 \hline
 GHP formalism & GKS formalism \\ [0.5ex] 
 \hline\hline
        $\text{\thorn}, \text{\thorn}'$ & $\nabc_4, \nabc_3$  \\  
          $\text{\eth} , \text{\eth}' $ &   $\DDc \hot , \ov{\DDc}\c $\\ 
          \hline
 \end{tabular}
\end{table}

\section{The Einstein-Maxwell equations in full generality}\label{section-equations}

The Einstein-Maxwell equations are given by
\bea 
\R_{\mu\nu}&=&2 \F_{\mu \lambda} {\F_\nu}^{\lambda}- \frac 1 2 \g_{\mu\nu} \F^{\alpha\beta} \F_{\alpha\beta}, \label{Einstein-1}\label{Einstein-Maxwell-eq}\\
\D_{[\mu} \F_{\nu\lambda]}&=&0, \qquad \D^\mu \F_{\mu\nu}=0. \label{Maxwell}
\eea
where $\R_{\mu\nu}$ denotes the Ricci curvature of $(\MM, \g)$ and $\F_{\mu\nu}$ is an antisymmetric 2-tensor.
In this section we derive the null Einstein-Maxwell equations in full generality, for a spacetime with a non-integrable null frame, therefore paying particular attention to the symmetric and antisymmetric part of $\chi$ and $\chib$.

\subsection{The Maxwell equations}

The equation $\D_{[\mu} \F_{\nu\lambda]}=0$ in \eqref{Maxwell} gives three independent equations. 
The first one is obtained in the following way, using \eqref{conn-coeff}:
\beaa
0&=& \D_a \F_{3 4}+\D_3 \F_{4a}+\D_4 \F_{a3}\\
&=& \nab_a(\F_{3 4})-\F(\chib_{ab} e_b+\ze_a e_3 , e_4)-\F(e_3, -\ze_a e_4+\chi_{ab} e_b)\\
&+&\nab_3 (\F_{4a})- \F(2\omb e_4+2\eta_b e_b, e_a)- \F(e_4, \eta_a e_3+\xib_a e_4)\\
&+&\nab_4 (\F_{a3})- \F(\etab_a e_4+\xi_a e_3, e_3)- \F(e_a, 2\om e_3+2\etab_b e_b)\\
&=& 2\nab_a(\rhoF)-\chib_{ab}\bF_b+\chi_{ab}  \bbF_b-\nab_3 \bF_a+2\omb \bF_a- 2 \in_{ab} \eta_b \dual\rhoF+2 \eta_a \rhoF\\
&+&\nab_4 \bbF_a+2\etab_a \rhoF-2\om \bbF_a+ 2 \in_{ab} \etab_b \dual\rhoF
\eeaa
By writing 
\beaa
\chi_{ab}  \bbF_b&=& \left(\chih_{ab}+\frac 1 2 \de_{ab}\, \trch+\frac 12 \in_{ab}\atrch\right) \bbF_b= \chih_{ab} \bbF_b+\frac 1 2  \trch \bbF_a+\frac 12 \atrch \dual \bbF_a
\eeaa
and using the definition of Hodge duals, we obtain 
\bea\label{Maxwell-1}
\begin{split}
\nab_3 \bF_a-\nab_4 \bbF_a&= -\frac 1 2  \trchb \bF_a+2\omb \bF_a-\frac 12 \atrchb \dual \bF_a\\
&+\frac 1 2  \trch \bbF_a-2\om \bbF_a+\frac 12 \atrch \dual \bbF_a\\
&+2\nab_a(\rhoF)+2\left( \eta_a+\etab_a\right) \rhoF+ 2 \left(\dual \etab_a -\dual \eta_a\right)\dual\rhoF\\
&-\chibh_{ab} \bF_b+\chih_{ab} \bbF_b
\end{split}
\eea

The second equation is obtained in the following way:
\beaa
0&=& \D_a \F_{b 3}+\D_b \F_{3a}+\D_3 \F_{ab}\\
&=& \nab_a (\F_{b 3})-\F(\frac 1 2 \chi_{ab} e_3+\frac 1 2  \chib_{ab}e_4, e_3)- \F(e_b, \chib_{ac} e_c +\ze_ae_3)\\
&+&\nab_b (\F_{3a})- \F(\chib_{bc} e_c +\ze_b e_3, e_a)-\F(e_3, \frac 1 2 \chi_{ba} e_3+\frac 1 2  \chib_{ba}e_4)\\
&+&\nab_3 (\F_{ab})-\F(\eta_a e_3+\xib_a e_4, e_b) -\F(e_a, \eta_b e_3+\xib_b e_4)\\
&=& \nab_a \bbF_b-\nab_b \bbF_a+  \left(\chib_{ab}-\chib_{ba}\right) \rhoF+\chib_{ac}\in_{bc} \dual \rhoF - \ze_a \bbF_b+ \chib_{bc} \in_{ca}\dual \rhoF \\
&& + \ze_b \bbF_a-\in_{ab}\nab_3 \dual \rhoF +\eta_a \bbF_b+\xib_a \bbF_b -\eta_b \bbF_a-\xib_b \bF_a
\eeaa
Contracting the above with $\in^{ab}$ and recalling that $\curl \bbF= \in^{ab} \nab_a \bbF_b$, we have
\beaa
\nab_3 \dual \rhoF-\curl \bbF&=&   - \left(\trchb \dual \rhoF-\atrchb \rhoF\right)  +\left(\eta-\ze \right) \c  \dual \bbF  +\xib \c \dual \bF
\eeaa
The third equation is obtained from symmetrization of the above:
\beaa
\nab_4 \dual \rhoF-\curl \bF&=&   - \left(\trch \dual \rhoF+\atrch \rhoF\right)  +\left(\etab+\ze \right) \c  \dual \bF  +\xi \c \dual \bbF
\eeaa

The equation $\D^\mu \F_{\mu\nu}=\de_{bc} \D_b \F_{c\nu}-\frac 1 2 \D_4 \F_{3\nu}-\frac 1 2 \D_3\F_{4\nu}=0$ gives three additional independent equations. 
The first one is obtained in the following way:
\beaa
0&=& \de_{bc} \D_b \F_{ca}-\frac 1 2 \D_4 \F_{3a}-\frac 1 2 \D_3\F_{4a}\\
&=&\de_{bc} \left(\in_{ac}\nab_b \dual\rhoF-\F\left(\frac 1 2 \chib_{bc} e_4+\frac 1 2 \chi_{bc} e_3, e_a\right)-\F\left(e_c, \frac 1 2 \chib_{ba} e_4+\frac 1 2 \chi_{ba} e_3\right)\right)\\
&&+\frac 1 2 \nab_4 \bbF_a+\frac 1 2  \F(2\om e_3+2\etab_c e_c, e_a)+\frac 1 2  \F(e_3, \etab_a e_4+\xi_a e_3)\\
&&+\frac 1 2 \nab_3\bF_a+\frac 1 2 \F(2\omb e_4+2\eta_c e_c, e_a)+\frac 1 2 \F(e_4, \eta_a e_3+\xib_a e_4)\\
&=& \in_{ac}\nab_c \dual\rhoF+ \frac 1 2 \trchb \bF_a+ \frac 1 2 \trch\bbF_a-  \frac 1 2 \chib_{ca}\bF_c- \frac 1 2 \chi_{ca} \bbF_c \\
&&+\frac 1 2 \nab_4 \bbF_a-  \om \bbF_a- \etab_c \dual \rhoF \in_{ca}+\etab_a \rhoF+\frac 1 2 \nab_3\bF_a-\omb \bF_a-\eta_c \dual \rhoF \in_{ca} -\eta_a \rhoF
\eeaa
By writing 
\beaa
\chi_{ca}  \bbF_c&=& \left(\chih_{ca}+\frac 1 2 \de_{ca}\, \trch+\frac 12 \in_{ca}\atrch\right) \bbF_c= \chih_{ca} \bbF_c+\frac 1 2  \trch \bbF_a-\frac 12 \atrch \dual  \bbF_a
\eeaa
we obtain
\bea\label{Maxwell-2}
\begin{split}
\nab_4 \bbF_a+\nab_3\bF_a&=-  \frac 1 2 \trchb \bF_a-\frac 12 \atrchb \dual  \bF_a+2\omb \bF_a\\
&-  \frac 1 2 \trch\bbF_a-\frac 12 \atrch \dual  \bbF_a +  2\om \bbF_a\\
&-2\in_{ac}\nab_c \dual\rhoF - 2\left(\dual \eta_a+ \dual \etab_a \right)\dual \rhoF+2\left(\eta_a-\etab_a \right)\rhoF\\
& +  \chibh_{ca} \bF_c+  \chih_{ca} \bbF_c
\end{split}
\eea
Summing and subtracting \eqref{Maxwell-1} and \eqref{Maxwell-2} we obtain
\beaa
\nab_3 \bF-\nab(\rhoF)+\dual \nab \dual\rhoF&=& - \frac 1 2  \left(\trchb \bF+ \atrchb \dual \bF\right)+2\omb \bF\\
&&+2 \left(\eta \rhoF- \dual \eta \dual \rhoF\right) +\chih \c \bbF 
\eeaa
and 
\beaa
\nab_4 \bbF+\nab(\rhoF)+\dual \nab \dual\rhoF&= &  -\frac 1 2 \left(\trch \bbF+ \atrch \dual \bbF\right)+2\om \bbF\\
&& -2\left(\etab \rhoF+ \dual \etab \dual\rhoF\right)+\chibh \c \bF
\eeaa

The last equation is obtained by 
\beaa
0&=& \de_{bc} \D_b \F_{c4}-\frac 1 2 \D_4 \F_{34}\\
&=&\de_{bc}\left(\nab_b \bF_c-\F\left(\frac 1 2 \chib_{bc} e_4+\frac 1 2 \chi_{bc} e_3, e_4\right)-\F(e_c, -\ze_b e_4+\chi_{ba} e_a)\right)\\
&&-\frac 1 2 \left(2\nab_4 \rhoF-\F(2\om e_3+2\etab_a e_a, e_4)-\F(e_3, -2\om e_4+2\xi_a e_a)\right)\\
&=&\div \bF-  \trch \rhoF+ \ze \c \bF+ \atrch \dual \rhoF- \nab_4 \rhoF+\etab \c \bF-\xi \c \bbF
\eeaa
which gives
\bea\label{nabb-4-rhoF-general}
\begin{split}
 \nab_4 \rhoF-\div \bF&=-  \left(\trch \rhoF- \atrch \dual \rhoF\right)+\left( \ze+\etab\right) \c \bF-\xi \c \bbF \\
 \nab_3 \rhoF+\div \bbF&=  -\left(\trchb \rhoF+ \atrchb \dual \rhoF\right)+\left( \ze-\eta\right) \c \bbF+\xib \c \bF 
\end{split}
\eea

We summarize the Maxwell equations in the following Proposition.
\begin{proposition}\label{prop-Maxwell} We have, 
\beaa
\nab_3 \bF-\nab(\rhoF)+\dual \nab \dual\rhoF&=& - \frac 1 2  \left(\trchb \bF+ \atrchb \dual \bF\right)+2\omb \bF\\
&&+2 \left(\eta \rhoF- \dual \eta \dual \rhoF\right) +\chih \c \bbF\\
\nab_4 \bbF+\nab(\rhoF)+\dual \nab \dual\rhoF&=& - \frac 1 2  \left(\trch \bbF+ \atrch \dual \bbF\right)+2\om \bbF\\
&&+2 \left(-\etab \rhoF- \dual \etab \dual \rhoF\right) +\chibh \c \bF\\
 \nab_4 \rhoF-\div \bF&=&-  \left(\trch \rhoF- \atrch \dual \rhoF\right)+\left( \ze+\etab\right) \c \bF-\xi \c \bbF \\
  \nab_3 \rhoF+\div \bbF&=&- \left(\trchb \rhoF+ \atrchb \dual \rhoF\right)+\left( \ze-\eta\right) \c \bbF+\xib \c \bF \\
 \nab_4 \dual \rhoF-\curl \bF&=&   - \left(\trch \dual \rhoF+\atrch \rhoF\right)  +\left(\etab+\ze \right) \c  \dual \bF  +\xi \c \dual \bbF\\
  \nab_3 \dual \rhoF-\curl \bbF&=&   - \left(\trchb \dual \rhoF-\atrchb \rhoF\right)  +\left(\eta-\ze \right) \c  \dual \bbF  +\xib \c \dual \bF
\eeaa
In complex notations and using conformal derivatives we have
\beaa
\nabc_3 \BF- \DDc\PF&=&- \frac 1 2 \tr\Xb \BF+2\PF H +\frac 1 2 \Xh \c \ov{\BBF}\\
\nabc_4 \BBF+ \DDc\ov{\PF}&=&- \frac 1 2 \tr X \BBF-2\ov{\PF} \Hb +\frac 1 2 \Xbh \c \ov{\BF}\\
\nabc_4 \PF- \frac 1 2 \ov{\DDc} \c \BF&=&- \ov{\tr X} \PF +\frac 1 2 \ov{\Hb} \c \BF   -\frac 1 2 \Xi \c \ov{\BBF}\\
\nabc_3 \PF+ \frac 1 2 \DDc \c \ov{\BBF}&=&- \tr \Xb \PF -\frac 1 2 H \c \ov{\BBF}   +\frac 1 2 \ov{\Xib} \c \BF
\eeaa
\end{proposition}
\begin{proof} We derive the equation for $\BF$. From the above equation for $\bF$ and its dual, we have 
\beaa
\nab_3 \BF&=& \nab_3(\bF+ i \dual \bF)\\
&=& \nab\rhoF-\dual \nab \dual\rhoF+i\left(\dual \nab\rhoF+\nab \dual\rhoF\right) \\
&&- \frac 1 2  \left(\trchb \bF+ \atrchb \dual \bF\right) - \frac 1 2 i \left(\trchb \dual \bF- \atrchb  \bF\right)+2\omb \left(\bF+i\dual \bF\right)\\
&&+2 \left(\eta \rhoF- \dual \eta \dual \rhoF\right)+2 i\left(\dual \eta \rhoF+ \eta \dual \rhoF\right) +\chih \c \bbF +i\dual (\chih \c \bbF)
\eeaa
which gives
\beaa
\nab_3 \BF- \DD\PF&=&- \frac 1 2 \tr\Xb \BF+2\omb \BF+2\PF H +\frac 1 2 \Xh \c \ov{\BBF}
\eeaa
From the equations for $\rhoF$ and $\dual \rhoF$ we obtain
\beaa
\nab_4 \PF&=& \nab_4 (\rhoF + i \dual \rhoF)\\
&=& \div \bF+ i \curl \bF-  \left(\trch \rhoF- \atrch \dual \rhoF\right)  - i\left(\trch \dual \rhoF+\atrch\rhoF\right)  \\
&&+\left( \ze+\etab\right) \c \bF+i\left(\etab+\ze \right) \c  \dual \bF-\xi \c \bbF    +i\xi \c \dual \bbF
\eeaa
which gives
\beaa
\nab_4 \PF- \frac 1 2 \ov{\DD} \c \BF&=&- \ov{\tr X} \PF +\frac 1 2  \left(\ov{Z+\Hb} \right)\c \BF   -\frac 1 2 \Xi \c \ov{\BBF}
\eeaa
as desired. The other equations are obtained by symmetrization. Using the fact that $\BF$ is conformal invariant of type $1$, $\BBF$ is conformal of type $-1$ and $\PF$ is conformal of type $0$, we easily deduce the equations with conformal derivatives.
\end{proof}

\subsection{The Ricci identities}

We now compute the Ricci curvature $\R_{\mu\nu}$ of $(\MM, \g)$ in terms of the decomposition in frames according to the Einstein equation \eqref{Einstein-1}:
\beaa
\R_{a3}&=& 2 \F_{a \lambda} {\F_3}^{\lambda}=2\de_{bc}\F_{ab}\F_{3c} - \F_{a3}\F_{34}=2\dual \rhoF \dual \bbF_a - 2\rhoF\bbF_a, \\
\R_{a4}&=&2\dual \rhoF \dual \bF_a + 2\rhoF\bF_a, \\
\R_{33}&=&2 g^{\lambda\mu}\F_{3 \lambda} \F_{3\mu}=2 \de_{ab}\F_{3a} \F_{3b}= 2 \bbF \c \bbF , \\
\R_{44}&=& 2 \bF\cdot  \bF\\
\R_{34}&=&(\F_{34})^2+2\bF \c \bbF  +\left(-2\rhoF^2+2\dual \rhoF^2-2 \bF \c \bbF \right)\\
&=&2\rhoF^2 +2\dual \rhoF^2 \\
\R_{ab}&=& - \F_{a 3} \F_{b4}- \F_{a 4} \F_{b3}+2 \de_{cd} \F_{a c} \F_{bd} - \frac 1 2 \de_{ab} \left( -2\rhoF^2+2\dual \rhoF^2-2 \bF \c \bbF \right)\\
&=& - \bbF_{a} \bF_{b}- \bF_{a} \bbF_{b}+2  \dual\rhoF^2\in_{ac} \in_{bc}- \frac 1 2 \de_{ab} \left( -2\rhoF^2+2\dual \rhoF^2-2 \bF \c \bbF \right)\\
&=& - 2(\bbF \hot \bbF)_{ab}+\left( \rhoF^2+\dual \rhoF^2 \right) \de_{ab} 
\eeaa

Using the decomposition of the Riemann curvature in Weyl curvature and Ricci tensor:
\bea
\label{WeylRiemanngeneral}
\R_{\a\b\gamma\delta}=\W_{\a\b\gamma\delta}+\frac 1 2 (g_{\b\delta}\R_{\a\gamma}+g_{\a\gamma}\R_{\b\delta}-g_{\b\gamma}\R_{\a\delta}-g_{\a\delta}\R_{\b\gamma}),
\eea
we compute the components of the Riemann tensor:
\beaa
\R_{a33b}&=&\W_{a33b}-\frac 1 2 \de_{ab} \R_{33}=-\aa_{ab}- \left( \bbF\cdot  \bbF \right) \de_{ab}, \\
\R_{a34b}&=&\W_{a34b}+\R_{ab}-\frac 1 2 \de_{ab}\R_{34}=\rho \de_{ab} -\dual \rho \in_{ab}-2(\bF \hot \bbF)_{ab}, \\
 \R_{a334}&=& \W_{a334}- \R_{a3}=2\bb_a-2\dual \rhoF \dual \bbF_a + 2\rhoF\bbF_a, \\
\R_{3434}&=&\W_{3434}+2\R_{34}=4\rho + 4\rhoF^2 +4\dual\rhoF^2, \\
\R_{a3cb}&=&\W_{a3cb}+\frac 1 2 (\de_{ac} \R_{3b}-\de_{ab} \R_{3c})\\
&=&\in_{cb} \dual \bb_a+ \de_{ac} (\dual \rhoF \dual \bbF_b - \rhoF\bbF_b)-\de_{ab} (\dual \rhoF \dual \bbF_c - \rhoF\bbF_c), 
\eeaa

The Ricci identities are obtained from the definition of Riemann curvature and are given by, see \cite{GKS}:
\beaa
\nab_3 \chib_{ba}&=&2\nab_b\xib_a-2\omb\, \chib_{ba}-\chib_{bc}\chib_{ca} 
+ 2\big(- 2\ze_b\xib_a+\eta_b\xib_a+ \etab_a\xib_b\big)+\R_{b33a},\\
\nab_3\chi_{ba}&=&2 \nab_b\eta_a+ 2 \omb\chi_{ba}-\chib_{bc}\chi_{ca}+2(\xib_b\xi_a+\eta_a\, \eta_b)+\R_{a43b},\\
 \nab_4 \chib_{ba}&=& 2 \nab_b\etab_a+2 \om\chib_{ba}-\chi_{bc}\chib_{ca}
 +2(\xi_b\xib_a+\etab_a\, \etab_b)+\R_{a34b},\\
\nab_4\chi_{ba}&=&2\nab_b\xi_a- 2 \om\,\chi_{ba}-\chi_{bc}\chi_{ca}
+ 2\big( 2\ze_b\xi_a+\etab_b\xi_a+ \eta_a\xi_b\big)+\R_{b44a}.\\
\nab_3 \ze_a+2\nab_a\omb&=&-\chib_{ab}(\ze_b+\eta_b)+ 2 \omb(\ze-\eta)_a+\chi_{ab}\xib_b +2 \om \xib_a -\frac 1 2 \R_{a334}.
\\
\nab_4 \ze_a-2\nab_a\om&=&\chi_{ab}(-\ze_b+\etab_b)+2 \om(\ze+\etab)_a-\chib_{ab}\xi_b -2 \omb \xi_a +\frac 1 2 \R_{a443},
\\
\nab_3 \etab_a -\nab_4\xib_a &=&-  \chib_{ba}(\etab-\eta)_b -4 \om \xib_a  +\frac 1 2 \R_{a334}, \\
\nab_4 \eta_a    -    \nab_3\xi_a &=&-       \chi_{ba}(\eta-\etab)_b-4\omb \xi_a +\frac 1 2 \R_{a443},\\
\nab_3\om+\nab_4\omb &=&4\om\omb +\xi\c \xib +(\eta-\etab)\c\ze -\eta\c\etab+   \frac 1 4 \R_{3434}, \\
\nab_a\chi_{bc}+\ze_a\chi_{bc}&=&\nab_b\chi_{ac}+\ze_b\chi_{ac} +(\chi_{ab}-\chi_{ba})\eta_c+(\chib_{ab}-\chib_{ba})\xi_c+\R_{b4ac},\\
\nab_a\chib_{bc}-\ze_a\chib_{bc}&=&\nab_b\chib_{ac}-\ze_b\chib_{ac} +(\chib_{ab}-\chib_{ba})\etab_c+(\chi_{ab}-\chi_{ba})\xib_c+\R_{b3ac}, 
\eeaa

We summarize the result of their complexification, with the above values of the Riemann curvature in the following.

\begin{proposition} 
\label{prop-nullstr:complex} In complex notations and using the conformal derivatives we have the following Ricci identities:
\beaa
\nabc_3\tr\Xb +\frac{1}{2}(\tr\Xb)^2 &=& \DDc\c\ov{\Xib}+\Xib\c\ov{\Hb}+\ov{\Xib}\c H-\frac{1}{2}\Xbh\c\ov{\Xbh}-\BBF \c \ov{\BBF},\\
\nabc_4\tr X +\frac{1}{2}(\tr X)^2 &=& \DDc\c\ov{\Xi}+\Xi\c\ov{H}+\ov{\Xi}\c H-\frac{1}{2}\Xh\c\ov{\Xh}-\BF \c \ov{\BF},\\
\nabc_3\tr X +\frac{1}{2}\tr\Xb\tr X &=& \DDc\c\ov{H}+H\c\ov{H}+2P+\Xib\c\ov{\Xi}-\frac{1}{2}\Xbh\c\ov{\Xh},\\
\nabc_4\tr\Xb +\frac{1}{2}\tr X\tr\Xb  &=& \DDc\c\ov{\Hb}+\Hb\c\ov{\Hb}+2\ov{P}+\Xi\c\ov{\Xib}-\frac{1}{2}\Xh\c\ov{\Xbh},
\eeaa
\beaa
\nabc_3\Xbh+\Re(\tr\Xb) \Xbh&=& \DDc\hot \Xib+   \Xib\hot(H+\Hb)-\Ab,\\
\nabc_4\Xh+\Re(\tr X)\Xh&=& \DDc\hot \Xi+  \Xi\hot(\Hb+H)-A\\
\nabc_3\widehat{X} +\frac{1}{2}\tr\Xb\, \widehat{X}&=& \DDc\hot H  +H\hot H -\frac{1}{2}\ov{\tr X} \widehat{\Xb}+\frac{1}{2}\Xib\hot\Xi-\frac 1 2 \BF \hot \BBF,\\
\nabc_4\widehat{\Xb} +\frac{1}{2}\tr X\, \widehat{\Xb} &=& \DDc\hot\Hb  +\Hb\hot\Hb -\frac{1}{2}\ov{\tr\Xb} \widehat{X}+\frac{1}{2}\Xi\hot\Xib-\frac 1 2 \BF \hot \BBF,
\eeaa
\beaa
\nabc_3\Hb -\nabc_4\Xib &=&  -\frac{1}{2}\ov{\tr\Xb}(\Hb-H) -\frac{1}{2}\Xbh\c(\ov{\Hb}-\ov{H}) +\Bb+\PF\BBF,\\
\nabc_4H -\nabc_3\Xi &=&  -\frac{1}{2}\ov{\tr X}(H-\Hb) -\frac{1}{2}\Xh\c(\ov{H}-\ov{\Hb}) -B-\ov{\PF} \BF,\\
\frac{1}{2}\ov{\DDc}\c\Xh &=& \frac{1}{2}\DDc\ov{\tr X}-i\Im(\tr X)H-i\Im(\tr \Xb)\Xi-B+\ov{\PF}\BF,\\
\frac{1}{2}\ov{\DDc}\c\Xbh &=& \frac{1}{2}\DDc\ov{\tr\Xb}-i\Im(\tr\Xb)\Hb-i\Im(\tr X)\Xib+\Bb-\PF \BBF.
\eeaa
Also, for the non-conformal $Z$, $\om$ and $\omb$:
\beaa
\nab_3Z +\frac{1}{2}\tr\Xb(Z+H)-2\omb(Z-H) &=& -2\DD\omb -\frac{1}{2}\widehat{\Xb}\c(\ov{Z}+\ov{H})+\frac{1}{2}\tr X\Xib+2\om\Xib \\
&&-\Bb-\PF \BBF+\frac{1}{2}\ov{\Xib}\c\Xh,\\
\nab_4Z +\frac{1}{2}\tr X(Z-\Hb)-2\om(Z+\Hb) &=& 2\DD\om +\frac{1}{2}\widehat{X}\c(-\ov{Z}+\ov{\Hb})-\frac{1}{2}\tr\Xb\Xi-2\omb\Xi \\
&&-B-\ov{\PF}\BF-\frac{1}{2}\ov{\Xi}\c\Xbh,\\
\nab_3\om+\nab_4\omb -4\om\omb -\xi\c \xib -(\eta-\etab)\c\ze +\eta\c\etab&=&   \rho+\rhoF^2+\dual\rhoF^2
\eeaa
\end{proposition}

Note that we are missing the traditional Gauss equation which, in the integrable case, connects the Gauss curvature of a sphere to a Riemann curvature component.  We now state the non-integrable analogue of the Gauss equation, see \cite{GKS}.

\begin{proposition}
The following identities hold true for $f\in \sk_1$ and $u\in\sk_2$:
\beaa
\big( \nab_a \nab_b- \nab_b \nab_a\big)  f_c&=&\frac 1 2 \in_{ab}(\atrch\nab_3+\atrchb \nab_4) f_c  -\frac  12 E_{cdab} f^d+ \R_{c d   ab}f^d\\
\big( \nab_a \nab_b- \nab_b \nab_a\big)  u_{st} &=&\frac 1 2 \in_{ab}(\atrch\nab_3+\atrchb \nab_4) u_{st}
 -\frac  12 E_{sdab} u_{dt}    -\frac 1 2 E_{tdab} u_{sd} \\
    &+& \R_{s d   ab}u_{dt}+\R_{td ab} u_{sd}
\eeaa
where
\bea\label{definition-E}
E_{cdab}:&=&  \chi_{ac}\chib_{bd} + \chib_{ac}\chi_{bd} - \chi_{bc}\chib_{ad}- \chib_{bc}\chi_{ad} 
\eea
\end{proposition}
 \begin{proof} See Proposition 2.34 in \cite{GKS}. 
 \end{proof}

\subsection{The Bianchi identities}

The Bianchi identities for the Weyl curvature are given by 
\bea\label{J-def}
 \D^\a \W_{\a\b\gamma\delta}&=&\frac 1 2 (\D_\gamma \R_{\b\delta}-\D_{\delta}\R_{\b\gamma})=:J_{\b\gamma\delta} 
\eea
The null Bianchi identities are given by, see \cite{GKS},
    \beaa
    \nab_3\a-2  \nab\hot \b&=&-\frac 1 2 \big(\trchb\a+\atrchb\dual \a)+4\omb \a+
 2 (\ze+4\eta)\hot \b - 3 (\rho\chih +\rhod\dual\chih)+\mathfrak{a},\\
     \nab_4\aa+2 \nab\hot \bb&=&-\frac 1 2 \big(\trch\aa-\atrch\dual \aa)+4\om \aa+2
 (\ze-4\etab)\hot \bb - 3  (\rho\chibh -\rhod\dual\chibh)+\underline{\mathfrak{a}}.
 \eeaa
 where
 \beaa
\mathfrak{a}_{ab}&=& J_{ba4}+J_{ab4}-\frac 1 2 \de_{ab}J_{4 34}, \qquad \underline{\mathfrak{a}}_{ab}= J_{ba3}+J_{ab3}-\frac 1 2 \de_{ab}J_{343} 
\eeaa
We also have
 \beaa
\nab_4\beta - \div\a &=&-2(\trch\beta-\atrch \dual \b) - 2  \om\b +\a\c  (2 \ze +\etab) + 3  (\xi\rho+\dual \xi\rhod)-J_{4a4},\\
    \nab_3\bb +\div\aa &=&-2(\trchb\,\bb-\atrchb \dual \bb)- 2  \omb\bb-\aa\c(-2\ze+\eta) - 3  (\xib\rho-\dual \xib \rhod)+J_{3a3},\\
     \nab_3\b+\div\varrho&=&-(\trchb \b+\atrchb \dual \b)+2 \omb\,\b+2\bb\c \chih+3 (\rho\eta+\rhod\dual \eta)+    \a\c\xib+J_{3a4},\\
          \nab_4\bb-\div\varoc&=&-(\trch \bb+ \atrch \dual \bb)+ 2\om\,\bb+2\b\c \chibh
    -3 (\rho\etab-\rhod\dual \etab)-    \aa\c\xi-J_{4a3}
    \eeaa
    where
\beaa
\div\varo&=&- (\nab\rho+\dual\nab\rhod), \qquad \div\varoc=- (\nab\rho-\dual\nab\rhod).
\eeaa
Finally,
    \beaa
 \nab_4 \rho-\div \b&=&-\frac 3 2 (\trch \rho+\atrch \rhod)+(2\etab+\ze)\c\b-2\xi\c\bb-\frac 1 2 \chibh \c\a-\frac 1 2 J_{434},\\
   \nab_4 \rhod+\curl\b&=&-\frac 3 2 (\trch \rhod-\atrch \rho)-(2\etab+\ze)\c\dual \b-2\xi\c\dual \bb+\frac 1 2 \chibh \c\dual \a-\frac 1 2\dual J_{434} , \\
     \nab_3 \rho+\div\bb&=&-\frac 3 2 (\trchb \rho -\atrchb \rhod) -(2\eta-\ze) \c\bb+2\xib\c\b-\frac{1}{2}\chih\c\aa-\frac 1 2 J_{343},
 \\
   \nab_3 \rhod+\curl\bb&=&-\frac 3 2 (\trchb \rhod+\atrchb \rho)- (2\eta-\ze) \c\dual \bb-2\xib\c\dual\b-\frac 1 2 \chih\c\dual \aa+\frac 1 2 \dual J_{343}
   \eeaa

We compute the $J$s terms through the Ricci curvature, and then use the Maxwell equations to simplify them. We summarize the final Bianchi identities in the following, and defer the proof to the Appendix.

\begin{proposition}\label{prop:bianchi:complex} 
    In complex notations and using the conformal derivatives we have the following Bianchi identities:
    \beaa
 \nabc_3A+\frac{1}{2}\tr\Xb A  &=&\DDc\hot B  +4 H\hot B -3\ov{P}\Xh-2\ov{\PF}\left( -\frac 1 2 \DDc \hot \BF+\PF\Xh  \right)\\
 &&+\frac 1 2\nab_4(\BF \hot \BBF)+\frac 1 2 \nabc_3(\BF \hot \BF)\\
 &&+ \left(- \frac 1 2 \tr X \BBF +\frac 1 2 \Xbh \c \ov{\BF}+\frac 1 2  \Xh \c \BBF+\PF \Xi \right)\hot \BF\\
  \nabc_4\Ab+\frac{1}{2}\tr X \Ab  &=&-\DDc\hot \Bb  -4 \Hb\hot \Bb -3P\Xbh+2\PF \left( -\frac 1 2 \DDc \hot \BBF -\ov{\PF}\Xbh \right)\\
&&+\frac 1 2\nab_3(\BBF \hot \BF)+\frac 1 2 \nabc_4(\BBF \hot \BBF)\\
 &&+\left(- \frac 1 2 \tr \Xb \BF +\frac 1 2 \Xh \c \ov{\BBF}+\frac 1 2  \Xbh \c \BF-\PF \Xib \right)\hot \BBF
\eeaa
We also have
 \beaa
\nabc_4B -\frac{1}{2}\ov{\DDc}\c A &=& -2\ov{\tr X} B  +\frac{1}{2}A\c  \ov{\Hb}+\left( 3\ov{P} -2\ov{\PF}\PF\right)\,\Xi+ \ov{\PF} \nabc_4(\BF)\\
&&+\frac 1 2  \DDc(\BF\c  \ov{\BF})  ,\\
\nabc_3\Bb +\frac{1}{2}\ov{\DDc}\c\AA &=& -2\ov{\tr\Xb}\,\Bb  -\frac{1}{2}\Ab\c \ov{H}-\left(3P-2\ov{\PF}\PF\right) \,\Xib+ \PF \nabc_3(\BBF)\\
&&+\frac 1 2  \DDc(\BBF\c  \ov{\BBF})  
\eeaa
and
\beaa
\nabc_3B-\DDc\ov{P} &=& -\tr\Xb B+\ov{\Bb}\c \Xh+3\ov{P}H +\frac{1}{2}A\c\ov{\Xib}\\
&&+\ov{\PF}\DDc(\PF) -\frac 1 2 \tr \Xb \ov{\PF} \BF -\ov{\tr X}\PF \BBF\\
&& +\frac 1 2( \ov{\DDc }\c \BF)\BBF-\PF \Xh \c \BBF-\frac 1 2 \ov{\PF} \Xbh \c \BF,\\
\nabc_4\Bb+\DDc P &=& -\tr X\Bb+\ov{B}\c \Xbh-3P\Hb -\frac{1}{2}\Ab\c\ov{\Xi}\\
&&-\PF\DDc(\ov{\PF}) -\frac 1 2 \tr X \PF\BBF -\ov{\tr \Xb}\ov{\PF} \BF\\
&&+\frac 1 2( \ov{\DDc }\c \BBF)\BF+\PF \Xbh \c \BF+\frac 1 2 \ov{\PF} \Xh \c \BBF
\eeaa
Finally,
\beaa
\nabc_4P -\frac{1}{2}\DDc\c \ov{B} &=& -\frac{3}{2}\tr X P -\tr X \PF \ov{\PF}+\Hb\c\ov{B} -\ov{\Xi}\c\Bb -\frac{1}{4}\Xbh\c \ov{A}\\
&&+ \frac 1 2 \PF \DDc \c \ov{\BF} +   \ov{H}  \c  \ov{\PF}\BF \\
&&+ \nabc_3(\BF \hot \BF)+ (-  \tr \Xb \BF - \frac 1 2 \Xh \c \BBF)\hot \BF \\
\nabc_3P +\frac{1}{2}\ov{\DDc}\c\Bb &=& -\frac{3}{2}\ov{\tr\Xb} P -\ov{\tr \Xb} \PF \ov{\PF} -\ov{H}\c\Bb +\Xib\c \ov{B} -\frac{1}{4}\ov{\Xh}\c\Ab\\
&&- \frac 1 2 \PF \ov{\DDc} \c \BF -   \Hb  \c  \ov{\PF}\BBF \\
&&+ \nabc_4(\BBF \hot \BBF)+ (-  \tr X \BBF - \frac 1 2 \Xbh \c \BF)\hot \BBF
\eeaa
\end{proposition}
\begin{proof} See Appendix \ref{proof-compl=Bianchi}. 
\end{proof}

    \section{The Kerr-Newman spacetime and its linear perturbations}\label{kn-section}
    
    In this section we introduce the Kerr-Newman spacetime and its representation within the formalism above introduced. For a more complete description of the Kerr-Newman spacetime see \cite{Exact}. 


\subsection{The Kerr-Newman metric}\label{sec:values-KN}

The Kerr-Newman black hole $\g_{M, a, Q}$ represents the most general explicit solution of a stationary, rotating (with spin $a$) and charged (with charge $Q$) black hole of mass $M$.
We consider the Kerr-Newman metric in standard Boyer-Lindquist coordinates $(t, r, \th, \vphi)$:
$$\g_{M, a, Q}=-\frac{\Delta}{|q|^2}\left( dt- a \sin^2\th d\vphi\right)^2+\frac{|q|^2}{\Delta}dr^2+|q|^2 d\th^2+\frac{\sin^2\th}{|q|^2}\left(a dt-(r^2+a^2) d\vphi \right)^2,$$
where 
\bea\label{definition-q}
q=r+ i a \cos\th, \qquad \ov{q}=r- i a \cos\th
\eea
and
\beaa
\Delta &=& r^2-2Mr+a^2+Q^2,\\
|q|^2 &=& r^2+a^2(\cos\theta)^2.
\eeaa

The metric $\g_{M, a, Q}$ is a solution to the Einstein-Maxwell equations \eqref{Einstein-1} and \eqref{Maxwell}, with electromagnetic tensor $\F=d\A$, and vector potential $\A$ given by
\beaa
\A&=& -\frac{Qr}{|q|^2} \left( dt - a \sin^2\th d\phi\right).
\eeaa
 We note that $\partial_t$ and $\partial_\vphi$ are Killing vectorfields of the metric.

The Kerr-Newman metric is of Petrov Type D, i.e. its Weyl curvature can be diagonalized with two linearly independent eigenvectors, the so-called principal null directions.
The principal null frame is given\footnote{There is an indeterminacy in the principal null frame as one may replace the pair $(e_3, e_4)$ with $(\la^{-1}e_3, \la e_4)$ for any $\lambda>0$. The formulas in this section correspond to the choice of $\lambda$ such that $e_4$ is geodesic. } by
\beaa
&& e_4=\frac{r^2+a^2}{\Delta}\pr_t+\pr_r+\frac{a}{\Delta}\pr_\varphi,\qquad e_3=\frac{r^2+a^2}{|q|^2}\pr_t-\frac{\Delta}{|q|^2}\pr_r+\frac{a}{|q|^2}\pr_\varphi,\\ && e_1=\frac{1}{\sqrt{|q|^2}}\pr_\th,\qquad e_2=\frac{a\sin\th}{\sqrt{|q|^2}}\pr_t+\frac{1}{\sqrt{|q|^2}\sin\th}\pr_\varphi.
\eeaa
With respect to the principal null frame, we have 
\beaa
&&\chih=\chibh=\xi=\xib=0, \qquad \bF= \bbF=0, \qquad \a=\b=\bb=\aa=0
\eeaa
or their complexified versions\footnote{In NP formalism, this corresponds to the vanishing of $\sigma=\lambda=\kappa=\nu=0$, $\Phi_0=\Phi_1=\Phi_3=\Phi_4=0$, and $\phi_0=\phi_2=0$.},
\bea\label{eq:vanishing}
&&\Xh=\Xbh=\Xi=\Xib=0, \qquad \BF=\BBF=0, \qquad A=B=\Bb=\Ab=0.
\eea
With the above choice of principal null frame, the Ricci coefficients are given by
\beaa
&& \trch=\frac{2r}{|q|^2},\quad \atrch=\frac{2a\cos\th}{|q|^2}
, \qquad  \trchb=-\frac{2r\Delta}{|q|^4}, \quad \atrchb=\frac{2a\Delta\cos\th}{|q|^4}\\
&& \omb = \frac{a^2\cos^2\th(r-M)+Mr^2-a^2r-Q^2r}{|q|^4}, \qquad \om=0, \qquad \etab=-\ze
\eeaa
Also, we have 
\beaa
\eta_1&=& -\frac{a^2\sin\th \cos\th}{|q|^3}, \qquad \eta_2=\frac{a\sin\th r}{|q|^3}, \\
\dual \eta_1&=& \frac{a\sin\th r}{|q|^3}, \qquad \dual \eta_2= \frac{a^2\sin\th \cos\th}{|q|^3},\\
\etab_1&=& -\frac{a^2\sin\th \cos\th }{|q|^3}, \qquad \etab_2 =-\frac{a\sin\th r}{|q|^3},\\
\dual \etab_1&=& -\frac{a\sin\th r}{|q|^3}, \qquad \dual \etab_2 =\frac{a^2\sin\th \cos\th}{|q|^3}.
\eeaa
Their complexified values are given by
\beaa
&&  \tr X=\frac{2}{q}, \qquad \tr\Xb=-\frac{2\Delta  }{q \ov{q}^2}, \qquad  \Hb=-Z, \\
&&H_1=\frac{ai\sin\th\, q}{|q|^3}, \qquad H_2=\frac{a\sin\th \, q}{|q|^3}, \\
&& Z_1=\frac{ai\sin\th\, \ov{q}}{|q|^3},\qquad\,\, Z_2=\frac{a\sin\th\,\ov{q}}{|q|^3}.
\eeaa
The non-vanishing electromagnetic components are given by
\beaa
&&\rhoF=\frac{Q(r^2- a^2\cos^2\th)}{|q|^4}, \qquad \dual\rhoF=\frac{2 a Q r \cos\th}{|q|^4}
\eeaa
with complexified value
\beaa
\PF=\frac{Q}{\ov{q}^2}.
\eeaa
The non-vanishing curvature components are given by
\beaa
\rho&=& \frac{1}{|q|^6} (-2Mr^3+2Q^2r^2+6M  a^2 \cos^2\th r-2Q^2a^2\cos^2\th),\\
 \dual \rho&=& \frac{a\cos\th }{|q|^6} (6Mr^2-4Q^2 r- 2Ma^2 \cos^2\th)
\eeaa
with complexified value
\beaa
&& P=-\frac{2M}{q^3}+\frac{2Q^2}{q^3 \ov{q}}. \nonumber
\eeaa

\subsection{The Einstein-Maxwell equations in Kerr-Newman}

 Using the vanishing of the Ricci, curvature and electromagnetic components given by \eqref{eq:vanishing},  one can see that many of the  Einstein-Maxwell equations obtained in Section \ref{section-equations} become trivial in Kerr-Newman. We denote those which are not trivially satisfied as reduced equations, and we collect them in the following. 

 \begin{proposition}\label{prop:reduced-equations}The reduced Maxwell equations in Kerr-Newman are
  \bea
\nabc_4 \PF&=&- \ov{\tr X} \PF , \qquad \nabc_3 \PF=- \tr \Xb \PF\label{nabc-4-PF-red}  \label{nabc-3-PF-red} \\
\DDc\PF&=&-2\PF H, \qquad \DDc\ov{\PF}=-2\ov{\PF} \Hb. \label{DDc-PF-red} \label{DDc-ov-PF-red}
\eea
The reduced Ricci identities in Kerr-Newman are
\bea
 \nabc_3\tr\Xb +\frac{1}{2}(\tr\Xb)^2 &=&0, \qquad \nabc_4\tr X +\frac{1}{2}(\tr X)^2 =0 \label{nabc-3-trXb-red} \\
\nabc_3\tr X +\frac{1}{2}\tr\Xb\tr X &=& \DDc\c\ov{H}+H\c\ov{H}+2P,\\
\nabc_4\tr\Xb +\frac{1}{2}\tr X\tr\Xb  &=& \DDc\c\ov{\Hb}+\Hb\c\ov{\Hb}+2\ov{P}, \label{nabc-4-trXb-red}\\
 \DDc\hot H  +H\hot H&=&0, \qquad  \DDc\hot\Hb  +\Hb\hot\Hb= 0\\
\nabc_3\Hb+\frac{1}{2}\ov{\tr\Xb}(\Hb-H) &=& 0 ,\qquad \nabc_4H+\frac{1}{2}\ov{\tr X}(H-\Hb) =  0 ,\label{nabc-3-Hb-red}\label{nabc-4-H-red}\\
\DDc\ov{\tr X} -(\tr X-\ov{\tr X})H&=&0 , \qquad \DDc\ov{\tr\Xb} -(\tr\Xb-\ov{\tr \Xb})\Hb= 0.\label{Codazzi-1-red}\label{Codazzi-2-red}
\eea
 The reduced Bianchi identities in Kerr-Newman are
\bea
\DDc\ov{P} &=& -\left(3\ov{P} -2\PF\ov{\PF} \right)H ,\qquad \DDc P =-\left(3P -2\PF\ov{\PF} \right)\Hb \label{DDc-ov-P-red}\label{DDc-P-red}\\
\nabc_4P  &=& -\frac{3}{2}\tr X P -\tr X \PF \ov{\PF}, \qquad \nabc_3P  = -\frac{3}{2}\ov{\tr\Xb} P -\ov{\tr \Xb} \PF \ov{\PF} \label{nabc-4-P-red}\label{nabc-3-P-red}
\eea
 \end{proposition}

From the above we deduce (see also \cite{GKS}) the following identities for $q=r+ia \cos\th$:
\bea\label{transport-for-q}
\begin{split}
\nab_3 q&= \frac 1 2 \ov{\tr \Xb} \,  q , \qquad \nab_4 q= \frac 1 2 \tr X q \qquad   \DDov q =  q\ov{H} , \qquad   \DD q=  q  \Hb
   \end{split}
\eea

 \subsection{The wave operators in Kerr-Newman spacetime}
 
 In what follows, we will need to express the equations governing electromagnetic-gravitational perturbations of Kerr-Newman in terms of wave operators applied to $k$-horizontal tensor fields. In this section, we collect useful formulas for those operators.

  Consider the wave operator for  $\Psi \in \sk_k(\mathbb{C})$ defined as
 \bea\label{eq:definition-squared}
 \squared_k\Psi_{ab}:= \g^{\mu\nu} \Ddot_\mu\Ddot_ \nu \Psi_{ab}.
\eea
where $\Ddot$ is the horizontal covariant derivative as defined in \eqref{eq:def-horizontal-cov-der}.

\begin{lemma}\label{lemma:expression-wave-operator}
The wave operator for $\Psi\in \sk_k(\mathbb{C})$ is given by
\bea\label{eq:general-wave}
\begin{split}
\squared_k \Psi&=-\frac 1 2 \big(\nab_3\nab_4\Psi+\nab_4 \nab_3 \Psi\big)+\lap_k \Psi +\left(\omb -\frac 1 2 \trchb\right) \nab_4\Psi+\left(\om -\frac 1 2 \trch\right) \nab_3\Psi + (\eta+\etab) \c\nab \Psi,
\end{split}
\eea
where $\lap_k=\de^{ab} \nab_a\nab_b$ is the Laplacian operator for horizontal $k$-tensors.
More precisely, for $F \in \sk_1(\CCC)$ we have
 \beaa
 \squared_1 F&=&- \nab_3\nab_4F+\lap_1 F +\left(2\omb -\frac 1 2 \trchb\right) \nab_4F -\frac 1 2 \trch \nab_3F + 2\eta \c\nab F +i \left(- \rhod+ \eta \wedge \etab  \right) F\\
 &=&- \nab_3\nab_4F+\frac 1 2    \DDb \c ( \DD \hot F) +\left(2\omb -\frac 1 2\ov{\tr\Xb} \right) \nab_4F -\frac 1 2 \ov{\tr X} \nab_3F+ (H+ \ov{H}) \c\nab F \\
 && +  \left(  \frac 14 \trch\trchb+\frac 1 4 \atrch\atrchb+ \rho-\rhoF^2-\dual\rhoF^2+i \left(- \rhod+ \eta \wedge \etab  \right)\right)  F
 \eeaa
and for $U \in \sk_2(\CCC)$ we have
 \beaa
 \squared_2 U&=&- \nab_3\nab_4U+\lap_2 U +\left(2\omb -\frac 1 2 \trchb\right) \nab_4U -\frac 1 2 \trch \nab_3U + 2\eta \c\nab U +i \left(- 2\rhod+2 \eta \wedge \etab  \right) U\\
 &=&- \nab_3\nab_4U+\frac 1 2  \DD \hot (\DDb \c U) +\left(2\omb -\frac 1 2\tr\Xb \right) \nab_4U -\frac 1 2\tr X \nab_3U+ (H+\ov{H}) \c\nab U\\
 &&+ \left(-\frac 1 2  \trch\trchb- \frac 1 2 \atrch\atrchb-2\rho+2\rhoF^2+2\dual\rhoF^2+i \left(- 2\rhod+2 \eta \wedge \etab  \right)\right)  U
 \eeaa
\end{lemma}
\begin{proof}
For the proof of \eqref{eq:general-wave} see Lemma 5.4. in \cite{GKS}. 

For $F \in \sk_1(\CCC)$ and $U \in \sk_2(\CCC)$, using the commutators, see Lemma 5.2 in \cite{GKS}:
 \beaa
  \, [\nab_3, \nab_4] F  &=& -2 \om \nab_3 F +2\omb \nab_4 F+ 2(\eta-\etab )\c \nab F +2i \left(- \rhod+ \eta \wedge \etab  \right) F\\
  \,[\nab_3, \nab_4]U &=& - 2\om \nab_3 U+ 2\omb \nab_4 U  + 2 (\eta-\etab) \c \nab U +4i \left(- \rhod+ \eta \wedge \etab  \right) U
 \eeaa
 we obtain 
 \beaa
  \squared_1 F&=&- \nab_3\nab_4F+\lap_1 F +\left(2\omb -\frac 1 2 \trchb\right) \nab_4F -\frac 1 2 \trch \nab_3F + 2\eta \c\nab F +i \left(- \rhod+ \eta \wedge \etab  \right) F\\
 \squared_2 U&=&- \nab_3\nab_4U+\lap_2 U +\left(2\omb -\frac 1 2 \trchb\right) \nab_4U -\frac 1 2 \trch \nab_3U + 2\eta \c\nab U +i \left(- 2\rhod+2 \eta \wedge \etab  \right) U
\eeaa
Using the following Gauss relations, see Lemma \ref{Laplacian-1}:
\beaa
   \DDb \c ( \DD \hot F)&=& 2\lap_1  F + i (\atrch\nab_3+\atrchb \nab_4) F\\
   && - \left(  \frac 12 \trch\trchb+\frac 1 2 \atrch\atrchb+ 2\rho-2\rhoF^2-2\dual\rhoF^2\right) F\\
\big(\DD\hot( \DDb \c U)\big)&=& 2\lap_2  U - i (\atrch\nab_3+\atrchb \nab_4) U\\
&& +   \left( \trch\trchb+ \atrch\atrchb+4\rho-4\rhoF^2-4\dual\rhoF^2\right) U
\eeaa
we obtain
\beaa
  \squared_1 F&=&- \nab_3\nab_4F+\frac 1 2    \DDb \c ( \DD \hot F) +\left(2\omb -\frac 1 2( \trchb+ i \atrchb) \right) \nab_4F -\frac 1 2 (\trch+i\atrch) \nab_3F\\
  &&+ 2\eta \c\nab F  +  \left(  \frac 14 \trch\trchb+\frac 1 4 \atrch\atrchb+ \rho-\rhoF^2-\dual\rhoF^2+i \left(- \rhod+ \eta \wedge \etab  \right)\right)  F
\eeaa
and 
\beaa
 \squared_2 U&=&- \nab_3\nab_4U+\frac 1 2  \DD \hot (\DDb \c U) +\left(2\omb -\frac 1 2( \trchb-i\atrchb)\right) \nab_4U -\frac 1 2 (\trch-i\atrch) \nab_3U\\
 && 2\eta \c\nab U+ \left(-\frac 1 2  \trch\trchb- \frac 1 2 \atrch\atrchb-2\rho+2\rhoF^2+2\dual\rhoF^2+i \left(- 2\rhod+2 \eta \wedge \etab  \right)\right)  U
\eeaa
By writing $ \trchb+ i \atrchb=\ov{\tr\Xb}$, $\trch+i\atrch=\ov{\tr X}$ and $2\eta=H+\ov{H}$, we obtain the stated expressions. 
 \end{proof}

 \subsection{Linear perturbations of Kerr-Newman}

In this section, we define the linear electromagnetic-gravitational perturbations of the Kerr-Newman spacetime. Recall that as Kerr-Newman is of Petrov Type D, the following coefficients,
 \beaa
 \Xh, \Xbh, \Xi, \Xib, \qquad A, B, \Bb, \Ab, \qquad \BF, \BBF
 \eeaa
 vanish in the background. For this reason, our definition of linear perturbations of Kerr-Newman consists in solutions to the Einstein-Maxwell equations where quadratic expressions in the above terms are neglected.

\begin{definition}\label{definition-linear-oerturbation} A linear electromagnetic-gravitational perturbation of the Kerr-Newman spacetime\footnote{This same definition can also be used to define linear electromagnetic-gravitational perturbations of any Petrov Type D spacetime.} is a solution to the Einstein-Maxwell equations of Section \ref{section-equations}, where quadratic expressions of terms which vanish in the background (i.e. $ \Xh, \Xbh, \Xi, \Xib, A, B, \Bb, \Ab, \BF, \BBF$) are neglected.
 \end{definition}

For example, consider the Maxwell equation:
 \beaa
 \nabc_3 \BF- \DDc\PF&=&- \frac 1 2 \tr\Xb \BF+2\PF H +\frac 1 2 \Xh \c \ov{\BBF}
 \eeaa
 The last term, $\frac 1 2 \Xh \c \ov{\BBF}$, is quadratic in $\Xh$ and $\BBF$, and therefore it is neglected in linear electromagnetic-gravitational perturbations of Kerr-Newman. The linearized version of the above Maxwell equation then reduces to
  \beaa
 \nabc_3 \BF- \DDc\PF&=&- \frac 1 2 \tr\Xb \BF+2\PF H 
 \eeaa
 A similar procedure can be applied to all the Einstein-Maxwell equations of Section \ref{section-equations}. We collect them in the following.

 \begin{proposition} A linear  electromagnetic-gravitational perturbation of the Kerr-Newman spacetime consists in a set of complex horizontal scalars, one-forms, 2-tensors
 \beaa
\tr X , \tr \Xb , P, \PF \in \sk_0(\CCC), \\
H, \Hb, Z, \Xi, \Xib, \BF, \BBF, B, \Bb \in \sk_1(\CCC), \\
\Xh, \Xbh, A, \Ab \in \sk_2(\CCC)
\eeaa
which satisfy the following linearized Einstein-Maxwell equations, comprised of the linearized Maxwell equations:
 \bea
\nabc_3 \BF- \DDc\PF&=&- \frac 1 2 \tr\Xb \BF+2\PF H \label{nabc-3-BF}\\
\nabc_4 \BBF+ \DDc\ov{\PF}&=&- \frac 1 2 \tr X \BBF-2\ov{\PF} \Hb  \label{nabc-4-BBF}\\
\nabc_4 \PF- \frac 1 2 \ov{\DDc} \c \BF&=&- \ov{\tr X} \PF +\frac 1 2 \ov{\Hb} \c \BF \label{nab-4-PF} \\
\nabc_3 \PF+ \frac 1 2 \DDc \c \ov{\BBF}&=&- \tr \Xb \PF -\frac 1 2 H \c \ov{\BBF}  ,
\eea
the linearized Ricci identities:
\bea
\nabc_3\tr\Xb +\frac{1}{2}(\tr\Xb)^2 &=& \DDc\c\ov{\Xib}+\Xib\c\ov{\Hb}+\ov{\Xib}\c H,\\
\nabc_4\tr X +\frac{1}{2}(\tr X)^2 &=& \DDc\c\ov{\Xi}+\Xi\c\ov{H}+\ov{\Xi}\c H,\\
\nabc_3\tr X +\frac{1}{2}\tr\Xb\tr X &=& \DDc\c\ov{H}+H\c\ov{H}+2P, \label{nabc-3-tr-X}\\
\nabc_4\tr\Xb +\frac{1}{2}\tr X\tr\Xb  &=& \DDc\c\ov{\Hb}+\Hb\c\ov{\Hb}+2\ov{P}, \label{nabc-4-tr-Xb}
\eea
\bea
\nabc_3\Xbh+\Re(\tr\Xb) \Xbh&=& \DDc\hot \Xib+   \Xib\hot(H+\Hb)-\Ab,\\
\nabc_4\Xh+\Re(\tr X)\Xh&=& \DDc\hot \Xi+  \Xi\hot(\Hb+H)-A. \label{nabc-4-Xh}\\
\nabc_3\widehat{X} +\frac{1}{2}\tr\Xb\, \widehat{X}&=& \DDc\hot H  +H\hot H -\frac{1}{2}\ov{\tr X} \widehat{\Xb},\\
\nabc_4\widehat{\Xb} +\frac{1}{2}\tr X\, \widehat{\Xb} &=& \DDc\hot\Hb  +\Hb\hot\Hb -\frac{1}{2}\ov{\tr\Xb} \widehat{X}
\eea
\bea
\nabc_3\Hb -\nabc_4\Xib &=&  -\frac{1}{2}\ov{\tr\Xb}(\Hb-H) -\frac{1}{2}\Xbh\c(\ov{\Hb}-\ov{H}) +\Bb+\PF\BBF,\\
\nabc_4H -\nabc_3\Xi &=&  -\frac{1}{2}\ov{\tr X}(H-\Hb) -\frac{1}{2}\Xh\c(\ov{H}-\ov{\Hb}) -B-\ov{\PF} \BF, \label{nabc-4-H}\\
\frac{1}{2}\ov{\DDc}\c\Xh &=& \frac{1}{2}\DDc\ov{\tr X}-i\Im(\tr X)H-i\Im(\tr \Xb)\Xi-B+\ov{\PF}\BF, \label{Codazzi-1}\\
\frac{1}{2}\ov{\DDc}\c\Xbh &=& \frac{1}{2}\DDc\ov{\tr\Xb}-i\Im(\tr \Xb)\Hb-i\Im(\tr X)\Xib+\Bb-\PF \BBF,
\eea
\bea\label{Ricci-om}
\nab_3\om+\nab_4\omb -4\om\omb -(\eta-\etab)\c\ze +\eta\c\etab&=&   \rho+\rhoF^2+\dual\rhoF^2
\eea
and the linearized Bianchi identities:
    \bea
 \nabc_3A+\frac{1}{2}\tr\Xb A  &=&\DDc\hot B  +4 H\hot B -3\ov{P}\Xh-2\ov{\PF}\left( -\frac 1 2 \DDc \hot \BF+\PF\Xh  \right)\label{nabc-3-A-noF}\\
  \nabc_4\Ab+\frac{1}{2}\tr X \Ab  &=&-\DDc\hot \Bb  -4 \Hb\hot \Bb -3P\Xbh+2\PF \left( -\frac 1 2 \DDc \hot \BBF -\ov{\PF}\Xbh \right) \label{nabc-4-Ab-noF}
\eea
\bea
\nabc_4B -\frac{1}{2}\ov{\DDc}\c A &=& -2\ov{\tr X} B  +\frac{1}{2}A\c  \ov{\Hb}+\left( 3\ov{P} -2\ov{\PF}\PF\right)\,\Xi+ \ov{\PF} \nabc_4(\BF),  \label{nabc-4-B}\\
\nabc_3\Bb +\frac{1}{2}\ov{\DDc}\c\Ab &=& -2\ov{\tr\Xb}\,\Bb  -\frac{1}{2}\Ab\c \ov{H}-\left(3P-2\ov{\PF}\PF\right) \,\Xib+ \PF \nabc_3(\BBF)
\eea
\bea
\nabc_3B-\DDc\ov{P} &=& -\tr\Xb B+3\ov{P}H +\ov{\PF}\DDc(\PF) -\frac 1 2 \tr \Xb \ov{\PF} \BF -\ov{\tr X}\PF \BBF, \label{nabc-3-B}\\
\nabc_4\Bb+\DDc P &=& -\tr X\Bb-3P\Hb -\PF\DDc(\ov{\PF}) -\frac 1 2 \tr X \PF\BBF -\ov{\tr \Xb}\ov{\PF} \BF
\eea
\bea
\nabc_4P -\frac{1}{2}\DDc\c \ov{B} &=& -\frac{3}{2}\tr X P -\tr X \PF \ov{\PF}+\Hb\c\ov{B} \nonumber\\
&& + \frac 1 2 \PF \DDc \c \ov{\BF}+   \ov{H} \c  \ov{\PF} \BF \label{nabc-4-P} \\
\nabc_3P +\frac{1}{2}\ov{\DDc}\c\Bb &=& -\frac{3}{2}\ov{\tr\Xb} P -\ov{\tr \Xb} \PF \ov{\PF} -\ov{H}\c\Bb \nonumber\\
&& - \frac 1 2 \PF \ov{\DDc} \c \BF -   \Hb  \c  \ov{\PF}\BBF 
\eea
 \end{proposition}

 \begin{remark} Observe that the above Definition of linear perturbations does not rely on a linear expansion of the metric of the form
 \beaa
 \g&=& \g_{m, a, Q}+ \ep \  \g^{(1)}+ \ep^2 \ \g^{(2)}+\dots
 \eeaa
 for a smallness parameter $\ep$. If one performs the above decomposition, applied to the Ricci, curvature and electromagnetic components, and selects the $\ep$-expansion of the Einstein-Maxwell equations, then a choice of gauge would be needed in order to evaluate the values of the background metric $ \g_{m, a, Q}$. The Definition \ref{definition-linear-oerturbation} instead is more general, and has the only effect of discarding the $\ep^2$ terms, without choosing a gauge in doing so.
 \end{remark}

\section{Gauge-invariant quantities in perturbations of Kerr-Newman}\label{section-gauge}

In this section we identify the gauge-invariant quantities in linear gravitational-electromagnetic perturbations of Kerr-Newman spacetime. Those quantities play a fundamental role in the resolution of the stability of Kerr-Newman, as, being affected only quadratically by a change of coordinates, they are good candidate to represent gravitational and electromagnetic radiation.

\subsection{Linear frame transformations}

Recall the rotations of class I, class II and class III which transform the basis vectors $\{e_3, e_4, e_a\}$ into $\{e'_3, e'_4, e'_a\}$ as introduced in Section \ref{conf-inv-section}. Those rotations depend on the one-forms $\mu$ and $\mub$, on the scalar function $\lambda$ and on the orthogonal matrix $O_{ab}$. Observe that the dependence of the equations and the coefficients on $\lambda$ has already been taken into account in the definitions of conformal derivatives, and the dependence on the matrix $O_{ab}$ is accounted for in the use of horizontal tensors as opposed to scalars. In this way, all the equations we have are already invariant to a rotation of class III. We now consider the change in the coefficients caused by rotations of class I and class II.

  In the context of linear perturbations of a spacetime, we consider linear frame transformations, i.e. those where quadratic expressions in the $\mu$ and $\mub$ are neglected. Combining the transformations given by \eqref{SSMe:general.e4} and \eqref{SSMe:general.e3}, and neglecting the quadratic terms $|\mu|^2$ and $|\mub|^2$, we obtain the general linear frame transformations as defined here.

\begin{definition} A linear frame transformation of the basis vectors $\{e_3, e_4, e_a\}$ into $\{e'_3, e'_4, e'_a\}$ is a transformation of the form 
\bea\label{eq:non-conf-transfs}
\begin{split}
e_4'&= e_4 + \mu_a e_a \\
e_3'&= e_3+\mub_a e_a  \\
e_a'&= e_a + \frac 1 2 \mub_a e_4 +\frac 1 2 \mu_a e_3
\end{split}
\eea
 where $\mu$ and $\mub$ are real one-forms.
\end{definition}

When a linear frame transformation is applied to the frame, the Ricci, curvature and electromagnetic components change accordingly. For example, the electromagnetic component $\bF$ is modified in the following way:
\beaa
\bF'_a&=& \F(e_a', e'_4)=\F(e_a + \frac 1 2 \mub_a e_4 +\frac 1 2 \mu_a e_3, e_4 + \mu_b e_b)\\
&=&\F(e_a , e_4 )+\frac 1 2 \mu_a \F( e_3, e_4 )+ \mu_b \F(e_a ,  e_b)+ \frac 1 2 \mub_a \mu_b \F( e_4 ,   e_b)+\frac 1 2 \mu_a\mu_b \F( e_3,   e_b)
\eeaa
By neglecting the quadratic terms in $\mu$ and $\mub$, we then obtain
\beaa
\bF'_a&=&\bF_a+ \mu_a \rhoF -\in_{ab} \mu_b  \dual\rhoF=\bF_a+ \mu_a \rhoF -\dual \mu_a \dual \rhoF
\eeaa
By considering the complexification of the $\bF$, i.e. $\BF$, and by defining $M:=\mu+ i \dual \mu$, we deduce
\beaa
\BF'&=& \bF'+ i \dual \bF'=\bF+\mu \rhoF - \dual \mu \dual \rhoF+ i \dual (\bF+\mu \rhoF - \dual \mu \dual \rhoF)\\
&=&\bF+ i \dual \bF+\mu \rhoF - \dual \mu \dual \rhoF+ i \dual \mu \rhoF + i  \mu \dual \rhoF\\
&=&\bF+ i \dual \bF+(\rhoF+ i \dual \rhoF) (\mu+ i \dual \mu)= \BF+\PF M.
\eeaa

In the same way we can compute how all the Ricci, curvature and electromagnetic components get transformed by a linear frame transformation of the form \eqref{eq:non-conf-transfs}. We collect those transformations in the following.

 \begin{lemma}\label{lemma-gauge-transformations} The linear frame transformation \eqref{eq:non-conf-transfs} modifies the Ricci, curvature and electromagnetic components in the following way:
 \beaa
\tr X'&=& \tr X +\frac 1 2 \DDc \c \ov{M} +\frac 1 2 \ov{H} \c M, \qquad \tr \Xb'= \tr \Xb +\frac 1 2 \DDc \c \ov{\Mb} +\frac 1 2 \ov{\Hb} \c \Mb \\
\Xh'&=& \Xh+\frac 1 2 \DDc \hot M+\frac 1 2 H \hot M, \qquad \Xbh'= \Xbh+\frac 1 2 \DDc \hot \Mb+\frac 1 2 \Hb \hot \Mb\\
H'&=& H+\frac 1 2 \nabc_3 M +\frac 1 4 \ov{\tr X} \Mb, \qquad \Hb'= \Hb+\frac 1 2 \nabc_4 \Mb +\frac 1 4 \ov{\tr \Xb} M\\
\Xi'&=& \Xi +\frac 1 2 \nabc_4 M+ \frac 1 4 \ov{\tr X} M, \qquad \Xib'= \Xib +\frac 1 2 \nabc_3 \Mb+\frac 1 4  \ov{\tr \Xb} \Mb
\eeaa
and
\beaa
\BF'= \BF+\PF M, \qquad \PF'= \PF, \qquad \BBF'=\BBF - \ov{\PF} \Mb \\
A'= A, \qquad B'= B+\frac 3 2 \ov{P} M, \qquad P'= P, \qquad \Bb'=\Bb -\frac 3 2 P \Mb, \qquad \Ab'= \Ab
\eeaa
 where $M=:=\mu+ i \dual \mu$ and $\Mb:=\mub+ i \dual \mub$
 \end{lemma}
 \begin{proof} See \cite{stabilitySchwarzschild}.
 \end{proof}

\subsection{Gauge-invariant quantities and their relations}

The linear frame transformations of the form \eqref{eq:non-conf-transfs} can be used to pick a gauge in the perturbations. In particular, the quantities which are not modified by such a linear frame transformation may have a physical meaning, since they do not depend at the linear level on the choice of coordinates. We call such quantities gauge-invariant, and they are good candidates to represent electromagnetic or gravitational radiation.
\begin{definition} A horizontal tensor $\Psi \in \sk(\CCC)$ is said to be gauge-invariant if it is not modified by a linear frame transformations of the form \eqref{eq:non-conf-transfs}, i.e. if $\Psi'=\Psi$. 
\end{definition}

One of the main steps in analyzing electromagnetic-gravitational perturbations of Kerr-Newman is then to identify the gauge-invariant quantities. We identify in the following lemma four gauge-invariant quantities and their symmetric version.

\begin{lemma} For a linear electromagnetic-gravitational perturbation of the Kerr-Newman spacetime, the following symmetric traceless 2-tensors
\beaa
&&A, \qquad \Ab
\eeaa
and 
\bea\label{definition-F-Bianchi}
\begin{split}
 \mathfrak{F}&=  -\frac 1 2 \DDc \hot \BF-\frac 3 2 H \hot \BF +\PF\Xh \\
  \mathfrak{\underline{F}}&=  -\frac 1 2 \DDc \hot \BBF-\frac 3 2 \Hb \hot \BBF -\ov{\PF}\Xbh
  \end{split}
  \eea
 and the following 1-tensors
\bea
\mathfrak{B}&=& 2\PF B - 3\ov{P} \BF, \qquad  \underline{\mathfrak{B}}= 2\ov{\PF} \Bb - 3 P \BBF \label{definition-mathfrak-B}
\eea
and
\bea
\mathfrak{X}=\nabc_4\BF+\frac 3 2 \ov{\tr X} \BF-2\PF \Xi, \qquad  \underline{\mathfrak{X}}=\nabc_3\BBF+\frac 3 2 \ov{\tr \Xb} \BBF+2\ov{\PF} \Xib\label{definition-mathfrak-X}
\eea
are gauge-invariant.
\end{lemma}
\begin{proof} The invariance of $A$ and $\Ab$ is straightforward from Lemma \ref{lemma-gauge-transformations}. We check the invariance of $\mathfrak{F}$: 
\beaa
\mathfrak{F}'&=&  -\frac 1 2 \DDc' \hot \BF' -\frac 3 2  H' \hot \BF'+\PF'\Xh'\\
&=&  -\frac 1 2 \DDc\hot \left(\BF+\PF M \right) -\frac 3 2  H \hot \left( \BF+\PF M \right)+\PF\left(\Xh+\frac 1 2 \DDc \hot M+\frac 1 2 H \hot M\right)\\
&=& \mathfrak{F} -\frac 1 2 \DDc \PF \hot M -\frac 1 2\PF  \DDc\hot  M  -\frac 3 2 \PF  H \hot  M +\PF\left(\frac 1 2 \DDc \hot M+\frac 1 2 H \hot M\right)\\
&=& \mathfrak{F} -\frac 1 2 (-2 \PF H) \hot M   - \PF  H \hot  M =\mathfrak{F}
\eeaa 
where we used\footnote{Observe that by using the linearized Maxwell equation \eqref{nabc-3-BF}, one obtains
\beaa
(\DDc \PF) \hot M&=&\left(\nabc_3 \BF+ \frac 1 2 \tr\Xb \BF-2\PF H \right) \hot M= - 2 \PF H \hot M \text{+ quadratic terms}
\eeaa  
since the terms in $\BF$ and $M$ are quadratic for linear perturbations of Kerr-Newman, and are therefore neglected. From now on, when multiplied by a quantity which vanishes on the background, we can then use the reduced equations of Proposition \ref{prop:reduced-equations}.} \eqref{DDc-PF-red}. Similarly for $  \mathfrak{\underline{F}}$.
We check the invariance of $\mathfrak{B}$:
\beaa
\mathfrak{B}'&=& 2\PF' B' - 3\ov{P}' \BF'= 2\PF \left(B+\frac 3 2 \ov{P} F \right) - 3\ov{P}\left(\BF+\PF F \right)\\
&=&\mathfrak{B}+ 2\PF \left(\frac 3 2 \ov{P} F \right) - 3\ov{P}\left(\PF F \right)=\mathfrak{B}
\eeaa
and similarly for $\underline{\mathfrak{B}}$. Finally, we check the invariance of $\mathfrak{X}$:
\beaa
\mathfrak{X}'&=&\nabc_4(\BF+\PF F)+\frac 3 2 \ov{\tr X} (\BF+\PF F)-2\PF( \Xi +\frac 1 2 \nabc_4 F+\frac 1 4  \ov{\tr X} F)\\
&=&\nabc_4\BF+\frac 3 2 \ov{\tr X} \BF-2\PF \Xi\\
&&- \ov{\tr X} \PF F+\PF \nabc_4( F)+\frac 3 2 \ov{\tr X} (\PF F)-2\PF( \frac 1 2 \nabc_4 F+ \frac 1 4 \ov{\tr X} F)\\
&=&\nabc_4\BF+\frac 3 2 \ov{\tr X} \BF-2\PF \Xi=\mathfrak{X}
\eeaa
and similarly for $\underline{\mathfrak{X}}$.
\end{proof}

\begin{remark}\label{remark:comp-NP-gauge} The identification of the above gauge-invariant quantities for linear perturbations of Kerr-Newman is a crucial new part of this work. In order to compare them in NP formalism, we collect here their equivalent, where the correspondence has to be understood through the projection to the first component, as explained in Section \ref{section:connection-NP}. 
\begin{table}[h!]
\centering
 \begin{tabular}{||c c ||} 
 \hline
 NP formalism & GKS formalism \\ [0.5ex] 
 \hline\hline
        $\Psi_0$ & $A$  \\  
          $\ff:= - \delta \phi_0+(2 \beta +3\tau) \phi_0 - 2\sigma \phi_1$ & $\mathfrak{F}$  \\ 
            $\bff:= 3 \phi_0 \Psi_2 - 2 \phi_1 \Psi_1$ & $\mathfrak{B}$  \\ 
       $\xf:= D\phi_0-(3 \rho+2\epsilon) \phi_0+2\kappa \phi_1$ & $\mathfrak{X}$  \\ 
 \hline
 \end{tabular}
\end{table}

We point out that in page 240 of \cite{Chandra}, equation (213), in the context of perturbations of Reissner-Nordstr\"om, Chandrasekhar notes ``parenthetically, that while a gauge, in which $\Psi_1$ and $\phi_1$ vanish simultaneously, cannot be chosen, the combination $2 \Psi_1 \phi_1-3\phi_0 \Psi_2$ is invariant to first order for infinitesimal rotations".

The complex scalar $\bff$ identified in \cite{Chandra} corresponds precisely to the projection to the first component of our gauge-invariant quantity $\mathfrak{B}$. 
Nevertheless, such quantity was not used in the subsequent analysis in \cite{Chandra}. Indeed, it was used to show that a gauge where $\Psi_1$ and $\phi_1$ vanish identically cannot be chosen, while a gauge where $\phi_0=\phi_2=0$, the so called phantom gauge, can be chosen.  In \cite{Chandra}, the equations governing the perturbations in the NP formalism were written in the phantom gauge, and all the analysis was performed in such a gauge. In particular, by choosing the phantom gauge, the above quantity was being reduced to a rescaled version of the curvature component $\Psi_1$.  

No previous mention of $\ff$ nor $\xf$ in the context of perturbations of Kerr-Newman  is known to the author. 

\end{remark}

\begin{remark}\label{rem:confornto-RN}In the case of linear electromagnetic-gravitational perturbations of Reissner-Nordstr\"om, the gauge-invariant quantities $A$, $\Ffr$, $\Bfr$ and $\Xfr$ respectively reduce to the quantities $\a$, $\ff$, $\tilde{\b}$ and $\mathfrak{x}$, first appearing in \cite{Giorgi4} and \cite{Giorgi5}.  More precisely, in Reissner-Nordstr\"om the Ricci coefficients $H, \Hb$, $\atrch, \atrchb$, $\dual \rhoF, \rhod$ vanish in the background, and therefore the terms $H\hot \BF$, $\Hb \hot \BBF$, $\dual \rhoF \Xh$, $\dual \rhoF \Xbh$ in the definition of $\Ffr$, $\Bfr$ and $\Xfr$ become quadratic for linear perturbations of Reissner-Nordstr\"om. The real parts of $A$, $\Ffr$, $\Bfr$ and $\Xfr$ reduce to 
\beaa
\Re( \mathfrak{F})&=&  - \nab \hot \bF +\rhoF\chih\text{+ quadratic terms} =\ff\\
\Re(\mathfrak{B})&=& 2\rhoF \b - 3\rho \bF\text{+ quadratic terms} =\tilde{\b}\\
\Re(\mathfrak{X})&=&\nab_4\bF+\frac 3 2 \trch \bF-2\rhoF \xi\text{+ quadratic terms}=\xf
\eeaa
with  $\ff$, $\tilde{\b}$ and $\mathfrak{x}$, as defined in \cite{Giorgi4} and \cite{Giorgi5}.

In the case of gravitational perturbations of Kerr, the only gauge-invariant quantity which has relevance is $\Psi_0, \Psi_4$ or $A, \Ab$, the well-known Teukolsky variables. The quantities $\Ffr$, $\Bfr$ and $\Xfr$, since contain electromagnetic components, only make sense for solutions to the Einstein-Maxwell equations.
\end{remark}

Observe that by adding and subtracting $3\ov{\PF} H \hot \BF$ and $3\PF \Hb \hot \BBF$ to the linearized Bianchi identities \eqref{nabc-3-A-noF} and \eqref{nabc-4-Ab-noF} respectively, using the definition of $ \mathfrak{F}$ and $\underline{ \mathfrak{F}}$ \eqref{definition-F-Bianchi}, those Bianchi identity become
\bea
 \nabc_3A+\frac{1}{2}\tr\Xb A  &=&\DDc\hot B  +H\hot \left(4B-3\ov{\PF} \BF \right) -3\ov{P}\Xh-2\ov{\PF} \mathfrak{F} \label{nabc-3-A}\\
   \nabc_4\Ab+\frac{1}{2}\tr X \Ab  &=&-\DDc\hot \Bb  -\Hb\hot \left(4B-3\PF \BBF \right) -3P\Xbh+2\PF\underline{\mathfrak{F}} 
 \eea

We summarize here some fundamental relations between the above gauge invariant quantities $A$, $\mathfrak{F}$, $\mathfrak{B}$ and $\mathfrak{X}$ obtained as consequence of the linearized Einstein-Maxwell equation. The relations between $\Ab$, $\underline{\mathfrak{F}}$, $\underline{\mathfrak{B}}$ and $\underline{\mathfrak{X}}$ can be obtained by symmetrization.

\begin{proposition}\label{proposition-relation-gauge-inv} In a linear  electromagnetic-gravitational perturbation of the Kerr-Newman spacetime, the following relations among the gauge invariant quantities $A$, $\mathfrak{F}$, $\mathfrak{B}$ and $\mathfrak{X}$ hold true\footnote{In NP formalism, the above relations have the following form: 
\beaa
2\phi_1\left(- \Delta \Psi_{0}+	\left(4\gamma - \mu\right)\Psi_{0} \right)&=&\delta \bff  - 2  ( \beta +3\tau) \bff+ \left( 3  \Psi_2+ 2\phi_1\overline{\phi_1} \right)\ff\\
D \ff	-\left(3 \rho+\overline{\rho}+3\epsilon- \overline{\epsilon}\right)\ff &=&-\delta \xf+\left(3 \beta +3\tau-\overline{\pi}  + \overline{\alpha}\right) \xf  - 2 \phi_1  \Psi_{0} \\
D \bff - 2  \left(\epsilon + 3\rho\right)\bff &=&- 2 \phi_1 \left( \overline{\delta}\Psi_{0}-4\alpha\Psi_{0} + \pi\Psi_{0} \right)+\left(3\Psi_2-2\phi_1\overline{\phi_1}\right) \xf \\
\Delta \xf+\left(\overline{\mu}-3\gamma-\overline{\gamma}\right)\xf	&=&  -\overline{\delta} \ff+ \left(\overline{\tau}+ 3 \alpha- \overline{\beta}\right) \ff+2\bff
\eeaa
where $\ff$, $\bff$, $\xf$ are defined as in Remark \ref{remark:comp-NP-gauge}.
}. 
\begin{itemize}
\item The following relation between the $\nabc_3$ derivative of $A$, the $\DDc$ derivative of $\Bfr$ and $\Ffr$:
\bea
 \PF \left(\nabc_3A+\frac{1}{2} \tr\Xb A \right)&=& \frac 1 2 \DDc \hot \mathfrak{B}+  3H  \hot  \mathfrak{B} -\left(3 \ov{P}+2\PF\ov{\PF}\right)\mathfrak{F}\label{relation-F-A-B}
 \eea
 \item The following relation between the $\nabc_4$ derivative of $\Ffr$, the $\DDc$ derivative of $\Xfr$ and $A$:
 \bea
\nabc_4 \mathfrak{F}+\left(\frac 3 2 \ov{\tr X} +\frac 1 2  \tr X\right)\mathfrak{F}&=&-\frac 1 2 \DDc \hot \mathfrak{X} - \frac 1 2 \left( 3    H+ \Hb \right) \hot \mathfrak{X}-\PF A   \label{relation-F-B-A}
\eea
\item The following relation between the $\nabc_4$ derivative of $\Bfr$, the $\DDbc$ derivative of $A$ and $\Xfr$:
    \bea
\nabc_4\mathfrak{B}+3\ov{\tr X} \mathfrak{B}&=&\PF \left(\ov{\DDc}\c A  +  \ov{ \Hb} \c A \right)- \left( 3\ov{P}-2 \PF \ov{\PF} \right)\mathfrak{X}\label{nabb-4-mathfrak-B}
\eea
\item The following relation between the $\nabc_3$ derivative of $\Xfr$, the $\DDbc$ derivative of $\Ffr$ and $\Bfr$:
\bea
\nabc_3 \mathfrak{X} +\frac 1 2 \ov{\tr \Xb} \ \mathfrak{X}&=&-\ov{\DDc} \c \mathfrak{F}  -\ov{H} \c \mathfrak{F}-2\mathfrak{B} \label{nabc-3-mathfrak-X}
\eea
\end{itemize}
\end{proposition}
\begin{proof} See Appendix \ref{proof-prop-app}. 
\end{proof}

\section{The system of Teukolsky equations}\label{Teukolsky-equations-section}

In this section, we state the first theorem of the paper, which contains the coupled system of Teukolsky equations for the gauge-invariant quantities $A$, $\mathfrak{F}$ and $\mathfrak{B}$. These equations govern the linear electromagnetic-gravitational perturbations of Kerr-Newman spacetime, and generalize the Teukolsky equation for $A$ in the case of Kerr. The system of Teukolsky equations for $\underline{\Bfr}$, $\underline{\Ffr}$ and $\Ab$ can be obtained by symmetry.

\begin{theorem}\label{Teukolsky-equations-theorem} Consider a linear electromagnetic-gravitational perturbation of Kerr--Newman spacetime $\g_{M, a, Q}$ as in Definition \ref{definition-linear-oerturbation}.  Then its associated complex tensors and gauge-invariant quantities $A$, $\mathfrak{F}$, $\mathfrak{B}$ and $\mathfrak{X}$, satisfy the following coupled system of Teukolsky equations:
\bea\label{final-system-schematic}
\TT_1(\mathfrak{B})&=&\M_1[\mathfrak{F}, \mathfrak{X}] \label{Teukolsky-B}\\
\TT_2(\mathfrak{F})&=&\M_2[A, \mathfrak{X}, \mathfrak{B}] \label{Teukolsky-F}\\
\TT_3(A)&=&\M_3[\Ffr, \Xfr]\label{Teukolsky-A}
\eea
where
\begin{itemize}
\item on the left hand side of the equations, $\TT$ denotes the Teukolsky differential operators, respectively given by
  \bea\label{operator-Teukolsky-B}
  \begin{split}
 \TT_1(\Bfr)&:=  - \nabc_3\nabc_4\Bfr+ \frac 1 2 \ov{\DDc}\c (\DDc \hot \Bfr)-3\ov{\tr X} \nabc_3\Bfr- \left(\frac{3}{2}\tr\Xb +\frac 1 2\ov{\tr\Xb}\right)\nabc_4\Bfr\\
   &+\left( 6H+ \ov{H}+ 3  \ov{ \Hb}  \right)\c  \nabc  \Bfr+\left(-\frac{9}{2}\tr\Xb \ov{\tr X} -4 \PF \ov{\PF}+9 \ov{ \Hb} \c  H\right)\Bfr
   \end{split}
 \eea
  \bea\label{operator-Teukolsky-F}
\begin{split}
\TT_2(\Ffr)&:= -\nabc_3\nabc_4 \Ffr+ \frac 1 2 \DDc \hot (\ov{\DDc} \c \Ffr)-\left(\frac 3 2 \ov{\tr X} +\frac 1 2  \tr X\right)\nabc_3\Ffr\\
&- \frac 1 2 \left(\tr \Xb+\ov{\tr \Xb}  \right)\nabc_4\Ffr  + \left(4   H+ \ov{H}+  \Hb \right)\c \nabc \Ffr\\
&+\left(- \frac 3 4 \tr \Xb  \ov{\tr X}- \frac 1 4\ov{\tr \Xb}   \tr X+3\ov{P} -P +4\PF\ov{\PF} - \frac 3 2\ov{\DDc}\c H\right)\Ffr+\frac 1 2 \Hb \hot ( \ov{H} \c \Ffr)
\end{split}
\eea
\bea\label{operator-Teukolsky-A}
\begin{split}
\TT_3(A)&:= - \nabc_4 \nabc_3A+\frac{1}{2}\DDc \hot (\ov{\DDc}\c A)  - \left(\frac 1 2 \tr X+2\ov{\tr X}\right)\nabc_3A-\frac{1}{2}\tr\Xb \nabc_4A \\
& +\left(  4 H+ \Hb+\ov{\Hb} \right) \c \nabc A+\left(-\ov{\tr X} \tr \Xb+2\ov{P}-2\PF\ov{\PF} \right)A +2 H\hot \left( \ov{\Hb} \c A \right) 
\end{split}
\eea
\item on the right hand side of the equations, $\M$ denotes the coupling terms, where the terms in squared parenthesis indicate the quantities involved in the expressions, respectively given by
\bea
\M_1[\mathfrak{F}, \mathfrak{X}]&:=& 2\PF\ov{\PF}\left(2\ov{\DDc}\c\mathfrak{F}+4\ov{\Hb}\c\mathfrak{F}- \left(2\tr \Xb -  \ov{\tr \Xb}\right) \ \mathfrak{X} \right)  \label{definition-MM1}\\
\M_2[A, \mathfrak{X}, \mathfrak{B}]&:=& -  \PF \left(\nabc_3A +\frac 1 2 \left(3\tr \Xb-\ov{\tr \Xb}  \right) A\right)    + \left(\frac 3 2  \nabc_3 H \right) \hot \mathfrak{X} \nonumber\\
&&+\left(2  H- \Hb \right) \hot \mathfrak{B} \label{definition-MM2}\\
\M_3[\Ffr, \Xfr]&:=&2\ov{\PF}\left(  2 \nabc_4\mathfrak{F} +2 \ov{\tr X} \mathfrak{F}+ \left(\Hb + H\right)\hot\mathfrak{X} \right)
\eea
\end{itemize}
\end{theorem}
\begin{proof} The derivation of the above Teukolsky equations relies on Proposition \ref{proposition-relation-gauge-inv}, and is obtained in Appendix \ref{proof-main-thm}.
\end{proof}

We collect here few remarks about Theorem \ref{Teukolsky-equations-theorem}. 
\begin{enumerate}
\item The Teukolsky operators $\TT_1$, $\TT_2$ and $\TT_3$ are wave-like operators, as it can be seen by comparing the expressions for $\squared_1$ and $\squared_2$ given in Lemma \ref{lemma:expression-wave-operator}. More precisely, their highest order terms are given by a wave operator, with the additional presence of first order terms. 
\item Observe that the system of equations \eqref{Teukolsky-B}--\eqref{Teukolsky-A} for $\Bfr$, $\Ffr$ and $A$ also involve the gauge-invariant quantity $\Xfr$. Nevertheless, $\Xfr$ is considered here an auxiliary quantity which only appears on the right hand side $\M_1$, $\M_2$ and $\M_3$ at the first order. More precisely, the system of Teukolsky equations \eqref{Teukolsky-B}--\eqref{Teukolsky-A}, combined with the transport equation \eqref{nabc-3-mathfrak-X} for $\Xfr$, gives a complete system of equations. 
\item  In the case of linear gravitational perturbations of Kerr spacetimes (which corresponds to Kerr-Newman for $\PF=0$), equation \eqref{Teukolsky-A} for $A$  reduces to $\TT_3(A)=0$, i.e. the Teukolsky equation of spin $+2$ in Kerr, as obtained in \cite{GKS}. 

\item In the case of linear electromagnetic-gravitational perturbations of Reissner-Nordstr\"om (which corresponds to Kerr-Newman for $H=\Hb=\atrch=\atrchb=0$), the real parts of the Teukolsky system \eqref{Teukolsky-B}--\eqref{Teukolsky-A} reduces to the following Teukolsky equations for $\tilde{\b}$, $\ff$ and $\alpha$
\beaa
\TT_1(\tilde{\beta})&=&2\rhoF^2\left(4\div \ff -\trchb \ \mathfrak{x}\right)\\
\TT_2(\ff)&=& -\rhoF\left(\nabc_3 \a+ \trchb  \a \right)\\
\TT_3(\alpha)&=& 4\rhoF\left(\nabc_4\ff+ \trch \ff\right) 
\eeaa
where
\beaa
\TT_1(\tilde{\beta})&:=&-\nabc_3\nabc_4\tilde{\b}-2\div \DDs_2\tilde{\b}- 3\trch \nabc_3\tilde{\b}-2 \trchb\nabc_4\tilde{\b}+ \left(-\frac {9}{2} \trch \trchb-4\rhoF^2\right) \tilde{\b}\\
\TT_2(\ff)&:=& - \nabc_{3}\nabc_{4} \ff -2\DDs_2 \div\ff- \trchb \nabc_4\ff-2 \trch\nabc_3\ff+(- \trch\trchb+2\rho+4\rhoF^2)\ff\\
\TT_3(\a)&:=&-\nabc_4\nabc_3\a-2\DDs_2\div\a- \frac 1 2 \trchb \,\nabc_4\a  -\frac 5 2 \trch \nabc_3 \a + \left(- \trch\trchb+2\rho-2\rhoF^2 \right)\,\a 
\eeaa
The quantities $\tilde{\b}$, $\ff$ and $\alpha$  (as recalled in Remark \ref{rem:confornto-RN}) and the above quantities were obtained in \cite{Giorgi4} and \cite{Giorgi5}. 
\end{enumerate}

Finally, we relate the above system of Teukolsky equations for the horizontal tensors $\Bfr$, $\Ffr$ and $A$ to the equations verified by their projection to the first component, as one would have obtained using the NP formalism. 

Using that for $F \in \sk_1(\CCC)$ and $U \in \sk_2(\CCC)$ with $f=F_{1}=F(e_1)$ and $u=u_{11}=u(e_1, e_1)$ their scalar projections, we have, see Appendix D of \cite{GKS},
\beaa
(\squared_2 F)_{11} &=& \square_\g f +i \frac{2}{|q|^2}\frac{\cos\th}{\sin^2\th}  \partial_\vphi f +\left(  -2  \frac{(r^2+a^2)^2}{|q|^6}\cot^2\th+ \frac{2a^2\cos^2\th \Delta }{|q|^6}\right) f\\
(\squared_2 U)_{11} &=& \square_\g u +i \frac{4}{|q|^2}\frac{\cos\th}{\sin^2\th}  \partial_\vphi u +\left(  -4  \frac{(r^2+a^2)^2}{|q|^6}\cot^2\th+ \frac{4a^2\cos^2\th \Delta }{|q|^6}\right) u
\eeaa
we can deduce the scalar Teukolsky equations satisfied by the projections of $\Bfr$, $\Ffr$ and $A$. 

An interesting aspect of the system of Teukolsky equations in Kerr-Newman is that we have to differentiate between the spin and the conformal type of the quantities $\Bfr$, $\Ffr$ and $A$. As $\Bfr$ is a horizontal 1-tensor and $\Ffr$ and $A$ are horizontal 2-tensors, their projections $\Bfr_{1}$ and $\Ffr_{11}$, $A_{11}$ will respectively be scalars of spin $1$ and spin $2$. On the other hand, $\Bfr$ and $\Ffr$ are of conformal type $1$, while $A$ is of conformal type $2$. We define the relevant rescaling of those projections up to some functions of $q$ and $\ov{q}$ so that we can relate them to the standard Teukolsky equation in the literature \cite{Teukolsky}. 

We define the following rescaled projected quantities:
\beaa
\bff= \frac{\ov{q}^{7/2}}{q^{ 1/ 2}} \Bfr_{1}, \qquad  \ff=\ov{q} \Ffr_{11}, \qquad \a=\frac{\ov{q}}{q}A_{11}
\eeaa
and we collect in the following table their respective spin and conformal type:
\begin{table}[h!]
\centering
\begin{tabular}{l | c | c  }
\hline
  & spin type $s$ & conformal type $c$  \\
\hline \hline
$\bff$ & 1 & 1 \\
\hline
$\ff$ & 2 & 1 \\
\hline
$\a$ & 2 & 2  \\ 
\hline
\end{tabular}
\end{table}

We define the following Teukolsky operator of spin type $s$ and conformal type $c$ in Kerr-Newman, applied to a scalar $\psi$ of spin type $s$ and conformal type $c$ to be given by
\bea\label{eq:scalar-Teuk-operators}
\begin{split}
\mathcal{T}^{[s,c]}(\psi)&:=\square_{\g_{M, a,Q}} \psi + \frac{2c}{|q|^2}(r-M) \partial_r \psi +\frac{2}{|q|^2} \left( c\frac{a(r-M)}{\Delta} + s i \frac{\cos\th}{\sin^2\th} \right) \partial_\vphi \psi \\
&+\frac{2}{|q|^2} \left(c \big(\frac{M(r^2-a^2)-Q^2 r}{\Delta} - r \big)- s i a \cos\th \right) \partial_t \psi + \frac{1}{|q|^2} (s-s^2 \cot^2 \th) \psi
\end{split}
\eea
In particular notice that the conformal type $c$ is relevant in the real parts of the coefficients of the first derivative, while the spin type $s$ is relevant in the imaginary parts. 
Observe that the above Teukolsky operator reduces to the standard one \cite{Teukolsky} in Kerr for spin $s$ applied to $\Psi_0$, $\Psi_4$, $\phi_0$, $\phi_2$ by using $c=s$, since these quantities have the same spin and conformal type.

One can then show that the projections to the first components of the Teukolsky differential operators $\TT_1$, $\TT_2$, $\TT_3$ given by \eqref{operator-Teukolsky-B}--\eqref{operator-Teukolsky-B} can be written in terms of the scalar Teukolsky operator \eqref{eq:scalar-Teuk-operators}. More precisely:
\beaa
\frac{\ov{q}^{7/2}}{q^{ 1/ 2}} ( \TT_1(\Bfr))_{1}= \TT^{[1,1]}(\bff), \qquad \ov{q} (\TT_2(\Ffr))_{11}=\TT^{[2,1]}(\ff), \qquad \frac{\ov{q}}{q}(\TT_3(A))_{11}=\TT^{[2,2]}(\a). 
\eeaa

Just like in Schwarzschild, Kerr and Reissner-Nordstr\"om, boundedness and decay for solutions to the Teukolsky equations cannot be obtained directly. Since in Kerr-Newman is crucial to avoid the decomposition in modes as recalled in the introduction, we proceed in deriving a system of generalized Regge-Wheeler equations from the Teukolsky ones.

\section{The system of generalized Regge-Wheeler equations}\label{Chandra-section}

In this section, we derive the Regge-Wheeler system of equations governing the electromagnetic-gravitational perturbations of Kerr-Newman spacetime, therefore proving the main result of the paper.

\subsection{The invariant quantities $\mathfrak{P}$, $\mathfrak{Q}$ and $\pf$, $\qf$}

We introduce here the crucial invariant quantities satisfying the Regge-Wheeler equations. Those quantities are derived from the gauge-invariant quantities $\Bfr$ and $\Ffr$, through the following conformal operator.
\begin{definition} Let $\Psi \in \sk_k(\CCC)$ be a gauge-invariant quantity of conformal type $s$. We define the Chandrasekhar operator $\mathcal{P}_{C}: \sk_k(\CCC) \to \sk_k (\CCC)$ to be
\bea\label{definition-P-C}
\mathcal{P}_{C}(\Psi)&:=& \nabc_3 \Psi + C  \Psi \in \sk_k(\CCC)
\eea
for a scalar function $C$ of conformal type $-1$. 
\end{definition}
We immediately observe that $\mathcal{P}_{C}(\Psi) \in \sk_k(\CCC)$ is gauge-invariant of conformal type $s-1$.

We define the invariant quantities $\Pfr$ and $\Qfr$ as the Chandrasekhar-transformed of the gauge-invariant quantities $\Bfr$ and $\Ffr$ respectively. In addition we allow for a rescaling of those quantities. 

\begin{definition} We define the  invariant quantities $\Pfr \in \sk_1(\CCC)$ and $\Qfr \in \sk_2(\CCC)$ as
\bea
\mathfrak{P}&:=&\mathcal{P}_{C_1}(\mathfrak{B})= \nabc_3 \Bfr + C_1  \Bfr \in \sk_1(\CCC), \label{definition-P-Q}\\
 \mathfrak{Q}&:=&\mathcal{P}_{C_2}(\mathfrak{F})= \nabc_3 \Ffr + C_2  \Ffr \in \sk_2(\CCC)\label{definition-P-Q2}
\eea
for scalar functions $C_1$ and $C_2$ to be determined. 
We also define their rescaled version $\pf \in \sk_1(\CCC)$ and $\qf^\F \in \sk_2(\CCC)$ as 
\bea
\mathfrak{p}&:=& f_1(q, \ov{q}) \mathfrak{P}= f_1(q, \ov{q})\big(\nabc_3 \Bfr + C_1  \Bfr \big) \in \sk_1(\CCC), \label{definition-pf}\\
 \qf^\F&:=&  f_2(q, \ov{q})\mathfrak{Q} = f_2(q, \ov{q})\big( \nabc_3 \Ffr + C_2  \Ffr \big) \in \sk_2(\CCC) \label{definition-qf-F}
\eea
where $f_1$ and $f_2$ are functions of $q=r+i a \cos\th$ and $\ov{q}=r-i a \cos\th$ to be determined.
\end{definition}
The quantities $\pf$ and $\qf$ can be seen as first order differential operators applied to the gauge-invariant quantities $\Bfr$ and $\Ffr$, which satisfy the Teukolsky system of equations. Observe that $\Pfr$, $\Qfr$ and $\pf$, $\qf^\F$ are all of conformal type $0$.

\subsection{Statement of the main theorem and remarks}

We now state the main result regarding the wave equations satisfied by $\pf$ and $\qf^\F$.

\begin{theorem}\label{main-theorem-RW} Consider a linear electromagnetic-gravitational perturbation of Kerr-Newman spacetime $\g_{M, a, Q}$ as in Definition \ref{definition-linear-oerturbation}, with associated gauge-invariant quantities $\mathfrak{B}$ and $\mathfrak{F}$.

Then there exist choices of complex scalar functions $C_1$, $C_2$ $f_1$, $f_2$, in the definitions of $\pf$ and $\qf^\F$, explicitly:
\beaa
\mathfrak{p}&=& q^{\frac 1 2 } \ov{q}^{\frac 9 2 } \left(\nabc_3 \Bfr + \big(2\trchb -\frac 5 2  i  \atrchb \big)  \Bfr \right) \in \sk_1(\CCC), \\
 \qf^\F&=& q \ov{q}^2\left( \nabc_3 \Ffr + \big(\trchb -3 i \atrchb \big) \Ffr \right) \in \sk_2(\CCC),
\eeaa
such that the invariant 1-tensor $\pf \in \sk_1(\CCC)$ and the symmetric traceless 2-tensor $\qf^\F \in \sk_2(\CCC)$ satisfy
 the following coupled system of wave equations:
\bea
 \squared_1\pf-i  \frac{2a\cos\th}{|q|^2}\nab_t \pf  -V_1  \pf &=&4Q^2 \frac{\ov{q}^3 }{|q|^5} \left(  \ov{\DD} \c  \qf^\F  \right) + L_\pf[\Bfr, \Ffr] \label{final-eq-1}\\
\squared_2\qf^\F-i  \frac{4a\cos\th}{|q|^2}\nab_t \qf^\F -V_2  \qf^\F &=&-   \frac 1 2\frac{q^3}{|q|^5} \left(  \DD \hot  \pf  -\frac 3 2 \left( H - \Hb\right)  \hot \pf \right) + L_{\qf^\F}[\Bfr, \Ffr]\label{final-eq-2}
 \eea
where
\begin{itemize}
\item $\squared_1$ and $\squared_2$ denote the wave operators for horizontal 1-tensors and 2-tensors respectively, as defined in \eqref{eq:definition-squared}, 
\item the potentials $V_1$ and $V_2$ are \textbf{real} positive scalar functions (whose precise expression is given by \eqref{potentials-explicity}), which for $a=0$ coincide with the potentials of the Regge-Wheeler system of equations in Reissner-Nordstr\"om \cite{Giorgi7}, i.e.
\beaa
V_1&=&- \frac {1}{ 4} \trch\trchb +5 \rhoF^2+ O\big(\frac{|a|}{r^4}\big), \qquad V_2=- \trch \trchb  +2\rhoF^2+ O\big(\frac{|a|}{r^4}\big),
\eeaa
\item $L_\pf[\Bfr, \Ffr] $ and $L_{\qf^\F}[\Bfr, \Ffr]$ are linear first order operators in $\Bfr$ and $\Ffr$, respectively given by
 \beaa
L_\pf[\Bfr, \Ffr] &=& q^{1/2} \ov{q}^{9/2} \Big[-Z^{\Bfr}_{a} \c \nabc \Bfr+ 2 \PF\ov{\PF} \ Y^{\Ffr}_{a} (\ov{\DDc}\c\mathfrak{F})\\
&&+ ( 2 \PF\ov{\PF}\ Y^{\Bfr}_0 - Z^{\Bfr}_0) \Bfr + 2 \PF\ov{\PF}\left( Y^{\Ffr}_0 \c \mathfrak{F}+Y^{\Xfr}_0 \ \mathfrak{X}\right)\Big] 
\eeaa
and
\beaa
L_{\qf^\F}[\Bfr, \Ffr] &=& q\ov{q}^2 \Big[ W_4^{\Ffr} \nabc_4 \mathfrak{F}+\left( W_a^{\Ffr}  -Z^{\Ffr}_{a} \right) \c \nabc \mathfrak{F}+W_a^{\Xfr} \DDc \hot \mathfrak{X}  + W_{a}^{\Bfr} \DDc\hot \Bfr\\
&&+ ( W_0^{\Ffr}-Z^{\Ffr}_0)  \mathfrak{F} +W_0^{\Bfr} \hot  \mathfrak{B}+ W_0^{\Xfr} \hot \mathfrak{X}  \Big]
\eeaa
where
\begin{itemize}
\item $W_4^{\Ffr}$ and $W_a^{\Xfr}$ are \textbf{real} functions,
\item  $Z^{\Bfr}_{a}$ and $(W_a^{\Ffr}  -Z^{\Ffr}_{a})$ are \textbf{real} one-forms, 
\item $W_{a}^{\Bfr}$ and $Y^{\Ffr}_{a}$ are \textbf{imaginary} functions given by $W_{a}^{\Bfr}= \frac 3 4  i  \atrchb $ and $Y^{\Ffr}_{a}= -3  i \atrchb $,
\end{itemize}
and $Y^{\Bfr}_0$, $Z^{\Bfr}_0$, $Y^{\Xfr}_0$, $( W_0^{\Ffr}-Z^{\Ffr}_0) $ are complex functions, and $Y^{\Ffr}_0$, $W_0^{\Bfr}$, $W_0^{\Xfr} $ are complex one-forms, all of which vanish for zero angular momentum.
\end{itemize}

We call the system of equations \eqref{final-eq-1}-\eqref{final-eq-2} a system of  \textbf{generalized Regge-Wheeler equations}.
\end{theorem}

We now remark what are the crucial structures of the system of generalized Regge-Wheeler equations \eqref{final-eq-1}--\eqref{final-eq-2} which make them analyzable in physical space. 
\begin{enumerate}
\item The only first order terms present in both equations is of the form $i  \nab_t$, as in the generalized Regge-Wheeler equation obtained in Kerr \cite{ma2}, \cite{TeukolskyDHR}, \cite{GKS}. Such first order term has good divergence properties in the derivation of the energy estimates. Schematically, when multiplying equation \eqref{final-eq-1} by $\nab_t \ov{\pf}$ and taking the real part, one obtains a cancellation from the first order term:
\beaa
i \nab_t \pf \c \nab_t \ov{\pf}+ \overline{i \nab_t \pf} \c \nab_t \pf= i \nab_t \pf \c \nab_t \ov{\pf}- i \nab_t\ov{\pf} \c \nab_t \pf=0
\eeaa
This allows to derive the energy estimates without loss of derivatives. Similarly for equation \eqref{final-eq-2}. 
\item The reality of the potentials $V_1$ and $V_2$ is also crucial in the derivation of the estimates. When deriving the energy estimates and multiplying equation \eqref{final-eq-1} by $\nab_t \ov{\pf}$ and taking the real part, one obtains
\beaa
V_1 \pf \c \nab_t \ov{\pf}+ \overline{V_1 \pf} \c \nab_t \pf&=&\Re(V_1)(  \pf \c \nab_t \ov{\pf}+ \overline{\pf} \c \nab_t \pf)+ i \Im(V_1)( \pf \c \nab_t \ov{\pf}- i \overline{\pf} \c \nab_t \pf)\\
&=& \frac 1 2 \Re(V_1)\nab_t(| \pf|^2) + i \Im(V_1)( \pf \c \nab_t \ov{\pf}- i \overline{\pf} \c \nab_t \pf)
\eeaa
If the imaginary part  of the potential is not zero, then the last term cannot be written as a boundary term, and the energy estimates cannot be closed. In addition, the positivity of the real part of the potentials give positive contribution to the energy in the boundary terms. Similarly for equation \eqref{final-eq-2}. 

\item The highest order coupling terms on the right hand side of the equations are of the form $\frac{\ov{q}^3}{|q|^5}(\DDov \c \qf^\F)$ and $-\frac{q^3}{|q|^5} (\DD\hot \pf)$, up to a multiplication by the positive constant $8Q^2$. In particular observe that the functions multiplying the operators $\DDov \c$ and $\DD\hot$ are complex conjugate. Such structure is crucial in the cancellation of those coupling terms once the estimates for the two equations are summed, since $\DDov \c$ and $\DD\hot$ are adjoint operators up to lower order terms, as obtained in Lemma \ref{lemma:adjoint-operators}.

Such lower order terms, together with the derivatives falling to the functions $\frac{\ov{q}^3 }{|q|^5} $ and $\frac{q^3 }{|q|^5} $, are crucial in treating the coupling terms $\frac{\ov{q}^3 }{|q|^5} \left(  \ov{\DD} \c  \qf^\F  \right)$ and $-\frac{q^3}{|q|^5} \left(  \DD \hot  \pf  -\frac 3 2 \left( H - \Hb\right)  \hot \pf \right) $, precisely  to cancel the term $-\frac 3 2 \left( H - \Hb\right)  \hot \pf $ in the estimates.

\item The first order operators $L_\pf[\Bfr, \Ffr] $ and $L_{\qf^\F}[\Bfr, \Ffr]$ contain terms which are lower order in differentiability with respect to $\pf$ and $\qf^\F$. 
In particular, in the derivation of the energy estimates it is crucial that the highest order terms relative to the corresponding equations have real coefficients. More precisely, in the equation for $\pf$, the operator $L_\pf[\Bfr, \Ffr] $ should have real coefficients for the quantities which are lower order with respect to $\pf$, i.e. $\nabc \Bfr$, while in the equation for $\qf^\F$, the operator $L_\qf^\F[\Bfr, \Ffr] $ should have real coefficients for the quantities which are lower order with respect to $\qf^\F$, i.e. $\nabc_4 \Ffr$ and $\nabc \Ffr$. 

On the other hand, the lower order terms which are coupled should cancel, similarly to the coupling terms above. It is therefore crucial to obtain a cancellation in the terms $W_{a}^{\Bfr}= \frac 3 4  i  \atrchb $ and $Y^{\Ffr}_{a}= -3  i \atrchb $.

\end{enumerate}

It is remarkable that a choice of complex functions $C_1$, $C_2$, $f_1$ and $f_2$ which realizes all the above exists and can be found. In particular, the freedom in the choice of these functions is not enough to impose each one of the above conditions. We instead will prove the Theorem by imposing condition 1 (i.e. the only first order term is of the form $i\nab_t$) and condition 2 (i.e. the potentials are real), and this will uniquely determine the functions $C_1$, $C_2$, $f_1$ and $f_2$. We then show that with those choices, conditions 3 and 4 are also satisfied. 

\medskip
We summarize here the main steps of the proof.
\begin{enumerate}
\item[(a)] We compute the commutator between the Chandrasekhar operator $\PP_C$ and the Teukolsky operators $\TT_1$ and $\TT_2$. In order to cancel the lower order terms in the commutator, we impose conditions on the real part of the functions $C_1$ and $C_2$, and obtain
\beaa
\Re(C_1)=2\trchb, \qquad \Re(C_2)=\trchb.
\eeaa
This is done in Section \ref{comm-section}. With such choice we can compute the wave equations for the Chandrasekhar-transformed $\Pfr$ and $\Qfr$ quantities, in Section \ref{sec:wave-Pfr-Qfr}.
\item[(b)] We compute the effect on the wave equations of the rescaling of $\Pfr$ and $\Qfr$ through functions $f_1$ and $f_2$. In order to get only first order terms of the form $i \nab_t$ (condition 1), we impose conditions on functions $f_1$ and $f_2$, and obtain
\beaa
f_1=(q)^{1/2} (\ov{q})^{9/2}, \qquad f_2=q \ov{q}^2.
\eeaa
This is done in Section \ref{sec:rescaling-f}. 

\item[(c)] We compute the right hand side of the respective equations, and show that the above choice of $f_1$ and $f_2$ implies the structure of the higher coupling terms (condition 3). This is done in Section \ref{rhs-section}. 

\item[(d)] We compute the potentials of the two equations, and impose the vanishing of their imaginary parts. This uniquely determines the imaginary parts of the functions $C_1$ and $C_2$, giving
\beaa
\Im(C_1)=-\frac 5 2 \atrchb, \qquad \Im(C_2)=-3i \atrchb. 
\eeaa
This is done in Section \ref{sec:potentials-equations}. 

\item[(e)] Finally, we compute the lower order terms in $L_\pf[\Bfr, \Ffr] $ and $L_{\qf^\F}[\Bfr, \Ffr]$ and show that the above choices for $C_1$ and $C_2$ imply the reality or the cancellation of the relevant coefficients. This is done in Section \ref{sec:lower-order-terms}. 
\end{enumerate}

\subsection{Proof of Theorem \ref{main-theorem-RW}}

In this section, we derive the proof of Theorem  \ref{main-theorem-RW} while relying on the computations in the Appendix.
Recall the Teukolsky equations \eqref{Teukolsky-B} and \eqref{Teukolsky-F} in Theorem \ref{Teukolsky-equations-theorem}, i.e.
\beaa
\TT_1(\mathfrak{B})&=&\M_1[\mathfrak{F}, \mathfrak{X}] \\
\TT_2(\mathfrak{F})&=&\M_2[A, \mathfrak{X}, \mathfrak{B}]. 
\eeaa
We apply the Chandrasekhar operators $\PP_{C_1}$ and $\PP_{C_2}$, for $C_1$ and $C_2$ to be determined, to the above equations respectively. Recalling that $\mathfrak{P}:=\mathcal{P}_{C_1}(\mathfrak{B})$ and $\mathfrak{Q}:=\mathcal{P}_{C_2}(\mathfrak{F})$, we obtain
\bea
\TT_1(\mathfrak{P})+[\mathcal{P}_{C_1}, \TT_1](\mathfrak{B})&=&\mathcal{P}_{C_1}\Big(\M_1[\mathfrak{F}, \mathfrak{X}]\Big) \label{intermediate-RW-B} \\
\TT_2(\mathfrak{Q})+[\mathcal{P}_{C_2}, \TT_2](\mathfrak{F})&=&\mathcal{P}_{C_2}\Big(\M_2[A, \mathfrak{X}, \mathfrak{B}] \Big)\label{intermediate-RW-F-A}
\eea

\subsubsection{The commutators $[\PP_C, \TT]$}\label{comm-section}

We compute the commutators between the Teukolsky operators $\TT_1$ and $\TT_2$ and the first order differential operator $\PP_C$ as defined in \eqref{definition-P-C} for any scalar function $C$. In order to eliminate the highest order terms which cannot be expressed in terms of $\Pfr$ or $\Qfr$ (i.e. $\nabc_4 \Bfr$ and $\nabc_4 \Ffr$), we need to impose conditions on the real part of the functions $C_1$ and $C_2$. We obtain the following.

\begin{proposition}\label{proposition-commutator}\label{proposition-commutator-F-A} 
Let $\mathfrak{P}=\mathcal{P}_{C_1}(\mathfrak{B})= \nabc_3 \Bfr + C_1  \Bfr$ and $\mathfrak{Q}=\mathcal{P}_{C_2}(\mathfrak{F})= \nabc_3 \Ffr + C_2  \Ffr$, such that $C_1$ and $C_2$ satisfy respectively
\bea
 \nabc_3C_1+\frac {1}{ 2}\left(\tr \Xb +\ov{\tr\Xb}\right)C_1 -  \tr\Xb\ov{\tr\Xb}=0 \label{transport-nabc-3-C1}\\
  \nabc_3 C_2+\frac {1}{ 2}\left(\tr \Xb +\ov{\tr\Xb}\right) C_2-   \frac 1 2 \tr \Xb \ov{\tr\Xb} =0 \label{transport-nabc-3-C2}
\eea
Then the commutators between the Chandrasekhar operators $\mathcal{P}_{C_1}$ and $\PP_{C_2}$ and the Teukolsky operators $\TT_1$ and $\TT_2$ are respectively given by
\beaa
\, [\mathcal{P}_{C_1}, \TT_1](\mathfrak{B})&=&2 \etab \c \nab\mathfrak{P} - \frac 1 2 \left(\tr \Xb +\ov{\tr\Xb}\right)  \nab_4\mathfrak{P}  +\hat{V}_1   \mathfrak{P}- \frac 1 2 \left(\tr \Xb +\ov{\tr\Xb}\right)\M_1[\mathfrak{F}, \mathfrak{X}]-L_{\Pfr}[\Bfr, \Ffr], \\
\, [\mathcal{P}_{C_2}, \TT_2](\mathfrak{F})&=&2 \etab \c \nab\mathfrak{Q} - \frac 1 2 \left(\tr \Xb +\ov{\tr\Xb}\right)  \nab_4\mathfrak{Q}  + \hat{V}_2  \mathfrak{Q}- \frac 1 2 \left(\tr \Xb +\ov{\tr\Xb}\right) \M_2[A, \mathfrak{X}, \mathfrak{B}] -L_{\Qfr}[\Bfr, \Ffr] ,
\eeaa
where
\begin{itemize}
\item the potentials $\hat{V}_1$ and $\hat{V}_2$ are given by
\bea
\hat{V}_1&=& I^{\Bfr}_{3} +J^{\Bfr}_3 +K^{\Bfr}_3+M^{\Bfr}_3=  - \frac 5 2 \trch \trchb -4\rho -2 \rhoF^2+ O\left(\frac{|a|}{r^3} \right) \label{V1hat}\\
\hat{V}_2&=& I^{\Ffr}_{3} +J^{\Ffr}_3 +K^{\Ffr}_3+M^{\Ffr}_3=  - \frac 3 2 \trch \trchb -4\rho -2 \rhoF^2+ O\left(\frac{|a|}{r^3} \right) \label{V2hat}
\eea
where the precise coefficients are given in Appendix \ref{proof-7.1-section}, 
\item $L_{\Pfr}[\Bfr, \Ffr]$ and $L_{\Qfr}[\Bfr, \Ffr]$ are linear first order operators in $\Bfr$ and $\Ffr$, given by
\bea
L_{\Pfr}[\Bfr, \Ffr]&:=&-Z^{\Bfr}_{a} \c \nabc \Bfr-\left( 4\trchb \rhoF^2 +Z^{\Bfr}_0\right) \Bfr, \label{eq:L-Pfr-lot}\\
L_{\Qfr}[\Bfr, \Ffr]&:=&-Z^{\Ffr}_{a} \c \nabc \Ffr+\left(4\trchb \rhoF^2-Z^{\Ffr}_0\right)  \Ffr \label{eq:L-Qfr-lot}
\eea
where $Z^{\Bfr}_{a}$ and $Z^{\Ffr}_{a}$ are complex one-forms and $Z^{\Bfr}_0$ and $Z^{\Bfr}_0$ are complex functions of $(r, \th)$, all of which vanish for zero angular momentum, having the following fall-off in $r$:
\beaa
Z^{\Bfr}_{a}, Z^{\Ffr}_{a}=O\left(\frac{|a|}{r^3} \right), \qquad Z^{\Bfr}_{0}, Z^{\Ffr}_{0}=O\left(\frac{|a|}{r^4} \right)
\eeaa
\end{itemize}
\end{proposition}
\begin{proof} See Appendix \ref{sec:proof-commutators}. 
\end{proof}

Observe that the transport equations \eqref{transport-nabc-3-C1} and \eqref{transport-nabc-3-C1} only impose conditions on the real parts of $C_1$ and $C_2$. Indeed, for any real constants $p_1, p_2$, the scalar functions $C_1$ and $C_2$ given by
\bea
C_1=2\trchb + i p_1 \atrchb, \qquad C_2=\trchb + i p_2 \atrchb \label{definition-C1-C22}\label{definition-C1-C2}
\eea
are of conformal type $-1$ and satisfy \eqref{transport-nabc-3-C1} and \eqref{transport-nabc-3-C1} respectively.

\subsubsection{The wave equations for $\Pfr$ and $\Qfr$}\label{sec:wave-Pfr-Qfr}

Recall the Teukolsky operators $\TT_1$ and $\TT_2$, see \eqref{operator-Teukolsky-B} and \eqref{operator-Teukolsky-F}. We then obtain for $\Pfr$ and $\Qfr$ respectively,
  \beaa
 \TT_1(\Pfr)&=&   -\nab_3\nab_4\Pfr+ \frac 1 2 \ov{\DD}\c (\DD \hot \Pfr)-3\ov{\tr X} \nab_3\Pfr- \left(\frac{3}{2}\tr\Xb +\frac 1 2\ov{\tr\Xb}-2\omb\right)\nab_4\Pfr\\
   &&+\left( 6H+ \ov{H}+ 3  \ov{ \Hb}  \right)\c  \nab  \Pfr+\left(-\frac{9}{2}\tr\Xb \ov{\tr X} -4 \PF \ov{\PF}+9 \ov{ \Hb} \c  H\right)\Pfr\\
\TT_2(\Qfr)&=& -\nab_3\nab_4 \Qfr+ \frac 1 2 \DD \hot (\ov{\DD} \c \Qfr)-\left(\frac 3 2 \ov{\tr X} +\frac 1 2  \tr X\right)\nab_3\Qfr- \frac 1 2 \left(\tr \Xb+\ov{\tr \Xb} -4\omb \right)\nab_4\Qfr  \\
&&+ \left(4   H+ \ov{H}+  \Hb \right)\c \nab \Qfr\\
&&+\left(- \frac 3 4 \tr \Xb  \ov{\tr X}- \frac 1 4\ov{\tr \Xb}   \tr X+3\ov{P} -P +4\PF\ov{\PF} - \frac 3 2\ov{\DDc}\c H+\eta \c \etab +i \eta \wedge \etab\right)\Qfr
\eeaa
where recall that $\Pfr$ and $\Qfr$ are both of conformal type $0$.

From the commuted equations \eqref{intermediate-RW-B} and \eqref{intermediate-RW-F-A} and using the formulas for the commutators given by Proposition \ref{proposition-commutator}, we obtain respectively, by writing $2\etab=\Hb+\ov{\Hb}$:
\beaa
&&  -\nab_3\nab_4\Pfr+ \frac 1 2 \ov{\DD}\c (\DD \hot \Pfr)-3\ov{\tr X} \nab_3\Pfr- \left(2\tr\Xb +\ov{\tr\Xb}-2\omb\right)\nab_4\Pfr\\
   &&+\left( 6H+ \ov{H}+\Hb+ 4  \ov{ \Hb}  \right)\c  \nab  \Pfr+\left(-\frac{9}{2}\tr\Xb \ov{\tr X} -4 \PF \ov{\PF}+9 \ov{ \Hb} \c  H+\hat{V}_1\right)\Pfr\\
   &=&\mathcal{P}_{C_1}\Big(\M_1[\mathfrak{F}, \mathfrak{X}]\Big) + \frac 1 2 \left(\tr \Xb +\ov{\tr\Xb}\right)\M_1[\mathfrak{F}, \mathfrak{X}]+L_{\Pfr}[\Bfr, \Ffr]
\eeaa
and
\beaa
&& -\nab_3\nab_4 \Qfr+ \frac 1 2 \DD \hot (\ov{\DD} \c \Qfr)-\left(\frac 3 2 \ov{\tr X} +\frac 1 2  \tr X\right)\nab_3\Qfr- \left(\tr \Xb+\ov{\tr \Xb} -2\omb \right)\nab_4\Qfr  \\
&&+ \left(4   H+ \ov{H}+ 2 \Hb +\ov{\Hb}\right)\c \nab \Qfr\\
&&+\left(- \frac 3 4 \tr \Xb  \ov{\tr X}- \frac 1 4\ov{\tr \Xb}   \tr X+3\ov{P} -P +4\PF\ov{\PF} - \frac 3 2\ov{\DDc}\c H+\eta \c \etab +i \eta \wedge \etab+ \hat{V}_2\right)\Qfr\\
&=&\mathcal{P}_{C_2}\Big(\M_2[A, \mathfrak{X}, \mathfrak{B}] \Big)   + \frac 1 2 \left(\tr \Xb +\ov{\tr\Xb}\right) \M_2[A, \mathfrak{X}, \mathfrak{B}] +L_{\Qfr}[\Bfr, \Ffr]
\eeaa
Using the formulas for the wave operator according to Lemma \ref{lemma:expression-wave-operator}, 
 \beaa
 \squared_1 \Pfr&=&- \nab_3\nab_4\Pfr+\frac 1 2    \DDb \c ( \DD \hot \Pfr) +\left(2\omb -\frac 1 2\ov{\tr\Xb} \right) \nab_4\Pfr -\frac 1 2 \ov{\tr X} \nab_3\Pfr+ (H+ \ov{H}) \c\nab \Pfr \\
 && +  \left(  \frac 14 \trch\trchb+\frac 1 4 \atrch\atrchb+ \rho-\rhoF^2-\dual\rhoF^2+i \left(- \rhod+ \eta \wedge \etab  \right)\right)  \Pfr\\
 \squared_2 \Qfr &=&- \nab_3\nab_4 \Qfr +\frac 1 2  \DD \hot (\DDb \c  \Qfr ) +\left(2\omb -\frac 1 2\tr\Xb \right) \nab_4 \Qfr  -\frac 1 2\tr X \nab_3 \Qfr + (H+\ov{H}) \c\nab  \Qfr \\
 &&+ \left(-\frac 1 2  \trch\trchb- \frac 1 2 \atrch\atrchb-2\rho+2\rhoF^2+2\dual\rhoF^2+i \left(- 2\rhod+2 \eta \wedge \etab  \right)\right)   \Qfr 
 \eeaa
we can rewrite the above as 
\bea\label{intermediate-square-mathfrak-P}
\begin{split}
  \squared_1 \Pfr   &= \frac 5 2 \ov{\tr X} \nab_3\Pfr+ \left(2\tr\Xb +\frac 1 2 \ov{\tr\Xb}\right)\nab_4\Pfr-\left( 5H+\Hb+ 4  \ov{ \Hb}  \right)\c  \nab  \Pfr+\tilde{V}_1\Pfr\\
 &\mathcal{P}_{C_1}\Big(\M_1[\mathfrak{F}, \mathfrak{X}]\Big) + \frac 1 2 \left(\tr \Xb +\ov{\tr\Xb}\right)\M_1[\mathfrak{F}, \mathfrak{X}]+L_{\Pfr}[\Bfr, \Ffr]
 \end{split}
\eea
where
\bea\label{tilde-V1}
\begin{split}
\tilde{V}_1&= \frac{9}{2}\tr\Xb \ov{\tr X} +4 \PF \ov{\PF}-9 \ov{ \Hb} \c  H-\hat{V}_1\\
&+ \frac 14 \trch\trchb+\frac 1 4 \atrch\atrchb+ \rho-\rhoF^2-\dual\rhoF^2+i \left(- \rhod+ \eta \wedge \etab  \right)
\end{split}
\eea
and 
\bea\label{intermediate=square2Q}
\begin{split}
  \squared_2 \Qfr&= \frac 3 2 \ov{\tr X} \nab_3\Qfr+ \left(\frac 1 2 \tr \Xb+\ov{\tr \Xb}  \right)\nab_4\Qfr - \left(3   H+ 2 \Hb +\ov{\Hb}\right)\c \nab \Qfr+\tilde{V}_2  \Qfr \\
&+\mathcal{P}_{C_2}\Big(\M_2[A, \mathfrak{X}, \mathfrak{B}] \Big)   + \frac 1 2 \left(\tr \Xb +\ov{\tr\Xb}\right) \M_2[A, \mathfrak{X}, \mathfrak{B}] +L_{\Qfr}[\Bfr, \Ffr]
\end{split}
\eea
where
\bea\label{tilde-V2}
\begin{split}
\tilde{V}_2&= \frac 3 4 \tr \Xb  \ov{\tr X}+ \frac 1 4\ov{\tr \Xb}   \tr X-3\ov{P} +P -4\PF\ov{\PF} + \frac 3 2\ov{\DDc}\c H-\eta \c \etab -i \eta \wedge \etab- \hat{V}_2\\
&-\frac 1 2  \trch\trchb- \frac 1 2 \atrch\atrchb-2\rho+2\rhoF^2+2\dual\rhoF^2+i \left(- 2\rhod+2 \eta \wedge \etab  \right)
\end{split}
\eea

\subsubsection{The rescaling from $\Pfr$ to $\pf$ and from $\Qfr$ to $\qf^\F$}\label{sec:rescaling-f}

Observe that the wave equations \eqref{intermediate-square-mathfrak-P} and \eqref{intermediate=square2Q} satisfied by $\Pfr$ and $\Qfr$ present first order derivatives $\nab_3$, $\nab_4$ and $\nab$ on their right hand side. In order to have only a first order term of the form $i\nab_t$, we need to define rescaled versions of $\Pfr$ and $\Qfr$. The rescaling is obtained through functions of $q=r+i a \cos\th$ and $\ov{q}=r-ia \cos\th$, i.e.
\beaa
\pf&=& f_1(q, \ov{q}) \Pfr \in \sk_1(\CCC), \qquad \qf^\F= f_2(q, \ov{q}) \Qfr \in \sk_2(\CCC)
\eeaa

\begin{proposition}\label{prop:rescaling-f} Let $f_1$ and $f_2$ be of the respective forms 
\beaa
f_1&=&(q)^{n_1} (\ov{q})^{5-n_1}, \qquad \text{for any real $n_1$} \\
f_2&=& (q)^{n_2} (\ov{q})^{3-n_2}, \qquad \text{for any real $n_2$},
\eeaa
Then
 \beaa
\squared_1\pf&=&i f_1 \Big[ \frac{2a\cos\th}{|q|^2}\nab_t \Pfr +(1 - 2n_1 )\Big(  \frac{2a\Delta\cos\th}{|q|^4}\nab_{r}\Pfr+   \frac{2a\sin\th r}{|q|^4}   \nab_\th \Pfr\Big) \Big]\\
&&+\left( \tilde{V}_1+ f_1^{-1}\square(f_1)\right) \pf + f_1 \Big[\mathcal{P}_{C_1}\Big(\M_1[\mathfrak{F}, \mathfrak{X}]\Big) + \frac 1 2 \left(\tr \Xb +\ov{\tr\Xb}\right)\M_1[\mathfrak{F}, \mathfrak{X}]+L_{\Pfr}[\Bfr, \Ffr] \Big]
\eeaa
and
   \beaa
\squared_2\qf^\F&=&i f_2 \Big[ \frac{4a\cos\th}{|q|^2}\nab_t \Qfr +(1 - n_2 )\Big(  \frac{4a\Delta\cos\th}{|q|^4}\nab_{r}\Qfr+   \frac{4a\sin\th r}{|q|^4}   \nab_\th \Qfr\Big) \Big]\\
&&+\left( \tilde{V}_2+ f_2^{-1}\square(f_2)\right) \qf^\F + f_2 \Big[\mathcal{P}_{C_2}\Big(\M_2[A, \mathfrak{X}, \mathfrak{B}] \Big)   + \frac 1 2 \left(\tr \Xb +\ov{\tr\Xb}\right) \M_2[A, \mathfrak{X}, \mathfrak{B}] +L_{\Qfr}[\Bfr, \Ffr]  \Big]
 \eeaa
 In particular, observe that for $n_1=\frac 1 2$ and $n_2=1$, i.e. 
 \bea\label{eq:definition-f1-f2}
 f_1=(q)^{1/2} (\ov{q})^{9/2}, \qquad f_2=q \ov{q}^2,
 \eea
  the dependence on the $\nab_r$ and $\nab_\th$ derivatives cancels out and we obtain
 \bea\label{intermediate-1-pf}
 \begin{split}
 \squared_1\pf&=i  \frac{2a\cos\th}{|q|^2}\nab_t \pf  +\left( \tilde{V}_1+ f_1^{-1}\square(f_1)\right) \pf \\
 &+ f_1 \Big[\mathcal{P}_{C_1}\Big(\M_1[\mathfrak{F}, \mathfrak{X}]\Big) + \frac 1 2 \left(\tr \Xb +\ov{\tr\Xb}\right)\M_1[\mathfrak{F}, \mathfrak{X}]+L_{\Pfr}[\Bfr, \Ffr] \Big]
 \end{split}
 \eea
 and 
    \bea\label{intermediate-1-qf}
    \begin{split}
\squared_2\qf^\F&=i  \frac{4a\cos\th}{|q|^2}\nab_t \qf^\F +\left( \tilde{V}_2+ f_2^{-1}\square(f_2)\right) \qf^\F \\
&+ f_2 \Big[\mathcal{P}_{C_2}\Big(\M_2[A, \mathfrak{X}, \mathfrak{B}] \Big)   + \frac 1 2 \left(\tr \Xb +\ov{\tr\Xb}\right) \M_2[A, \mathfrak{X}, \mathfrak{B}] +L_{\Qfr}[\Bfr, \Ffr]  \Big]
\end{split}
 \eea
\end{proposition}
\begin{proof} See Appendix \ref{section-proof-rescaling}.  \end{proof}

\subsubsection{The right hand side of the equations}\label{rhs-section}

\begin{proposition}\label{right-hand-side-lemma}  Let $\M_1[\mathfrak{F}, \mathfrak{X}]$ and $\M_2[A, \mathfrak{X}, \mathfrak{B}]$ be the right hand sides of the Teukolsky equations, as defined in \eqref{definition-MM1} and \eqref{definition-MM2}.
Let $\PP_{C_1}$ and $\PP_{C_2}$ be the operators defined in \eqref{definition-P-C} with $C_1$ and $C_2$ given by \eqref{definition-C1-C2} and \eqref{definition-C1-C22}.
Then the following relations hold true:
\beaa
\mathcal{P}_{C_1}\Big(\M_1[\mathfrak{F}, \mathfrak{X}]\Big)+ \frac 1 2 \left(\tr \Xb +\ov{\tr\Xb}\right)\M_1[\mathfrak{F}, \mathfrak{X}]&=& 4 \PF\ov{\PF} \left( \ov{\DD}\c \mathfrak{Q}+\left( 2 \ov{\Hb}+\ov{H} \right) \c \mathfrak{Q}\right)\\
&&+L_{\M_1}[\Bfr, \Ffr, \Xfr]\\
\mathcal{P}_{C_2}\Big(\M_2[A, \mathfrak{X}, \mathfrak{B}] \Big)+\frac 1 2 \left(\tr \Xb +\ov{\tr\Xb}\right) \M_2[A, \mathfrak{X}, \mathfrak{B}]&=&\left(3 \ov{P}+2\PF\ov{\PF}\right)\mathfrak{Q}\\
&&-   \frac 1 2 \left( \DD \hot \mathfrak{P}  + (3H+2\Hb)  \hot  \mathfrak{P}\right)\\
&& +L_{\M_2}[\Bfr, \Ffr, \Xfr]
\eeaa
where $L_{\M_1}[\Ffr, \Xfr]$ and $L_{\M_2}[\Bfr, \Ffr, \Xfr]$ are linear first order operator in $\Bfr$, $\Ffr$ and $\Xfr$, given by
\bea\label{eq:L=M1-lot}
L_{\M_1}[\Bfr, \Ffr, \Xfr]&:=&( 4 \trchb \rhoF^2 )\ \Bfr+ 2 \PF\ov{\PF}\left( Y^{\Ffr}_{a} \ov{\DDc}\c\mathfrak{F}+Y^{\Ffr}_0 \c \mathfrak{F}+Y^{\Bfr}_0 \Bfr+Y^{\Xfr}_0 \ \mathfrak{X}\right)
\eea
and
\bea\label{eq:L=M2-lot}
\begin{split}
L_{\M_2}[\Bfr, \Ffr, \Xfr]&:=  - ( 4 \trchb \rhoF^2 ) \ \mathfrak{F}  +W_4^{\Ffr} \nabc_4 \mathfrak{F}+W_a^{\Ffr}  \c \nabc \mathfrak{F} + W_0^{\Ffr} \mathfrak{F}\\
&+ W_{a}^{\Bfr} \DDc\hot \Bfr +W_0^{\Bfr} \hot  \mathfrak{B}+W_a^{\Xfr} \DDc \hot \mathfrak{X} + W_0^{\Xfr} \hot \mathfrak{X}  
\end{split}
\eea
where $Y^{\Ffr}_{a}$, $Y^{\Bfr}_0$ and $Y^{\Xfr}_0 $ are complex functions of $(r, \th)$ and $Y^{\Ffr}_0$ is a complex one-form, all of which vanish for zero angular momentum, having the following fall-off in $r$:
\beaa
Y^{\Ffr}_{a}, Y^{\Bfr}_0=O\left(\frac{|a|}{r^2} \right), \qquad Y^{\Ffr}_0, Y^{\Xfr}_0=O\left(\frac{|a|}{r^3} \right)
\eeaa
and $W_4^{\Ffr}$, $W_0^{\Ffr}$, $W_{a}^{\Bfr} $ and $W_a^{\Xfr}  $ are complex functions of $(r, \th)$ and $W_a^{\Ffr} $, $W_0^{\Bfr}$ and $W_0^{\Xfr}$ are complex one-forms, all of which vanish for zero angular momentum, having the following fall-off in $r$:
\beaa
W_{a}^{\Bfr}=O\left(\frac{|a|}{r^2} \right), \qquad W_4^{\Ffr}, W_a^{\Ffr},W_0^{\Bfr}, W_a^{\Xfr}=O\left(\frac{|a|}{r^3} \right), \qquad  W_0^{\Ffr}, W_0^{\Xfr}=O\left(\frac{|a|}{r^4} \right)
\eeaa
\end{proposition}
\begin{proof} See Appendix \ref{section-proof-lemma-rhs-1}.
\end{proof}

Using \eqref{intermediate-1-pf} and \eqref{intermediate-1-qf} and the above Proposition, we deduce
 \bea
 \begin{split}
 \squared_1\pf&=i  \frac{2a\cos\th}{|q|^2}\nab_t \pf  +\left( \tilde{V}_1+ f_1^{-1}\square(f_1)\right) \pf \\
 &+ f_1 \Big[4 \PF\ov{\PF} \left( \ov{\DD}\c \mathfrak{Q}+\left( 2 \ov{\Hb}+\ov{H} \right) \c \mathfrak{Q}\right)+L_{\M_1}[\Bfr, \Ffr, \Xfr]+L_{\Pfr}[\Bfr, \Ffr] \Big]
 \end{split}
 \eea
 and 
    \bea
    \begin{split}
\squared_2\qf^\F&=i  \frac{4a\cos\th}{|q|^2}\nab_t \qf^\F +\left( \tilde{V}_2+ f_2^{-1}\square(f_2)+3 \ov{P}+2\PF\ov{\PF}\right) \qf^\F \\
&+ f_2 \Big[-   \frac 1 2 \left( \DD \hot \mathfrak{P}  + (3H+2\Hb)  \hot  \mathfrak{P}\right) +L_{\M_2}[\Bfr, \Ffr, \Xfr]+L_{\Qfr}[\Bfr, \Ffr]  \Big]
\end{split}
 \eea

We are now left to express the right hand side in terms of $\pf=f_1 \Pfr$ and $\qf^\F=f_2\Qfr$. We write
\beaa
 \ov{\DDc} \c \Qfr&=&f_2^{-1}( \ov{\DDc} \c  \qf^\F) + \ov{\DDc}(f_2^{-1}) \c  \qf^\F=f_2^{-1}( \ov{\DDc} \c  \qf^\F) -f_2^{-2} \ov{\DDc}(f_2) \c  \qf^\F\\
 \DDc \hot \Pfr&=&f_1^{-1}(  \DDc \hot  \pf) +\DDc(f_1^{-1}) \hot  \pf=f_1^{-1}(  \DDc \hot  \pf) -f_1^{-2}\DDc(f_1) \hot  \pf
\eeaa
This implies
\bea\label{almost-final-squared-pf}
 \begin{split}
 \squared_1\pf&=i  \frac{2a\cos\th}{|q|^2}\nab_t \pf  +V_1 \pf + 4 \PF\ov{\PF}( f_1 f_2^{-1} )\left(  \ov{\DD} \c  \qf^\F +\left( 2 \ov{\Hb}+\ov{H} - f_2^{-1}\ov{\DDc}(f_2) \right) \c  \qf^\F \right) + L_\pf[\Bfr, \Ffr] 
 \end{split}
 \eea
 and 
    \bea\label{almost-final-squared-qf}
    \begin{split}
\squared_2\qf^\F&=i  \frac{4a\cos\th}{|q|^2}\nab_t \qf^\F +V_2 \qf^\F -   \frac 1 2( f_2f_1^{-1}) \left(  \DD \hot  \pf + \left(3H+2\Hb-f_1^{-1} \DDc(f_1)\right)  \hot \pf \right) + L_{\qf^\F}[\Bfr, \Ffr] 
\end{split}
 \eea
 where we define
      \bea
        V_1&:=&  \tilde{V}_1+ f_1^{-1}\square(f_1) \label{eq:definition-V1}\\
        V_2&:=&  \tilde{V}_2+ f_2^{-1}\square(f_2)+3 \ov{P}+2\PF\ov{\PF}\label{eq:definition-V2}
        \eea
        and
 \bea
L_\pf[\Bfr, \Ffr] &:=& f_1 \Big[L_{\M_1}[\Bfr, \Ffr, \Xfr]+L_{\Pfr}[\Bfr, \Ffr] \Big]  \label{eq:def-Lpf}\\
L_{\qf^\F}[\Bfr, \Ffr] &:=& f_2 \Big[ L_{\M_2}[\Bfr, \Ffr, \Xfr]+L_{\Qfr}[\Bfr, \Ffr]  \Big]\label{eq:def-Lqf}
\eea
We now simplify the coupling terms on the right hand sides of the above. From \eqref{transport-for-q}, we deduce
\bea\label{DD-q-n-ov-q-m}
\begin{split}
\DD( q^n \ov{q}^m)&= n q^{n-1} (\DD q ) \ov{q}^m+m q^n  \ov{q}^{m-1} (\DD\ov{q})=(  n \Hb +m H) q^n \ov{q}^m\\
\ov{\DD}(q^n \ov{q}^m )&=(  m \ov{\Hb} +n \ov{H}) q^n\ov{q}^m 
\end{split}
\eea
We therefore have, for $f_1=(q)^{1/2} (\ov{q})^{9/2}$, and  $f_2=q \ov{q}^2$:
\beaa
2 \ov{\Hb}+\ov{H}-f_2^{-1} \ov{\DDc}(f_2)&=& 2 \ov{\Hb}+\ov{H}- (  2 \ov{\Hb} + \ov{H}) =0, \\
3H+2\Hb-f_1^{-1} \DDc(f_1)&=& 3H+2\Hb- (  \frac 1 2  \Hb +\frac 9 2  H)= -\frac 3 2 H +\frac 3 2 \Hb
\eeaa
We also write
\beaa
4 \PF\ov{\PF}( f_1 f_2^{-1} )&=& 4 \frac{Q}{\ov{q}^2}  \frac{Q}{q^2}(q)^{1/2} (\ov{q})^{9/2} (q)^{-1} (\ov{q})^{-2}= 4Q^2 \frac{\ov{q}^{1/2}}{q^{5/2}}=\frac{ 4Q^2}{|q|^5} \ov{q}^3 \\
( f_2f_1^{-1})&=& q \ov{q}^2(q)^{-1/2} (\ov{q})^{-9/2}=\frac{q^{1/2}}{\ov{q}^{5/2}}=\frac{q^3}{|q|^5}
\eeaa
We therefore finally obtain
\bea
 \squared_1\pf-i  \frac{2a\cos\th}{|q|^2}\nab_t \pf  -V_1  \pf &=&4Q^2 \frac{\ov{q}^3 }{|q|^5} \left(  \ov{\DD} \c  \qf^\F  \right) + L_\pf[\Bfr, \Ffr]  \label{final-squared-pf}\label{final-squared-pf-lot}\\
\squared_2\qf^\F-i  \frac{4a\cos\th}{|q|^2}\nab_t \qf^\F -V_2  \qf^\F &=&-   \frac 1 2\frac{q^3}{|q|^5} \left(  \DD \hot  \pf  -\frac 3 2 \left( H - \Hb\right)  \hot \pf \right) + L_{\qf^\F}[\Bfr, \Ffr]. \label{final-squared-qf}\label{final-squared-qf-lot}
 \eea

        \subsubsection{The potentials of the equations}\label{sec:potentials-equations}
        
   In this section we compute the potentials $V_1$ and $V_2$ as obtained in \eqref{eq:definition-V1} and \eqref{eq:definition-V2}. We determine the imaginary parts of the complex functions $C_1=2\trchb + i p_1 \atrchb$ and $C_2=\trchb + i p_2 \atrchb$ given in \eqref{definition-C1-C22} such that the imaginary part of the potentials vanish. 
   
   \begin{proposition}\label{prop:potentials-real} Choosing $p_1=-\frac 5 2$ and $p_2=-3$ in the definition of $C_1$ and $C_2$ \eqref{definition-C1-C22}, i.e. for
    \bea
C_1&=&2\trchb -\frac 5 2  i  \atrchb, \label{definition-C1-C2-final}\\
 C_2&=&\trchb -3 i \atrchb \label{definition-C1-C22-final}
\eea
the potentials $V_1$ and $V_2$ in equations \eqref{final-squared-pf} and \eqref{final-squared-qf} are real, i.e. $\Im(V_1)=\Im(V_2)=0$, and are given by
\beaa
V_1&=& -\frac {1}{ 4} \trch\trchb  + \div \etab +\frac 1 2  \eta \c \etab+\frac 1 2 |\etab|^2+5\rhoF^2+5\dual\rhoF^2  \\
V_2&=&- \trch\trchb+4\div \etab+2\eta \c \etab+2|\etab|^2+2\rhoF^2+2\dual\rhoF^2
\eeaa 
In particular,  modulo $O(|a|)$ terms, we have
\beaa
  V_1  &=& - \frac{1}{4}\trch\trchb +5 \rhoF^2   +O\left(\frac{|a|}{r^3} \right)\\
V_2&=& -\trch\trchb  +2\rhoF^2 +O\left(\frac{|a|}{r^3} \right)
\eeaa
   \end{proposition}
       \begin{proof} See Appendix \ref{sec:proof-potentials}. 
       \end{proof}
      
Using the values in Kerr-Newman given in Section \ref{sec:values-KN}, we obtain
\beaa
 -\frac {1}{ 4} \trch\trchb  + \div \etab +\frac 1 2  \eta \c \etab+\frac 1 2 |\etab|^2&=&  \frac{r^4-2Mr^3+(2-3\cos^2\th)a^2r^2+Q^2r^2-2a^4\cos^2\th}{|q|^6} 
\eeaa
and therefore, explicitly,
\bea\label{potentials-explicity}
\begin{split}
V_1&=\frac{r^4-2Mr^3+(2-3\cos^2\th)a^2r^2+Q^2r^2-2a^4\cos^2\th}{|q|^6} +\frac{5Q^2}{|q|^4}  \\
V_2&=4\frac{r^4-2Mr^3+(2-3\cos^2\th)a^2r^2+Q^2r^2-2a^4\cos^2\th}{|q|^6} +\frac{2Q^2}{|q|^4}.
\end{split}
\eea

\subsubsection{The lower order terms}\label{sec:lower-order-terms}

We finally simplify the lower order terms $L_\pf[\Bfr, \Ffr] $ and $L_{\qf^\F}[\Bfr, \Ffr] $ as defined in \eqref{eq:def-Lpf} and \eqref{eq:def-Lqf}. 

Using  \eqref{eq:L-Pfr-lot} and \eqref{eq:L=M1-lot}, we obtain:
\beaa
L_\pf[\Bfr, \Ffr] &=& f_1 \Big[L_{\M_1}[\Bfr, \Ffr, \Xfr]+L_{\Pfr}[\Bfr, \Ffr] \Big] \\
&=& q^{1/2} \ov{q}^{9/2} \Big[( 4 \trchb \rhoF^2 )\ \Bfr+ 2 \PF\ov{\PF}\left( Y^{\Ffr}_{a} \ov{\DDc}\c\mathfrak{F}+Y^{\Ffr}_0 \c \mathfrak{F}+Y^{\Bfr}_0 \Bfr+Y^{\Xfr}_0 \ \mathfrak{X}\right)\\
&&-Z^{\Bfr}_{a} \c \nabc \Bfr-\left( 4\trchb \rhoF^2 +Z^{\Bfr}_0\right) \Bfr\Big] \\
&=& q^{1/2} \ov{q}^{9/2} \Big[-Z^{\Bfr}_{a} \c \nabc \Bfr+ 2 \PF\ov{\PF} \ Y^{\Ffr}_{a} (\ov{\DDc}\c\mathfrak{F})\\
&&+ ( 2 \PF\ov{\PF}\ Y^{\Bfr}_0 - Z^{\Bfr}_0) \Bfr + 2 \PF\ov{\PF}\left( Y^{\Ffr}_0 \c \mathfrak{F}+Y^{\Xfr}_0 \ \mathfrak{X}\right)\Big] 
\eeaa
Using  \eqref{eq:L-Qfr-lot} and \eqref{eq:L=M2-lot}, we obtain:
\beaa
L_{\qf^\F}[\Bfr, \Ffr] &=& f_2 \Big[ L_{\M_2}[\Bfr, \Ffr, \Xfr]+L_{\Qfr}[\Bfr, \Ffr]  \Big]\\
&=& q\ov{q}^2 \Big[  - ( 4 \trchb \rhoF^2 ) \ \mathfrak{F}  +W_4^{\Ffr} \nabc_4 \mathfrak{F}+W_a^{\Ffr}  \c \nabc \mathfrak{F} + W_0^{\Ffr} \mathfrak{F}\\
&&+ W_{a}^{\Bfr} \DDc\hot \Bfr +W_0^{\Bfr} \hot  \mathfrak{B}+W_a^{\Xfr} \DDc \hot \mathfrak{X} + W_0^{\Xfr} \hot \mathfrak{X} -Z^{\Ffr}_{a} \c \nabc \Ffr+\left(4\trchb \rhoF^2-Z^{\Ffr}_0\right)  \Ffr  \Big]\\
&=& q\ov{q}^2 \Big[ W_4^{\Ffr} \nabc_4 \mathfrak{F}+\left( W_a^{\Ffr}  -Z^{\Ffr}_{a} \right) \c \nabc \mathfrak{F}+W_a^{\Xfr} \DDc \hot \mathfrak{X}  + W_{a}^{\Bfr} \DDc\hot \Bfr\\
&&+ ( W_0^{\Ffr}-Z^{\Ffr}_0)  \mathfrak{F} +W_0^{\Bfr} \hot  \mathfrak{B}+ W_0^{\Xfr} \hot \mathfrak{X}  \Big]
\eeaa

We now summarize in the following the structure of the above terms.
\begin{lemma}\label{lemma:lot-terms} With the above choices of $C_1$ and $C_2$, we have that the highest order terms in $L_\pf[\Bfr, \Ffr]$ and $L_{\qf^\F}[\Bfr, \Ffr]$ satisfies the following:
\begin{itemize}
\item $W_4^{\Ffr}$ and $W_a^{\Xfr}$ are real functions,
\item $Z^{\Bfr}_{a}$ and $W_a^{\Ffr}  -Z^{\Ffr}_{a}$ are real one-forms, 
\item $W_{a}^{\Bfr}= \frac 3 4  i  \atrchb $ and $Y^{\Ffr}_{a}= -3  i \atrchb $.
\end{itemize}
\end{lemma}
\begin{proof} See Appendix \ref{sec:proof-lemma-lot-terms}. 
\end{proof}

\subsection{Sketch of boundedness of the energy}\label{section-bdn-energy}

We here sketch how to prove boundedness of the energy for the system of equations \eqref{final-eq-1} and \eqref{final-eq-2} as obtained in Theorem \ref{main-theorem-RW}:
\beaa
 \squared_1\pf-i  \frac{2a\cos\th}{|q|^2}\nab_t \pf  -V_1  \pf &=&4Q^2 \frac{\ov{q}^3 }{|q|^5} \left(  \ov{\DD} \c  \qf^\F  \right) + L_\pf[\Bfr, \Ffr] \\
\squared_2\qf^\F-i  \frac{4a\cos\th}{|q|^2}\nab_t \qf^\F -V_2  \qf^\F &=&- \frac 1 2   \frac{q^3}{|q|^5} \left(  \DD \hot  \pf -\frac 3 2 \left(  H - \Hb\right)  \hot \pf \right)+ L_{\qf^\F} [\Bfr, \Ffr] 
 \eeaa
 To fully close energy estimates, we need to combine them with spacetime local integrated Morawetz estimates, which will be done in a future work \cite{Giorgi9} by making use of the hidden symmetry in Kerr-Newman to avoid decomposition in modes. Nevertheless, in this section we show that all the crucial structures obtained in Theorem \ref{main-theorem-RW} are precisely what one needs to perform energy estimates, once the Morawetz estimate, which is less sensitive to the structure of the lower order terms, are achieved. 
 
 As a general rule\footnote{In the case of Kerr-Newman, in the ergoregion we need to multiply by the timelike $\partial_t + \frac{a}{r^2+a^2} \partial_\vphi$. The analysis is identical since the term involving the $\partial_\vphi$ can be absorbed for small $a$ by the non-degenerate Morawetz estimates away from the trapping region. }, in order to obtain energy estimates for the wave equation $\square \psi=0$, we multiply the equation by $\nab_t\psi$, and integrate by parts. Since we are dealing with complex tensors, we then multiply the equation $\squared_1 \pf$ by $\nab_t \overline{\pf}$ and the equation $\squared_2\qf^\F$ by $\nab_t\overline{\qf^\F}$ respectively, and then add the conjugate of each one to take the real part.
 
 Doing so, we obtain from each one of the above equations the following:
 \bea\label{eq:second-pf}
 \begin{split}
  \squared_1\pf \c \nab_t \overline{\pf}+\squared_1\overline{\pf} \c \nab_t \pf&= i  \frac{2a\cos\th}{|q|^2}\nab_t \pf \c \nab_t \overline{\pf} +\overline{ i  \frac{2a\cos\th}{|q|^2}}\nab_t \pf \c \nab_t \overline{\pf} +V_1  \pf \c \nab_t \overline{\pf}+\overline{V_1}  \overline{\pf} \c \nab_t \pf\\
  &+4Q^2 \frac{\ov{q}^3 }{|q|^5} \left(  \ov{\DD} \c  \qf^\F  \right) \c \nab_t \overline{\pf}+ 4Q^2\frac{ q^3}{|q|^5} \big(  \DD \c  \overline{\qf^\F}  \big) \c \nab_t \pf\\
  &+ L_\pf[\Bfr, \Ffr]  \c \nab_t \overline{\pf}+\overline{L_\pf[\Bfr, \Ffr]}  \c \nab_t \pf
  \end{split}
 \eea
 and 
 \bea\label{eq:second-qfF}
 \begin{split}
 \squared_2\qf^\F \c \nab_t \ov{\qf^\F}+  \squared_2\overline{\qf^\F} \c \nab_t \qf^\F&=i  \frac{4a\cos\th}{|q|^2}\nab_t \qf^\F\c \nab_t \ov{\qf^\F}+\overline{i  \frac{4a\cos\th}{|q|^2}}\nab_t \qf^\F\c \nab_t \ov{\qf^\F} +V_2  \qf^\F\c \nab_t \ov{\qf^\F} +\overline{V_2 }\overline{ \qf^\F}\c \nab_t \qf^\F \\
 &- \frac 1 2  \frac{q^3}{|q|^5} \left(  \DD \hot  \pf  \right)\c \nab_t \ov{\qf^\F}- \frac 1 2   \frac{\ov{q}^3}{|q|^5} \big(  \ov{\DD \hot  \pf}  \big)\c \nab_t \qf^\F\\
  &+  \frac 3 4 \frac{ q^3}{|q|^5}\left(  \left(  H - \Hb\right)  \hot \pf \right)\c \nab_t \ov{\qf^\F} +  \frac 3 4  \frac{ \ov{q}^3}{|q|^5}\left(  \left(  \ov{H} - \ov{\Hb}\right)  \hot \ov{\pf} \right)\c \nab_t \qf^\F \\
&+L_{\qf^\F} [\Bfr, \Ffr] \c \nab_t \ov{\qf^\F}+ \ov{L_{\qf^\F} [\Bfr, \Ffr] }\c \nab_t \qf^\F
\end{split}
 \eea
 
 We now analyze each term on the left hand side.
 
 \begin{enumerate}
 
\item  The structure of the first order terms $\nab_t$ in the equations of the form $i f(r, \th)\nab_t$, for a real function $f(r, \th)$ is crucial for the cancellation of these terms. Indeed,
 \beaa
  i  \frac{2a\cos\th}{|q|^2}\nab_t \pf \c \nab_t \overline{\pf} +\overline{ i  \frac{2a\cos\th}{|q|^2}}\nab_t \pf \c \nab_t \overline{\pf} &=&  i  \frac{2a\cos\th}{|q|^2}\nab_t \pf \c \nab_t \overline{\pf}- i  \frac{2a\cos\th}{|q|^2}\nab_t \pf \c \nab_t \overline{\pf} =0, \\
  i  \frac{4a\cos\th}{|q|^2}\nab_t \qf^\F\c \nab_t \ov{\qf^\F}+\overline{i  \frac{4a\cos\th}{|q|^2}}\nab_t \qf^\F\c \nab_t \ov{\qf^\F} &=& i  \frac{4a\cos\th}{|q|^2}\nab_t \qf^\F\c \nab_t \ov{\qf^\F}-i  \frac{4a\cos\th}{|q|^2}\nab_t \qf^\F\c \nab_t \ov{\qf^\F} =0
 \eeaa
 
\item  The reality of the potentials $V_1$ and $V_2$ allows to write the terms involving the potential as boundary terms in the usual way:
 \beaa
 V_1  \pf \c \nab_t \overline{\pf}+\overline{V_1}  \overline{\pf} \c \nab_t \pf&=& V_1 \big( \pf \c \nab_t \overline{\pf}+  \overline{\pf} \c \nab_t \pf \big)= V_1 \pr_t( |\pf|^2) =\pr_t( V_1 |\pf|^2) \\
 V_2  \qf^\F\c \nab_t \ov{\qf^\F} +\overline{V_2 }\overline{ \qf^\F}\c \nab_t \qf^\F&=& V_2 \big(  \qf^\F\c \nab_t \ov{\qf^\F} +\overline{ \qf^\F}\c \nab_t \qf^\F \big)= V_2\pr_t (|\qf^\F|^2) = \pr_t  (V_2 |\qf^\F|^2)
 \eeaa
 Being $V_1$ and $V_2$ positive for $|a|/M \ll 1$, they give a coercive contribution to the energies.

\item In order to obtain cancellation for the terms involving coupling, we need to sum the estimates for the two equations. Observe that the complex functions which multiply the coupling terms, i.e. $\frac{\ov{q}^3}{|q|^5}$ and $\frac{q^3}{|q|^5}$, are conjugate complex functions, and such structure is crucial for the cancellation.
Since the coupling terms differ by a constant factor $8Q^2$,
we multiply the second identity \eqref{eq:second-qfF} by $8Q^2$ and sum to \eqref{eq:second-pf} and obtain
\beaa
&&  \squared_1\pf \c \nab_t \overline{\pf}+\squared_1\overline{\pf} \c \nab_t \pf- \pr_t( V_1 |\pf|^2)+8Q^2 \big( \squared_2\qf^\F \c \nab_t \ov{\qf^\F}+  \squared_2\overline{\qf^\F} \c \nab_t \qf^\F- \pr_t  (V_2 |\qf^\F|^2)\big)\\
&=&4Q^2 \frac{\ov{q}^3 }{|q|^5} \left(  \ov{\DD} \c  \qf^\F  \right) \c \nab_t \overline{\pf}+ 4Q^2\frac{ q^3}{|q|^5} \big(  \DD \c  \overline{\qf^\F}  \big) \c \nab_t \pf-  4Q^2 \frac{q^3}{|q|^5} \left(  \DD \hot  \pf  \right)\c \nab_t \ov{\qf^\F}-  4Q^2 \frac{\ov{q}^3}{|q|^5} \big(  \ov{\DD \hot  \pf}  \big)\c \nab_t \qf^\F\\
  &&+ 6Q^2\frac{q^3}{|q|^5} \left(  \left(  H - \Hb\right)  \hot \pf \right)\c \nab_t \ov{\qf^\F} +  6Q^2 \frac{\ov{q}^3}{|q|^5} \left(  \left(  \ov{H} - \ov{\Hb}\right)  \hot \ov{\pf} \right)\c \nab_t \qf^\F \\
    &&+ L_\pf[\Bfr, \Ffr]  \c \nab_t \overline{\pf}+\overline{L_\pf[\Bfr, \Ffr]}  \c \nab_t \pf+8Q^2(L_{\qf^\F} [\Bfr, \Ffr] \c \nab_t \ov{\qf^\F}+ \ov{L_{\qf^\F} [\Bfr, \Ffr] }\c \nab_t \qf^\F)
 \eeaa
 We now consider the first two lines on the right hand side of the above. We put together the terms which are multiplied by the function $\frac{\ov{q}^3 }{|q|^5}$ and those multiplied by $\frac{q }{|q|^5}$.
 We first integrate by parts in $t$ in the first term, and obtain:
  \beaa
  &&2Q^2 \big(  2\frac{\ov{q}^3 }{|q|^5} \left(  \ov{\DD} \c   \qf^\F  \right) \c  \nab_t\overline{\pf}-  2 \frac{\ov{q}^3}{|q|^5} \big(  \ov{\DD \hot  \pf}  \big)\c \nab_t \qf^\F +   3\frac{\ov{q}^3}{|q|^5}\left(  \left(  \ov{H} - \ov{\Hb}\right)  \hot \ov{\pf} \right)\c \nab_t \qf^\F \big) \\
 &&+2Q^2 \big( 2 \frac{ q^3}{|q|^5} \big(  \DD \c  \overline{\qf^\F}  \big) \c \nab_t \pf- 2  \frac{q^3}{|q|^5} \left(  \DD \hot  \pf  \right)\c \nab_t \ov{\qf^\F}+3 \frac{q^3}{|q|^5} \left(  \left(  H - \Hb\right)  \hot \pf \right)\c \nab_t \ov{\qf^\F}\big) \\
 &=&2Q^2 \big(  -2\frac{\ov{q}^3 }{|q|^5} \left(  \ov{\DD} \c  \nab_t \qf^\F  \right) \c  \overline{\pf}-  2 \frac{\ov{q}^3}{|q|^5} \big(  \ov{\DD \hot  \pf}  \big)\c \nab_t \qf^\F +   3\frac{\ov{q}^3}{|q|^5}\left(  \left(  \ov{H} - \ov{\Hb}\right)  \hot \ov{\pf} \right)\c \nab_t \qf^\F \big) \\
 &&+2Q^2 \big( -2 \frac{ q^3}{|q|^5} \big(  \DD \c  \nab_t\overline{\qf^\F}  \big) \c  \pf- 2  \frac{q^3}{|q|^5} \left(  \DD \hot  \pf  \right)\c \nab_t \ov{\qf^\F}+3 \frac{q^3}{|q|^5} \left(  \left(  H - \Hb\right)  \hot \pf \right)\c \nab_t \ov{\qf^\F}\big) 
 \eeaa
Recall Lemma \ref{lemma:adjoint-operators} that relates the operator $\DD\hot$ and $\DD \c$. Applying it to $F=\pf$, $U=\nab_t\qf^\F$, we obtain
   \beaa
 ( \DD \hot   \pf) \c   \nab_t\ov{\qf^\F}  &=&  -\pf \c (\DD \c \nab_t\ov{\qf^\F}) -( (H+\Hb ) \hot \pf)\c \nab_t\ov{\qf^\F} +\D_\a (\pf \c \nab_t\ov{\qf^\F})
 \eeaa
Using the above we write, modulo spacetime divergence terms:
 \beaa
  -2 \frac{ q^3}{|q|^5} \big(  \DD \c  \nab_t\overline{\qf^\F}  \big) \c  \pf&=&  2  \DD(\frac{ q^3}{|q|^5})  \nab_t\overline{\qf^\F}   \c  \pf +2 \frac{ q^3}{|q|^5}   \nab_t\overline{\qf^\F}   \c ( \DD\hot \pf) +2 \frac{ q^3}{|q|^5}((H+\Hb) \hot \pf )\c   \nab_t\overline{\qf^\F} \\
  &=&  2  \DD( q^{1/2}(\ov{q})^{-5/2})  \nab_t\overline{\qf^\F}   \c  \pf +2 \frac{ q^3}{|q|^5}   \nab_t\overline{\qf^\F}   \c ( \DD\hot \pf)+2 \frac{ q^3}{|q|^5}((H+\Hb) \hot \pf )\c   \nab_t\overline{\qf^\F}\\
  &=&   (   \Hb - 5   H)  \frac{ q^3}{|q|^5}  \nab_t\overline{\qf^\F}   \c  \pf +2 \frac{ q^3}{|q|^5}   \nab_t\overline{\qf^\F}   \c ( \DD\hot \pf)+2 \frac{ q^3}{|q|^5}((H+\Hb) \hot \pf )\c   \nab_t\overline{\qf^\F}\\
   &=&   (  3 \Hb - 3   H)  \frac{ q^3}{|q|^5}  \nab_t\overline{\qf^\F}   \c  \pf +2 \frac{ q^3}{|q|^5}   \nab_t\overline{\qf^\F}   \c ( \DD\hot \pf)
 \eeaa
 since from \eqref{DD-q-n-ov-q-m} we have $\DD( q^{1/2} \ov{q}^{-5/2})=(  \frac 1 2  \Hb -\frac 5 2  H)  \frac{ q^3}{|q|^5} $. Similarly,
  \beaa
  -2 \frac{ \ov{q}^3}{|q|^5} \left(  \ov{\DD} \c  \nab_t \qf^\F  \right) \c  \ov{\pf}  &=&   (  3 \ov{\Hb} - 3  \ov{ H})  \frac{ \ov{q}^3}{|q|^5}  \nab_t\qf^\F   \c  \ov{\pf} +2 \frac{ \ov{q}^3}{|q|^5}   \nab_t\qf^\F   \c ( \ov{\DD\hot \pf})
 \eeaa
We finally obtain
   \beaa
 &&  2Q^2 \big(  -2\frac{\ov{q}^3 }{|q|^5} \left(  \ov{\DD} \c  \nab_t \qf^\F  \right) \c  \overline{\pf}-  2 \frac{\ov{q}^3}{|q|^5} \big(  \ov{\DD \hot  \pf}  \big)\c \nab_t \qf^\F +   3\frac{\ov{q}^3}{|q|^5}\left(  \left(  \ov{H} - \ov{\Hb}\right)  \hot \ov{\pf} \right)\c \nab_t \qf^\F \big) \\
 &&+2Q^2 \big( -2 \frac{ q^3}{|q|^5} \big(  \DD \c  \nab_t\overline{\qf^\F}  \big) \c  \pf- 2  \frac{q^3}{|q|^5} \left(  \DD \hot  \pf  \right)\c \nab_t \ov{\qf^\F}+3 \frac{q^3}{|q|^5} \left(  \left(  H - \Hb\right)  \hot \pf \right)\c \nab_t \ov{\qf^\F}\big) \\
 &=&2Q^2 \frac{ \ov{q}^3}{|q|^5}  \big( (  3 \ov{\Hb} - 3  \ov{ H})  \nab_t\qf^\F   \c  \ov{\pf} +2  \nab_t\qf^\F   \c ( \ov{\DD\hot \pf})-  2 \left(  \ov{\DD \hot  \pf}  \right)\c \nab_t \qf^\F +   3\left(  \left(  \ov{H} - \ov{\Hb}\right)  \hot \ov{\pf} \right)\c \nab_t \qf^\F \big) \\
 &&+2Q^2  \frac{ q^3}{|q|^5} \big( (  3 \Hb - 3   H)  \nab_t\overline{\qf^\F}   \c  \pf +2    \nab_t\overline{\qf^\F}   \c ( \DD\hot \pf)- 2  \left(  \DD \hot  \pf  \right)\c \nab_t \ov{\qf^\F}+3  \left(  \left(  H - \Hb\right)  \hot \pf \right)\c \nab_t \ov{\qf^\F}\big) \\
  &=&0
 \eeaa

  Observe that upon a spacetime integration, the coupling terms cancel out, and therefore they only give contributions to boundary terms. Even though those terms do not have a definite sign, the modified energy terms are positive in the case of Reissner-Nordstr\"om for $|Q|<M$, as proved in \cite{Giorgi7}. In particular, for small angular momentum $|a| \ll M$ they remain positive in Kerr-Newman.
 
 By putting the above together we have
 \beaa
&&  \squared_1\pf \c \nab_t \overline{\pf}+\squared_1\overline{\pf} \c \nab_t \pf- \pr_t( V_1 |\pf|^2)+8Q^2 \big( \squared_2\qf^\F \c \nab_t \ov{\qf^\F}+  \squared_2\overline{\qf^\F} \c \nab_t \qf^\F- \pr_t  (V_2 |\qf^\F|^2)\big)\\
&&- \partial_t \big( \frac{ 4Q^2}{|q|^5} \ov{q}^3\left(  \ov{\DD} \c  \qf^\F  \right) \c  \overline{\pf}+ \frac{ 4Q^2}{|q|^5} q^3\left(  \DD \c  \overline{\qf^\F}  \right) \c  \pf \big)\\
&=&  L_\pf[\Bfr, \Ffr]  \c \nab_t \overline{\pf}+\overline{L_\pf[\Bfr, \Ffr]}  \c \nab_t \pf+8Q^2(L_{\qf^\F} [\Bfr, \Ffr] \c \nab_t \ov{\qf^\F}+ \ov{L_{\qf^\F} [\Bfr, \Ffr] }\c \nab_t \qf^\F)
 \eeaa
 \item In order to absorb the lower order terms on the right of the above estimates, one needs to combine the above energy estimates with boundedness of trapped spacetime energies, as given by Morawetz estimates. Moreover, through transport estimates one can show to bound all first derivatives of $\Bfr$, $\Ffr$ and $\Xfr$ by a degenerate Morawetz bulk for $\pf$ and $\qf^\F$.

 Assuming such estimates, we briefly explain how to absorb the lower order terms above. Recall that
 \beaa
L_\pf[\Bfr, \Ffr] &=& q^{1/2} \ov{q}^{9/2} \Big[-Z^{\Bfr}_{a} \c \nabc \Bfr+ 2 \PF\ov{\PF} \ Y^{\Ffr}_{a} (\ov{\DDc}\c\mathfrak{F})\\
&&+ ( 2 \PF\ov{\PF}\ Y^{\Bfr}_0 - Z^{\Bfr}_0) \Bfr + 2 \PF\ov{\PF}\left( Y^{\Ffr}_0 \c \mathfrak{F}+Y^{\Xfr}_0 \ \mathfrak{X}\right)\Big] 
\eeaa
and
\beaa
L_{\qf^\F}[\Bfr, \Ffr] &=& q\ov{q}^2 \Big[ W_4^{\Ffr} \nabc_4 \mathfrak{F}+\left( W_a^{\Ffr}  -Z^{\Ffr}_{a} \right) \c \nabc \mathfrak{F}+W_a^{\Xfr} \DDc \hot \mathfrak{X}  + W_{a}^{\Bfr} \DDc\hot \Bfr\\
&&+ ( W_0^{\Ffr}-Z^{\Ffr}_0)  \mathfrak{F} +W_0^{\Bfr} \hot  \mathfrak{B}+ W_0^{\Xfr} \hot \mathfrak{X}  \Big]
\eeaa
The terms on the second line of the above expressions (i.e. the lowest order terms) can be absorbed for small $|a|\ll M$, by integration by parts in $t$ and then bounding by Cauchy-Schwarz. For example,
\beaa
q^{1/2} \ov{q}^{9/2}( 2 \PF\ov{\PF}\ Y^{\Bfr}_0 - Z^{\Bfr}_0) \Bfr  \c \nab_t \overline{\pf}&=& - q^{1/2} \ov{q}^{9/2}( 2 \PF\ov{\PF}\ Y^{\Bfr}_0 - Z^{\Bfr}_0) \nab_t \Bfr  \c \overline{\pf}\\
&\leq& O( a r)\big( | \nab_t \Bfr|^2+ | \pf|^2\big)
\eeaa
Both terms on the right hand side appear without degeneracy at the trapping region in the Morawetz bulks, and therefore they can be absorbed by that for small $|a|\ll M$. The same will be true for the other terms of lower order, which contains only one derivative of $\Bfr$, $\Ffr$ or $\Xfr$. 
 
 In what follows, we therefore only look at the terms which highest number of derivatives, since the lower order terms can be treated as above.
 We now consider 
 \beaa
q\ov{q}^2  W_4^{\Ffr}\nabc_4 \Ffr \c \nab_t \ov{\qf^\F}+ \ov{q\ov{q}^2  W_4^{\Ffr} \nabc_4 \Ffr}\c \nab_t \qf^\F
 \eeaa
 Since $W_4^{\Ffr}$ is real, we have
 \beaa
&=& W_4^{\Ffr}  \big(q\ov{q}^2  \nabc_4 \Ffr \c \nab_t \ov{\qf^\F}+ q^2\ov{q}  \nabc_4 \ov{\Ffr}\c \nab_t \qf^\F\big) \\
&=& W_4^{\Ffr}  \big(q\ov{q}^2 \nab_t \Ffr \c    \nab_4 \ov{\qf^\F}+ q^2\ov{q}  \nab_t  \ov{\Ffr}\c \nab_4 \qf^\F\big) \\
&=& W_4^{\Ffr}  \big(q\ov{q}^2 \nab_3 \Ffr \c    \nab_4 \ov{\qf^\F}+ q^2\ov{q}  \nab_3  \ov{\Ffr}\c \nab_4 \qf^\F\big) + \dots
 \eeaa
 Writing $q\ov{q}^2 \nab_3\Ffr=\qf^\F+ \lot$, we obtain
  \beaa
&=& W_4^{\Ffr}  \big(\qf^\F \c    \nab_4 \ov{\qf^\F}+ \ov{\qf^\F}\c \nab_4 \qf^\F\big) + \dots= W_4^{\Ffr}\nab_4 ( |\qf^\F|^2) 
 \eeaa
 which gives a boundary term. The same happens for the terms $W_a^{\Xfr} \DDc \hot \mathfrak{X} $, $Z^{\Bfr}_{a} \c \nabc \Bfr$ and $\left( W_a^{\Ffr}  -Z^{\Ffr}_{a} \right) \c \nabc \mathfrak{F}$, which because of the reality of the coefficients, can be written as boundary terms.

 We now look at the coupling terms in the lower order terms, i.e.
 \beaa
&&(q^{1/2} \ov{q}^{9/2}2 \PF\ov{\PF} \ Y^{\Ffr}_{a} (\ov{\DDc}\c\mathfrak{F})) \c \nab_t \overline{\pf}+\overline{(q^{1/2} \ov{q}^{9/2}2 \PF\ov{\PF} \ Y^{\Ffr}_{a} (\ov{\DDc}\c\mathfrak{F}))}  \c \nab_t \pf\\
&&+8Q^2(q\ov{q}^2 W_{a}^{\Bfr} \DDc\hot \Bfr \c \nab_t \ov{\qf^\F}+ \ov{q\ov{q}^2 W_{a}^{\Bfr} \DDc\hot \Bfr }\c \nab_t \qf^\F)
 \eeaa
 Writing that $\PF \ov{\PF}=\frac{Q^2}{|q|^4}$ and $W_{a}^{\Bfr}= \frac 3 4  i  \atrchb $ and $Y^{\Ffr}_{a}= -3  i \atrchb $, we have
  \beaa
&=&-6Q^2 \frac{q^{1/2} \ov{q}^{9/2}}{|q|^4} \    i \atrchb (\ov{\DDc}\c\mathfrak{F}) \c \nab_t \overline{\pf}+6Q^2 \frac{q^{9/2} \ov{q}^{1/2}}{|q|^4} \    i \atrchb (\DDc\c\ov{\mathfrak{F}}) \c \nab_t \pf\\
&&+6Q^2 q\ov{q}^2  i  \atrchb \DDc\hot \Bfr \c \nab_t \ov{\qf^\F}-6Q^2 q^2\ov{q}  i  \atrchb \ov{\DDc\hot \Bfr} \c \nab_t \qf^\F\\
&=&-6Q^2  i \atrchb \big[ \frac{q^{1/2} \ov{q}^{9/2}}{|q|^4} \   (\ov{\DDc}\c\mathfrak{F}) \c \nab_t \overline{\pf}+q^2\ov{q}   \ov{\DDc\hot \Bfr} \c \nab_t \qf^\F\big] \\
&&+6Q^2   i \atrchb  \big[ \frac{q^{9/2} \ov{q}^{1/2}}{|q|^4} \ (\DDc\c\ov{\mathfrak{F}}) \c \nab_t \pf +q\ov{q}^2  \DDc\hot \Bfr \c \nab_t \ov{\qf^\F}\big]
 \eeaa
 Now by recalling that $q\ov{q}^2 \nab_3\Ffr=\qf^\F+ \lot$ and $q^{1/2}\ov{q}^{9/2} \nab_3\Bfr=\pf+ \lot$, we obtain, only looking at the highest order terms:
 \beaa
&&  \frac{q^{1/2} \ov{q}^{9/2}}{|q|^4} \   (\ov{\DDc}\c\mathfrak{F}) \c \nab_t \overline{\pf}+q^2\ov{q}   \ov{\DDc\hot \Bfr} \c \nab_t \qf^\F\\
&=&-  \frac{q^{1/2} \ov{q}^{9/2}}{|q|^4} \   (\ov{\DDc}\c \nab_t\mathfrak{F}) \c  \overline{\pf}-q^2\ov{q}   \ov{\DDc\hot\nab_t \Bfr} \c  \qf^\F\\
&=&-  \frac{q^{1/2} \ov{q}^{9/2}}{|q|^4} \   (\ov{\DDc}\c \nab_3\mathfrak{F}) \c  \overline{\pf}-q^2\ov{q}   \ov{\DDc\hot\nab_3 \Bfr} \c  \qf^\F\\
&=&-  \frac{q^{1/2} \ov{q}^{9/2}}{|q|^4} \  \frac{1}{q\ov{q}^2}  (\ov{\DDc}\c \qf^\F) \c  \overline{\pf}-q^2\ov{q} \frac{1}{\ov{q}^{1/2}q^{9/2}}  \ov{\DDc\hot\pf} \c  \qf^\F\\
&=&- \frac{\ov{q}^{1/2}}{q^{5/2}}\Big[ (\ov{\DDc}\c \qf^\F) \c  \overline{\pf}+ \ov{\DDc\hot\pf} \c  \qf^\F \Big]=0
 \eeaa
 The remaining terms are therefore only of lower order, and can be absorbed as shown before for small $|a|\ll M$ by Cauchy-Schwarz.

\end{enumerate}

       \appendix

\small
\section{Explicit computations}\label{appendix-section-a}

We collect here some explicit computations needed in Section 2--5. 

\subsection{Derivation of the Bianchi identities}\label{proof-compl=Bianchi}

\begin{lemma}\label{lemma-formulas-J} We have the following for the decomposition in frames of $J_{\b\gamma\delta}$ in \eqref{J-def}, 
     \bea
   J_{434}    &=& -\nab_4(\rhoF^2 +\dual\rhoF^2)+2\left(\etab-2\eta\right)\c \left(\dual \rhoF\dual \bF+ \rhoF \bF\right) \nonumber\\
    &&+2 \xi \c \left( \dual \rhoF\dual \bbF -  \rhoF\bbF\right) +\nab_3(\bF\c  \bF)-4\omb( \bF\cdot  \bF) \label{first-J434}\\
  J_{ab4} &=& \nab_b(\dual \rhoF \dual \bF_a + \rhoF\bF_a) +(\ze_b+\etab_b) (\dual \rhoF \dual \bF_a + \rhoF\bF_a)\nonumber\\
  &&+\etab_a  (\dual \rhoF \dual \bF_b + \rhoF\bF_b)-\frac 1 2 \nab_4 (\rhoF^2 +\dual\rhoF^2)\de_{ab}-\chi_{ba}  (\rhoF^2 +\dual\rhoF^2)\nonumber\\
 &&+\nab_4(\bF \hot \bbF)_{ab}+\frac 1 2 \chi_{bc}(\bF \hot \bbF)_{ca}- \frac 1 2 \chib_{ba} \bF\cdot  \bF\nonumber\\
 &&+\xi_a(\dual \rhoF \dual \bbF_b -\rhoF\bbF_b)+\xi_b(\dual \rhoF \dual \bbF_a - \rhoF\bbF_a)\label{Jab4}\\
J_{4a4}&=&  -\nabc_4(\dual \rhoF \dual \bF_a + \rhoF\bF_a)- \trch(\dual \rhoF \dual \bF_a + \rhoF\bF_a)\nonumber\\
&&+  \atrch(\dual \rhoF \bF_a - \rhoF \dual \bF_a)+2\left( \rhoF^2+\dual \rhoF^2 \right) \xi_a \nonumber\\
    &&+\nab_a(\bF\c  \bF) +(2\ze_a+\etab_a) \bF \c \bF - 2 \xi_b (\bbF \hot \bbF)_{ab} \label{J4a4}\\
J_{3a4}    &=&\nab_a(\rhoF^2+\dual \rhoF^2)-\frac 1 2 \trch (\dual \rhoF \dual \bbF_a - \rhoF\bbF_a)\nonumber\\
&&-\frac 1 2  \trchb (\dual \rhoF \dual \bF_a + \rhoF\bF_a) -\frac 12\atrch\dual (\dual \rhoF \dual \bbF_a - \rhoF\bbF_a)\nonumber\\
&&- \frac 12 \atrchb\dual (\dual \rhoF \dual \bF_a + \rhoF\bF_a)-\nabc_4(\dual \rhoF \dual \bbF_a - \rhoF\bbF_a)\nonumber\\
&&+2 \left( \rhoF^2+\dual \rhoF^2 \right)\etab_a+\xi_a \bbF \c \bbF-\chih_{ab} (\dual \rhoF \dual \bbF_b - \rhoF\bbF_b)\nonumber\\
&&-\chibh (\dual \rhoF \dual \bF_b + \rhoF\bF_b)-2\etab_b (\bbF \hot \bbF)_{ab}\label{J3a4}\\
\dual J_{434}&=& 2  \curl (\dual \rhoF \dual \bF + \rhoF\bF)+2\ze \c \dual (\dual \rhoF \dual \bF + \rhoF\bF) \nonumber\\
 &&-2\atrch (\rhoF^2 +\dual\rhoF^2) + \atrchb \bF\cdot  \bF \label{dualJ434}
\eea
 The other quantities are obtained by symmetrization, where in interchanging the $3$ with the $4$, one interchanges $\bF \leftrightarrow \bbF$, $\rhoF \leftrightarrow -\rhoF$, $\dual \rhoF\leftrightarrow\dual \rhoF$, $\ze \leftrightarrow -\ze$, $\eta\leftrightarrow\etab$, $\chih\leftrightarrow\chib$ and $\trch\leftrightarrow\trchb$, $\om\leftrightarrow\omb$. 
\end{lemma}
\begin{proof} 
We compute $J_{434}$:
 \beaa
  2 J_{434}&=& \D_3 \R_{44}-\D_{4}\R_{43}\\
  &=&\nab_3(\R_{44})-2\R(\D_3 e_4, e_4)-\nab_4(\R_{34})+\R(\D_4 e_4, e_3)+\R(e_4, \D_4 e_3)\\
 &=& 2\nab_3(\bF\c  \bF)-\nab_4(2\rhoF^2 +2\dual\rhoF^2)-2\R(2\omb e_4+2\eta_a e_a, e_4)\\
 &&+\R(-2\om e_4+2\xi_a e_a, e_3) +\R(e_4, 2\om e_3+ 2\etab_a e_a) \\
  &=& -2\nab_4(\rhoF^2 +\dual\rhoF^2)-4\omb\R_{44}+2\left(\etab_a-2\eta_a\right)\R_{a4}+2\xi_a\R_{a3} +2\nab_3(\bF\c  \bF)\\
    &=& -2\nab_4(\rhoF^2 +\dual\rhoF^2)+4\dual \rhoF\left(\etab-2\eta\right)\c \dual \bF+4 \rhoF\left(\etab-2\eta\right)\c \bF\\
    &&+4\dual \rhoF \xi \c  \dual \bbF - 4\rhoF \xi \c \bbF +2\nab_3(\bF\c  \bF)-8\omb( \bF\cdot  \bF)
    \eeaa
    This proves \eqref{first-J434}.
    We compute $J_{ab4}$:
 \beaa
  2J_{ab4}&=&\D_b \R_{4a}-\D_{4}\R_{ab}\\
  &=&\nab_b(\R_{4a}) -\R(\D_b e_4, e_a)-\R(e_4, \D_b e_a)-\nab_4(\R_{ab})+\R(\D_4 e_a, e_b)+\R(e_a, \D_{4} e_b)\\
 &=& \nab_b(2\dual \rhoF \dual \bF_a + 2\rhoF\bF_a)-\nab_4(-2(\bF \hot \bbF)_{ab}+  (\rhoF^2 +\dual\rhoF^2)\de_{ab})+\\
 && -\R(-\ze_b e_4+\chi_{bc} e_c, e_a)-\R(e_4, \frac 1 2 \chib_{ba} e_4 +\frac 1 2 \chi_{ba} e_3)+\R(\etab_a  e_4+\xi_a e_3, e_b)+\R(e_a, \etab_b  e_4+\xi_b e_3)\\
 &=& \nab_b(2\dual \rhoF \dual \bF_a + 2\rhoF\bF_a)-\nab_4(-2(\bF \hot \bbF)_{ab}+  (\rhoF^2 +\dual\rhoF^2)\de_{ab})+\\
 && +\ze_b \R_{a4}-\chi_{bc}\R_{ca}-\frac 1 2 \chib_{ba}\R_{44}-\frac 1 2 \chi_{ba}\R_{43}+\etab_a\R_{4b}+\xi_a\R_{3b}+\etab_b\R_{a4}+\xi_b\R_{a3}
 \eeaa
Using the Ricci decomposition, this proves \eqref{Jab4}.
We compute $J_{4a4}$:
  \beaa
    2J_{4a4}&=&\D_a \R_{44}-\D_{4}\R_{4a}\\
    &=&\nab_a(\R_{44}) -2\R(\D_a e_4, e_4)-\nab_4(\R_{4a})+\R(\D_4 e_4, e_a)+\R(e_4, \D_{4} e_a)\\
 &=& 2\nab_a(\bF\c  \bF) -2\nab_4(\dual \rhoF \dual \bF_a + \rhoF\bF_a)-2\R(-\ze_a e_4+\chi_{ab} e_b, e_4)\\
 &&+\R(-2\om e_4+2\xi_b e_b, e_a) +\R(e_4, \etab_a e_4+\xi_a e_3)\\
 &=& -2\nab_4(\dual \rhoF \dual \bF_a + \rhoF\bF_a)-2\chi_{ab}(2\dual \rhoF \dual \bF_b + 2\rhoF\bF_b)\\
 &&-2\om (2\dual \rhoF \dual \bF_a + 2\rhoF\bF_a)+4\left( \rhoF^2+\dual \rhoF^2 \right) \xi_a\\
  &&+2\nab_a(\bF\c  \bF) +(4\ze_a+2\etab_a) \bF \c \bF - 4 \xi_b (\bbF \hot \bbF)_{ab}
 \eeaa
 which proves \eqref{J4a4}. 
We compute $J_{3a4}$:
\beaa
2J_{3a4}&=& \D_a \R_{43}-\D_{4}\R_{3a}\\
  &=&\nab_a(\R_{34}) -\R(\D_a e_4, e_3)-\R(e_4, \D_a e_3)-\nab_4(\R_{3a})+\R(\D_4 e_3, e_a)+\R(e_3, \D_{4} e_a)\\
  &=&\nab_a(2\rhoF^2+2\dual \rhoF^2) -\R(\chi_{ab} e_b-\ze_a e_4, e_3)-\R(e_4, \chib_{ab} e_b+\ze_a e_3)\\
  &&-\nab_4(2\dual \rhoF \dual \bbF_a - 2\rhoF\bbF_a)+\R(2\om e_3+2\etab_b e_b, e_a)+\R(e_3, \etab_a e_4+\xi_a e_3)\\
    &=&2\nab_a(\rhoF^2+\dual \rhoF^2) -\chi_{ab}(2\dual \rhoF \dual \bbF_b - 2\rhoF\bbF_b)- \chib_{ab}(2\dual \rhoF \dual \bF_b + 2\rhoF\bF_b)\\
  &&-2\nab_4(\dual \rhoF \dual \bbF_a - \rhoF\bbF_a)+2\om(2\dual \rhoF \dual \bbF_a - 2\rhoF\bbF_a)\\
  &&+2\etab_b( - 2(\bbF \hot \bbF)_{ab}+\left( \rhoF^2+\dual \rhoF^2 \right) \de_{ab} )+\etab_a (2\rhoF^2+2\dual\rhoF^2)+2\xi_a \bbF \c \bbF
\eeaa
which proves \eqref{J3a4}. We compute $\dual J_{434}$ using \eqref{Jab4}:  
\beaa
\dual J_{434}&=& -2 J_{ab4} \in_{ab}\\
 &=& 2  \curl (\dual \rhoF \dual \bF + \rhoF\bF)+2\ze \c \dual (\dual \rhoF \dual \bF + \rhoF\bF)\\
 &&-2\atrch (\rhoF^2 +\dual\rhoF^2)+ \atrchb \bF\cdot  \bF
\eeaa
which proves \eqref{dualJ434}. 
\end{proof}

\begin{lemma}\label{lemma:simplify-J-Maxwell} 
Using Maxwell equations as in Proposition \ref{prop-Maxwell}, we obtain
\bea\label{J434}
\begin{split}
      J_{434}       &= -2\div (\dual \rhoF \dual \bF + \rhoF\bF)- 2\left( \ze+4\eta\right) \c \left(\rhoF \bF +\dual \rhoF \dual \bF\right)\\
      &+2\trch (\rhoF^2+\dual \rhoF^2) + 2\left(2\nabc_3 \bF + \frac 1 2  \trchb \bF -\chih \c \bbF\right)\c  \bF  
      \end{split}
      \eea
\end{lemma}
\begin{proof} We compute
 \beaa
  \nab_4 (\rhoF^2 +\dual\rhoF^2)&=& 2\rhoF \nab_4 \rhoF +2\dual\rhoF\nab_4\dual\rhoF\\
  &=& 2\rhoF\left(\div \bF-  \left(\trch \rhoF- \atrch \dual \rhoF\right)+\left( \ze+\etab\right) \c \bF-\xi \c \bbF \right) \\
  &&+2\dual\rhoF\left(\curl \bF   - \left(\trch \dual \rhoF+\atrch \rhoF\right)  +\left(\etab+\ze \right) \c  \dual \bF  +\xi \c \dual \bbF \right)\\
   &=&2\rhoF\div \bF+2\dual\rhoF\curl \bF-2\trch (\rhoF^2+\dual \rhoF^2)\\
   &&+ 2\left( \ze+\etab\right) \c \left(\rhoF \bF +\dual \rhoF \dual \bF\right) +2\xi \c \left(\dual\rhoF  \dual \bbF- \rhoF \bbF\right) 
 \eeaa
 We now compute $\div (\dual \rhoF \dual \bF + \rhoF\bF)$. It is given by
 \beaa
 \div (\dual \rhoF \dual \bF + \rhoF\bF)&=&\nab \dual \rhoF\c  \dual \bF +\dual \rhoF \div\dual \bF + \nab \rhoF \c \bF+ \rhoF \div \bF \\
 &=& \rhoF \div \bF+\dual \rhoF \curl\bF + \left(\nab \rhoF -\dual \nab \dual \rhoF \right)\c  \bF 
 \eeaa
Using the Maxwell equation
\beaa
\nab(\rhoF)-\dual \nab \dual\rhoF&=&\nab_3 \bF + \frac 1 2  \left(\trchb \bF+ \atrchb \dual \bF\right)-2\omb \bF-2 \left(\eta \rhoF- \dual \eta \dual \rhoF\right) -\chih \c \bbF
\eeaa
we obtain
\bea\label{divv-pfbf}
\begin{split}
 \div (\dual \rhoF \dual \bF + \rhoF\bF)&= \rhoF \div \bF+\dual \rhoF \curl\bF -2 \eta \c \left(\rhoF \bF +  \dual \rhoF \dual \bF\right) \\
 &+ \left(\nabc_3 \bF + \frac 1 2  \trchb \bF -\chih \c \bbF\right)\c  \bF 
 \end{split}
\eea
This therefore gives
 \beaa
  \nab_4 (\rhoF^2 +\dual\rhoF^2)   &=&2 \div (\dual \rhoF \dual \bF + \rhoF\bF)-2\trch (\rhoF^2+\dual \rhoF^2)\\
   &&+ 2\left( \ze+\etab+2\eta\right) \c \left(\rhoF \bF +\dual \rhoF \dual \bF\right)\\
   && +2\xi \c \left(\dual\rhoF  \dual \bbF- \rhoF \bbF\right)- 2\left(\nabc_3 \bF + \frac 1 2  \trchb \bF -\chih \c \bbF\right)\c  \bF  
 \eeaa
Using \eqref{first-J434} and the above, we deduce \eqref{J434}.
\end{proof}

The complexified Bianchi identity for $A$ is given by 
\bea\label{compl-Bianchi-a}
\nab_3A -\DD\hot B &=& -\frac{1}{2}\tr\Xb A+4\omb A +(Z+4H)\hot B -3\ov{P}\Xh+ \mathfrak{a} + i \dual \mathfrak{a}
\eea
 We compute $\mathfrak{a}$, using \eqref{first-J434} and \eqref{Jab4}:
\bea\label{mathfrak-a}
\begin{split}
 \mathfrak{a} &= 2 \nab \hot (\dual \rhoF \dual \bF + \rhoF\bF) +2(\ze+2\etab)\hot (\dual \rhoF \dual \bF + \rhoF\bF_b)-2(\rhoF^2 +\dual\rhoF^2)\chih \\
 &+2\nab_4(\bF \hot \bbF)+\nabc_3(\bF \hot \bF)+ \chi \c (\bF \hot \bbF)+2\xi \hot (\dual \rhoF \dual \bbF -2\rhoF\bbF)
 \end{split}
\eea
Using \eqref{mathfrak-a}, we have
 \beaa
 \mathfrak{a} + i \dual \mathfrak{a} &=& \DD \hot (\left(\dual \rhoF \dual \bF + \rhoF\bF\right) + i \dual \left(\dual \rhoF \dual \bF + \rhoF\bF\right))\\
 && +((\ze+2\etab) + i \dual (\ze+2\etab))\hot ((\dual \rhoF \dual \bF + \rhoF\bF)+ i \dual (\dual \rhoF \dual \bF + \rhoF\bF)) \\
 &&-2(\rhoF^2 +\dual\rhoF^2)(\chih+ i \dual \chih) \\
 &&+\nab_4(\BF \hot \BBF)+\frac 1 2 \nabc_3(\BF \hot \BF)+\frac 1 2  \Xh \c (\BF \hot \BBF)+\PF (\Xi \hot  \BF)
 \eeaa
 Observe that 
 \beaa
 \ov{\PF} \BF&=& \left(\dual \rhoF \dual \bF + \rhoF\bF\right) + i \dual \left(\dual \rhoF \dual \bF + \rhoF\bF\right)
 \eeaa
 and $\PF \ov{\PF}=  \rhoF^2 +\dual\rhoF^2$.
We can therefore write
 \beaa
 \mathfrak{a} + i \dual \mathfrak{a}&=& \DD \hot (\ov{\PF} \BF) +(Z+2\Hb)\hot (\ov{\PF} \BF)-2\ov{\PF}\PF\Xh\\
 &&+\nab_4(\BF \hot \BBF)+\frac 1 2 \nabc_3(\BF \hot \BF)+\frac 1 2  \Xh \c (\BF \hot \BBF)+\PF (\Xi \hot  \BF) \\
 &=&\ov{\PF} \DD \hot ( \BF)+ \DD \ov{\PF} \hot \BF +(Z+2\Hb)\hot (\ov{\PF} \BF)-2\ov{\PF}\PF\Xh\\
 &&+\nab_4(\BF \hot \BBF)+\frac 1 2 \nabc_3(\BF \hot \BF)+ \frac 1 2  \Xh \c (\BF \hot \BBF)+\PF (\Xi \hot  \BF)
 \eeaa
 Using the Maxwell equation for $\DDc \ov{\PF}$ we write
 \beaa
 \mathfrak{a} + i \dual \mathfrak{a}&=&\ov{\PF} \DD \hot ( \BF)+ (-2\ov{\PF} \Hb )\hot \BF +(Z+2\Hb)\hot (\ov{\PF} \BF)-2\ov{\PF}\PF\Xh+\frac 1 2\nab_4(\BF \hot \BBF)\\
 &&+\frac 1 2 \nabc_3(\BF \hot \BF)+ (- \frac 1 2 \tr X \BBF +\frac 1 2 \Xbh \c \ov{\BF}+\frac 1 2  \Xh \c \BBF+\PF \Xi)\hot \BF \\
  &=&-2\ov{\PF} \left( -\frac 1 2 (\DD+Z) \hot \BF +\PF\Xh \right)+\frac 1 2\nab_4(\BF \hot \BBF)+\frac 1 2 \nabc_3(\BF \hot \BF)\\
  &&+ (- \frac 1 2 \tr X \BBF +\frac 1 2 \Xbh \c \ov{\BF}+\frac 1 2  \Xh \c \BBF+\PF \Xi)\hot \BF
 \eeaa
From \eqref{compl-Bianchi-a}, we obtain
\beaa
\nab_3A -\DD\hot B &=& -\frac{1}{2}\tr\Xb A+4\omb A +(Z+4H)\hot B -3\ov{P}\Xh-2\ov{\PF} \left( -\frac 1 2 (\DD+Z) \hot \BF +\PF\Xh \right)\\
&&+\frac 1 2\nab_4(\BF \hot \BBF)+\frac 1 2 \nabc_3(\BF \hot \BF)\\
&&+ (- \frac 1 2 \tr X \BBF +\frac 1 2 \Xbh \c \ov{\BF}+\frac 1 2  \Xh \c \BBF+\PF \Xi)\hot \BF
\eeaa
which proves the first equation. 

 The complexified Bianchi identity for $B$ is given by
 \beaa
 \nab_4B -\frac{1}{2}\ov{\DD}\c A &=& -2\ov{\tr X} B -2\om B +\frac{1}{2}A\c  (\ov{2Z +\Hb})+3\ov{P} \,\Xi-\left(J_{4a4}+i \dual J_{4a4} \right)
 \eeaa
 Using \eqref{J4a4}, we have 
 \beaa
  \dual  J_{4a4} &=& -\nab_4\dual (\dual \rhoF \dual \bF_a + \rhoF\bF_a)- \trch(-\dual \rhoF  \bF_a + \rhoF \dual \bF_a)\\
  &&+  \atrch(\dual \rhoF \dual \bF_a +\rhoF \bF_a)-2\om \dual (\dual \rhoF \dual \bF_a + \rhoF\bF_a)\\
  &&+2\left( \rhoF^2+\dual \rhoF^2 \right) \dual\xi_a+\dual \nab_a(\bF\c  \bF) +\dual (2\ze_a+\etab_a) \bF \c \bF - 2 \dual \xi_b (\bbF \hot \bbF)_{ab}
 \eeaa
and therefore
 \beaa
J_{4a4}+i \dual J_{4a4}&=& -\nab_4(\ov{\PF} \BF_a)- \tr X\ov{\PF} \BF_a-2\om \ov{\PF} \BF_a+2\ov{\PF}\PF \Xi_a\\
&&+\frac 1 4  \DD(\BF\c  \ov{\BF}) +\frac 1 2 (2Z+\Hb) \BF \c \ov{\BF} -\frac 1 2  \ov{ \Xib} \c (\BBF \hot \BBF)
 \eeaa
 Using the Maxwell equation for $\nabc_4 \PF$  we obtain
 \beaa
 J_{4a4}+i \dual J_{4a4}&=& -\ov{\PF} \nab_4(\BF_a)-2\om \ov{\PF} \BF_a+2\ov{\PF}\PF \Xi_a+\frac 1 2  \DD(\BF\c  \ov{\BF}) +Z \BF \c \ov{\BF} 
 \eeaa
which proves the second equation. 
The other complexified Bianchi identity for $\b$ is given by 
\beaa
\nab_3B-\DD\ov{P} &=& -\tr\Xb B+2\omb B+\ov{\Bb}\c \Xh+3\ov{P}H +\frac{1}{2}A\c\ov{\Xib}+\left(J_{3a4}+i \dual J_{3a4} \right)
\eeaa
Using \eqref{J3a4}, we have
\beaa
J_{3a4}+i \dual J_{3a4}&=& \DD(\PF \ov{\PF})+\frac 1 2 \tr X \PF \BBF-2\om\PF \BBF-\frac 1 2 \tr \Xb \ov{\PF} \BF +\nab_4(\PF \BBF)\\
&&+2 \PF \ov{\PF}\Hb +\frac 1 2 \Xi (\BBF \c \ov{\BBF})-\PF \Xh \c \BBF-\ov{\PF} \Xbh \c \BF- \ov{\Hb} \c(\BBF \hot \BBF)
\eeaa
Using the Maxwell equations, we obtain
\beaa
J_{3a4}+i \dual J_{3a4}&=& \DD(\PF) \ov{\PF}+ \PF \DD(\ov{\PF})+\frac 1 2 \tr X \PF \BBF-2\om\PF \BBF-\frac 1 2 \tr \Xb \ov{\PF} \BF \\
  &&-\ov{\tr X}\PF \BBF+\PF( -\DD\ov{\PF}- \frac 1 2 \tr X \BBF+2\om \BBF-2\ov{\PF} \Hb )+2 \PF \ov{\PF}\Hb \\
  &&+\frac 1 2( \ov{\DDc }\c \BF)\BBF-\PF \Xh \c \BBF-\frac 1 2 \ov{\PF} \Xbh \c \BF\\
  &=& \ov{\PF}\DD(\PF) -\frac 1 2 \tr \Xb \ov{\PF} \BF -\ov{\tr X}\PF \BBF\\
   &&+\frac 1 2( \ov{\DDc }\c \BF)\BBF-\PF \Xh \c \BBF-\frac 1 2 \ov{\PF} \Xbh \c \BF
\eeaa
which gives the desired formula. 
The complexified Bianchi identity for $P$ is given by
\beaa
\nab_4P -\frac{1}{2}\DD\c \ov{B} &=& -\frac{3}{2}\tr X P +\frac{1}{2}(2\Hb+Z)\c\ov{B} -\ov{\Xi}\c\Bb -\frac{1}{4}\Xbh\c \ov{A}-\frac 1 2 \left(J_{434} + i \dual J_{434} \right)
\eeaa
We compute using \eqref{J434} and \eqref{dualJ434}:
\beaa
J_{434} + i \dual J_{434}&=& -2 \div (\dual \rhoF \dual \bF + \rhoF\bF)+  2i  \curl (\dual \rhoF \dual \bF + \rhoF\bF) \\
&&- 2\left( \ze+ i \dual \ze +4\eta\right) \c \left(\rhoF \bF +\dual \rhoF \dual \bF\right)+2(\trch- i \atrch)  (\rhoF^2+\dual \rhoF^2) \\
&&+ 2\left(2\nabc_3 \bF + \frac 1 2  \trchb \bF -\chih \c \bbF\right)\c  \bF  +i  \atrchb \bF\cdot  \bF
\eeaa
 Observe that 
 \beaa
 \PF \ov{\BF}&=& \left(\dual \rhoF \dual \bF + \rhoF\bF\right) - i \dual \left(\dual \rhoF \dual \bF + \rhoF\bF\right)
 \eeaa
and therefore
\beaa
J_{434} + i \dual J_{434}&=& - \DD \c ( \PF \ov{\BF})-  Z \c  \PF \ov{\BF} +2\tr X \PF \ov{\PF}- 2\left(  H+\ov{H} \right) \c \left(\PF \ov{\BF} + \ov{\PF} \BF\right)\\
&&+2 \nabc_3(\BF \hot \BF)+ (- \frac 1 2 \tr \Xb \BF - \Xh \c \BBF)\hot \BF
\eeaa
Using the Maxwell equations for $\DDc\PF$ we obtain
\beaa
J_{434} + i \dual J_{434}&=&2\tr X \PF \ov{\PF}-  \PF \DD \c \ov{\BF}-    Z  \c  \PF\ov{\BF} - 2 \ov{H}  \c  \ov{\PF} \BF\\
&&+ \nabc_3(\BF \hot \BF)+ (-  \tr \Xb \BF - \frac 1 2 \Xh \c \BBF)\hot \BF
\eeaa
which gives the desired formula.

\subsection{Proof of Proposition \ref{proposition-relation-gauge-inv}}\label{proof-prop-app}

Here we derive \eqref{relation-F-A-B}, \eqref{relation-F-B-A}, \eqref{nabb-4-mathfrak-B}, \eqref{nabc-3-mathfrak-X}. 

\subsubsection{Derivation of \eqref{relation-F-A-B}}
Multiply the Bianchi identity  \eqref{nabc-3-A} by $\PF$:
\beaa
 \PF\nabc_3A+\frac{1}{2}\tr\Xb \PF A  &=&\PF\DDc\hot B  +H\hot \left(4\PF B-3\ov{\PF}\PF \BF \right) -3\ov{P}\PF\Xh-2\ov{\PF}\PF \mathfrak{F}
\eeaa
Multiply the definition of $\mathfrak{F}$ \eqref{definition-F-Bianchi} by $3\ov{P}$:
\beaa
3 \ov{P}\mathfrak{F}&=&  -\frac 3 2  \ov{P}\DDc \hot \BF-\frac 9 2  H \hot  \ov{P}\BF +3 \ov{P}\PF\Xh
\eeaa
Summing the above we obtain the cancellation of $3 \ov{P}\PF\Xh$:
\beaa
&&\left(3 \ov{P}+2\PF\ov{\PF}\right)\mathfrak{F}+ \PF\nabc_3A+\frac{1}{2}\PF\tr\Xb A \\
&=& \frac 1 2 \left( 2\PF\DDc\hot B-3  \ov{P}\DDc \hot \BF\right)  +H\hot \left(4\PF B-3\ov{\PF}\PF \BF-\frac 9 2   \ov{P}\BF \right)
\eeaa
On the other hand
\beaa
\DDc \hot \mathfrak{B}&=& \DDc \hot \left( 2\PF B - 3\ov{P} \BF \right)\\
&=& \left(2\PF \DDc \hot B - 3\ov{P}\DDc \hot \BF\right) +  2\DDc\PF \hot  B - 3\DDc \ov{P}\hot  \BF\\
&=& \left(2\PF \DDc \hot B - 3\ov{P}\DDc \hot \BF\right) -  4H  \hot \PF B + \left(9  \ov{P } -6\PF \ov{\PF} \right)H  \hot  \BF
\eeaa
Therefore
\beaa
&&\left(3 \ov{P}+2\PF\ov{\PF}\right)\mathfrak{F}+ \PF\nabc_3A+\frac{1}{2}\PF\tr\Xb A \\
&=& \frac 1 2 \left(\DDc \hot \mathfrak{B}+  4H  \hot \PF B - \left(9  \ov{P } -6\PF \ov{\PF} \right)H  \hot  \BF\right) \\
&& +H\hot \left(4\PF B-3\ov{\PF}\PF \BF-\frac 9 2   \ov{P}\BF \right)= \frac 1 2 \DDc \hot \mathfrak{B}+  3H  \hot  \mathfrak{B}
\eeaa
which proves \eqref{relation-F-A-B}.

\subsubsection{Derivation of \eqref{relation-F-B-A}}
 Using  the definition of $\mathfrak{F}$ \eqref{definition-F-Bianchi} we compute
 \beaa
\nabc_4 \mathfrak{F}&=& \nabc_4 \left( -\frac 1 2 \DDc \hot \BF-\frac 3 2  H \hot \BF +\PF\Xh \right)\\
&=& -\frac 1 2  \nabc_4\DDc \hot \BF-\frac 3 2  \nabc_4H \hot \BF-\frac 3 2  H \hot \nabc_4\BF +\nabc_4\PF\Xh +\PF\nabc_4\Xh 
\eeaa
Recall the following commutator formula, for $F=f+i \dual f \in \sk_1(\CCC)$ of conformal type $s$, see Lemma 5.3 in \cite{GKS}:
\bea
 \, [\nabc_4 , \DDc \hot ]F &=&- \frac 1 2 \tr X( \DDc\hot F + (1-s)\Hb\hot F)+ \underline{H} \hot \nabc_4 F \label{commutator-nabc-4-F-formula}
 \eea
Applying \eqref{commutator-nabc-4-F-formula} to $F=\BF$ and $s=1$, using \eqref{nabc-4-H-red}, \eqref{nabc-4-PF-red} and \eqref{nabc-4-Xh} we obtain
\beaa
\nabc_4 \mathfrak{F}&=& -\frac 1 2 \DDc \hot  \nabc_4\BF -\frac 1 2 \left( -\frac 1 2 \tr X \DDc \hot \BF+\Hb \hot \nabc_4 \BF\right)-\frac 3 2  \left(  -\frac{1}{2}\ov{\tr X}(H-\Hb)\right) \hot \BF\\
&&-\frac 3 2  H \hot \nabc_4\BF -\ov{\tr X}\PF\Xh +\PF\left(-\frac 1 2 (\tr X +\ov{\tr X})\Xh +\DDc \hot \Xi +(H+\Hb) \hot \Xi -A\right)\\
&=& -\frac 1 2 \DDc \hot  \nabc_4\BF +\frac 1 4  \tr X \DDc \hot \BF-\frac 1 2\left(3H+ \Hb \right)\hot \nabc_4 \BF\\
&&+\frac 3 4  \ov{\tr X}(H-\Hb)\hot \BF -\frac 3 2 \ov{\tr X}\PF\Xh-\frac 1 2 \tr X \PF \Xh +\PF\left( \DDc \hot \Xi +(H+\Hb) \hot \Xi -A\right)
\eeaa
On the other hand, using the definition of $\Xfr$ \eqref{definition-mathfrak-X} we compute using \eqref{Codazzi-1-red} and \eqref{DDc-PF-red}:
\beaa
\DDc \hot \mathfrak{X}&=& \DDc \hot \left(\nabc_4\BF+\frac 3 2 \ov{\tr X} \BF-2\PF \Xi \right)\\
&=& \DDc \hot \nabc_4\BF+\frac 3 2 \DDc\ov{\tr X} \hot \BF+\frac 3 2 \ov{\tr X} \DDc \hot \BF-2\DDc \PF\hot  \Xi -2\PF \DDc \hot\Xi \\
&=& \DDc \hot \nabc_4\BF+\frac 3 2 (\tr X-\ov{\tr X})H \hot \BF+\frac 3 2 \ov{\tr X} \DDc \hot \BF+4\PF H\hot  \Xi -2\PF \DDc \hot\Xi 
\eeaa
This implies
\beaa
-\frac 1 2 \DDc \hot \nabc_4\BF&=&-\frac 1 2 \DDc \hot \mathfrak{X} +\frac 3 4 (\tr X-\ov{\tr X})H \hot \BF+\frac 3 4 \ov{\tr X} \DDc \hot \BF\\
&&+2\PF H\hot  \Xi -\PF \DDc \hot\Xi 
\eeaa
By plugging in to the above expression for $\nabc_4 \Ffr$, we obtain
\beaa
\nabc_4 \mathfrak{F}&=&-\frac 1 2 \DDc \hot \mathfrak{X} +\frac 3 4 (\tr X-\ov{\tr X})H \hot \BF+\frac 3 4 \ov{\tr X} \DDc \hot \BF+2\PF H\hot  \Xi -\PF \DDc \hot\Xi \\
&& +\frac 1 4  \tr X \DDc \hot \BF-\frac 1 2\left(3H+ \Hb \right)\hot \nabc_4 \BF\\
&&+\frac 3 4  \ov{\tr X}(H-\Hb)\hot \BF -\frac 3 2 \ov{\tr X}\PF\Xh-\frac 1 2 \tr X \PF \Xh +\PF\left( \DDc \hot \Xi +(H+\Hb) \hot \Xi -A\right)\\
&=& -\frac 1 2 \DDc \hot \mathfrak{X} +\left(\frac 3 4 \ov{\tr X} +\frac 1 4  \tr X\right)\DDc \hot \BF-\left(\frac 3 2 \ov{\tr X}+\frac 1 2 \tr X \right)\PF\Xh\\
&&+\frac 3 4 \tr XH \hot \BF-\frac 3 4  \ov{\tr X}\Hb\hot \BF -\frac 1 2\left(3H+ \Hb \right)\hot \nabc_4 \BF+\PF( (3H+\Hb) \hot \Xi -A)
\eeaa
Using again the definition of $\Ffr$ \eqref{definition-F-Bianchi} to write $ -\frac 1 2 \DDc \hot \BF +\PF\Xh=\mathfrak{F}+\frac 3 2  H \hot \BF$, we finally obtain
\beaa
\nabc_4 \mathfrak{F}&=&-\left(\frac 3 2 \ov{\tr X} +\frac 1 2  \tr X\right)\mathfrak{F} -\frac 1 2 \DDc \hot \mathfrak{X} -\left(\frac 3 2 \ov{\tr X} +\frac 1 2  \tr X\right)\frac 3 2  H \hot \BF\\
&&+\frac 3 4 \tr XH \hot \BF-\frac 3 4  \ov{\tr X}\Hb\hot \BF  -\frac 1 2\left(3H+ \Hb \right)\hot \nabc_4 \BF+\PF( (3H+\Hb) \hot \Xi -A)\\
&=&-\left(\frac 3 2 \ov{\tr X} +\frac 1 2  \tr X\right)\mathfrak{F} -\frac 1 2 \DDc \hot \mathfrak{X}  -\frac 1 2\left(3H+ \Hb \right) \hot \left(\nabc_4 \BF+\frac 3 2 \ov{\tr X}\BF-2\PF \Xi \right) -\PF A
\eeaa
Using again the definition of $\Xfr$ \eqref{definition-mathfrak-X}, this proves \eqref{relation-F-B-A}. 

\subsubsection{Derivation of \eqref{nabb-4-mathfrak-B}}
 Using the definition of $\Bfr$ \eqref{definition-mathfrak-B}, we compute
\beaa
\nabc_4\mathfrak{B}+3\ov{\tr X} \mathfrak{B}&=& 2\nabc_4\PF B+2\PF \nabc_4B - 3\nabc_4\ov{P} \BF - 3\ov{P} \nabc_4\BF\\
&&+3\ov{\tr X} \left(2\PF B - 3\ov{P} \BF \right)
\eeaa
Using \eqref{nabc-4-B}, \eqref{nabc-4-PF-red}, \eqref{nabc-4-P-red} we obtain
\beaa
\nabc_4\mathfrak{B}+3\ov{\tr X} \mathfrak{B}&=& -2\ov{\tr X}\PF B\\
&&+2\PF \left(\frac{1}{2}\ov{\DDc}\c A -2\ov{\tr X} B  +\frac{1}{2}A\c  \ov{ \Hb}+ 3\ov{P}\,\Xi+ \ov{\PF}\left( \nabc_4\BF-2 \PF \Xi \right) \right)\\
&& - 3\left(-\frac 3 2 \ov{\tr X} \ \ov{P}-\ov{\tr X} \PF \ov{\PF} \right)\BF - 3\ov{P} \nabc_4\BF+3\ov{\tr X} \left(2\PF B - 3\ov{P} \BF \right)
\eeaa
which gives
\beaa
\nabc_4\mathfrak{B}+3\ov{\tr X} \mathfrak{B}&=&\PF \left(\ov{\DDc}\c A  +  \ov{ \Hb} \c A \right) - \left( 3\ov{P}-2 \PF \ov{\PF} \right)\left(\nabc_4\BF+ \frac 3 2 \ov{\tr X}\BF-2\PF \Xi  \right)
\eeaa
which, using the definition of $\Xfr$ \eqref{definition-mathfrak-X}, proves \eqref{nabb-4-mathfrak-B}. 

\subsubsection{Derivation of \eqref{nabc-3-mathfrak-X}}
 Using the definition of $\Xfr$ \eqref{definition-mathfrak-X}, we compute
\beaa
\nabc_3 \mathfrak{X}&=& \nabc_3 \left(\nabc_4\BF+\frac 3 2 \ov{\tr X} \BF-2\PF \Xi  \right)\\
&=& \nabc_4\nabc_3\BF+[\nabc_3, \nabc_4]\BF+\frac 3 2 \ov{\tr X} \nabc_3\BF+\frac 3 2 \nabc_3\ov{\tr X} \BF\\
&&-2\nabc_3\PF \Xi -2\PF \nabc_3\Xi 
\eeaa
We compute each term. Using \eqref{nabc-3-BF} and \eqref{nabc-4-trXb-red}, we obtain
\beaa
\nabc_4\nabc_3\BF&=&  \nabc_4\left(- \frac 1 2 \tr\Xb \BF+\DDc\PF+2\PF H \right)\\
&=&  - \frac 1 2 \tr\Xb \nabc_4\BF- \frac 1 2 \nabc_4\tr\Xb \BF+\nabc_4\DDc\PF+2\PF \nabc_4H+2\nabc_4\PF H \\
&=&  - \frac 1 2 \tr\Xb \nabc_4\BF+\left( \frac{1}{4}\tr X\tr\Xb-\frac 1 2 \DDc\c\ov{\Hb}-\frac 1 2 \Hb\c\ov{\Hb}-\ov{P} \right) \BF+\DDc \nabc_4\PF\\
&&+[\nabc_4, \DDc]\PF+2\PF \nabc_4H+2\nabc_4\PF H \\
&=& L_1+L_2+L_3+L_4+L_5
\eeaa
We simplify $L_1$ by making use of the definition of $\mathfrak{X}$ and writing $\nabc_4 \BF=\mathfrak{X}-\frac 3 2 \ov{\tr X} \BF+2\PF \Xi $. We obtain
\beaa
L_1&=&  - \frac 1 2 \tr\Xb \nabc_4\BF+\left( \frac{1}{4}\tr X\tr\Xb-\frac 1 2 \DDc\c\ov{\Hb}-\frac 1 2 \Hb\c\ov{\Hb}-\ov{P} \right) \BF\\
&=&  - \frac 1 2 \tr\Xb \left(\mathfrak{X}-\frac 3 2 \ov{\tr X} \BF+2\PF \Xi \right)+\left( \frac{1}{4}\tr X\tr\Xb-\frac 1 2 \DDc\c\ov{\Hb}-\frac 1 2 \Hb\c\ov{\Hb}-\ov{P} \right) \BF
\eeaa
which gives
\bea\label{L1first}
L_1&=&  - \frac 1 2 \tr\Xb \  \mathfrak{X}  +\left( \frac{1}{4}\tr X\tr\Xb+\frac 3 4 \ov{\tr X} \tr \Xb -\frac 1 2 \DDc\c\ov{\Hb}-\frac 1 2 \Hb\c\ov{\Hb}-\ov{P} \right) \BF+ \PF \left( -\tr\Xb\Xi \right)
\eea
We compute $L_2$ using \eqref{nab-4-PF}:
\beaa
L_2&=& \DDc \nabc_4\PF= \DDc \left(- \ov{\tr X} \PF +\frac 1 2 \ov{\DDc} \c \BF+\frac 1 2 \ov{\Hb} \c \BF  \right)\\
&=& -  \DDc\ov{\tr X} \PF- \ov{\tr X}  \DDc\PF +\frac 1 2 \DDc( \ov{\DDc} \c \BF)+\frac 1 2  \DDc(\ov{\Hb} \c \BF)
\eeaa
Using \eqref{Codazzi-1} to write 
\beaa
\DDc\ov{\tr X}= \ov{\DDc}\c\Xh +(\tr X - \ov{\tr X})H+ (\tr \Xb - \ov{\tr \Xb})\Xi+2B-2\ov{\PF}\BF
\eeaa
 and using the Leibniz rules to write $\DDc (\ov{\Hb} \c \BF)= (\DDc \c \ov{\Hb} ) \BF +\ov{\Hb} \c \DDc \BF$, we obtain
\beaa
L_2&=&- \ov{\tr X}  \DDc\PF -(\tr X - \ov{\tr X})H \PF +\frac 1 2 \DDc( \ov{\DDc} \c \BF) +\frac 1 2 \ov{\Hb} \c \DDc \BF -2\PF B\\
&&+\left(2\PF \ov{\PF} +\frac 1 2 \DDc \c \ov{\Hb} \right)\BF +  \PF\left( -\ov{\DDc}\c\Xh- (\tr \Xb - \ov{\tr \Xb})\Xi \right) 
\eeaa
Using the definition of $\Bfr$ to write $2\PF B=\mathfrak{B}+3\ov{P} \BF $, we finally obtain
\bea\label{L2first}
\begin{split}
L_2&=-\mathfrak{B}- \ov{\tr X}  \DDc\PF -(\tr X - \ov{\tr X})H \PF +\frac 1 2 \DDc( \ov{\DDc} \c \BF) +\frac 1 2 \ov{\Hb} \c \DDc \BF \\
&+\left(-3\ov{P}+2\PF \ov{\PF} +\frac 1 2 \DDc \c \ov{\Hb} \right)\BF +  \PF\left( -\ov{\DDc}\c\Xh- (\tr \Xb - \ov{\tr \Xb})\Xi \right) 
\end{split}
\eea
To compute $L_3$ we first have the following.
\begin{lemma}
Let $G=g_1 + i g_2 \in \sk_0(\CCC)$ be a $0$-conformal invariant scalar function. Then 
 \bea
  \,[\nabc_4, \DDc] G &=&-\frac 1 2 \tr X \DDc G +\Hb \nabc_4 G+\Xi \nabc_3 G-\frac 1 2 \Xh \c \DDbc G \label{commutator-1}
\eea
\end{lemma}
\begin{proof}  From the commutators, see Lemma 2.39 in \cite{GKS}, 
\beaa
   \,[\nab_4, \nab_a] G &=&-\frac 1 2 \left(\trch \nab_a G+\atrch \dual \nab G\right)+(\etab_a+\ze_a) \nab_4 G-\chih_{ab}\nab_b G +\xi_a \nab_3 G\\ 
 \,[\nab_4, \dual \nab_a] G &=&-\frac 1 2 \left(\trch \dual \nab_a G-\atrch  \nab G\right)+\dual (\etab_a+\ze_a) \nab_4 G-\dual \chih_{ab}\nab_b G +\dual \xi_a \nab_3 G
\eeaa
we obtain
\beaa
  \,[\nab_4, \DD] G &=&-\frac 1 2 \left(\trch \nab_a G+\atrch \dual \nab G\right)+(\Hb+Z) \nab_4 G-\chih_{ab}\nab_b G +\Xi \nab_3 G\\
  &&-\frac 1 2 i\left(\trch \dual \nab_a G-\atrch  \nab G\right)-i\dual \chih_{ab}\nab_b G 
\eeaa
Writing $\nab=\frac 1 2 \DD +\frac 1 2 \DDb$, we obtain the desired formula with non conformal derivatives. Using conformal derivatives we have 
\beaa
  \,[\nabc_4, \DDc] G &=& \nabc_4 \DDc G - \DDc \nabc_4 G= \nab_4 \DD G - \DD \nab_4 G - Z \nab_4 G
\eeaa
which gives the desired formula. 
\end{proof}

We compute $L_3$ by applying \eqref{commutator-1} to $\PF$ (which is of conformal type $0$) and using \eqref{nab-4-PF}, \eqref{nabc-3-PF-red}, and \eqref{DDc-ov-PF-red}, we obtain
\beaa
L_3&=& \,[\nabc_4, \DDc] \PF \\
&=&-\frac 1 2 \tr X \DDc \PF +\Hb \nabc_4 \PF+\Xi \nabc_3 \PF-\frac 1 2 \Xh \c \DDbc \PF\\
 &=&-\frac 1 2 \tr X \DDc \PF +\Hb \left(-\ov{tr X} \PF +\frac 1 2 \ov{\DDc} \c \BF +\frac 1 2 \ov{\Hb} \c \BF \right)-\tr \Xb \PF \Xi-\frac 1 2 \Xh \c(-2\ov{\Hb} \PF)
\eeaa
which gives, 
\bea\label{L3first}
\begin{split}
L_3 &=-\frac 1 2 \tr X \DDc \PF- \ov{\tr X} \PF \Hb  +\frac 1 2  \Hb  \c \ov{\DDc} \BF+\frac 1 2  (\Hb \c \ov{\Hb}) \BF + \PF \left( \Xh \c \ov{\Hb} - \tr \Xb\Xi \right)
  \end{split}
\eea
We compute $L_4$ using \eqref{nabc-4-H}
\beaa
L_4&=& 2\PF \nabc_4H= 2\PF \left(  -\frac{1}{2}\ov{\tr X}(H-\Hb)+\nabc_3\Xi -\frac{1}{2}\Xh\c(\ov{H}-\ov{\Hb}) -B-\ov{\PF} \BF \right)
\eeaa
which can be written as 
\bea\label{L4first}
L_4&=&-\mathfrak{B} - \PF\ov{\tr X}(H-\Hb)+\left(-3\ov{P}-2\ov{\PF}\PF \right) \BF+  \PF \left(  2\nabc_3\Xi -\Xh\c(\ov{H}-\ov{\Hb})  \right)
\eea
We compute $L_5$ using \eqref{nab-4-PF}, 
\beaa
L_5&=& 2\nabc_4\PF H= 2 \left(- \ov{\tr X} \PF +\frac 1 2 \ov{\DDc} \c \BF+\frac 1 2 \ov{\Hb} \c \BF  \right) H
\eeaa
which can be written as 
\bea\label{L5first}
L_5&=&-2\ov{\tr X} \PF H +H \c \ov{\DDc} \BF+ \left(H \c \ov{\Hb}\right) \BF
\eea
Putting together \eqref{L1first}, \eqref{L2first}, \eqref{L3first}, \eqref{L4first} and \eqref{L5first} we obtain
\beaa
\nabc_4\nabc_3\BF&=&  - \frac 1 2 \tr\Xb \  \mathfrak{X} -2 \mathfrak{B} -\left(\frac 1 2 \tr X+ \ov{\tr X} \right)\left( \DDc\PF +2H \PF\right) \\
&&+\frac 1 2 \DDc( \ov{\DDc} \c \BF) +\frac 1 2 \ov{\Hb} \c \DDc \BF +\left(H+\frac 1 2  \Hb \right)  \c \ov{\DDc} \BF\\
&&+\left( \frac{1}{4}\tr X\tr\Xb+\frac 3 4 \ov{\tr X} \tr \Xb -7\ov{P}+H \c \ov{\Hb} \right) \BF\\
&&+  \PF\left( -\ov{\DDc}\c\Xh+\Xh \c(-\ov{H}+2\ov{\Hb} )+ 2\nabc_3\Xi - (3\tr \Xb - \ov{\tr \Xb})\Xi   \right) 
\eeaa
Using \eqref{nabc-3-BF} to substitute $\DDc \PF+2 \PF H=\nabc_3 \BF+\frac 1 2 \tr\Xb \BF$ we obtain
\bea\label{term1}
\begin{split}
\nabc_4\nabc_3\BF&=  - \frac 1 2 \tr\Xb \  \mathfrak{X} -2 \mathfrak{B} -\left(\frac 1 2 \tr X+ \ov{\tr X} \right) \nabc_3 \BF +\frac 1 2 \DDc( \ov{\DDc} \c \BF) \\
&+\frac 1 2 \ov{\Hb} \c \DDc \BF +\left(H+\frac 1 2  \Hb \right)  \c \ov{\DDc} \BF+\left(\frac 1 4 \ov{\tr X} \tr \Xb -7\ov{P}+H \c \ov{\Hb} \right) \BF\\
&+  \PF\left( -\ov{\DDc}\c\Xh+\Xh \c(-\ov{H}+2\ov{\Hb} )+ 2\nabc_3\Xi - (3\tr \Xb - \ov{\tr \Xb})\Xi   \right)
\end{split} 
\eea

In order to compute $[\nabc_3, \nabc_4] \BF$, we first have the following.

\begin{lemma}
Let $F=f+i\dual f \in \sk_1(\CCC)$ be of conformal type $s$. Then
 \bea\label{commutator-3-4-F}
\begin{split}
  \,[\nabc_3, \nabc_4] F   &=\frac 1 2 (H-\Hb) \c \ov{\DDc} F +\frac 1 2 (\ov{H}-\ov{\Hb} ) \c \DDc F  \\
  &+\left((s-1)P+(s+1)\ov{P}+2s \PF \ov{\PF}-\frac{ s+1}{ 2}(H \c \ov{\Hb}  )  +\frac {1-s}{ 2} (\ov{H} \c \Hb )\right)  F
  \end{split}
\eea
\end{lemma}
\begin{proof} 
 From, see Lemma 2.39 in \cite{GKS}
\beaa
 \, [\nab_4, \nab_3] f_a&=2 \om \nab_3 f_a -2\omb \nab_4 f_a+ 2(\etab_b-\eta_b ) \nab_b f_a +2(\etab \c f ) \eta_{a} -2 (\eta \c f )\etab_{a} -2 \dual \rho \dual f_a\\
 \, [\nab_4, \nab_3] \dual f_a&=2 \om \nab_3 \dual f_a -2\omb \nab_4 \dual f_a+ 2(\etab_b-\eta_b ) \nab_b \dual f_a +2(\etab \c \dual f ) \eta_{a} -2 (\eta \c \dual f )\etab_{a} +2 \dual \rho  f_a
 \eeaa 
 we  derive for $F=f+i\dual f$,
 \beaa
  \, [\nab_4, \nab_3] F_a&=& 2 \om \nab_3 F_a -2\omb \nab_4 F_a+ 2(\etab_b-\eta_b ) \nab_b F_a +(\ov{\Hb} \c F ) \eta_{a} - (\ov{H} \c F )\etab_{a} +(P-\ov{P}) F_a\\
  &=& 2 \om \nab_3 F_a -2\omb \nab_4 F_a+ 2(\etab_b-\eta_b ) \nab_b F_a +\frac 1 2 (\ov{\Hb} \c H ) F_{a} -\frac 1 2  (\ov{H} \c \Hb )F_{a} +(P-\ov{P}) F_a
 \eeaa
 We have 
 \beaa
 [\nabc_3, \nabc_4] F &=&[\nab_3, \nab_4] F-2 \omb \nab_4 F+2\om \nab_3 F  +2s(\nab_3 \om+ \nab_4\omb -4\om\omb)F
 \eeaa
 and using \eqref{Ricci-om}, we obtain  
 \beaa
  [\nabc_3, \nabc_4] F &=&-2 \om \nab_3 F +2\omb \nab_4 F- 2(\etab-\eta ) \c \nab F  -\frac 1 2 (\ov{\Hb} \c H ) F_{a} +\frac 1 2  (\ov{H} \c \Hb )F_{a}  -(P-\ov{P}) F\\
  &&-2 \omb \nab_4 F+2\om \nab_3 F  +2s(  \frac 1 2 (P+\ov{P}) +\PF \ov{\PF} +(\eta-\etab)\c\ze -\eta\c\etab)F\\
&=&2(\eta-\etab ) \c \nabc F  +\left((s-1)P+(s+1)\ov{P}+2s \PF \ov{\PF}\right) F -\frac 1 2 (\ov{\Hb} \c H ) F_{a} +\frac 1 2  (\ov{H} \c \Hb )F_{a} \\
&&   -\frac s 2 (H \c \ov{\Hb} + \ov{H} \c \Hb)F
 \eeaa
Finally observe that $2(\eta-\etab) \c \nabc F=\frac 1 2 (H - \Hb) \c \DDbc F + \frac 1 2 (\ov{H} - \ov{\Hb} ) \c \DDc F$. 
 \end{proof}

Specializing \eqref{commutator-3-4-F} to $F=\BF$ and $s=1$, we compute
\bea\label{term2}
\begin{split}
  \,[\nabc_3, \nabc_4] \BF  &=\frac 1 2 (H-\Hb) \c \ov{\DDc} \BF +\frac 1 2 (\ov{H}-\ov{\Hb} ) \c \DDc \BF\\
  & +\left(2\ov{P}+2 \PF \ov{\PF}-H \c \ov{\Hb}  \right)  \BF
  \end{split}
\eea
We also compute using \eqref{nabc-3-tr-X}
\bea\label{term3}
\frac 3 2 \nabc_3\ov{\tr X} \BF&=&  \left(-\frac 3 4 \ov{\tr\Xb} \ \ov{\tr X}+ \frac 3 2\ov{\DDc}\c H+\frac 3 2H\c\ov{H}+3\ov{P} \right)\BF
\eea
and using \eqref{nabc-3-PF-red}
\bea\label{term4}
 -2\nabc_3\PF \Xi &=& 2\tr \Xb\PF \Xi 
\eea
Using \eqref{term1}, \eqref{term2}, \eqref{term3}, \eqref{term4} we obtain
\beaa
&&\nabc_3 \mathfrak{X}\\
&=& \nabc_4\nabc_3\BF+[\nabc_3, \nabc_4]\BF+\frac 3 2 \ov{\tr X} \nabc_3\BF+\frac 3 2 \nabc_3\ov{\tr X} \BF\\
&&-2\nabc_3\PF \Xi -2\PF \nabc_3\Xi \\
&=& - \frac 1 2 \tr\Xb \  \mathfrak{X} -2 \mathfrak{B} -\left(\frac 1 2 \tr X+ \ov{\tr X} \right) \nabc_3 \BF +\frac 1 2 \DDc( \ov{\DDc} \c \BF)\\
&& +\frac 1 2 \ov{\Hb} \c \DDc \BF +\left(H+\frac 1 2  \Hb \right)  \c \ov{\DDc} \BF+\left(\frac 1 4 \ov{\tr X} \tr \Xb -7\ov{P}+H \c \ov{\Hb} \right) \BF\\
&&+  \PF\left( -\ov{\DDc}\c\Xh+\Xh \c( -\ov{H}+2\ov{\Hb} )+ 2\nabc_3\Xi - (3\tr \Xb - \ov{\tr \Xb})\Xi   \right)\\
&&+\frac 1 2 (H-\Hb) \c \ov{\DDc} \BF +\frac 1 2 (\ov{H}-\ov{\Hb} ) \c \DDc \BF +\left(2\ov{P}+2 \PF \ov{\PF}-H \c \ov{\Hb}    \right)  \BF\\
&&+\frac 3 2 \ov{\tr X} \nabc_3\BF+ \left(-\frac 3 4 \ov{\tr\Xb} \ \ov{\tr X}+ \frac 3 2\ov{\DDc}\c H+\frac 3 2H\c\ov{H}+3\ov{P} \right)\BF+ 2\tr \Xb\PF \Xi  -2\PF \nabc_3\Xi 
\eeaa
which gives
\bea\label{intermediate-nabc-3-X}
\begin{split}
&\nabc_3 \mathfrak{X} + \frac 1 2 \tr\Xb \  \mathfrak{X} +2 \mathfrak{B}  \\
&=\frac 1 2 \DDc( \ov{\DDc} \c \BF) +\frac 1 2 \ov{H} \c \DDc \BF +\frac 3 2 H \c \ov{\DDc} \BF-\frac 1 2\left( \tr X-  \ov{\tr X} \right) \nabc_3 \BF\\
&+\left(\frac 1 4 \ov{\tr X} \tr \Xb-\frac 3 4\ov{\tr X} \ \ov{\tr\Xb} -2\ov{P}+2 \PF \ov{\PF} + \frac 3 2\ov{\DDc}\c H+\frac 3 2H\c\ov{H} \right) \BF\\
&+  \PF\left( -\ov{\DDc}\c\Xh+\Xh \c(-\ov{H}+2\ov{\Hb} )- (\tr \Xb - \ov{\tr \Xb})\Xi   \right)
\end{split}
\eea

We now write the right hand side of the equation in terms of $\mathfrak{F}$. We have the following lemma.  

    \begin{lemma}\label{divergence-F} The following formula for the divergence of $\mathfrak{F}$ holds:
                   \beaa
   \ov{\DDc} \c \mathfrak{F}  +\ov{H} \c \mathfrak{F}   &=&-\frac 1 2(\ov{\tr \Xb} - \tr \Xb) \mathfrak{X}\\
   && -\frac 1 2  \DDc (\ov{\DDc} \c \BF) -\frac 3 2  H \c \ov{\DDc}  \BF -\frac 1 2\ov{H} \c \DDc \BF -\frac 1 2(\ov{\tr X} - \tr X)\nabc_3 \BF\\
  &&+ \Big[  \frac 1 4 \tr X \ov{\tr \Xb}+\frac 3 4\ov{\tr \Xb}  \ov{\tr X}- \frac 1 2 \tr \Xb  \ov{\tr X} -\omb(\ov{\tr X} - \tr X) +\frac 1 2 \DD\c \ov{Z}-\frac 1 2  \ov{\DD}\c Z\\
  &&+ P + \ov{P}-2\PF \ov{\PF}-\frac 3 2 \ov{\DDc} \c H-\frac 3 2 \ov{H} \c H\Big] \BF \\
  &&+\PF \left( \ov{\DDc} \c \Xh+\left( \ov{H}-2\ov{\Hb} \right)\c \Xh-(\ov{\tr \Xb} - \tr \Xb)\PF \Xi \right)
    \eeaa
     \end{lemma}
    \begin{proof} Using the definition of $\Ffr$ \eqref{definition-F-Bianchi} we compute
    \beaa
    \ov{\DDc} \c \mathfrak{F}  +\ov{H} \c \mathfrak{F}&=&\ov{\DDc} \c \left( -\frac 1 2 \DDc \hot \BF-\frac 3 2 H \hot \BF +\PF\Xh \right)\\
    &&+\ov{H} \c \left( -\frac 1 2 \DDc \hot \BF-\frac 3 2 H \hot \BF +\PF\Xh \right)\\
    &=& -\frac 1 2 \ov{\DDc} \c(\DDc \hot \BF)-\frac 3 2 \ov{\DDc} \c(H \hot \BF) +\PF\ov{\DDc} \c \Xh+\ov{\DDc}\PF \c \Xh\\
    && -\frac 1 2\ov{H} \c \DDc \hot \BF -\frac 3 2\ov{H} \c \left( H \hot \BF  \right)+ \PF\Xh \c \ov{H}
    \eeaa
    Using Leibniz rules, the above simplifies to
       \beaa
    \ov{\DDc} \c \mathfrak{F}  +\ov{H} \c \mathfrak{F}    &=& -\frac 1 2 \ov{\DDc} \c(\DDc \hot \BF) -\frac 3 2  H \c \ov{\DDc}  \BF -\frac 1 2\ov{H} \c \DDc \BF \\
    &&+\left(-\frac 3 2 \ov{\DDc} \c H-\frac 3 2 \ov{H} \c H \right)\BF+\PF \left( \ov{\DDc} \c \Xh+\left( \ov{H}-2\ov{\Hb} \right)\c \Xh \right)
    \eeaa
    Since $\BF$ is of conformal type $1$, we have
    \beaa
    \ov{\DDc} \c(\DDc \hot \BF)&=&\ov{\DD}\c(\DD \hot \BF)+ Z \c \ov{\DD} \BF+ \ov{Z}\c\DD  \BF+( \ov{\DD}\c Z+\ov{Z}\c Z) \BF \\
    \DDc (\ov{\DDc} \c \BF)&=&   \DD (\ov{\DD} \c \BF)+ Z \c \ov{\DD} \BF+ \ov{Z}\c\DD  \BF+(\DD\c \ov{Z}+\ov{Z}\c Z) \BF
    \eeaa
     Using the relation \eqref{relation-two-laplacians} of Lemma \ref{Laplacian-1}, which we prove below, applied to $\BF$:
\beaa
  \DDb \c ( \DD \hot \BF)&= & \DD (\ov{\DD} \c \BF) +(\ov{\tr X} - \tr X)\nab_3 \BF+(\ov{\tr \Xb} - \tr \Xb) \nab_4 \BF\\
  &&- \left(  \frac 1 2 \tr X \ov{\tr \Xb}+ \frac 1 2 \tr \Xb  \ov{\tr X} + 2P +2 \ov{P}-4\PF \ov{\PF}\right) \BF
\eeaa
we obtain
\beaa
  \ov{\DDc} \c(\DDc \hot \BF)&=& \DD (\ov{\DD} \c \BF)+ Z \c \ov{\DD} \BF+ \ov{Z}\c\DD  \BF+( \ov{\DD}\c Z+\ov{Z}\c Z) \BF\\
  && +(\ov{\tr X} - \tr X)\nab_3 \BF+(\ov{\tr \Xb} - \tr \Xb) \nab_4 \BF\\
  &&- \left(  \frac 1 2 \tr X \ov{\tr \Xb}+ \frac 1 2 \tr \Xb  \ov{\tr X} + 2P +2 \ov{P}-4\PF \ov{\PF}\right)\\
  &=& \DDc (\ov{\DDc} \c \BF) +(\ov{\tr X} - \tr X)\left(\nabc_3 \BF+2\omb \BF\right)+(\ov{\tr \Xb} - \tr \Xb) \nab_4 \BF\\
  &&- \left(  \frac 1 2 \tr X \ov{\tr \Xb}+ \frac 1 2 \tr \Xb  \ov{\tr X} + 2P +2 \ov{P}-4\PF \ov{\PF}+\DD\c \ov{Z}- \ov{\DD}\c Z\right)
\eeaa
    This finally gives
           \beaa
    \ov{\DDc} \c \mathfrak{F}  +\ov{H} \c \mathfrak{F}    &=& -\frac 1 2  \DDc (\ov{\DDc} \c \BF) -\frac 3 2  H \c \ov{\DDc}  \BF -\frac 1 2\ov{H} \c \DDc \BF\\
    && -\frac 1 2(\ov{\tr X} - \tr X)\nabc_3 \BF-\frac 1 2(\ov{\tr \Xb} - \tr \Xb) \nabc_4 \BF\\
  &&+ \Big[  \frac 1 4 \tr X \ov{\tr \Xb}+ \frac 1 4 \tr \Xb  \ov{\tr X} -\omb(\ov{\tr X} - \tr X) + P + \ov{P}-2\PF \ov{\PF}\\
  &&-\frac 3 2 \ov{\DDc} \c H-\frac 3 2 \ov{H} \c H+\frac 1 2 \DD\c \ov{Z}-\frac 1 2  \ov{\DD}\c Z\Big] \BF \\
    &&+\PF \left( \ov{\DDc} \c \Xh+\left( \ov{H}-2\ov{\Hb} \right)\c \Xh \right)
    \eeaa
    Using the definition of $\Xfr$ to write $\nabc_4 \BF=\mathfrak{X}-\frac 3 2 \ov{\tr X} \BF+2\PF \Xi $ we obtain the final expression. 
    
    \end{proof}
    
    We can therefore write the right hand side of \eqref{intermediate-nabc-3-X} as
    \beaa
&&    \nabc_3 \mathfrak{X} + \frac 1 2 \tr\Xb \  \mathfrak{X} +2 \mathfrak{B}  \\
&=& - \ov{\DDc} \c \mathfrak{F}  -\ov{H} \c \mathfrak{F} -\frac 1 2(\ov{\tr \Xb} - \tr \Xb) \mathfrak{X} \\
&& + \left(  \frac 1 4 \tr X \ov{\tr \Xb}- \frac 1 4 \tr \Xb  \ov{\tr X} -\omb(\ov{\tr X} - \tr X) +\frac 1 2 \DD\c \ov{Z}-\frac 1 2  \ov{\DD}\c Z+ P - \ov{P}\right) \BF
    \eeaa
    Observe that the coefficient of $\BF$ vanishes. We therefore obtain \eqref{nabc-3-mathfrak-X}.

    We are now left to prove Lemma \ref{Laplacian-1}. We first recall the following Gauss equation.
    \begin{proposition}\label{Gauss-equation-self-dual} We have
\begin{enumerate}
\item  For $\psi \in\sk_1(\CCC)$:
\bea\label{Gauss-eq-1}
\begin{split}
 \big( \nab_1 \nab_2- \nab_2 \nab_1\big)  \psi&=\frac 1 2 (\atrch\nab_3+\atrchb \nab_4) \psi \\
 &+ i \left(  \frac 1 4 \trch \trchb+ \frac 1 4\atrch \atrchb +\rho-\rhoF^2-\dual \rhoF^2\right) \psi
 \end{split}
 \eea
 or also:
 \bea\label{Gauss-eq-1-complex}
 \begin{split}
 \big( \nab_1 \nab_2- \nab_2 \nab_1\big)  \psi&=\frac 1 2 (\atrch\nab_3+\atrchb \nab_4) \psi \\
 &+ \frac 1 2 i \left(  \frac 1 4 \tr X \ov{\tr \Xb}+ \frac 1 4 \tr \Xb  \ov{\tr X} + P + \ov{P}-2\PF \ov{\PF}\right) \psi
 \end{split}
 \eea

 \item  For $\Psi\in \sk_2(\CCC)$:
\bea
\begin{split}
 \big( \nab_1 \nab_2- \nab_2 \nab_1\big)  \Psi &=\frac 1 2(\atrch\nab_3+\atrchb \nab_4) \Psi \\
 & + 2 i  \left( \frac 1 4\trch\trchb+ \frac 1 4 \atrch\atrchb+\rho-\rhoF^2-\dual \rhoF^2\right) \Psi
    \end{split}
    \eea
    or also 
    \bea\label{Gauss-eq-2-complex}
\begin{split}
 \big( \nab_1 \nab_2- \nab_2 \nab_1\big)  \Psi &=\frac 1 2(\atrch\nab_3+\atrchb \nab_4) \Psi \\
 & +  i  \left(  \frac 1 4 \tr X \ov{\tr \Xb}+ \frac 1 4 \tr \Xb  \ov{\tr X} + P + \ov{P}-2\PF \ov{\PF}\right) \Psi
    \end{split}
    \eea
 \end{enumerate}
 \end{proposition}
\begin{proof} See Proposition 5.5  in \cite{GKS}. 
\end{proof}

We are now left to prove the following.

   \begin{lemma}\label{Laplacian-1}  Let $\psi \in \sk_1(\mathbb{C})$ and $\Psi \in \sk_2(\mathbb{C})$. The following relations hold true. 
   \bea
   \DD (\ov{\DD} \c \psi)&=& 2\lap_1 \psi - i (\atrch\nab_3+\atrchb \nab_4) \psi \nonumber\\
   &&+ \left(  \frac 1 4 \tr X \ov{\tr \Xb}+ \frac 1 4 \tr \Xb  \ov{\tr X} + P + \ov{P}-2\PF \ov{\PF}\right) \psi  \label{relation-DD-DDb-lap}\\
   \DDb \c ( \DD \hot \psi)&=& 2\lap_1  \psi + i (\atrch\nab_3+\atrchb \nab_4) \psi \nonumber\\
   && - \left(  \frac 1 4 \tr X \ov{\tr \Xb}+ \frac 1 4 \tr \Xb  \ov{\tr X} + P + \ov{P}-2\PF \ov{\PF}\right) \psi \label{relation-DDb-DD-hot-lap}\\
   \DD \hot (\DDb \c \Psi)&=& 2\lap_2  \Psi - i (\atrch\nab_3+\atrchb \nab_4) \Psi \nonumber\\
   && +  \left(  \frac 1 2 \tr X \ov{\tr \Xb}+ \frac 1 2 \tr \Xb  \ov{\tr X} + 2P + 2\ov{P}-4\PF \ov{\PF}\right) \Psi  \label{relation-DD-hot-DDb-lap}
   \eea
   In particular, 
 \bea\label{relation-two-laplacians}
 \begin{split}
  \DDb \c ( \DD \hot \psi)&=  \DD (\ov{\DD} \c \psi) +(\ov{\tr X} - \tr X)\nab_3 \psi+(\ov{\tr \Xb} - \tr \Xb) \nab_4 \psi\\
  &- \left(  \frac 1 2 \tr X \ov{\tr \Xb}+ \frac 1 2 \tr \Xb  \ov{\tr X} + 2P +2 \ov{P}-4\PF \ov{\PF}\right) \psi 
  \end{split}
 \eea
   \end{lemma}
\begin{proof} 
Define $Z_a:=( \DD (\ov{\DD} \c \psi))_a=\DD_a \ov{\DD}^b  \psi_b$ and evaluate it in the frame. We have
\beaa
Z_1&=& \DD_1 \ov{\DD}_1  \psi_1+ \DD_1 \ov{\DD}_2  \psi_2\\
&=& \big(\nab_1+i\dual\nab_1\big)\big(\nab_1-i\dual\nab_1\big)  \psi_1+  \big(\nab_1+i\dual\nab_1\big) \big(\nab_2-i\dual\nab_2\big)  \psi_2\\
&=& \big(\nab_1+i\nab_2\big)\big(\nab_1-i\nab_2\big)  \psi_1+  \big(\nab_1+i\nab_2\big) \big(\nab_2+i\nab_1\big)  \psi_2\\
&=&\left(\nab_1\nab_1+\nab_2\nab_2- i (\nab_1\nab_2-\nab_2\nab_1)  \right) \psi_1+  \left(\nab_1\nab_2-\nab_2\nab_1+ i(\nab_1\nab_1+\nab_2\nab_2)  \right)\psi_2
\eeaa
Using that $\psi_{2}=-i \psi_{1}$, we obtain $Z_1=2\lap_1 \psi_1-2 i (\nab_1\nab_2-\nab_2\nab_1)  \psi_1$.
Also,
\beaa
Z_2&=& \DD_2 \ov{\DD}_1  \psi_1+ \DD_2 \ov{\DD}_2  \psi_2\\
&=& \big(\nab_2+i\dual\nab_2\big)\big(\nab_1-i\dual\nab_1\big)  \psi_1+  \big(\nab_2+i\dual\nab_2\big) \big(\nab_2-i\dual\nab_2\big)  \psi_2\\
&=& \big(\nab_2-i\nab_1\big)\big(\nab_1-i\nab_2\big)  \psi_1+  \big(\nab_2-i\nab_1\big) \big(\nab_2+i\nab_1\big)  \psi_2\\
&=&\left(\nab_2\nab_1-\nab_1\nab_2- i (\nab_1\nab_1+\nab_2\nab_2)  \right) \psi_1+  \left(\nab_1\nab_1+\nab_2\nab_2+ i(\nab_2\nab_1-\nab_1\nab_2)  \right)\psi_2\\
&=&  \left(2\nab_1\nab_1+2\nab_2\nab_2+2 i(\nab_2\nab_1-\nab_1\nab_2)  \right)\psi_2
\eeaa
which gives
\beaa
( \DD (\ov{\DD} \c \psi))_a&=& 2\lap_1 \psi_a-2 i (\nab_1\nab_2-\nab_2\nab_1)  \psi_a
\eeaa
Using the Gauss equation \eqref{Gauss-eq-1-complex} we obtain \eqref{relation-DD-DDb-lap}. Define $Y_{a}:=2\big(\ov{\DD} \c ( \DD \hot \psi)\big)_{a}$, and evaluate it in coordinates, i.e. $Y_a=  \ov{\DD}^b \DD_a\psi_b + \ov{\DD}^b \DD_b \psi_a -\de_{ab}  \ov{\DD}^b\DD^d \psi_d$. 
We have
\beaa
Y_1&=&  \ov{\DD}^b \DD_1\psi_b + \ov{\DD}^b \DD_b \psi_1 -\de_{1b}  \ov{\DD}^b\DD^d \psi_d\\
&=& \left( \ov{\DD}_1 \DD_1+ \ov{\DD}_2 \DD_2\right) \psi_1+ \left(\ov{\DD}_2 \DD_1-  \ov{\DD}_1\DD_2\right) \psi_2\\
&=& \left( \big(\nab_1-i\dual\nab_1\big)  \big(\nab_1+i\dual\nab_1\big)+\big(\nab_2-i\dual\nab_2\big)\big(\nab_2+i\dual\nab_2\big)\right) \psi_1\\
&&+ \left(\big(\nab_2-i\dual\nab_2\big)\big(\nab_1+i\dual\nab_1\big)-   \big(\nab_1-i\dual\nab_1\big) \big(\nab_2+i\dual\nab_2\big)\right) \psi_2\\
&=& \left( \big(\nab_1-i\nab_2\big)  \big(\nab_1+i\nab_2\big)+\big(\nab_2+i\nab_1\big)\big(\nab_2-i\nab_1\big)\right) \psi_1\\
&&+ \left(\big(\nab_2+i\nab_1\big)\big(\nab_1+i\nab_2\big)-   \big(\nab_1-i\nab_2\big) \big(\nab_2-i\nab_1\big)\right) \psi_2\\
&=& \left(2\nab_1\nab_1+2\nab_2\nab_2+2i (\nab_1\nab_2-\nab_2\nab_1)\right) \psi_1+ \left(2\nab_2\nab_1-2\nab_1\nab_2+2 i(\nab_1\nab_1+\nab_2\nab_2)\right) \psi_2
\eeaa
Using that $\psi_{2}=-i \psi_{1}$, we obtain $Y_1= 4\lap_1 \psi_1+4i (\nab_1\nab_2-\nab_2\nab_1)\psi_1$. 
Also, 
\beaa
Y_2&=&  \ov{\DD}^b \DD_2\psi_b + \ov{\DD}^b \DD_b \psi_2 -\de_{2b}  \ov{\DD}^b\DD^d \psi_d= \left( \ov{\DD}_1 \DD_2 -  \ov{\DD}_2\DD_1 \right)\psi_1+\left(  \ov{\DD}_1 \DD_1+\ov{\DD}_2 \DD_2 \right)\psi_2 
\eeaa
which gives
\beaa
2\big(\ov{\DD} \c ( \DD \hot \psi)\big)_{a}&=& 4\lap_1  \psi_a +4 i (\nab_1\nab_2-\nab_2\nab_1) \psi_a
\eeaa
Using the Gauss equation \eqref{Gauss-eq-1-complex} we obtain \eqref{relation-DDb-DD-hot-lap}. Finally, \eqref{relation-DD-hot-DDb-lap} is proved in Proposition 5.6 in \cite{GKS}. 
\end{proof}

    \section{Derivation of the Teukolsky equations}\label{proof-main-thm}

In this section, we derive the system of Teukolsky equations for $\Bfr$,  $\Ffr$, $A$.

\subsection{The Teukolsky equation for $\mathfrak{B}$}

 Recall the relation \eqref{nabb-4-mathfrak-B}:
\beaa
 \nabc_4\mathfrak{B}+3\ov{\tr X} \mathfrak{B}&=&\PF \left(\ov{\DDc}\c A  +  \ov{ \Hb} \c A \right)+\left(2 \PF \ov{\PF} - 3\ov{P}\right)\mathfrak{X}
 \eeaa
 We apply $\nabc_3$ to the above, and using \eqref{nabc-3-tr-X} we obtain
 \beaa
  \nabc_3\nabc_4\mathfrak{B}&=&-3\ov{\tr X} \nabc_3\mathfrak{B}+\left(\frac{3}{2}\ov{\tr\Xb} \ \ov{\tr X}-3 \ov{\DDc}\c H-3H\c\ov{H}-6\ov{P}\right)\mathfrak{B}+I_1+I_2+I_3
 \eeaa
where
\beaa
I_1&=& \nabc_3\left( \PF \left(\ov{\DDc}\c A  +  \ov{ \Hb} \c A \right)\right), \qquad I_2= \nabc_3\left(2 \PF \ov{\PF} - 3\ov{P}\right)\mathfrak{X} \\
 I_3&=& \left(2 \PF \ov{\PF} - 3\ov{P}\right)\nabc_3\mathfrak{X}
\eeaa
We compute $I_1$. Using \eqref{nabc-3-PF-red} we obtain
\beaa
I_1&=& \nabc_3 \PF \left(\ov{\DDc}\c A  +  \ov{ \Hb} \c A \right)+\PF \nabc_3\left(\ov{\DDc}\c A  +  \ov{ \Hb} \c A \right)\\
&=& -\tr\Xb \PF \left(\ov{\DDc}\c A  +  \ov{ \Hb} \c A \right)+\PF \left(\ov{\DDc}\c \nabc_3A +[\nabc_3,\ov{\DDc}\c] A +  \ov{ \Hb} \c \nabc_3A +  \nabc_3\ov{ \Hb} \c A\right)
\eeaa
Recall the following commutator formula, for $U=u + i \dual u \in \sk_2(\CCC)$, see Lemma 5.3 in \cite{GKS}:
 \bea
\, [\nabc_3, \ov{\DDc}\c] U&=- \frac 1 2\ov{\tr\Xb}\, ( \ov{\DDc} \c U +(s-  2) \ov{H} \c U)+\ov{H} \c \nab_3 U \label{commutator-nabc-3-ov-DDc-U}
\eea
Applying \eqref{commutator-nabc-3-ov-DDc-U} to $A$ with $s=2$ we obtain
\beaa
I_1&=& -\tr\Xb \PF \left(\ov{\DDc}\c A  +  \ov{ \Hb} \c A \right) - \frac 1 2\ov{\tr\Xb}\,\PF   \ov{\DDc} \c A+ \PF \nabc_3\ov{ \Hb} \c A \\
&&+\PF \left(\ov{\DDc}\c \nabc_3A+\left(\ov{H} +\ov{\Hb} \right)\c \nabc_3 A  \right)
\eeaa
We write the last term of the above using \eqref{DDc-ov-PF-red} as
\beaa
&&\PF \left(\ov{\DDc}\c \nabc_3A  +\left(\ov{H}+  \ov{ \Hb}  \right)\c \nabc_3 A  \right)\\
&=& \ov{\DDc}\c(\PF \nabc_3A )- \ov{\DDc}\PF \c \nabc_3A  +\left(\ov{H}+  \ov{ \Hb}  \right)\c  \PF \nabc_3 A \\
&=& \ov{\DDc}\c(\PF \nabc_3A ) +\left(\ov{H}+ 3 \ov{ \Hb}  \right)\c  \PF \nabc_3 A  
\eeaa
Using \eqref{relation-F-A-B} to write
\beaa
 \PF\nabc_3A &=&-\frac{1}{2}\PF\tr\Xb A+ \frac 1 2 \DDc \hot \mathfrak{B}+  3H  \hot  \mathfrak{B}-\left(3 \ov{P}+2\PF\ov{\PF}\right)\mathfrak{F}
\eeaa
we compute
\beaa
 \ov{\DDc}\c( \PF\nabc_3A) &=&-\frac{1}{2}\PF\tr\Xb  \ov{\DDc}\c A-\frac{1}{2} \tr\Xb\ov{\DDc}\PF \c  A-\frac{1}{2}\PF \ov{\DDc}\tr\Xb \c A\\
 &&+ \frac 1 2 \ov{\DDc}\c \DDc \hot \mathfrak{B}+  3H   \ov{\DDc}\c \mathfrak{B}+  3\ov{\DDc} \c H   \mathfrak{B}\\
 &&-\left(3 \ov{P}+2\PF\ov{\PF}\right)\ov{\DDc}\c\mathfrak{F}-\ov{\DDc}\left(3 \ov{P}+2\PF\ov{\PF}\right)\c\mathfrak{F}
\eeaa
Using  \eqref{DDc-ov-PF-red}, \eqref{Codazzi-2-red}, \eqref{DDc-P-red}, the above becomes
\beaa
 \ov{\DDc}\c( \PF\nabc_3A) &=&-\frac{1}{2}\PF\tr\Xb  \ov{\DDc}\c A+\PF\left(\frac{3}{2} \tr\Xb-\frac{1}{2}\ov{\tr \Xb} \right)\ov{\Hb} \c A+ \frac 1 2 \ov{\DDc}\c \DDc \hot \mathfrak{B}\\
 &&+  3H   \ov{\DDc}\c \mathfrak{B}+  3(\ov{\DDc} \c H )  \mathfrak{B}-\left(3 \ov{P}+2\PF\ov{\PF}\right)\ov{\DDc}\c\mathfrak{F}\\
 &&+\left(\left(9\ov{P} -2\PF\ov{\PF} \right)\ov{\Hb}+4\PF\ov{\PF} \ov{H} \right)\c\mathfrak{F}
\eeaa
We therefore obtain
\beaa
I_1&=& -\left(\frac 3 2 \tr\Xb +\frac 1 2\ov{\tr\Xb}\right) \PF \left(\ov{\DDc}\c A   \right)+\PF \left(\frac 1 2 \tr\Xb \ov{H}-\frac{1}{2}\ov{\tr \Xb}\ov{\Hb}\right) \c A \\
&&+ \frac 1 2 \ov{\DDc}\c \DDc \hot \mathfrak{B}+  3H   \ov{\DDc}\c \mathfrak{B}+  3(\ov{\DDc} \c H )  \mathfrak{B}\\
 &&-\left(3 \ov{P}+2\PF\ov{\PF}\right)\ov{\DDc}\c\mathfrak{F}+\left(\left(9\ov{P} -2\PF\ov{\PF} \right)\ov{\Hb}+4\PF\ov{\PF} \ov{H} \right)\c\mathfrak{F}\\
 &&+\left(\ov{H}+ 3 \ov{ \Hb}  \right)\c \left( -\frac{1}{2}\PF\tr\Xb A+ \frac 1 2 \DDc \hot \mathfrak{B}+  3H  \hot  \mathfrak{B}-\left(3 \ov{P}+2\PF\ov{\PF}\right)\mathfrak{F}\right)
\eeaa
which finally gives
\beaa
I_1&=& -\left(\frac 3 2 \tr\Xb +\frac 1 2\ov{\tr\Xb}\right) \PF \left(\ov{\DDc}\c A  + \ov{\Hb} \c A \right) \\
 &&+ \frac 1 2 \ov{\DDc}\c \DDc \hot \mathfrak{B}+  3H   \ov{\DDc}\c \mathfrak{B}+\left(\ov{H}+ 3 \ov{ \Hb}  \right)\c \left(  \frac 1 2 \DDc \hot \mathfrak{B}\right)+  3\left(\ov{\DDc} \c H+\ov{H} \c H \right)  \mathfrak{B}\\
 &&+9 \ov{ \Hb} \c \left(  H  \hot  \mathfrak{B}\right)-\left(3 \ov{P}+2\PF\ov{\PF}\right)\ov{\DDc}\c\mathfrak{F}+\left(\left( -8\PF\ov{\PF} \right)\ov{\Hb}+\left(-3\ov{P}+2\PF\ov{\PF} \right)\ov{H} \right)\c\mathfrak{F}
\eeaa
We compute $I_2$. Using \eqref{nabc-3-PF-red} and \eqref{nabc-3-P-red} we have
\beaa
I_2&=&\left(2  \nabc_3\PF \ov{\PF} +2 \PF  \nabc_3\ov{\PF} - 3 \nabc_3\ov{P}\right)\mathfrak{X} =\left(\frac{9}{2}\tr\Xb \ov{P}+ \left(\tr \Xb -2  \ov{\tr \Xb}\right) \PF \ov{\PF} \right)\mathfrak{X} \\
&=& \left(\tr \Xb -2  \ov{\tr \Xb}\right) \PF \ov{\PF} \  \mathfrak{X} +\left(\frac{3}{2}\tr\Xb  \right)\left(3\ov{P}\mathfrak{X} \right)
\eeaa
We use \eqref{nabb-4-mathfrak-B} again to write
\beaa
3\ov{P}\mathfrak{X}&=&- \nabc_4\mathfrak{B}-3\ov{\tr X} \mathfrak{B}+\PF \left(\ov{\DDc}\c A  +  \ov{ \Hb} \c A \right)+ 2 \PF \ov{\PF} \mathfrak{X}
 \eeaa
 and substituting in the above we obtain
\beaa
I_2&=&- \frac{3}{2}\tr\Xb \nabc_4\mathfrak{B}-\frac{9}{2}\tr\Xb \ov{\tr X} \mathfrak{B}+\frac{3}{2}\tr\Xb \PF \left(\ov{\DDc}\c A  +  \ov{ \Hb} \c A \right)+ \left(4\tr \Xb -2  \ov{\tr \Xb}\right) \PF \ov{\PF} \  \mathfrak{X} 
\eeaa
We compute $I_3$. Using \eqref{nabc-3-mathfrak-X} we obtain
\beaa
I_3&=& \left(2 \PF \ov{\PF} - 3\ov{P}\right)\nabc_3\mathfrak{X}\\
&=& -\frac 1 2 \ov{\tr \Xb}\left(2 \PF \ov{\PF} - 3\ov{P}\right) \mathfrak{X}  -\left(2 \PF \ov{\PF} - 3\ov{P}\right) \ov{\DDc} \c \mathfrak{F}\\
&&  - \left(2 \PF \ov{\PF} - 3\ov{P}\right) \ov{H} \c \mathfrak{F} + \left(-4 \PF \ov{\PF} +6\ov{P}\right)\mathfrak{B} 
\eeaa
Putting the above together we obtain
  \beaa
  \nabc_3\nabc_4\mathfrak{B}&=&-3\ov{\tr X} \nabc_3\mathfrak{B}- \frac{3}{2}\tr\Xb \nabc_4\mathfrak{B}+\left(-\frac{9}{2}\tr\Xb \ov{\tr X} +\frac{3}{2}\ov{\tr\Xb} \ \ov{\tr X}-4 \PF \ov{\PF}\right)\mathfrak{B}\\
   &&+ \frac 1 2 \ov{\DDc}\c \DDc \hot \mathfrak{B}+  3H   \ov{\DDc}\c \mathfrak{B}+\left(\ov{H}+ 3 \ov{ \Hb}  \right)\c \left(  \frac 1 2 \DDc \hot \mathfrak{B}\right)+9 \ov{ \Hb} \c \left(  H  \hot  \mathfrak{B}\right)\\
  && - \frac 1 2\ov{\tr\Xb} \PF \left(\ov{\DDc}\c A  + \ov{\Hb} \c A \right) -\left(4\PF\ov{\PF}\right)\ov{\DDc}\c\mathfrak{F}+\left(\left( -8\PF\ov{\PF} \right)\ov{\Hb}\right)\c\mathfrak{F}\\
&&+ \left(4\tr \Xb -  3\ov{\tr \Xb}\right) \PF \ov{\PF} \  \mathfrak{X}  +\frac 3 2 \ov{\tr \Xb} \ \ov{P}\mathfrak{X} 
 \eeaa
Finally using \eqref{nabb-4-mathfrak-B} to write
\beaa
\PF \left(\ov{\DDc}\c A  +  \ov{ \Hb} \c A \right)&=& \nabc_4\mathfrak{B}+3\ov{\tr X} \mathfrak{B}-\left(2 \PF \ov{\PF} - 3\ov{P}\right)\mathfrak{X}
 \eeaa
 we obtain
  \beaa
  \nabc_3\nabc_4\mathfrak{B}&=&-3\ov{\tr X} \nabc_3\mathfrak{B}- \left(\frac{3}{2}\tr\Xb +\frac 1 2\ov{\tr\Xb}\right)\nabc_4\mathfrak{B}+\left(-\frac{9}{2}\tr\Xb \ov{\tr X} -4 \PF \ov{\PF}\right)\mathfrak{B}\\
   &&+ \frac 1 2 \ov{\DDc}\c \DDc \hot \mathfrak{B}+  3H   \ov{\DDc}\c \mathfrak{B}+\left(\ov{H}+ 3 \ov{ \Hb}  \right)\c \left(  \frac 1 2 \DDc \hot \mathfrak{B}\right)+9 \ov{ \Hb} \c \left(  H  \hot  \mathfrak{B}\right)\\
 &&-\left(4\PF\ov{\PF}\right)\ov{\DDc}\c\mathfrak{F}+\left(\left( -8\PF\ov{\PF} \right)\ov{\Hb}\right)\c\mathfrak{F}+ \left(4\tr \Xb -  2\ov{\tr \Xb}\right) \PF \ov{\PF} \  \mathfrak{X}  
 \eeaa
 which gives the operator $\TT_1$ as in \eqref{operator-Teukolsky-B}. This completes the derivation of the Teukolsky equation for $\mathfrak{B}$.

\subsection{The Teukolsky equation for $\mathfrak{F}$}

 Recall the relation \eqref{relation-F-B-A}:
\beaa
\nabc_4 \mathfrak{F}+\left(\frac 3 2 \ov{\tr X} +\frac 1 2  \tr X\right)\mathfrak{F}&=& -\frac 1 2 \DDc \hot \mathfrak{X} - \left(\frac 3 2   H+ \frac 1 2 \Hb \right) \hot \mathfrak{X} -\PF A
\eeaa
We apply $\nabc_3$ to the above, and using \eqref{nabc-3-tr-X} we obtain
\beaa
\nabc_3\nabc_4 \mathfrak{F}&=&-\left(\frac 3 2 \ov{\tr X} +\frac 1 2  \tr X\right)\nabc_3\mathfrak{F}\\
&&+\left(\frac 1 4\tr\Xb\tr X +\frac 3 4\ov{\tr\Xb} \ \ov{\tr X}-3\ov{P} -P  - \frac 3 2\ov{\DDc}\c H- \frac 1 2\DDc\c\ov{H}-2H\c\ov{H}\right)\mathfrak{F}\\
&&+K_1 +K_2 +K_3
\eeaa
where 
\beaa
K_1&=&  -\frac 1 2 \nabc_3(\DDc \hot \mathfrak{X}), \qquad K_2= -\nabc_3\left( \left(\frac 3 2   H+ \frac 1 2 \Hb \right) \hot \mathfrak{X}\right), \qquad K_3= -\nabc_3(\PF A)
\eeaa

We compute $K_1$. 
Recall the following commutator formula, for $F=f+i \dual f \in \sk_1(\CCC)$, see Lemma 5.3 in \cite{GKS}:
 \bea
 \, [\nabc_3, \DDc \hot ]F &=&- \frac 1 2 \tr \Xb \left( \DDc \hot F + (1+s)H \hot F \right)  + H \hot \nabc_3 F  \label{commutator-nabc-3-F-formula}
 \eea

Using  \eqref{nabc-3-mathfrak-X} and \eqref{commutator-nabc-3-F-formula} for $s=2$, we have
\beaa
K_1&=&  -\frac 1 2 \DDc \hot (\nabc_3\mathfrak{X}) -\frac 1 2 [\nabc_3, \DDc \hot] \mathfrak{X}\\
&=&  -\frac 1 2 \DDc \hot (\nabc_3\mathfrak{X}) + \frac 1 4 \tr \Xb \left( \DDc \hot  \mathfrak{X}+ 3H \hot  \mathfrak{X} \right)  -\frac 1 2  H \hot \nabc_3  \mathfrak{X} \\
&=&  -\frac 1 2 \DDc \hot (-\frac 1 2 \ov{\tr \Xb} \ \mathfrak{X}-\ov{\DDc} \c \mathfrak{F}-\ov{H} \c \mathfrak{F}-2\mathfrak{B} ) + \frac 1 4 \tr \Xb \left( \DDc \hot  \mathfrak{X}+ 3H \hot  \mathfrak{X} \right)\\
&&  -\frac 1 2  H \hot (-\frac 1 2 \ov{\tr \Xb} \ \mathfrak{X}-\ov{\DDc} \c \mathfrak{F} -\ov{H} \c \mathfrak{F}-2\mathfrak{B} )
\eeaa
which gives using \eqref{Codazzi-2-red}
\beaa
K_1&=&  \frac 1 2 \DDc \hot (\ov{\DDc} \c \mathfrak{F} +\ov{H} \c \mathfrak{F} )  +\frac 1 2  H \hot (\ov{\DDc} \c \mathfrak{F} +\ov{H} \c \mathfrak{F}) \\
&& + \DDc \hot \mathfrak{B} +  H \hot \mathfrak{B}  + \frac 1 4 \left(\tr \Xb+\ov{\tr \Xb}  \right) \DDc \hot  \mathfrak{X}  +\left( \frac 1 4  (\tr\Xb-\ov{\tr \Xb})\Hb+ \frac 3 4 \tr \Xb  H+\frac 1 4  \ov{\tr \Xb}  H \right)\hot  \mathfrak{X}    
\eeaa
Using  \eqref{relation-F-B-A}  to write
\beaa
 \DDc \hot \mathfrak{X}&=&-2\nabc_4 \mathfrak{F}-\left( 3  \ov{\tr X} +  \tr X\right)\mathfrak{F}- \left(3   H+  \Hb \right) \hot \mathfrak{X}-2\PF A    
  \eeaa
we obtain
\beaa
K_1&=&  \frac 1 2 \DDc \hot (\ov{\DDc} \c \mathfrak{F} +\ov{H} \c \mathfrak{F} )  +\frac 1 2  H \hot (\ov{\DDc} \c \mathfrak{F} +\ov{H} \c \mathfrak{F})  + \DDc \hot \mathfrak{B} +  H \hot \mathfrak{B} \\
&& + \frac 1 4 \left(\tr \Xb+\ov{\tr \Xb}  \right) \big(-2\nabc_4 \mathfrak{F}-\left( 3  \ov{\tr X} +  \tr X\right)\mathfrak{F}- \left(3   H+  \Hb \right) \hot \mathfrak{X}-2\PF A  \big) \\
&& +\left( \frac 1 4  (\tr\Xb-\ov{\tr \Xb})\Hb+ \frac 3 4 \tr \Xb  H+\frac 1 4  \ov{\tr \Xb}  H \right)\hot  \mathfrak{X}    \\
&=&- \frac 1 2 \left(\tr \Xb+\ov{\tr \Xb}  \right)\nabc_4 \mathfrak{F}- \frac 1 4\left( \tr \Xb    \tr X+3 \tr \Xb  \ov{\tr X}+3  \ov{\tr X}\ov{\tr \Xb}+\ov{\tr \Xb}   \tr X\right)\mathfrak{F}\\
&& + \frac 1 2 \DDc \hot (\ov{\DDc} \c \mathfrak{F}+\ov{H} \c \mathfrak{F} )  +\frac 1 2  H \hot (\ov{\DDc} \c \mathfrak{F} +\ov{H} \c \mathfrak{F})  + \DDc \hot \mathfrak{B} +  H \hot \mathfrak{B} \\
&& - \frac 1 2 \left(\tr \Xb+\ov{\tr \Xb}  \right)\PF A   -\frac 1 2  \ov{\tr \Xb}\left( \Hb+  H \right)\hot  \mathfrak{X}  
\eeaa
Using  \eqref{relation-F-A-B} to write
\beaa
\DDc \hot \mathfrak{B}&=& 2 \PF\nabc_3A+\PF\tr\Xb A -  6H  \hot  \mathfrak{B} +2\left(3 \ov{P}+2\PF\ov{\PF}\right)\mathfrak{F}
  \eeaa
we obtain
\beaa
K_1&=&- \frac 1 2 \left(\tr \Xb+\ov{\tr \Xb}  \right)\nabc_4 \mathfrak{F}- \frac 1 4\left( \tr \Xb    \tr X+3 \tr \Xb  \ov{\tr X}+3  \ov{\tr X}\ov{\tr \Xb}+\ov{\tr \Xb}   \tr X\right)\mathfrak{F}+2\left(3 \ov{P}+2\PF\ov{\PF}\right)\mathfrak{F}\\
&&  \frac 1 2 \DDc \hot (\ov{\DDc} \c \mathfrak{F} +\ov{H} \c \mathfrak{F} )  +\frac 1 2  H \hot (\ov{\DDc} \c \mathfrak{F} +\ov{H} \c \mathfrak{F})  + 2 \PF\nabc_3A \\
&&+\frac 1 2 \left(\tr \Xb-\ov{\tr \Xb}  \right)\PF A  -  5H  \hot  \mathfrak{B}  -\frac 1 2  \ov{\tr \Xb}\left( \Hb+  H \right)\hot  \mathfrak{X}    
\eeaa
We compute $K_2$. Using \eqref{nabc-3-Hb-red} and \eqref{nabc-3-mathfrak-X}
\beaa
K_2&=& -\nabc_3 \left(\frac 3 2   H+ \frac 1 2 \Hb \right) \hot \mathfrak{X} -\left(\frac 3 2   H+ \frac 1 2 \Hb \right) \hot \nabc_3\mathfrak{X}\\
&=& \left(-\frac 3 2  \nabc_3 H+\frac{1}{4}\ov{\tr\Xb}(\Hb-H) \right) \hot \mathfrak{X} -\left(\frac 3 2   H+ \frac 1 2 \Hb \right) \hot \left(-\frac 1 2 \ov{\tr \Xb} \ \mathfrak{X}-\ov{\DDc} \c \mathfrak{F} -\ov{H} \c \mathfrak{F}-2\mathfrak{B}  \right)\\
&=&\left(\frac 3 2   H+ \frac 1 2 \Hb \right) \hot \ov{\DDc} \c \mathfrak{F}  +\left(\frac 3 2   H+ \frac 1 2 \Hb \right) \hot \left(\ov{H} \c \mathfrak{F} \right)\\
&&+ \left(-\frac 3 2  \nabc_3 H+\frac{1}{2}\ov{\tr\Xb}(\Hb+H) \right) \hot \mathfrak{X} +\left(3   H+ \Hb \right) \hot \mathfrak{B}  
\eeaa
We compute $K_3$. Using \eqref{nabc-3-PF-red} we obtain
\beaa
K_3&=& -\nabc_3(\PF) A -\PF \nabc_3A=  -\PF \nabc_3A+\tr\Xb \PF A
\eeaa
We therefore obtain
\beaa
\nabc_3\nabc_4 \mathfrak{F}&=&-\left(\frac 3 2 \ov{\tr X} +\frac 1 2  \tr X\right)\nabc_3\mathfrak{F}- \frac 1 2 \left(\tr \Xb+\ov{\tr \Xb}  \right)\nabc_4 \mathfrak{F}\\
&&+\left(- \frac 3 4 \tr \Xb  \ov{\tr X}- \frac 1 4\ov{\tr \Xb}   \tr X+3\ov{P} -P +4\PF\ov{\PF} - \frac 3 2\ov{\DDc}\c H- \frac 1 2\DDc\c\ov{H}-2H\c\ov{H}\right)\mathfrak{F}\\
&&+  \frac 1 2 \DDc \hot (\ov{\DDc} \c \mathfrak{F} +\ov{H} \c \mathfrak{F} )  +\left(2   H+ \frac 1 2 \Hb \right) \hot (\ov{\DDc} \c \mathfrak{F} +\ov{H} \c \mathfrak{F} ) \\
&& +  \PF\nabc_3A +\frac 1 2 \left(3\tr \Xb-\ov{\tr \Xb}  \right)\PF A    + \left(-\frac 3 2  \nabc_3 H \right) \hot \mathfrak{X} +\left(-2  H+ \Hb \right) \hot \mathfrak{B}  
\eeaa
 which gives the operator $\TT_2$ as in \eqref{operator-Teukolsky-F} and this ends the derivation of the Teukolsky equation for $\mathfrak{F}$.

\subsection{The Teukolsky equation for $A$}\label{sec:teuyk-A}
This derivation is similar to the one obtained in \cite{GKS}. 
Recall \eqref{nabc-3-A}:
\beaa
 \nabc_3A+\frac{1}{2}\tr\Xb A  &=&\DDc\hot B  +H\hot \left(4B-3\ov{\PF} \BF \right) -3\ov{P}\Xh-2\ov{\PF} \mathfrak{F}
 \eeaa
 We apply $\nabc_4$ to the above, and using \eqref{nabc-4-tr-Xb} we obtain
 \beaa
 \nabc_4 \nabc_3 A&=& -\frac{1}{2}\tr\Xb \nabc_4A +\left(\frac 1 4 \tr X\tr\Xb-\frac 1 2  \DDc\c\ov{\Hb}-\frac 1 2 \Hb\c\ov{\Hb}-\ov{P} \right) A  \\
 &&+J_1  +J_2 +J_3+J_4
 \eeaa
 where 
 \beaa
 J_1&=& \nabc_4\DDc\hot B, \qquad  J_2= \nabc_4\left(H\hot \left(4B-3\ov{\PF} \BF \right)\right)\\
 J_3&=&  -3\nabc_4(\ov{P}\Xh), \qquad J_4= -2\nabc_4(\ov{\PF} \mathfrak{F})
 \eeaa
We compute $J_1$. Using \eqref{commutator-nabc-4-F-formula} for $s=1$ we obtain
\beaa
 J_1&=& \DDc\hot (\nabc_4B)+[\nabc_4, \DDc \hot] B= \DDc\hot (\nabc_4B)- \frac 1 2 \tr X( \DDc\hot B )+ \underline{H} \hot \nabc_4 B
\eeaa
Observe that the Bianchi identity \eqref{nabc-4-B} can be written in terms of $\mathfrak{X}$ as
\bea\label{nabc-4-B-X}
 \nabc_4B+2\ov{\tr X} B +\frac 3 2 \ov{\tr X} \ov{\PF}\BF-3\ov{P} \,\Xi  &=&\frac{1}{2}\ov{\DDc}\c A  +\frac{1}{2}A\c  \ov{\Hb}+ \ov{\PF}\mathfrak{X}
 \eea
This gives
\beaa
 \DDc\hot (\nabc_4B)&=&  -2\ov{\tr X} \DDc\hot B-2\DDc\ov{\tr X} \hot  B -\frac 3 2 \DDc\ov{\tr X} \hot \ov{\PF}\BF -\frac 3 2 \ov{\tr X} \DDc \ov{\PF} \hot \BF\\
 && -\frac 3 2 \ov{\tr X} \ov{\PF} \DDc \hot \BF+3\ov{P} \, \DDc \hot \Xi +3\DDc\ov{P} \hot \Xi + \ov{\PF}\DDc \hot \mathfrak{X} + \DDc \ov{\PF}\hot\mathfrak{X} \\
 &&+\frac{1}{2}\DDc \hot \ov{\DDc}\c A  +\frac{1}{2}\DDc \hot (A\c  \ov{\Hb})
\eeaa
We therefore obtain using again \eqref{nabc-4-B-X}:
\beaa
 J_1 &=& - \left(\frac 1 2 \tr X+2\ov{\tr X}\right) \DDc\hot B +\left(-2\DDc\ov{\tr X}-2\ov{\tr X} \Hb \right)\hot  B \\
 && -\frac 3 2 \ov{\tr X} \ov{\PF} \DDc \hot \BF -\frac 3 2 \left(\ov{\PF}\DDc\ov{\tr X} +\ov{\tr X} \DDc \ov{\PF}+ \ov{\tr X} \ov{\PF} \Hb \right)\hot \BF\\
 &&+3\ov{P} \, \DDc \hot \Xi +3\left(\DDc\ov{P} + \Hb \ov{P}\right)\hot \Xi + \ov{\PF}\DDc \hot \mathfrak{X} + \left(\DDc \ov{\PF}+ \Hb \ov{\PF}\right)\hot\mathfrak{X} \\
 &&+\frac{1}{2}\DDc \hot \ov{\DDc}\c A  +\frac{1}{2}\DDc \hot (A\c  \ov{\Hb})+ \frac{1}{2}\underline{H} \hot \left(\ov{\DDc}\c A   \right)+ \frac{1}{2}\underline{H} \hot \left(A\c  \ov{\Hb} \right)
\eeaa
We compute $J_2$. Using \eqref{nabc-4-H-red},  \eqref{nabc-4-B-X} and  \eqref{nabc-4-PF-red} we obtain
\beaa
 J_2&=& \nabc_4 H\hot \left(4B-3\ov{\PF} \BF \right)+ H\hot \left(4 \nabc_4B-3\ov{\PF} \nabc_4\BF-3\nabc_4\ov{\PF} \BF \right)\\
 &=&  -\frac{1}{2}\ov{\tr X}(H-\Hb) \hot \left(4B-3\ov{\PF} \BF \right)+ H\hot \Big[4 (-2\ov{\tr X} B -\frac 3 2 \ov{\tr X} \ov{\PF}\BF+3\ov{P} \,\Xi)\\
 && +4(\frac{1}{2}\ov{\DDc}\c A  +\frac{1}{2}A\c  \ov{\Hb}+ \ov{\PF}\mathfrak{X})-3\ov{\PF} \nabc_4\BF-3(-\tr X \ov{\PF} )\BF \Big]
\eeaa
Writing $\nabc_4 \BF=\mathfrak{X}-\frac 3 2 \ov{\tr X} \BF+2\PF \Xi $, the above becomes
\beaa
 J_2 &=&\left(  -2\ov{\tr X}(5H-\Hb)  \right)\hot B  +\frac{3}{2} \left(-\ov{\tr X} \ \ov{\PF}\Hb+ 2 \tr X \ov{\PF}H \right) \hot  \BF \\
 &&+ \left(12 \ov{P}-6\PF \ov{\PF}\right) H \hot \Xi + \ov{\PF} H \hot \mathfrak{X}+ 2H\hot \left( \ov{\DDc}\c A  \right)+2 H\hot \left(A\c  \ov{\Hb} \right)
\eeaa

We compute $J_3$. Using \eqref{nabc-4-P-red} and \eqref{nabc-4-Xh}, we obtain
\beaa
 J_3&=&  -3\nabc_4(\ov{P})\Xh -3\ov{P}\nabc_4(\Xh)\\
 &=&  -3\left( -\frac{3}{2}\ov{\tr X} \ \ov{P} -\ov{\tr X} \PF \ov{\PF} \right)\Xh -3\ov{P}\left(-\frac 1 2 (\tr X + \ov{\tr X}) \Xh+ \DDc\hot \Xi+  \Xi\hot(\Hb+H)-A \right)\\
  &=&3 \ov{P} A +  \left( \left(\frac 3 2 \tr X+  6\ov{\tr X} \right)\ov{P} +3\ov{\tr X} \PF \ov{\PF} \right)\Xh -3\ov{P}\left( \DDc\hot \Xi+  \Xi\hot(\Hb+H) \right)
\eeaa

We compute $J_4$. Using \eqref{nabc-4-PF-red} we obtain
\beaa
J_4&=& -2\nabc_4(\ov{\PF}) \mathfrak{F} -2\ov{\PF} \nabc_4\mathfrak{F}=2\ov{\PF}\left(  - \nabc_4\mathfrak{F}+ \tr X  \mathfrak{F}\right)
\eeaa

Summing the expressions obtained for $J_1$, $J_2$, $J_3$ and $J_4$ we obtain
 \beaa
&& \nabc_4 \nabc_3 A\\
&=& -\frac{1}{2}\tr\Xb \nabc_4A +\left(\frac 1 4 \tr X\tr\Xb-\frac 1 2  \DDc\c\ov{\Hb}-\frac 1 2 \Hb\c\ov{\Hb}+2\ov{P} \right) A  \\
 &&+\frac{1}{2}\DDc \hot \ov{\DDc}\c A  +\frac{1}{2}\DDc \hot (A\c  \ov{\Hb})+ \frac{1}{2}\underline{H} \hot \left(\ov{\DDc}\c A   \right)+ \frac{1}{2}\underline{H} \hot \left(A\c  \ov{\Hb} \right)+ 2H\hot \left( \ov{\DDc}\c A  \right)\\
 &&+2 H\hot \left(A\c  \ov{\Hb} \right) - \left(\frac 1 2 \tr X+2\ov{\tr X}\right) \DDc\hot B +\left(-2\DDc\ov{\tr X}  -10\ov{\tr X}H \right)\hot  B \\
 && -\frac 3 2 \ov{\tr X} \ov{\PF} \DDc \hot \BF -\frac 3 2 \left(\ov{\PF}\DDc\ov{\tr X} +\ov{\tr X} \DDc \ov{\PF}+2 \ov{\tr X} \ov{\PF} \Hb -2 \tr X \ov{\PF}H\right)\hot \BF\\
 && +3\left(\DDc\ov{P} +\left(3 \ov{P}-2\PF \ov{\PF}\right) H\right)\hot \Xi + \ov{\PF}\DDc \hot \mathfrak{X} + \left(\DDc \ov{\PF}+ \Hb \ov{\PF}+ \ov{\PF} H\right)\hot\mathfrak{X} \\
 && +  \left( \left(\frac 3 2 \tr X+  6\ov{\tr X} \right)\ov{P} +3\ov{\tr X} \PF \ov{\PF} \right)\Xh +2\ov{\PF}\left(  - \nabc_4\mathfrak{F}+ \tr X  \mathfrak{F}\right)
 \eeaa
 Using \eqref{nabc-3-A} to write
\beaa
\DDc\hot B  &=& \nabc_3A+\frac{1}{2}\tr\Xb A -H\hot \left(4B-3\ov{\PF} \BF \right) +3\ov{P}\Xh+2\ov{\PF} \mathfrak{F} 
 \eeaa
the above becomes
 \beaa
 &&\nabc_4 \nabc_3 A\\
 &=& - \left(\frac 1 2 \tr X+2\ov{\tr X}\right)\nabc_3A-\frac{1}{2}\tr\Xb \nabc_4A +\left(-\ov{\tr X} \tr \Xb -\frac 1 2  \DDc\c\ov{\Hb}-\frac 1 2 \Hb\c\ov{\Hb}+2\ov{P} \right) A  \\
 &&+\frac{1}{2}\DDc \hot \ov{\DDc}\c A  +\frac{1}{2}\DDc \hot (A\c  \ov{\Hb})+ \frac{1}{2}\underline{H} \hot \left(\ov{\DDc}\c A   \right)+ \frac{1}{2}\underline{H} \hot \left(A\c  \ov{\Hb} \right)+ 2H\hot \left( \ov{\DDc}\c A  \right)\\
 &&+2 H\hot \left(A\c  \ov{\Hb} \right) +\left(-2\DDc\ov{\tr X} +2 \tr X H -2\ov{\tr X}H \right)\hot  B  -\frac 3 2 \ov{\tr X} \ov{\PF} \DDc \hot \BF\\
 && -\frac 3 2 \left(\ov{\PF}\DDc\ov{\tr X} +\ov{\tr X} \DDc \ov{\PF}+2 \ov{\tr X} \ov{\PF} \Hb +4\ov{\tr X} H - \tr X \ov{\PF}H\right)\hot \BF\\
 && +3\left(\DDc\ov{P} +\left(3 \ov{P}-2\PF \ov{\PF}\right) H\right)\hot \Xi+ \ov{\PF}\DDc \hot \mathfrak{X} + \left(\DDc \ov{\PF}+ \Hb \ov{\PF}+ \ov{\PF} H\right)\hot\mathfrak{X} \\
 && +  \left( 3\ov{\tr X} \PF \ov{\PF} \right)\Xh +2\ov{\PF}\left(  - \nabc_4\mathfrak{F}+\left(\frac 1 2  \tr X -2\ov{\tr X} \right)\mathfrak{F}\right)
 \eeaa
 Observe that using \eqref{Codazzi-1-red} and \eqref{DDc-ov-P-red}, the coefficients of $B$ and $\Xi$ vanish. Finally writing $ \DDc \hot \BF =-2\mathfrak{F}  -3  H \hot \BF +2\PF\Xh$,  we obtain
  \beaa
&& \nabc_4 \nabc_3 A\\
&=& - \left(\frac 1 2 \tr X+2\ov{\tr X}\right)\nabc_3A-\frac{1}{2}\tr\Xb \nabc_4A +\left(-\ov{\tr X} \tr \Xb -\frac 1 2  \DDc\c\ov{\Hb}-\frac 1 2 \Hb\c\ov{\Hb}+2\ov{P} \right) A  \\
 &&+\frac{1}{2}\DDc \hot \ov{\DDc}\c A  +\frac{1}{2}\DDc \hot (A\c  \ov{\Hb})+ \frac{1}{2}\underline{H} \hot \left(\ov{\DDc}\c A   \right)+ \frac{1}{2}\underline{H} \hot \left(A\c  \ov{\Hb} \right)+ 2H\hot \left( \ov{\DDc}\c A  \right)+2 H\hot \left(A\c  \ov{\Hb} \right)\\
 && -\frac 3 2 \left(\ov{\PF}\DDc\ov{\tr X} +\ov{\tr X} \DDc \ov{\PF}+2 \ov{\tr X} \ov{\PF} \Hb +\ov{\tr X}\ov{\PF} H - \tr X \ov{\PF}H\right)\hot \BF\\
 && + \ov{\PF}\DDc \hot \mathfrak{X} + \left(\DDc \ov{\PF}+ \Hb \ov{\PF}+ \ov{\PF} H\right)\hot\mathfrak{X}   +2\ov{\PF}\left(  - \nabc_4\mathfrak{F}+\left(\frac 1 2  \tr X -\frac 1 2 \ov{\tr X} \right)\mathfrak{F}\right)
 \eeaa
Observe that using  \eqref{Codazzi-1-red} and \eqref{DDc-ov-PF-red}, the coefficient of $\BF$ vanishes. Finally we can make use of \eqref{relation-F-B-A} to write
\beaa
 \DDc \hot \mathfrak{X}&=&-2\nabc_4 \mathfrak{F}-\left( 3  \ov{\tr X} +  \tr X\right)\mathfrak{F}- \left(3   H+  \Hb \right) \hot \mathfrak{X}-2\PF A    
\eeaa
and obtain
  \beaa
 \nabc_4 \nabc_3 A&=& - \left(\frac 1 2 \tr X+2\ov{\tr X}\right)\nabc_3A-\frac{1}{2}\tr\Xb \nabc_4A\\
&& +\left(-\ov{\tr X} \tr \Xb -\frac 1 2  \DDc\c\ov{\Hb}-\frac 1 2 \Hb\c\ov{\Hb}+2\ov{P}-2\PF\ov{\PF} \right) A  \\
 &&+\frac{1}{2}\DDc \hot \ov{\DDc}\c A  +\frac{1}{2}\DDc \hot (A\c  \ov{\Hb})+ \frac{1}{2}\underline{H} \hot \left(\ov{\DDc}\c A   \right)+ \frac{1}{2}\underline{H} \hot \left(A\c  \ov{\Hb} \right)\\
 &&+ 2H\hot \left( \ov{\DDc}\c A  \right)+2 H\hot \left(A\c  \ov{\Hb} \right)  -2\ov{\PF}\left(  2 \nabc_4\mathfrak{F} +2 \ov{\tr X} \mathfrak{F}+ \left(\Hb + H\right)\hot\mathfrak{X} \right)
 \eeaa
which gives the operator $\TT_3(A)$, as in \cite{GKS}. This completes the derivation of the Teukolsky equation for $A$.

\section{Derivation of the generalized Regge-Wheeler system}\label{proof-7.1-section}\label{RWeqapp}

\subsection{Proof of Proposition \ref{proposition-commutator}}\label{sec:proof-commutators}

\subsubsection{The commutators for $\PP_C$}

\begin{lemma}\label{commutator-RW} Let $\Psi \in \sk_k(\CCC)$ of conformal type $s$. Recall the definition of $\PP_C(\Psi)$, see \eqref{definition-P-C}, 
\beaa
\mathcal{P}_{C}(\Psi)= \nabc_3 \Psi + C  \Psi \in \sk_k(\CCC)
\eeaa 
Let $F \in \sk_1(\CCC)$ of conformal type $s$. Then the following commutators hold:
\beaa
\, [\PP_C, \nabc_3]F&=&  - (\nabc_3C)  F\\
\, [\PP_C, \nabc_4]F&=&2(\eta-\etab) \c \nabc F +\left(2s\left(\rho+\rhoF^2+\dual\rhoF^2-\eta\c\etab\right)+2 i \left(-\rhod+\eta \wedge \etab\right)-\nabc_4 C\right)  F  \\
\, [\PP_C, \nabc_a]F&=&-\frac  1 2   \trchb\, \nabc_a F-\frac 1 2 \atrchb\, \dual \nabc_a  F+\eta_a \nabc_3 F+( \mathcal{V}^s_{[3,a]} - \nabc_a C ) F \\
  \, [\PP_C, \DDc \hot ]F&=& - \frac 1 2 \tr \Xb \DDc \hot F + H \hot \nabc_3 F +\left( - \DDc  C  - \frac 1 2(s+1) \tr \Xb  H \right)\hot F 
  \eeaa
Let $U \in \sk_2(\CCC)$ of conformal type $s$. The following commutators hold:
\beaa
\, [\PP_C, \nabc_3]U&=&  -( \nabc_3C) U\\
\, [\PP_C, \nabc_4]U&=&2(\eta-\etab) \c \nabc U +\left(2s\left(\rho+\rhoF^2+\dual\rhoF^2-\eta\c\etab\right)+4 i \left(-\rhod+\eta \wedge \etab\right)-\nabc_4 C\right)  U\\
\, [\PP_C, \nabc_a]U&=&-\frac  1 2   \trchb\, \nabc_a U-\frac 1 2 \atrchb\, \dual \nabc_a  U+\eta_a \nabc_3 U+( \mathcal{V}^s_{[3,a]} - \nabc_a C ) U\\
\, [\PP_C, \ov{\DDc} \c ]U&=& - \frac 1 2\ov{\tr\Xb}\,  \ov{\DDc} \c U+\ov{H} \c \nabc_3 U+\left(- \ov{\DDc} C - \frac 1 2(s-  2) \ov{\tr\Xb}\,  \ov{H}\right)\c U
\eeaa
\end{lemma}
\begin{proof} We compute 
\beaa
\, [\PP_C, \nabc_3]F&=&(\nabc_3 + C )(\nabc_3 F)- \nabc_3 (\nabc_3F + C  F)= - (\nabc_3C)  F
\eeaa
and similarly for $U$. We compute
\beaa
\, [\PP_C, \nabc_4]F&=&(\nabc_3 + C )(\nabc_4 F)- \nabc_4 (\nabc_3F + C \ F)=[\nabc_3,  \nabc_4]  F - (\nabc_4 C ) F 
\eeaa
Recall from \eqref{commutator-3-4-F} and from \cite{GKS}:
 \bea
  \,[\nabc_3, \nabc_4] F&=2(\eta-\etab) \c \nabc F +\left(2s\left(\rho+\rhoF^2+\dual\rhoF^2-\eta\c\etab\right)+2 i \left(-\rhod+\eta \wedge \etab\right)\right)  F\label{commutator-3-4-F-real}\\
   \,[\nabc_3, \nabc_4] U&=2(\eta-\etab) \c \nabc U +\left(2s\left(\rho+\rhoF^2+\dual\rhoF^2-\eta\c\etab\right)+4 i \left(-\rhod+\eta \wedge \etab\right)\right)  U\label{commutator-3-4-U}
\eea
Using \eqref{commutator-3-4-F-real} and \eqref{commutator-3-4-U}, we obtain the stated expressions. 

We compute 
\beaa
\, [\PP_C, \nabc_a]F&=&(\nabc_3 + C )(\nabc_a F)- \nabc_a (\nabc_3F + C \ F)=[\nabc_3,  \nabc_a]  F - (\nabc_a C ) F 
\eeaa
Using that, see Lemma 5.3 in \cite{GKS},
 \bea
\, [\nabc_3, \nabc_a] F_{b} &=&-\frac  1 2   \trchb\, \nabc_a F_{b}-\frac 1 2 \atrchb\, \dual \nabc_a  F_{b}+\eta_a \nabc_3 F_{b}+\mathcal{V}^s_{[3,a]}(F) \label{commutator-3-a-F}\\
\,  [\nabc_3, \nabc_a] U_{bc} &=&-\frac  1 2   \trchb\, \nabc_a U_{bc}-\frac 1 2 \atrchb\, \dual \nabc_a  U_{bc}+\eta_a \nabc_3 U_{bc}+\mathcal{V}^s_{[3,a]}(U)\label{commutator-3-a-U}
\eea
where
\beaa
\mathcal{V}^s_{[3,a]}(F)&=& -\frac 1 2 \trchb \big(  s \eta_aU_b+\eta_b U_a-\de_{ab} \eta \c U \big) -\frac 1 2 \atrchb \big( s \dual \eta_a U_b+\eta_b \dual U_a-\in_{ab} \eta\c U\big)\\
\mathcal{V}^s_{[3,a]}(U)&=& -\frac  1 2   \trchb\, \Big(s(\eta_a) U_{bc}+\eta_bU_{ac}+\eta_c U_{ab}-\de_{a b}(\eta \c U)_c-\de_{a c}(\eta \c U)_b \Big)\\
&&-\frac 1 2 \atrchb\, \Big(s (\dual\eta_a) U_{bc} +\eta_b \dual U_{ac}+\eta_c \dual U_{ab}- \in_{a b}(\eta \c  U)_c- \in_{a c}(\eta \c  U)_b \Big)
\eeaa
we obtain the stated expressions.

We compute
\beaa
  \, [\PP_C, \DDc \hot ]F&=& (\nabc_3 + C )(\DDc \hot F) - \DDc \hot (\nabc_3F + C F)=[\nabc_3 , \DDc \hot] F -( \DDc  C) \hot  F
\eeaa
Using \eqref{commutator-nabc-3-F-formula}, we obtain the stated expression. 
We compute
\beaa
\, [\PP_C, \ov{\DDc} \c ]U&=& (\nabc_3 + C )(\ov{\DDc} \c U) - \ov{\DDc} \c (\nabc_3U + C U)=[\nabc_3, \ov{\DDc} \c] U- (\ov{\DDc} C) \c U
\eeaa
Using \eqref{commutator-nabc-3-ov-DDc-U}, we obtain the stated expression. 
\end{proof}

\subsubsection{The commutators for $[\PP_{C_1}, \TT_1]$ and $[\PP_{C_2}, \TT_2]$}

Using \eqref{operator-Teukolsky-B} and \eqref{operator-Teukolsky-F}, we separate the computations of $[  \PP_{C_1}, \TT_1]$ and $[  \PP_{C_2}, \TT_2]$ into the following terms:
\bea
\, [\PP_{C_1}, \TT_1]&=& I^{\Bfr}+J^{\Bfr}+K^{\Bfr}+L^{\Bfr}+M^{\Bfr}+N^{\Bfr} \label{comm-sum-1}\\
\, [\PP_{C_2}, \TT_2]&=& I^{\Ffr}+J^{\Ffr}+K^{\Ffr}+L^{\Ffr}+M^{\Ffr}+N^{\Ffr} \label{comm-sum-2}
\eea
where
\beaa
I^{\Bfr}&=&  -[\mathcal{P}_{C_1}, \nabc_3\nabc_4] \Bfr, \qquad I^{\Ffr}=  -[\mathcal{P}_{C_2}, \nabc_3\nabc_4] \Ffr \\
J^{\Bfr}&=&\frac 1 2 [\mathcal{P}_{C_1}, \ov{\DDc}\c  \DDc \hot ]\mathfrak{B}, \qquad J^{\Ffr}=\frac 1 2 [\mathcal{P}_{C_2}, \DDc \hot\ov{\DDc}\c ]\mathfrak{F}\\
K^{\Bfr}&=&[\PP_{C_1}, -3\ov{\tr X} \nabc_3]\Bfr, \qquad K^{\Ffr}=[\PP_{C_2},-\left(\frac 3 2 \ov{\tr X} +\frac 1 2  \tr X\right) \nabc_3]\Ffr\\
L^{\Bfr}&=&[\PP_{C_1},- \left(\frac{3}{2}\tr\Xb +\frac 1 2\ov{\tr\Xb}\right)\nabc_4]\Bfr, \qquad L^{\Ffr}=[\PP_{C_2},- \frac 1 2 \left(\tr \Xb+\ov{\tr \Xb}  \right)\nabc_4]\Ffr \\
M^{\Bfr}&=&[\PP_{C_1},\left( 6H+ \ov{H}+ 3  \ov{ \Hb}  \right)\c  \nabc ]\Bfr, \qquad M^{\Ffr}=[\PP_{C_2},\left(4   H+ \ov{H}+  \Hb \right)\c \nabc ]\Ffr \\
N^{\Bfr}&=&[\PP_{C_1},\left(-\frac{9}{2}\tr\Xb \ov{\tr X} -4 \PF \ov{\PF}+9 \ov{ \Hb} \c  H\right) ]\Bfr, \\
 N^{\Ffr}&=&[\PP_{C_2},+\left(- \frac 3 4 \tr \Xb  \ov{\tr X}- \frac 1 4\ov{\tr \Xb}   \tr X+3\ov{P} -P +4\PF\ov{\PF} - \frac 3 2\ov{\DDc}\c H\right) ]\Ffr +\frac 1 2 [\PP_{C_2}, \Hb \hot  \ov{H} \c] \Ffr
\eeaa

\subsubsection{Expressions for $I^{\Bfr}$ and $I^{\Ffr}$}

We have 
\beaa
I^{\Bfr}&=&  -[\mathcal{P}_{C_1}, \nabc_3\nabc_4] \Bfr=  -[\mathcal{P}_{C_1},  \nabc_3] \nabc_4\Bfr - \nabc_3([ \mathcal{P}_{C_1}, \nabc_4]\Bfr)\\
&=&  (  \nabc_3C_1 ) \nabc_4\Bfr - \nabc_3([ \mathcal{P}_{C_1}, \nabc_4]\Bfr)
\eeaa
Using Lemma \ref{commutator-RW} applied to $F=\Bfr$ of conformal type $s=1$,
\bea\label{PPC_1nabc4}
\begin{split}
\, [\PP_{C_1}, \nabc_4]\Bfr&=2(\eta-\etab) \c \nabc \Bfr \\
&+\left(2\left(\rho+\rhoF^2+\dual\rhoF^2-\eta\c\etab\right)+2 i \left(-\rhod+\eta \wedge \etab\right)-\nabc_4 C_1\right)  \Bfr
\end{split}
\eea
we deduce 
\beaa
\nabc_3([ \mathcal{P}_{C_1}, \nabc_4]\Bfr)&=& 2(\eta-\etab) \c \nabc_3\nabc \Bfr+2\nabc_3(\eta-\etab) \c \nabc \Bfr\\
&&+\left(2\left(\rho+\rhoF^2+\dual\rhoF^2-\eta\c\etab\right)+2 i \left(-\rhod+\eta \wedge \etab\right)-\nabc_4 C_1\right)  \nabc_3 \Bfr  \\
&&+\nabc_3\left(2\left(\rho+\rhoF^2+\dual\rhoF^2-\eta\c\etab\right)+2 i \left(-\rhod+\eta \wedge \etab\right)-\nabc_4 C_1\right)  \Bfr 
\eeaa
Using \eqref{commutator-3-a-F} we write the above as
\beaa
&&\nabc_3([ \mathcal{P}_{C_1}, \nabc_4]\Bfr)\\
&=& 2(\eta-\etab) \c \nabc\nabc_3 \Bfr+ \big( 2\nabc_3(\eta-\etab)-\trchb (\eta-\etab)+\atrchb \dual (\eta-\etab) \big)\c \nabc \Bfr\\
&&+\left(2\left(\rho+\rhoF^2+\dual\rhoF^2-\eta\c\etab+ \eta \c (\eta-\etab)\right)+2 i \left(-\rhod+\eta \wedge \etab\right)-\nabc_4 C_1\right)  \nabc_3 \Bfr  \\
&&+\Big[\nabc_3\left(2\left(\rho+\rhoF^2+\dual\rhoF^2-\eta\c\etab\right)+2 i \left(-\rhod+\eta \wedge \etab\right)-\nabc_4 C_1\right)+ 2 (\eta-\etab) \c  \mathcal{V}^{s=1}_{[3,a]} \Big]\Bfr 
\eeaa
We therefore obtain
\bea\label{I-B}
I^{\Bfr}&=&-2(\eta-\etab) \c \nabc\nabc_3 \Bfr+  I^{\Bfr}_4 \nabc_4\Bfr +I^{\Bfr}_{3} \nabc_3 \Bfr +I^{\Bfr}_a\c \nabc_a \Bfr+ I^{\Bfr}_0 \Bfr 
\eea
where
\bea
I^{\Bfr}_4&=& \nabc_3C_1 \nonumber\\
I^{\Bfr}_{3}&=& -2\rho -2\rhoF^2 -2\dual\rhoF^2 -2 \eta \c (\eta-2\etab)+ i \left(2\rhod-2\eta \wedge \etab\right)+\nabc_4 C_1 \nonumber \\
I^{\Bfr}_a&=& -  2\nabc_3(\eta-\etab)+\trchb (\eta-\etab)-\atrchb \dual (\eta-\etab) \label{eq:I-Bfr-a}\\
I^{\Bfr}_0&=& \nabc_3\left[-2\left(\rho+\rhoF^2+\dual\rhoF^2-\eta\c\etab\right)+2 i \left(\rhod-\eta \wedge \etab\right)+\nabc_4 C_1\right]- 2 (\eta-\etab) \c  \mathcal{V}^{s=1}_{[3,a]} \nonumber
\eea

We have 
\beaa
I^{\Ffr}&=&  (  \nabc_3C_2 ) \nabc_4\Ffr - \nabc_3([ \mathcal{P}_{C_2}, \nabc_4]\Ffr)
\eeaa
Using Lemma \ref{commutator-RW} applied to $U=\Ffr$ of conformal type $s=1$,
\bea\label{PPC_2nabc4}
\begin{split}
\, [\PP_{C_2}, \nabc_4]\Ffr&=2(\eta-\etab) \c \nabc \Ffr \\
&+\left(2\left(\rho+\rhoF^2+\dual\rhoF^2-\eta\c\etab\right)+4 i \left(-\rhod+\eta \wedge \etab\right)-\nabc_4 C_2\right)  \Ffr
\end{split}
\eea
We similarly obtain
\bea\label{I-F}
I^{\Ffr}&=&-2(\eta-\etab) \c \nabc\nabc_3 \Ffr+  I^{\Ffr}_4 \nabc_4\Ffr +I^{\Ffr}_{3} \nabc_3 \Ffr +I^{\Ffr}_a\c \nabc_a \Ffr+ I^{\Ffr}_0 \Ffr 
\eea
where
\bea
I^{\Ffr}_4&=& \nabc_3C_2 \nonumber\\
I^{\Ffr}_{3}&=& -2\rho -2\rhoF^2 -2\dual\rhoF^2 -2 \eta \c (\eta-2\etab)+ i \left(4\rhod-4\eta \wedge \etab\right)+\nabc_4 C_2 \nonumber \\
I^{\Ffr}_a&=& -  2\nabc_3(\eta-\etab)+\trchb (\eta-\etab)-\atrchb \dual (\eta-\etab) \label{eq:I-Ffr-a}\\
I^{\Ffr}_0&=& \nabc_3\left[-2\left(\rho+\rhoF^2+\dual\rhoF^2-\eta\c\etab\right)+4 i \left(\rhod-\eta \wedge \etab\right)+\nabc_4 C_2\right]- 2 (\eta-\etab) \c  \mathcal{V}^{s=1}_{[3,a]} \nonumber
\eea

\subsubsection{Expressions for $J^{\Bfr}$ and $J^{\Ffr}$}

We have
\beaa
J^{\Bfr}&=&\frac 1 2 [\mathcal{P}_{C_1}, \ov{\DDc}\c ] (\DDc \hot \mathfrak{B})+ \frac 1 2 \ov{\DDc}\c([ \mathcal{P}_{C_1} , \DDc \hot] \mathfrak{B})
\eeaa
Using Lemma \ref{commutator-RW} applied to $U=\DDc \hot \Bfr$ of conformal type $s=1$, we have 
\beaa
&&\, [\PP_{C_1}, \ov{\DDc} \c ](\DDc \hot \mathfrak{B})\\
&=& - \frac 1 2\ov{\tr\Xb}\,  \ov{\DDc} \c (\DDc \hot \mathfrak{B})+\ov{H} \c \nab_3 (\DDc \hot \mathfrak{B})+\left(- \ov{\DDc} C_1 + \frac 1 2\ov{\tr\Xb}\, \ov{H}\right)\c (\DDc \hot \mathfrak{B})\\
&=& - \frac 1 2\ov{\tr\Xb}\,  \ov{\DDc} \c (\DDc \hot \mathfrak{B})+\ov{H} \c \DDc \hot \nab_3 \mathfrak{B}+\ov{H} \c\left(- \frac 1 2 \tr \Xb \left( \DDc\hot \Bfr + 2H \hot \Bfr \right)  + H \hot \nabc_3 \Bfr \right)\\
&&+\left(- \ov{\DDc} C_1 + \frac 1 2\ov{\tr\Xb}\, \ov{H}\right)\c (\DDc \hot \mathfrak{B})\\
&=& - \frac 1 2\ov{\tr\Xb}\,  \ov{\DDc} \c (\DDc \hot \mathfrak{B})+2\ov{H} \c \nabc  \nabc_3  \mathfrak{B}+\left(- 2\ov{\DDc} C_1 +(\ov{\tr\Xb}-\tr\Xb )\, \ov{H}\right)\c \nabc  \mathfrak{B}\\
&&+ (H \c \ov{H}) \nabc_3 \Bfr -  \tr \Xb (\ov{H} \c H) \Bfr
\eeaa
We also have
\beaa
&&\ov{\DDc}\c([ \mathcal{P}_{C_1} , \DDc \hot] \mathfrak{B})\\
&=&\ov{\DDc}\c( - \frac 1 2 \tr \Xb \DDc \hot  \mathfrak{B} + H \hot \nabc_3  \mathfrak{B} +\left( - \DDc  C_1  -  \tr \Xb  H \right)\hot  \mathfrak{B})\\
&=& - \frac 1 2 \tr \Xb  \ov{\DDc}\c(\DDc \hot  \mathfrak{B}) + 2 H \c  \nabc \nabc_3  \mathfrak{B} + \left( - 2\DDc  C_1  - 2 \tr \Xb  H -  (\ov{\tr\Xb}-\tr \Xb)\ov{\Hb} \right) \c  \nabc\mathfrak{B}\\
&&+(\ov{\DDc} \c H)\nabc_3 \Bfr+\ov{\DDc} \c\left( - \DDc  C_1  -  \tr \Xb  H \right) \mathfrak{B}
\eeaa
where we used $\DDc\ov{\tr\Xb} -(\tr\Xb-\ov{\tr \Xb})\Hb= 0$.

Putting the above together we obtain
\beaa
J^{\Bfr}&=&- \frac 1 2(\tr\Xb+\ov{\tr\Xb})\,  \left( \frac 12 \ov{\DDc} \c (\DDc \hot \mathfrak{B})\right)+2\eta \c \nabc  \nabc_3  \mathfrak{B}+ \tilde{J}^{\Bfr}_3 \nabc_3 \Bfr  + \tilde{J}^{\Bfr}_a \c  \nabc\mathfrak{B}+\tilde{J}^{\Bfr}_0 \Bfr
\eeaa
where
\beaa
\tilde{J}^{\Bfr}_3&=& \frac 1 2 (\ov{\DDc} \c H)+ \frac 1 2 (H \c \ov{H})\\
\tilde{J}^{\Bfr}_a&=& - 2\nabc  C_1  -  \tr \Xb  H - \frac 1 2 (\ov{\tr\Xb}-\tr \Xb)\ov{\Hb} +\frac 1 2 (\ov{\tr\Xb}-\tr\Xb )\, \ov{H}\\
\tilde{J}^{\Bfr}_0&=& \frac 1 2 \ov{\DDc} \c\left( - \DDc  C_1  -  \tr \Xb  H \right)-  \frac 1 2 \tr \Xb (\ov{H} \c H)
\eeaa

Using the Teukolsky equation for $\mathfrak{B}$ given by $\TT_1(\mathfrak{B})=\M_1[\mathfrak{F}, \mathfrak{X}]$ and the expression for the Teukolsky operator \eqref{operator-Teukolsky-B} we can write 
\beaa
 \frac 1 2 \ov{\DDc}\c \DDc \hot \mathfrak{B}&=&   \nabc_3\nabc_4\mathfrak{B}+3\ov{\tr X} \nabc_3\mathfrak{B}+ \left(\frac{3}{2}\tr\Xb +\frac 1 2\ov{\tr\Xb}\right)\nabc_4\mathfrak{B}\\
   &&-\left( 6H+ \ov{H}+ 3  \ov{ \Hb}  \right)\c  \nabc  \Bfr+\left(\frac{9}{2}\tr\Xb \ov{\tr X} +4 \PF \ov{\PF}-9 \ov{ \Hb} \c  H\right)\mathfrak{B}+\M_1[\mathfrak{F}, \mathfrak{X}]\\
   &=&   \nabc_4\nabc_3\mathfrak{B}+3\ov{\tr X} \nabc_3\mathfrak{B}+ \left(\frac{3}{2}\tr\Xb +\frac 1 2\ov{\tr\Xb}\right)\nabc_4\mathfrak{B}\\
   &&+\Big(2(\eta-\etab)  -\left( 6H+ \ov{H}+ 3  \ov{ \Hb}  \right)\Big) \c  \nabc  \Bfr\\
   &&+\left(\frac{9}{2}\tr\Xb \ov{\tr X} +2\ov{P}+6 \PF \ov{\PF}-10 \ov{ \Hb} \c  H\right)\mathfrak{B}+\M_1[\mathfrak{F}, \mathfrak{X}]
   \eeaa
where we used \eqref{commutator-3-4-F} to write $  \,[\nabc_3, \nabc_4] \Bfr   =2(\eta-\etab) \c \nabc \Bfr +\left(2\ov{P}+2 \PF \ov{\PF}-(H \c \ov{\Hb}  ) \right)  \Bfr$. 
We therefore obtain
\bea\label{J-B}
\begin{split}
J^{\Bfr}&=- \frac 1 2(\tr\Xb+\ov{\tr\Xb})\,   \nabc_4\nabc_3\mathfrak{B} +2\eta \c \nabc  \nabc_3  \mathfrak{B}\\
&+J^{\Bfr}_4\nabc_4\mathfrak{B}+J^{\Bfr}_3 \nabc_3 \Bfr  + J^{\Bfr}_a \c  \nabc\mathfrak{B}+J^{\Bfr}_0 \Bfr- \frac 1 2(\tr\Xb+\ov{\tr\Xb})\M_1[\mathfrak{F}, \mathfrak{X}]
\end{split}
\eea
where
\bea
J^{\Bfr}_3&=& - \frac 3 2 \ov{\tr X}\left(\tr \Xb +\ov{\tr\Xb}\right)+ \tilde{J}^{\Bfr}_3 \nonumber\\
J^{\Bfr}_4&=& - \frac 1 2 \left(\tr \Xb +\ov{\tr\Xb}\right)\left(\frac{3}{2}\tr\Xb +\frac 1 2\ov{\tr\Xb}\right) \nonumber\\
J^{\Bfr}_a&=&- \frac 1 2(\tr\Xb+\ov{\tr\Xb})\Big(2(\eta-\etab)  -\left( 6H+ \ov{H}+ 3  \ov{ \Hb}  \right)\Big) + \tilde{J}^{\Bfr}_a \label{eq:J-Bfr-a}\\
J^{\Bfr}_0&=& - \frac 1 2 \left(\tr \Xb +\ov{\tr\Xb}\right)\left(\frac{9}{2}\tr\Xb \ov{\tr X} +2\ov{P}+6 \PF \ov{\PF}-10 \ov{ \Hb} \c  H\right)+\tilde{J}^{\Bfr}_0 \nonumber
\eea

We have
\beaa
J^{\Ffr}&=& \frac 1 2 [\mathcal{P}_{C_2}, \DDc \hot ](\ov{\DDc}\c \mathfrak{F})+ \frac 1 2\DDc \hot( [\mathcal{P}_{C_2},  \ov{\DDc}\c ]\mathfrak{F})
\eeaa
Using Lemma \ref{commutator-RW} applied to $F=\ov{\DDc}\c \mathfrak{F}$ of conformal type $s=1$, we have 
\beaa
&&  \, [\PP_{C_2}, \DDc \hot ](\ov{\DDc}\c \mathfrak{F})\\
&=& - \frac 1 2 \tr \Xb \DDc \hot (\ov{\DDc}\c \mathfrak{F}) + H \hot \nabc_3 (\ov{\DDc}\c \mathfrak{F}) +\left( - \DDc  C_2  -  \tr \Xb  H \right)\hot (\ov{\DDc}\c \mathfrak{F})\\
&=& - \frac 1 2 \tr \Xb \DDc \hot (\ov{\DDc}\c \mathfrak{F}) + H \hot( \ov{\DDc}\c  \nabc_3\mathfrak{F}- \frac 1 2\ov{\tr\Xb}\, ( \ov{\DDc} \c \Ffr - \ov{H} \c \Ffr)+\ov{H} \c \nab_3 \Ffr)\\
&& +\left( - \DDc  C_2  -  \tr \Xb  H \right)\hot (\ov{\DDc}\c \mathfrak{F})\\
&=& - \frac 1 2 \tr \Xb \DDc \hot (\ov{\DDc}\c \mathfrak{F}) +2 H \c \nabc  \nabc_3\mathfrak{F}+( H \c \ov{H} ) \nab_3 \Ffr\\
&& +\left( -2 \DDc  C_2  - (2 \tr \Xb+\ov{\tr\Xb} ) H \right)\c \nabc \mathfrak{F}+ \frac 1 2\ov{\tr\Xb}\,( H \c \ov{H} )  \Ffr
\eeaa
 We also have
\beaa
&&\DDc \hot( [\mathcal{P}_{C_2},  \ov{\DDc}\c ]\mathfrak{F})\\
&=& \DDc \hot(  - \frac 1 2\ov{\tr\Xb}\,  \ov{\DDc} \c \mathfrak{F}+\ov{H} \c \nabc_3 \mathfrak{F}+\left(- \ov{\DDc} C_2 + \frac 1 2\ov{\tr\Xb}\,  \ov{H}\right)\c \mathfrak{F})\\
&=& - \frac 1 2\ov{\tr\Xb} \DDc \hot(   \ov{\DDc} \c \mathfrak{F})+ 2 \ov{H} \c  \nabc \nabc_3 \mathfrak{F}+  \left(- 2\ov{\DDc} C_2 + \ov{\tr\Xb}\,  \ov{H}-\left( \tr\Xb - \ov{\tr\Xb}\right) \Hb\right)\c  \nabc \mathfrak{F} \\
&&+(\DDc \c \ov{H})  \nabc_3 \mathfrak{F}+ \DDc \c \left(- \ov{\DDc} C_2 + \frac 1 2\ov{\tr\Xb}\,  \ov{H}\right) \mathfrak{F}
\eeaa
where we used $\DDc\ov{\tr\Xb} -(\tr\Xb-\ov{\tr \Xb})\Hb= 0$.

Putting the above together we obtain
\beaa
J^{\Ffr}&=&  - \frac 1 2 \left(\tr \Xb +\ov{\tr\Xb}\right)\left( \frac 1 2 \DDc \hot (\ov{\DDc}\c \mathfrak{F}) \right) + 2\eta \c \nabc  \nabc_3\mathfrak{F}+\tilde{J}^{\Ffr}_3 \nab_3 \Ffr +\tilde{J}^{\Ffr}_a \c \nabc \mathfrak{F}+\tilde{J}^{\Ffr}_0 \Ffr
\eeaa
where
\beaa
\tilde{J}^{\Ffr}_3&=& \frac 1 2(\DDc \c \ov{H})+ \frac 1 2 (H \c \ov{H})\\
\tilde{J}^{\Ffr}_a&=&- 2\nabc  C_2+ \frac 1 2 \ov{\tr\Xb}\,  \ov{H}-\frac 1 2\left( \tr\Xb - \ov{\tr\Xb}\right) \Hb  -\frac 1 2  (2 \tr \Xb+\ov{\tr\Xb} ) H\\
\tilde{J}^{\Ffr}_0&=&\frac 1 2  \DDc \c \left(- \ov{\DDc} C_2 + \frac 1 2\ov{\tr\Xb}\,  \ov{H}\right) + \frac 1 4\ov{\tr\Xb}\,( H \c \ov{H} )
\eeaa
Using the Teukolsky equation for $\mathfrak{F}$ given by $\TT_2(\mathfrak{F})=\M_2[A, \mathfrak{X}, \mathfrak{B}]$ and the expression for the Teukolsky operator \eqref{operator-Teukolsky-F} we can write
\beaa
 \frac 1 2 \DDc \hot (\ov{\DDc} \c \mathfrak{F})&=& \nabc_3\nabc_4 \mathfrak{F}+\left(\frac 3 2 \ov{\tr X} +\frac 1 2  \tr X\right)\nabc_3\mathfrak{F}+ \frac 1 2 \left(\tr \Xb+\ov{\tr \Xb}  \right) \nabc_4 \mathfrak{F}\\
&&+\left( \frac 3 4 \tr \Xb  \ov{\tr X}+\frac 1 4\ov{\tr \Xb}   \tr X-3\ov{P} +P -4\PF\ov{\PF} +\frac 3 2\ov{\DDc}\c H\right)\mathfrak{F}-\frac 1 2 \Hb \hot ( \ov{H} \c \mathfrak{F})\\
&&- \left(4   H+ \ov{H}+  \Hb \right)\c \nabc \Ffr+\M_2[A, \mathfrak{X}, \mathfrak{B}]\\
&=& \nabc_4 \nabc_3\mathfrak{F}+\left(\frac 3 2 \ov{\tr X} +\frac 1 2  \tr X\right)\nabc_3\mathfrak{F}+ \frac 1 2 \left(\tr \Xb+\ov{\tr \Xb}  \right) \nabc_4 \mathfrak{F}\\
&&+\left( 2(\eta-\etab ) - \left(4   H+ \ov{H}+  \Hb \right)\right) \c  \nabc  \Ffr +\M_2[A, \mathfrak{X}, \mathfrak{B}]\\
&&+\left( \frac 3 4 \tr \Xb  \ov{\tr X}+\frac 1 4\ov{\tr \Xb}   \tr X -2\PF\ov{\PF} +\frac 3 2\ov{\DDc}\c H-3\eta\c\etab  +3i \eta \wedge \etab  \right)\mathfrak{F}
\eeaa
where we used \eqref{commutator-3-4-F} to write $ [\nabc_3, \nabc_4] \Ffr  =  2(\eta-\etab ) \c  \nabc  \Ffr  +\Big(-P+3\ov{P} +2 \PF \ov{\PF}-2\eta\c\etab  +4i \eta \wedge \etab  \Big)\Ffr$, and $\frac 1 2 \Hb   \hot (\ov{H} \c \Ffr)=\left( \eta \c \etab +i \eta \wedge \etab \right)\Ffr$.

We therefore obtain
\bea\label{J-F}
\begin{split}
J^{\Ffr}&=- \frac 1 2(\tr\Xb+\ov{\tr\Xb})\,   \nabc_4\nabc_3\mathfrak{F} +2\eta \c \nabc  \nabc_3  \mathfrak{F}\\
&+J^{\Ffr}_4\nabc_4\mathfrak{F}+J^{\Ffr}_3 \nabc_3 \Ffr  + J^{\Ffr}_a \c  \nabc\mathfrak{F}+J^{\Ffr}_0 \Ffr- \frac 1 2(\tr\Xb+\ov{\tr\Xb})\M_2[A, \mathfrak{X}, \mathfrak{B}]
\end{split}
\eea
where
\bea
J^{\Ffr}_3&=& - \frac 1 2   \left(\tr \Xb +\ov{\tr\Xb}\right)  \left(\frac 3 2 \ov{\tr X} +\frac 1 2  \tr X\right)+ \tilde{J}^{\Ffr}_3 \nonumber \\
J^{\Ffr}_4&=&  - \frac 1 4 \left(\tr \Xb +\ov{\tr\Xb}\right)   \left(\tr \Xb+\ov{\tr \Xb}  \right)\nonumber \\
J^{\Ffr}_a&=&- \frac 1 2(\tr\Xb+\ov{\tr\Xb})\Big(2(\eta-\etab)  - \left(4   H+ \ov{H}+  \Hb \right)\Big) + \tilde{J}^{\Ffr}_a \label{eq:J-Ffr-a}\\
J^{\Ffr}_0&=& - \frac 1 2 \left(\tr \Xb +\ov{\tr\Xb}\right)\left( \frac 3 4 \tr \Xb  \ov{\tr X}+\frac 1 4\ov{\tr \Xb}   \tr X -2\PF\ov{\PF} +\frac 3 2\ov{\DDc}\c H-3\eta\c\etab  +3i \eta \wedge \etab  \right)+\tilde{J}^{\Ffr}_0\nonumber
\eea

\subsubsection{Expressions for $K^{\Bfr}$ and $K^{\Ffr}$}

Observe that
\bea\label{formula-PC_1}
\mathcal{P}_{C}( g F)&=& (\nabc_3 + C )(gF)=(\nabc_3 g) F + g \nabc_3 F + C g F=  g \mathcal{P}_{C}(F)+(\nabc_3 g) F
\eea
We have, for $g=-3\ov{\tr X}$
\beaa
K^{\Bfr}&=&[\mathcal{P}_{C_1}, g \nabc_3] \Bfr= g [\mathcal{P}_{C_1}, \nabc_3] \Bfr+(\nabc_3 g) \nabc_3 \Bfr =(\nabc_3 g) \nabc_3 \Bfr- g  \nabc_3C_1 \Bfr
\eeaa
We therefore obtain
\bea\label{K-B}
K^{\Bfr}&=&K^{\Bfr}_3 \nabc_3 \mathfrak{B}+K^{\Bfr}_0 \mathfrak{B}
\eea
where
\beaa
K^{\Bfr}_3&=& -3 \nabc_3( \ov{\tr X}), \qquad K^{\Bfr}_0= 3\ov{\tr X}\nabc_3C_1
\eeaa
Similarly,
\bea\label{K-F}
K^{\Ffr}&=&K^{\Ffr}_3 \nabc_3 \mathfrak{F}+K^{\Ffr}_0  \mathfrak{F}
\eea
where
\beaa
K^{\Ffr}_3&=& -\nabc_3\left(\frac 3 2 \ov{\tr X} +\frac 1 2  \tr X\right), \qquad K^{\Ffr}_0= \left(\frac 3 2 \ov{\tr X} +\frac 1 2  \tr X\right)  \nabc_3C_2
\eeaa

\subsubsection{Expressions for $L^{\Bfr}$ and $L^{\Ffr}$}
Using \eqref{formula-PC_1}, we obtain
\beaa
L&=&[\mathcal{P}_{C}, g \nabc_4] \Psi=\mathcal{P}_{C}( g \nabc_4 \Psi) -g \nabc_4 (\mathcal{P}_{C} \Psi)= g [\mathcal{P}_{C}, \nabc_4] \Psi+(\nabc_3 g) \nabc_4 \Psi 
\eeaa
Using  \eqref{PPC_1nabc4} with $g=- \left(\frac{3}{2}\tr\Xb +\frac 1 2\ov{\tr\Xb}\right)$, we obtain
\bea\label{L-B}
L^{\Bfr}&=&L^{\Bfr}_4 \nabc_4 \Bfr + L^{\Bfr}_a \c \nabc \Bfr+ L^{\Bfr}_0 \Bfr
\eea
where
\bea
L^{\Bfr}_4&=& -\nabc_3 \left(\frac{3}{2}\tr\Xb +\frac 1 2\ov{\tr\Xb}\right) \nonumber\\
L^{\Bfr}_a&=&- \left( 3\tr\Xb +\ov{\tr\Xb}\right) (\eta-\etab) \label{eq:L-Bfr-a}\\
L^{\Bfr}_0&=& - \left(\frac{3}{2}\tr\Xb +\frac 1 2\ov{\tr\Xb}\right)\left(2\left(\rho+\rhoF^2+\dual\rhoF^2-\eta\c\etab\right)+2 i \left(-\rhod+\eta \wedge \etab\right)-\nabc_4 C_1\right) \nonumber
\eea

Similarly, using \eqref{PPC_2nabc4} with $g=- \frac 1 2 \left(\tr \Xb+\ov{\tr \Xb}  \right)$
we obtain 
\bea\label{L-F}
L^{\Ffr}&=&L^{\Ffr}_4\nabc_4 \mathfrak{F}+ L^{\Ffr}_a \c \nabc \Ffr+ L^{\Ffr}_0 \Ffr
\eea
where
\bea
L^\Ffr_4&=&- \frac 1 2 \nabc_3\left(\tr \Xb+\ov{\tr \Xb}  \right) \nonumber\\
L^{\Ffr}_a&=&-  \left(\tr \Xb+\ov{\tr \Xb}  \right) (\eta-\etab)  \label{eq:L-Ffr-a}\\
L^{\Ffr}_0&=& - \frac 1 2 \left(\tr \Xb+\ov{\tr \Xb}  \right)\left(2\left(\rho+\rhoF^2+\dual\rhoF^2-\eta\c\etab\right)+4 i \left(-\rhod+\eta \wedge \etab\right)-\nabc_4 C_2\right) \nonumber
\eea

\subsubsection{Expressions for $M^{\Bfr}$ and $M^{\Ffr}$}
Observe that
\beaa
\mathcal{P}_{C_1}( F \c U )&=& (\nabc_3 + C_1 )(F \c U )=\nabc_3 F \c U + F \c \nabc_3 U  + C_1  F \c U =  F \c \mathcal{P}_{C_1}(U)+(\nabc_3 F) \c  U
\eeaa
We have, for $F=\left( 6H+ \ov{H}+ 3  \ov{ \Hb}  \right)$, and using Lemma \ref{commutator-RW}
\beaa
M^{\Bfr}&=&[\PP_{C_1},F\c  \nabc ]\Bfr=F \c [\mathcal{P}_{C_1}, \nabc]\Bfr+(\nabc_3 F) \c  \nabc\Bfr\\
&=& -\frac  1 2   \trchb\, F \c  \nabc \Bfr-\frac 1 2 \atrchb\,F \c   \dual \nabc  \Bfr+F \c  \eta \nabc_3 \Bfr+F \c  ( \mathcal{V}^s_{[3,a]} - \nabc C_1 ) \Bfr\\
&&+(\nabc_3 F) \c  \nabc\Bfr
\eeaa
We therefore obtain
\bea\label{M-B}
M^{\Bfr}&=& M^{\Bfr}_3 \nabc_3 \Bfr+M_a^{\Bfr} \c  \nabc\Bfr+M^{\Bfr}_0 \Bfr
\eea
where
\bea
M^{\Bfr}_3&=& \eta \c \left( 6H+ \ov{H}+ 3  \ov{ \Hb}  \right) \nonumber\\
M_a^{\Bfr}&=& \nabc_3 \left( 6H+ \ov{H}+ 3  \ov{ \Hb}  \right)-\frac  1 2   \trchb\, \left( 6H+ \ov{H}+ 3  \ov{ \Hb}  \right)+\frac 1 2 \atrchb\, \dual \left( 6H+ \ov{H}+ 3  \ov{ \Hb}  \right)  \label{eq:M-a-Bfr}\\
M^{\Bfr}_0&=& \left( 6H+ \ov{H}+ 3  \ov{ \Hb}  \right) \c  ( \mathcal{V}^s_{[3,a]} - \nabc C_1 )\nonumber
\eea

Similarly, we obtain
\bea\label{M-F}
M^{\Ffr}&=& M^{\Bfr}_3 \nabc_3 \Ffr+M_a^{\Ffr} \c  \nabc\Ffr+M^{\Ffr}_0 \Ffr
\eea
where
\bea
M^{\Ffr}_3&=& \eta \c \left(4   H+ \ov{H}+  \Hb \right)\nonumber \\
M_a^{\Ffr}&=& \nabc_3 \left(4   H+ \ov{H}+  \Hb \right)-\frac  1 2   \trchb\, \left(4   H+ \ov{H}+  \Hb \right)+\frac 1 2 \atrchb\, \dual \left(4   H+ \ov{H}+  \Hb \right)  \label{eq:M-a-Ffr} \\
M^{\Ffr}_0&=& \left(4   H+ \ov{H}+  \Hb \right) \c  ( \mathcal{V}^s_{[3,a]} - \nabc C_2 )\nonumber
\eea

\subsubsection{Expressions for $N^{\Bfr}$ and $N^{\Ffr}$}
Using \eqref{formula-PC_1} we have
\beaa
N&=& [\mathcal{P}_{C}, g] \Psi=( \nabc_3g )\Psi 
\eeaa
We therefore obtain
\bea\label{N-B}
N^{\Bfr}&=& N^{\Bfr}_0 \mathfrak{B}
\eea
where
\beaa
N^\Bfr_0&=& \nabc_3 \left(-\frac{9}{2}\tr\Xb \ov{\tr X} -4 \PF \ov{\PF}+9 \ov{ \Hb} \c  H\right) 
\eeaa
and
\bea\label{N-F}
N^{\Ffr}&=& N^{\Ffr}_0 \mathfrak{F}
\eea
where
\beaa
N^\Ffr_0&=& \nabc_3 \left(- \frac 3 4 \tr \Xb  \ov{\tr X}- \frac 1 4\ov{\tr \Xb}   \tr X+3\ov{P} -P +4\PF\ov{\PF} - \frac 3 2\ov{\DDc}\c H+ \eta \c \etab +i \eta \wedge \etab  \right) 
\eeaa

\subsubsection{The sum}

From \eqref{comm-sum-1} and \eqref{comm-sum-2}, we obtain
\beaa
\, [\mathcal{P}_{C_1}, \TT_1](\mathfrak{B})&=& \eqref{I-B}+\eqref{J-B}+\eqref{K-B}+\eqref{L-B}+\eqref{M-B}+\eqref{N-B}\\
&=&2\etab \c \nabc\nabc_3 \Bfr- \frac 1 2(\tr\Xb+\ov{\tr\Xb})\,   \nabc_4\nabc_3\mathfrak{B}\\
&&+ \left(  I^{\Bfr}_4+J^{\Bfr}_4+L^{\Bfr}_4  \right) \nabc_4\Bfr +\left( I^{\Bfr}_{3} +J^{\Bfr}_3 +K^{\Bfr}_3+M^{\Bfr}_3 \right) \nabc_3 \Bfr \\
&&+\left( I^{\Bfr}_a + J^{\Bfr}_a + L^{\Bfr}_a+M_a^{\Bfr}\right) \c \nabc \Bfr+\left( I^{\Bfr}_0+J^{\Bfr}_0+K^{\Bfr}_0+ L^{\Bfr}_0 +M^{\Bfr}_0+ N^{\Bfr}_0\right) \Bfr \\
&&- \frac 1 2(\tr\Xb+\ov{\tr\Xb})\M_1[\mathfrak{F}, \mathfrak{X}]
\eeaa
and
\beaa
 [\mathcal{P}_{C_2}, \TT_2](\mathfrak{F})&=& \eqref{I-F}+\eqref{J-F}+\eqref{K-F}+\eqref{L-F}+\eqref{M-F}+\eqref{N-F}\\
 &=&2\etab \c \nabc\nabc_3 \Ffr- \frac 1 2(\tr\Xb+\ov{\tr\Xb})\,   \nabc_4\nabc_3\mathfrak{F}\\
&&+ \left(  I^{\Ffr}_4+J^{\Ffr}_4+L^{\Ffr}_4  \right) \nabc_4\Ffr +\left( I^{\Ffr}_{3} +J^{\Ffr}_3 +K^{\Ffr}_3+M^{\Ffr}_3 \right) \nabc_3 \Ffr \\
&&+\left( I^{\Ffr}_a + J^{\Ffr}_a + L^{\Ffr}_a+M_a^{\Ffr}\right) \c \nabc \Ffr+\left( I^{\Ffr}_0+J^{\Ffr}_0+K^{\Ffr}_0+ L^{\Ffr}_0 +M^{\Ffr}_0+ N^{\Ffr}_0\right) \Ffr \\
&&- \frac 1 2(\tr\Xb+\ov{\tr\Xb})\M_2[A, \mathfrak{X},  \mathfrak{B}]
\eeaa
Using \eqref{definition-P-Q} and \eqref{definition-P-Q2} to write 
\beaa
\nabc_3 \mathfrak{B}&=&\mathfrak{P}- C_1 \ \mathfrak{B}, \qquad \nabc_3 \mathfrak{F}=\mathfrak{Q}- C_2 \ \mathfrak{F}
\eeaa
and therefore
\beaa
\nabc_4\nabc_3\mathfrak{B}&=& \nabc_4\left(\mathfrak{P}- C_1 \ \mathfrak{B} \right)= \nabc_4\mathfrak{P} -  C_1 \nabc_4 \mathfrak{B} -(\nabc_4 C_1) \mathfrak{B}\\
\nabc \nabc_3\mathfrak{B}&=& \nabc \left(\mathfrak{P}- C_1 \ \mathfrak{B} \right)= \nabc \mathfrak{P} -  C_1 \nabc \mathfrak{B} -(\nabc C_1) \mathfrak{B}\\
 \nabc_4\nabc_3\mathfrak{F}&=& \nabc_4\left(\mathfrak{Q}- C_2 \ \mathfrak{F} \right)= \nabc_4\mathfrak{Q} -  C_2 \nabc_4 \mathfrak{F} -(\nabc_4 C_2) \mathfrak{F}\\
 \nabc \nabc_3\mathfrak{F}&=& \nabc \left(\mathfrak{Q}- C_2 \ \mathfrak{F} \right)= \nabc \mathfrak{Q} -  C_2 \nabc \mathfrak{F} -(\nabc C_2) \mathfrak{F}
\eeaa
Hence we obtain
\beaa
\, [\mathcal{P}_{C_1}, \TT_1](\mathfrak{B})&=&2\etab \c \nabc \mathfrak{P}  - \frac 1 2(\tr\Xb+\ov{\tr\Xb})\,   \nabc_4\mathfrak{P}+ \left(  I^{\Bfr}_4+J^{\Bfr}_4+L^{\Bfr}_4 +\frac 1 2(\tr\Xb+\ov{\tr\Xb})C_1 \right) \nabc_4\Bfr\\
&& +\hat{V}_1 \mathfrak{P} +Z^{\Bfr}_{a} \c \nabc \Bfr+\tilde{Z}^{\Bfr}_0 \Bfr - \frac 1 2(\tr\Xb+\ov{\tr\Xb})\M_1[\mathfrak{F}, \mathfrak{X}]\\
\, [\mathcal{P}_{C_1}, \TT_1](\mathfrak{B})&=&2\etab \c \nabc \mathfrak{Q}  - \frac 1 2(\tr\Xb+\ov{\tr\Xb})\,   \nabc_4\mathfrak{Q}+ \left(  I^{\Ffr}_4+J^{\Ffr}_4+L^{\Ffr}_4 +\frac 1 2(\tr\Xb+\ov{\tr\Xb})C_2 \right) \nabc_4\Ffr\\
&& +\hat{V}_2 \mathfrak{P} +Z^{\Ffr}_{a} \c \nabc \Ffr+\tilde{Z}^{\Ffr}_0 \Ffr - \frac 1 2(\tr\Xb+\ov{\tr\Xb})\M_2[A, \mathfrak{X},  \mathfrak{B}]
\eeaa
where
\bea
\hat{V}_1&=& I^{\Bfr}_{3} +J^{\Bfr}_3 +K^{\Bfr}_3+M^{\Bfr}_3 \\
Z^{\Bfr}_{a}&=& I^{\Bfr}_a + J^{\Bfr}_a + L^{\Bfr}_a+M_a^{\Bfr}-2\etab \c C_1 \label{eq:Z-Bfr-a} \\
\tilde{Z}^{\Bfr}_0&=& I^{\Bfr}_0+J^{\Bfr}_0+K^{\Bfr}_0+ L^{\Bfr}_0 +M^{\Bfr}_0+ N^{\Bfr}_0+\frac 1 2(\tr\Xb+\ov{\tr\Xb})\nabc_4 C_1 \nonumber\\
&&-C_1\left( I^{\Bfr}_{3} +J^{\Bfr}_3 +K^{\Bfr}_3+M^{\Bfr}_3 \right)-2\etab \c \nabc C_1
\eea
and 
\bea
\hat{V}_2&=& I^{\Ffr}_{3} +J^{\Ffr}_3 +K^{\Ffr}_3+M^{\Ffr}_3 \\
Z^{\Ffr}_{a}&=& I^{\Ffr}_a + J^{\Ffr}_a + L^{\Ffr}_a+M_a^{\Ffr}-2\etab \c C_2  \label{eq:Z-Ffr-a}\\
\tilde{Z}^{\Ffr}_0&=& I^{\Ffr}_0+J^{\Ffr}_0+K^{\Ffr}_0+ L^{\Ffr}_0 +M^{\Ffr}_0+ N^{\Ffr}_0+\frac 1 2(\tr\Xb+\ov{\tr\Xb})\nabc_4 C_2 \nonumber\\
&&-C_2\left( I^{\Ffr}_{3} +J^{\Ffr}_3 +K^{\Ffr}_3+M^{\Ffr}_3 \right)-2\etab \c \nabc C_2
\eea

Observe that the coefficients of $\nabc_4 \mathfrak{B}$ and $\nabc_4 \mathfrak{F}$
are given by
\beaa
&& I^\Bfr_4+J^\Bfr_4+L^\Bfr_4+\frac 1 2\left(\tr \Xb +\ov{\tr\Xb}\right)C_1 \\
&=&  \nabc_3C_1- \frac 1 2 \left(\tr \Xb +\ov{\tr\Xb}\right)\left(\frac{3}{2}\tr\Xb +\frac 1 2\ov{\tr\Xb}\right)-\nabc_3 \left(\frac{3}{2}\tr\Xb +\frac 1 2\ov{\tr\Xb}\right)+\frac 1 2\left(\tr \Xb +\ov{\tr\Xb}\right)C_1 \\
&=&  \nabc_3C_1+\frac {1}{ 2}\left(\tr \Xb +\ov{\tr\Xb}\right) C_1-  \tr\Xb\ov{\tr\Xb}
\eeaa
and
\beaa
&& I^\Ffr_4+J^\Ffr_4+L^\Ffr_4+\frac 1 2\left(\tr \Xb +\ov{\tr\Xb}\right)C_2 \\
&=& \nabc_3 C_2-   \frac 1 4 \left(\tr \Xb +\ov{\tr\Xb}\right)   \left(\tr \Xb +\ov{\tr\Xb}\right)  -\frac 1 2 \nabc_3 \left(\tr \Xb +\ov{\tr\Xb}\right)  +\frac 1 2\left(\tr \Xb +\ov{\tr\Xb}\right)C_2\\
&=& \nabc_3 C_2+\frac {1}{ 2}\left(\tr \Xb +\ov{\tr\Xb}\right)C_2-   \frac 1 2  \tr \Xb \ov{\tr\Xb}
\eeaa
which give conditions \eqref{transport-nabc-3-C1} and \eqref{transport-nabc-3-C2} for the vanishing of those coefficients.

\subsubsection{The lower order terms }

Defining 
\beaa
L_{\Pfr}[\Bfr, \Ffr]&:=&-Z^{\Bfr}_{a} \c \nabc \Bfr-\tilde{Z}^{\Bfr}_0 \Bfr, \\
L_{\Qfr}[\Bfr, \Ffr]&:=&-Z^{\Ffr}_{a} \c \nabc \Ffr-\tilde{Z}^{\Ffr}_0 \Ffr
\eeaa
to complete the proof of the Proposition, we need to compute the terms $\tilde{Z}^{\Bfr}_0$ and $\tilde{Z}^{\Ffr}_0$.

Observe that, according to \eqref{definition-C1-C2} and \eqref{definition-C1-C22}, we can write
\beaa
C_1&=&2\trchb + O\left(\frac{|a|}{r^2} \right) \qquad C_2=\trchb  + O\left(\frac{|a|}{r^2} \right)
 \eeaa
 This gives
 \beaa
 \nabc_3 C_1&=& -\trchb^2+ O\left(\frac{|a|}{r^3} \right), \qquad  \nabc_3 C_2= -\frac 1 2 \trchb^2+ O\left(\frac{|a|}{r^3} \right), \\
  \nabc_4 C_1&=& - \trch\trchb +4\rho+ O\left(\frac{|a|}{r^3} \right), \qquad  \nabc_4 C_2= -\frac1 2  \trch\trchb +2\rho+ O\left(\frac{|a|}{r^3} \right)
 \eeaa
We compute
\bea
I^{\Bfr}_{3}&=& -2\rho -2\rhoF^2 -2\dual\rhoF^2 -2 \eta \c (\eta-2\etab)+ i \left(2\rhod-2\eta \wedge \etab\right)+\nabc_4 C_1 \label{I-Bfr-3}\\
&=&- \trch\trchb +2\rho -2\rhoF^2+ O\left(\frac{|a|}{r^3} \right) \nonumber\\
J^\Bfr_3&=& - \frac 3 2 \ov{\tr X}\left(\tr \Xb +\ov{\tr\Xb}\right)+\frac 1 2 (\ov{\DDc} \c H)+ \frac 1 2 (H \c \ov{H})= - 3  \trch\trchb + O\left(\frac{|a|}{r^3} \right)\label{J-Bfr-3}\\
K^\Bfr_3&=& -3 \nabc_3( \ov{\tr X})=\frac 3 2 \trch\trchb-6\rho+ O\left(\frac{|a|}{r^3} \right)\label{K-Bfr-3}\\
M^{\Bfr}_3&=& \eta \c \left( 6H+ \ov{H}+ 3  \ov{ \Hb}  \right)=O\left(\frac{|a|}{r^4} \right) \label{M-Bfr-3}
\eea
which gives
\beaa
I^\Bfr_3+J^\Bfr_3+K^\Bfr_3+M^{\Bfr}_3&=&  - \frac 5 2 \trch \trchb -4\rho -2 \rhoF^2+ O\left(\frac{|a|}{r^3} \right)
\eeaa
Similarly, 
\bea
I^{\Ffr}_3&=& -2\rho -2\rhoF^2 -2\dual\rhoF^2 -2 \eta \c (\eta-2\etab)+ i \left(4\rhod-4\eta \wedge \etab\right)+\nabc_4 C_2 \label{I-Ffr-3}\\
&=&-\frac1 2  \trch\trchb -2\rhoF^2 + O\left(\frac{|a|}{r^3} \right)\nonumber\\
J^{\Ffr}_3&=&- \frac 1 2   \left(\tr \Xb +\ov{\tr\Xb}\right)  \left(\frac 3 2 \ov{\tr X} +\frac 1 2  \tr X\right)+\frac 1 2(\DDc \c \ov{H})+ \frac 1 2 (H \c \ov{H}) \label{J-Ffr-3}\\
&=&-2\trch\trchb  + O\left(\frac{|a|}{r^3} \right)\nonumber\\
K^{\Ffr}_3&=& -\nabc_3\left(\frac 3 2 \ov{\tr X} +\frac 1 2  \tr X\right)= \trch\trchb-4\rho + O\left(\frac{|a|}{r^3} \right) \label{K-Ffr-3}\\
M^{\Ffr}_3&=& \eta \c \left(4   H+ \ov{H}+  \Hb \right)=O\left(\frac{|a|}{r^4} \right) \label{M-Ffr-3}
\eea
which gives
\beaa
I^\Ffr_3+J^\Ffr_3+K^\Ffr_3+M^{\Ffr}_3&=&  - \frac 3 2 \trch \trchb -4\rho -2 \rhoF^2+ O\left(\frac{|a|}{r^3} \right)
\eeaa

We also compute
\beaa
I^{\Bfr}_0&=& \nabc_3\left[-2\left(\rho+\rhoF^2+\dual\rhoF^2-\eta\c\etab\right)+2 i \left(\rhod-\eta \wedge \etab\right)+\nabc_4 C_1\right]- 2 (\eta-\etab) \c  \mathcal{V}^{s=1}_{[3,a]}\\
&=& -2( -\frac{3}{2}\trchb \rho -\trchb \rhoF^2)+4\trchb \rhoF^2+\nabc_3 \left( - \trch\trchb +4\rho \right)+ O\left(\frac{|a|}{r^4} \right) \\
&=&\trch\trchb^2-5\trchb \rho +2\trchb \rhoF^2  + O\left(\frac{|a|}{r^4} \right) \\
J^{\Bfr}_0&=& - \frac 1 2 \left(\tr \Xb +\ov{\tr\Xb}\right)\left(\frac{9}{2}\tr\Xb \ov{\tr X} +2\ov{P}+6 \PF \ov{\PF}-10 \ov{ \Hb} \c  H\right)+ O\left(\frac{|a|}{r^4} \right) \\
&=&-\frac 9 2 \trch\trchb^2-2\trchb \rho -6\trchb \rhoF^2 +  O\left(\frac{|a|}{r^4} \right)\\
K^{\Bfr}_0&=&3\ov{\tr X}\nabc_3C_1=-3\trch\trchb^2+ O\left(\frac{|a|}{r^4} \right)\\
 L^{\Bfr}_0 &=&- \left(\frac{3}{2}\tr\Xb +\frac 1 2\ov{\tr\Xb}\right)\left(2\left(\rho+\rhoF^2+\dual\rhoF^2-\eta\c\etab\right)+2 i \left(-\rhod+\eta \wedge \etab\right)-\nabc_4 C_1\right)\\
  &=&- 2 \trch\trchb^2 +4\trchb \rho-4\trchb \rhoF^2 +  O\left(\frac{|a|}{r^4} \right) \\
 M^{\Bfr}_0&=&  \left( 6H+ \ov{H}+ 3  \ov{ \Hb}  \right) \c  ( \mathcal{V}^s_{[3,a]} - \nabc C_1 )= O\left(\frac{|a|}{r^4} \right) \\
  N^{\Bfr}_0&=&  \nabc_3 \left(-\frac{9}{2}\tr\Xb \ov{\tr X} -4 \PF \ov{\PF}+9 \ov{ \Hb} \c  H\right)\\
  &=&\frac 9 2  \trch\trchb^2 -9\trchb \rho+8\trchb \rhoF^2 + O\left(\frac{|a|}{r^4} \right)
\eeaa
which gives
\beaa
I^{\Bfr}_0+J^{\Bfr}_0+K^{\Bfr}_0+ L^{\Bfr}_0 +M^{\Bfr}_0+ N^{\Bfr}_0&=& -4  \trch\trchb^2 -12\trchb \rho+ O\left(\frac{|a|}{r^4} \right)
\eeaa
This finally implies
\beaa
\tilde{Z}^{\Bfr}_0&=& I^{\Bfr}_0+J^{\Bfr}_0+K^{\Bfr}_0+ L^{\Bfr}_0 +M^{\Bfr}_0+ N^{\Bfr}_0+\frac 1 2(\tr\Xb+\ov{\tr\Xb})\nabc_4 C_1\\
&&-C_1\left( I^{\Bfr}_{3} +J^{\Bfr}_3 +K^{\Bfr}_3+M^{\Bfr}_3 \right)-2\etab \c \nabc C_1
\\
&=& -4  \trch\trchb^2 -12\trchb \rho+\trchb ( - \trch\trchb +4\rho)-2\trchb \left(  - \frac 5 2 \trch \trchb -4\rho -2 \rhoF^2 \right)+ O\left(\frac{|a|}{r^4} \right)\\
&=& 4\trchb \rhoF^2+ O\left(\frac{|a|}{r^4} \right)
\eeaa

We compute
\beaa
 I^{\Ffr}_0&=& \nabc_3\left[-2\left(\rho+\rhoF^2+\dual\rhoF^2-\eta\c\etab\right)+4 i \left(\rhod-\eta \wedge \etab\right)+\nabc_4 C_2\right]- 2 (\eta-\etab) \c  \mathcal{V}^{s=1}_{[3,a]} \\
&=& -2( -\frac{3}{2}\trchb \rho -\trchb \rhoF^2)+4\trchb \rhoF^2+\nabc_3 \left( -\frac 1 2  \trch\trchb +2\rho \right)+ O\left(\frac{|a|}{r^4} \right) \\
&=&\frac 1 2 \trch\trchb^2-\trchb \rho +4\trchb \rhoF^2  + O\left(\frac{|a|}{r^4} \right) \\
 J^{\Ffr}_0&=& - \frac 1 2 \left(\tr \Xb +\ov{\tr\Xb}\right)\left( \frac 3 4 \tr \Xb  \ov{\tr X}+\frac 1 4\ov{\tr \Xb}   \tr X -2\PF\ov{\PF} +\frac 3 2\ov{\DDc}\c H-3\eta\c\etab  +3i \eta \wedge \etab  \right) + O\left(\frac{|a|}{r^4} \right) \\
 &=& - \trch\trchb^2 +2\trchb \rhoF^2  + O\left(\frac{|a|}{r^4} \right)\\
 K^{\Ffr}_0&=& \left(\frac 3 2 \ov{\tr X} +\frac 1 2  \tr X\right)  \nabc_3C_2= - \trch\trchb^2 + O\left(\frac{|a|}{r^4} \right)\\
  L^{\Ffr}_0 &=& - \frac 1 2 \left(\tr \Xb+\ov{\tr \Xb}  \right)\left(2\left(\rho+\rhoF^2+\dual\rhoF^2-\eta\c\etab\right)+4 i \left(-\rhod+\eta \wedge \etab\right)-\nabc_4 C_2\right) \\
  &=& - \frac1 2  \trch\trchb^2 -2\trchb \rhoF^2 + O\left(\frac{|a|}{r^4} \right) \\
  M^{\Ffr}_0&=&\left(4   H+ \ov{H}+  \Hb \right) \c  ( \mathcal{V}^s_{[3,a]} - \nabc C_2 )=O\left(\frac{|a|}{r^4} \right)\\
   N^{\Ffr}_0&=& \nabc_3 \left(- \frac 3 4 \tr \Xb  \ov{\tr X}- \frac 1 4\ov{\tr \Xb}   \tr X+3\ov{P} -P +4\PF\ov{\PF} - \frac 3 2\ov{\DDc}\c H+ \eta \c \etab +i \eta \wedge \etab  \right) \\
   &=&\trch\trchb-5\trchb\rho -10\trchb\rhoF^2  +O\left(\frac{|a|}{r^4} \right)
\eeaa
which gives
\beaa
 I^{\Ffr}_0+J^{\Ffr}_0+K^{\Ffr}_0+ L^{\Ffr}_0 +M^{\Ffr}_0+ N^{\Ffr}_0&=&- \trch\trchb-6\trchb\rho -6\trchb\rhoF^2  +O\left(\frac{|a|}{r^4} \right)
\eeaa
This finally implies
\beaa
\tilde{Z}^{\Ffr}_0&=& I^{\Ffr}_0+J^{\Ffr}_0+K^{\Ffr}_0+ L^{\Ffr}_0 +M^{\Ffr}_0+ N^{\Ffr}_0+\frac 1 2(\tr\Xb+\ov{\tr\Xb})\nabc_4 C_2\\
&&-C_2\left( I^{\Ffr}_{3} +J^{\Ffr}_3 +K^{\Ffr}_3+M^{\Ffr}_3 \right)-2\etab \c \nabc C_2\\
&=&- \trch\trchb-6\trchb\rho -6\trchb\rhoF^2  +\trchb (-\frac1 2  \trch\trchb +2\rho)-\trchb \left(  - \frac 3 2 \trch \trchb -4\rho -2 \rhoF^2 \right)+O\left(\frac{|a|}{r^4} \right)\\
&=& - 4 \trchb \rhoF^2+O\left(\frac{|a|}{r^4} \right)
\eeaa
which concludes the proof of Proposition \ref{proposition-commutator}.

\subsection{Proof of Proposition \ref{prop:rescaling-f}}\label{section-proof-rescaling}

Let $f$ be given by 
 \beaa
 f&=& (q)^n (\ov{q})^{m}.
 \eeaa
Recall, see Proposition 8.9 in \cite{GKS},
 \beaa
 \nab_3(f)&=& \left(\frac n 2 \ov{\tr \Xb} +\frac m 2   \tr \Xb \right) f, \\ 
  \nab_4(f) &=& \left(\frac n 2 \tr X +\frac m 2   \ov{\tr X} \right) f,\\
  2  \nab f    &=&  \left(m H+n\ov{H}+n \Hb+ m  \ov{\Hb}  \right)f.
    \eeaa
and for $\Psi \in \sk_k(\CCC)$:
 \bea\label{square-f-Q-1}
\squared_k ( f \Psi )&=& \square(f) \Psi+f \squared_k\Psi- \nab_3 f \nab_4\Psi- \nab_4f \nab_3 \Psi +2\nab f \c \nab \Psi.
\eea
We then obtain for $\pf=f_1 \Pfr$, using  \eqref{intermediate-square-mathfrak-P}:
\beaa
\squared_1\pf&=& f_1 \Big[ \frac 5 2 \ov{\tr X} \nab_3\Pfr+ \left(2\tr\Xb +\frac 1 2 \ov{\tr\Xb}\right)\nab_4\Pfr-\left( 5H+\Hb+ 4  \ov{ \Hb}  \right)\c  \nab  \Pfr+\tilde{V}_1\Pfr\\
&&\mathcal{P}_{C_1}\Big(\M_1[\mathfrak{F}, \mathfrak{X}]\Big) + \frac 1 2 \left(\tr \Xb +\ov{\tr\Xb}\right)\M_1[\mathfrak{F}, \mathfrak{X}]+L_{\Pfr}[\Bfr, \Ffr] \Big]+ \square(f_1) \Pfr\\
&&-  \left(\frac n 2 \ov{\tr \Xb} +\frac m 2   \tr \Xb \right) f_1 \nab_4\Pfr- \left(\frac n 2 \tr X +\frac m 2   \ov{\tr X} \right) f_1 \nab_3 \Pfr + \left(m H+n\ov{H}+n \Hb+ m  \ov{\Hb}  \right)f_1 \c \nab \Pfr
\eeaa
which gives
\beaa
\squared_1\pf&=&\left((\frac 1 2 -\frac n 2) \ov{\tr \Xb}+(2 -\frac m 2  ) \tr \Xb \right) f_1 \nab_4\Pfr+ \left(-\frac n 2 \tr X +(\frac 5 2-\frac m 2  ) \ov{\tr X} \right) f_1 \nab_3 \Pfr \\
&&+ \left((m-5) H+n\ov{H}+(n-1) \Hb+ (m-4)  \ov{\Hb}  \right)f_1 \c \nab \Pfr+\left( \tilde{V}_1+ f_1^{-1}\square(f_1)\right) \pf \\
&&+ f_1 \Big[\mathcal{P}_{C_1}\Big(\M_1[\mathfrak{F}, \mathfrak{X}]\Big) + \frac 1 2 \left(\tr \Xb +\ov{\tr\Xb}\right)\M_1[\mathfrak{F}, \mathfrak{X}]+L_{\Pfr}[\Bfr, \Ffr] \Big]
\eeaa
 Observe that the real part of the coefficients of all the first derivatives are multiple of $m+n-5$. To cancel their real part we then take $m=5-n$, which implies $f_1= (q)^n (\ov{q})^{5-n}$, and gives
 \beaa
\squared_1\pf&=&i f_1 \Big[ (1 - n ) \atrchb   \nab_4\Pfr+   n \atrch   \nab_3 \Pfr + \left(-2n \dual \eta +2(n-1) \dual \etab  \right) \c \nab \Pfr \Big]\\
&&+\left( \tilde{V}_1+ f_1^{-1}\square(f_1)\right) \pf + f_1 \Big[\mathcal{P}_{C_1}\Big(\M_1[\mathfrak{F}, \mathfrak{X}]\Big) + \frac 1 2 \left(\tr \Xb +\ov{\tr\Xb}\right)\M_1[\mathfrak{F}, \mathfrak{X}]+L_{\Pfr}[\Bfr, \Ffr] \Big]
\eeaa

Similarly, for $\qf^\F=f_2 \Qfr$, using \eqref{intermediate=square2Q}, we obtain
\beaa
\squared_2\qf^\F&=& f_2 \Big[ \frac 3 2 \ov{\tr X} \nab_3\Qfr+ \left(\frac 1 2 \tr \Xb+\ov{\tr \Xb}  \right)\nab_4\Qfr - \left(3   H+ 2 \Hb +\ov{\Hb}\right)\c \nab \Qfr+\tilde{V}_2  \Qfr \\
&&+\mathcal{P}_{C_2}\Big(\M_2[A, \mathfrak{X}, \mathfrak{B}] \Big)   + \frac 1 2 \left(\tr \Xb +\ov{\tr\Xb}\right) \M_2[A, \mathfrak{X}, \mathfrak{B}] +L_{\Qfr}[\Bfr, \Ffr] \Big]+ \square(f_2) \Qfr\\
&&-  \left(\frac n 2 \ov{\tr \Xb} +\frac m 2   \tr \Xb \right) f_2 \nab_4\Qfr- \left(\frac n 2 \tr X +\frac m 2   \ov{\tr X} \right) f_2 \nab_3 \Qfr + \left(m H+n\ov{H}+n \Hb+ m  \ov{\Hb}  \right)f_2 \c \nab \Qfr
\eeaa
which gives
\beaa
\squared_2\qf^\F&=&\left((1  -\frac n 2) \ov{\tr \Xb}+(\frac 1 2 -\frac m 2  ) \tr \Xb \right) f_2 \nab_4\Qfr+ \left(-\frac n 2 \tr X +(\frac 3 2-\frac m 2  ) \ov{\tr X} \right) f_2 \nab_3 \Qfr \\
&&+ \left((m-3) H+n\ov{H}+(n-2) \Hb+ (m-1)  \ov{\Hb}  \right)f_2 \c \nab \Qfr+\left( \tilde{V}_2+ f_2^{-1}\square(f_2)\right) \qf^\F \\
&&+ f_2 \Big[\mathcal{P}_{C_2}\Big(\M_2[A, \mathfrak{X}, \mathfrak{B}] \Big)   + \frac 1 2 \left(\tr \Xb +\ov{\tr\Xb}\right) \M_2[A, \mathfrak{X}, \mathfrak{B}] +L_{\Qfr}[\Bfr, \Ffr]  \Big]
\eeaa
Observe that the real part of the coefficients of all the first derivatives are multiple of $m+n-3$. To cancel their real part we then take $m=3-n$, which implies $f_2= (q)^n (\ov{q})^{3-n}$, and gives
 \beaa
\squared_2\qf^\F&=&i f_2 \Big[ (2 - n ) \atrchb   \nab_4\Qfr+   n \atrch   \nab_3 \Qfr + \left(-2n \dual \eta +2(n-2) \dual \etab  \right) \c \nab \Qfr \Big]\\
&&+\left( \tilde{V}_2+ f_2^{-1}\square(f_2)\right) \qf^\F + f_2 \Big[\mathcal{P}_{C_2}\Big(\M_2[A, \mathfrak{X}, \mathfrak{B}] \Big)   + \frac 1 2 \left(\tr \Xb +\ov{\tr\Xb}\right) \M_2[A, \mathfrak{X}, \mathfrak{B}] +L_{\Qfr}[\Bfr, \Ffr]  \Big]
\eeaa

 Using the values in Kerr-Newman:
 \beaa
 \atrchb e_4&=& \frac{2a\Delta\cos\th}{|q|^4}\nab_{r}+ \frac{2a\cos\th(r^2+ a^2)}{|q|^4}\nab_t + \frac{2a^2\cos\th}{|q|^4} \nab_{\vphi} \\
 \atrch e_3&=&- \frac{2a\Delta\cos\th}{|q|^4}\nab_{r} +\frac{2a\cos\th(r^2+a^2)}{|q|^4}\nab_t+ \frac{2a^2\cos\th}{|q|^4} \nab_{\vphi}\\
\dual \eta_1&=& \frac{a\sin\th r}{|q|^3}, \qquad \dual \eta_2= \frac{a^2\sin\th \cos\th}{|q|^3},\\
\dual \etab_1&=& -\frac{a\sin\th(r)}{|q|^3}, \qquad \dual \etab_2 =\frac{a^2\sin\th \cos\th}{|q|^3}.
\eeaa
we respectively obtain
\beaa
&&(1 - n ) \atrchb   \nab_4+   n \atrch   \nab_3  + \left(-2n \dual \eta +2(n-1) \dual \etab  \right) \c \nab \\
&=&(1 - n )\Big(  \frac{2a\Delta\cos\th}{|q|^4}\nab_{r}+ \frac{2a\cos\th(r^2+ a^2)}{|q|^4}\nab_t + \frac{2a^2\cos\th}{|q|^4} \nab_{\vphi}\Big)\\
&&+   n \Big(- \frac{2a\Delta\cos\th}{|q|^4}\nab_{r} +\frac{2a\cos\th(r^2+a^2)}{|q|^4}\nab_t+ \frac{2a^2\cos\th}{|q|^4} \nab_{\vphi} \Big)  \\
&&+ \left(-2n \dual \eta_1 +2(n-1) \dual \etab_1  \right)  \nab_1+ \left(-2n \dual \eta_2 +2(n-1) \dual \etab_2  \right)  \nab_2\\
&=&(1 - 2n ) \frac{2a\Delta\cos\th}{|q|^4}\nab_{r}+ \frac{2a\cos\th(r^2+ a^2)}{|q|^4}\nab_t + \frac{2a^2\cos\th}{|q|^4} \nab_{\vphi}\\
&&+ 2(1-2n)  \frac{a\sin\th r}{|q|^3}   \nab_1 -2 \frac{a^2\sin\th \cos\th}{|q|^3}   \nab_2
\eeaa
and 
\beaa
&&(2 - n ) \atrchb   \nab_4+   n \atrch   \nab_3  + \left(-2n \dual \eta +2(n-2) \dual \etab  \right) \c \nab \\
&=&(2 - 2n ) \frac{2a\Delta\cos\th}{|q|^4}\nab_{r}+2 \frac{2a\cos\th(r^2+ a^2)}{|q|^4}\nab_t +2 \frac{2a^2\cos\th}{|q|^4} \nab_{\vphi}\\
&&+ 2(2-2n)  \frac{a\sin\th r}{|q|^3}   \nab_1 -4 \frac{a^2\sin\th \cos\th}{|q|^3}   \nab_2
\eeaa
 Writing that $\nab_1=\frac{1}{|q|}\nab_\th$, and $\nab_2=\frac{a\sin\th}{|q|}\nab_t+\frac{1}{|q|\sin\th}\nab_\vphi $, we finally respectively have
 \beaa
&&(1 - n ) \atrchb   \nab_4+   n \atrch   \nab_3  + \left(-2n \dual \eta +2(n-1) \dual \etab  \right) \c \nab \\
&=& \frac{2a\cos\th}{|q|^2}\nab_t +(1 - 2n )\Big(  \frac{2a\Delta\cos\th}{|q|^4}\nab_{r}+   \frac{2a\sin\th r}{|q|^4}   \nab_\th\Big)
\eeaa
and 
\beaa
&&(2 - n ) \atrchb   \nab_4+   n \atrch   \nab_3  + \left(-2n \dual \eta +2(n-2) \dual \etab  \right) \c \nab \\
&=& \frac{4a\cos\th}{|q|^2}\nab_t +(1 - n )\Big(  \frac{4a\Delta\cos\th}{|q|^4}\nab_{r}+   \frac{4a\sin\th r}{|q|^4}   \nab_\th\Big)
\eeaa
which completes the proof.

\subsection{Proof of Proposition  \ref{right-hand-side-lemma}}\label{section-proof-lemma-rhs-1}

We compute here the right hand sides of the main equations.

\subsubsection{The right hand side of the equation for $\pf$}

Using the definition \eqref{definition-MM1} of $\M_1[\mathfrak{F}, \mathfrak{X}]$, we can write
\beaa 
\M_1[\mathfrak{F}, \mathfrak{X}]=(2\PF\ov{\PF}) \tilde{ \M_1}[\mathfrak{F}, \mathfrak{X}], \qquad \tilde{ \M_1}[\mathfrak{F}, \mathfrak{X}]=2\ov{\DDc}\c\mathfrak{F}+4\ov{\Hb}\c\mathfrak{F}- \left(2\tr \Xb -  \ov{\tr \Xb}\right) \ \mathfrak{X}
\eeaa
Using \eqref{formula-PC_1}, we can therefore compute
\beaa
&&\mathcal{P}_{C_1}\Big(\M_1[\mathfrak{F}, \mathfrak{X}]\Big)+ \frac 1 2 \left(\tr \Xb +\ov{\tr\Xb}\right)\M_1[\mathfrak{F}, \mathfrak{X}]\\
&=& \mathcal{P}_{C_1}\Big(2\PF\ov{\PF}\tilde{ \M_1}[\mathfrak{F}, \mathfrak{X}]\Big)+ \frac 1 2 \left(\tr \Xb +\ov{\tr\Xb}\right)(2\PF\ov{\PF}) \tilde{ \M_1}[\mathfrak{F}, \mathfrak{X}]\\
&=&( 2\PF\ov{\PF}) \mathcal{P}_{C_1}\left(\tilde{ \M_1}[\mathfrak{F}, \mathfrak{X}]\right) +2\nabc_3(\PF\ov{\PF})\tilde{ \M_1}[\mathfrak{F}, \mathfrak{X}]+ \frac 1 2 \left(\tr \Xb +\ov{\tr\Xb}\right)(2\PF\ov{\PF}) \tilde{ \M_1}[\mathfrak{F}, \mathfrak{X}]\\
&=& 2 \PF\ov{\PF}\Big[ \mathcal{P}_{C_1}\left(\tilde{ \M_1}[\mathfrak{F}, \mathfrak{X}]\right)-\frac 1 2 (\tr \Xb +\ov{\tr\Xb}) \left(\tilde{ \M_1}[\mathfrak{F}, \mathfrak{X}] \right)\Big]
\eeaa
where we used that $\nabc_3(\PF \ov{\PF})=-(\tr \Xb +\ov{\tr\Xb}) \PF \ov{\PF}$.
We then compute $\mathcal{P}_{C_1}\left(\tilde{ \M_1}[\mathfrak{F}, \mathfrak{X}]\right)$, using Lemma \ref{commutator-RW}:
\beaa
\mathcal{P}_{C_1}\left(\tilde{ \M_1}[\mathfrak{F}, \mathfrak{X}]\right)&=& \mathcal{P}_{C_1}\left(2\ov{\DDc}\c\mathfrak{F}+4\ov{\Hb}\c\mathfrak{F}- \left(2\tr \Xb -  \ov{\tr \Xb}\right) \ \mathfrak{X} \right)\\
&=&2 \ov{\DDc}\c(\mathcal{P}_{C_1}\mathfrak{F})+2[ \mathcal{P}_{C_1}, \ov{\DDc}\c]\mathfrak{F}+4 \ov{\Hb}\c \mathcal{P}_{C_1}(\mathfrak{F})+4 \nabc_3 \ov{\Hb}\c\mathfrak{F}\\
&&- \left(2\tr \Xb -  \ov{\tr \Xb}\right) \mathcal{P}_{C_1}( \mathfrak{X})- \nabc_3\left(2\tr \Xb -  \ov{\tr \Xb}\right) \ \mathfrak{X}\\
&=&2 \ov{\DDc}\c(\mathcal{P}_{C_1}\mathfrak{F}) - \ov{\tr\Xb}\,  \ov{\DDc} \c \Ffr+4 \ov{\Hb}\c \mathcal{P}_{C_1}(\mathfrak{F})+2\ov{H} \c \nabc_3 \Ffr\\
&&+\left(- 2\ov{\DDc} C_1 +4 \nabc_3 \ov{\Hb}+ \ov{\tr\Xb}\,  \ov{H}\right)\c \Ffr\\
&&- \left(2\tr \Xb -  \ov{\tr \Xb}\right) \mathcal{P}_{C_1}( \mathfrak{X})- \nabc_3\left(2\tr \Xb -  \ov{\tr \Xb}\right) \ \mathfrak{X}
\eeaa
We now write:
\beaa
\mathcal{P}_{C_1}\mathfrak{F}&=&\nabc_3 \mathfrak{F}+C_1 \mathfrak{F}=( \nabc_3 \mathfrak{F}+ C_2 \Ffr)+(C_1-C_2) \mathfrak{F}=\mathfrak{Q}+(C_1-C_2) \mathfrak{F}\\
\ov{\DDc}\c(\mathcal{P}_{C_1}\mathfrak{F})&=& \ov{\DDc}\c(\mathfrak{Q}+(C_1-C_2) \mathfrak{F})= \ov{\DDc}\c \mathfrak{Q}+(C_1-C_2)  \ov{\DDc}\c\mathfrak{F}+(\ov{\DDc} C_1-\ov{\DDc}C_2)\c \mathfrak{F} \\
\nabc_3 \Ffr&=&\Qfr - C_2 \Ffr
\eeaa
Using \eqref{nabc-3-mathfrak-X} we can write
\beaa
\mathcal{P}_{C_1}( \mathfrak{X})&=& \nabc_3\mathfrak{X} + C_1 \mathfrak{X}=\left(-\frac 1 2 \ov{\tr \Xb}+C_1\right) \ \mathfrak{X}-\ov{\DDc} \c \mathfrak{F}  -\ov{H} \c \mathfrak{F}-2\mathfrak{B}
\eeaa
By substituting in the above expression we obtain
\beaa
\mathcal{P}_{C_1}\left(\tilde{ \M_1}[\mathfrak{F}, \mathfrak{X}]\right)&=&2  \ov{\DDc}\c \mathfrak{Q}+\left( 4 \ov{\Hb}+2\ov{H} \right) \c \mathfrak{Q}+(2C_1-2C_2+2\tr \Xb-2 \ov{\tr\Xb})  \ov{\DDc}\c\mathfrak{F}\\
&&+\left(-2\ov{\DDc}C_2+4 \nabc_3 \ov{\Hb}+ 4(C_1-C_2)\ov{\Hb}+(2\tr \Xb -2 C_2 )\ov{H}\right)\c \mathfrak{F}\\
&&+ \left(4\tr \Xb -  2\ov{\tr \Xb}\right) \mathfrak{B} -\Big(  \nabc_3\left(2\tr \Xb -  \ov{\tr \Xb}\right)+\left(2\tr \Xb -  \ov{\tr \Xb}\right)\left(-\frac 1 2 \ov{\tr \Xb}+C_1\right) \Big) \ \mathfrak{X}
\eeaa
This gives
\beaa
&& \mathcal{P}_{C_1}\left(\tilde{ \M_1}[\mathfrak{F}, \mathfrak{X}]\right)-\frac 1 2 (\tr \Xb +\ov{\tr\Xb}) \left(\tilde{ \M_1}[\mathfrak{F}, \mathfrak{X}]\right) \\
&=& 2  \ov{\DDc}\c \mathfrak{Q}+\left( 4 \ov{\Hb}+2\ov{H} \right) \c \mathfrak{Q}+(2C_1-2C_2+\tr \Xb-3 \ov{\tr\Xb})  \ov{\DDc}\c\mathfrak{F}\\
&&+\left(-2\ov{\DDc}C_2+4 \nabc_3 \ov{\Hb}+ (4C_1-4C_2-2\tr \Xb -2\ov{\tr\Xb})\ov{\Hb}+(2\tr \Xb -2 C_2 )\ov{H}\right)\c \mathfrak{F}\\
&&+ \left(4\tr \Xb -  2\ov{\tr \Xb}\right) \mathfrak{B} -\Big(  \nabc_3\left(2\tr \Xb -  \ov{\tr \Xb}\right)+\left(2\tr \Xb -  \ov{\tr \Xb}\right)\left(-\frac 1 2 \tr \Xb - \ov{\tr \Xb}+C_1\right) \Big) \ \mathfrak{X}
\eeaa
Using \eqref{nabc-3-trXb-red} and \eqref{nabc-3-Hb-red}
we simplify the above to 
\beaa
 \mathcal{P}_{C_1}\left(\tilde{ \M_1}[\mathfrak{F}, \mathfrak{X}]\right)-\frac 1 2 (\tr \Xb +\ov{\tr\Xb}) \left(\tilde{ \M_1}[\mathfrak{F}, \mathfrak{X}]\right) &=& 2 \left( \ov{\DDc}\c \mathfrak{Q}+\left( 2 \ov{\Hb}+\ov{H} \right) \c \mathfrak{Q}\right)+ \left(4\tr \Xb -  2\ov{\tr \Xb}\right) \mathfrak{B} \\
 &&+Y^{\Ffr}_{a} \ov{\DDc}\c\mathfrak{F}+Y^{\Ffr}_0 \c \mathfrak{F}+Y^{\Xfr}_0 \ \mathfrak{X}
\eeaa
where
\bea
Y^{\Ffr}_{a}&=& 2C_1-2C_2+\tr \Xb-3 \ov{\tr\Xb} \label{eq:Y-Ffr-a}\\
Y^{\Ffr}_0&=& -2\ov{\DDc}C_2+ (4C_1-4C_2-4\tr \Xb -2\ov{\tr\Xb})\ov{\Hb}+(4\tr \Xb -2 C_2 )\ov{H} \\
Y^{\Xfr}_0&=& 2(\tr\Xb)^2 -  \frac 3 2 \ov{\tr \Xb}^2+\frac 3 2 \tr \Xb\ov{\tr \Xb}-\left(2\tr \Xb -  \ov{\tr \Xb}\right)C_1
\eea
To complete the proof of the first part of the Proposition, we need to compute the terms $Y^{\Ffr}_{a}$ and $Y^{\Xfr}_0$. 
Recall that, according to \eqref{definition-C1-C2} and \eqref{definition-C1-C22}, we can write
\beaa
C_1&=&2\trchb + O\left(\frac{|a|}{r^2} \right) \qquad C_2=\trchb  + O\left(\frac{|a|}{r^2} \right)
 \eeaa
We then have
\beaa
Y^{\Ffr}_{a}&=& 2C_1-2C_2+\tr \Xb-3 \ov{\tr\Xb}= 4\trchb-2\trchb+\trchb-3 \trchb+ O\left(\frac{|a|}{r^2} \right)=O\left(\frac{|a|}{r^2} \right) \\
Y^{\Xfr}_0&=& 2(\tr\Xb)^2 -  \frac 3 2 \ov{\tr \Xb}^2+\frac 3 2 \tr \Xb\ov{\tr \Xb}-\left(2\tr \Xb -  \ov{\tr \Xb}\right)C_1\\
&=& 2\trchb^2 -  \frac 3 2 \trchb^2+\frac 3 2 \trchb^2-\left(2\trchb -  \trchb\right)2\trchb+O\left(\frac{|a|}{r^3} \right) =O\left(\frac{|a|}{r^3} \right) 
\eeaa
Finally by writing $ 4 \PF\ov{\PF}\left(2\tr \Xb -  \ov{\tr \Xb}\right) \mathfrak{B}= 4 \trchb \rhoF^2 \Bfr+ (2 \PF \ov{\PF}) Y^{\Bfr}_0 \Bfr$, for $Y^{\Bfr}_0=O\left(\frac{|a|}{r^2} \right)$, we obtain the stated relation.

\subsubsection{The right hand side of the equation for $\qf^\F$}
We have
\beaa
&&\mathcal{P}_{C_2}\Big(\M_2[A, \mathfrak{X}, \mathfrak{B}] \Big)+\frac 1 2 \left(\tr \Xb +\ov{\tr\Xb}\right) \M_2[A, \mathfrak{X}, \mathfrak{B}]\\
&=& \nabc_3\Big(\M_2[A, \mathfrak{X}, \mathfrak{B}] \Big)+\left(C_2+ \frac 1 2 \left(\tr \Xb +\ov{\tr\Xb}\right)\right) \M_2[A, \mathfrak{X}, \mathfrak{B}]
\eeaa
Using the definition \eqref{definition-MM2} of $\M_2[A, \mathfrak{X}, \mathfrak{B}]$, we compute
\beaa
&& \nabc_3\Big(\M_2[A, \mathfrak{X}, \mathfrak{B}] \Big)\\
&=&  \nabc_3\Big( -  \PF \left(\nabc_3A +\frac 1 2 \left(3\tr \Xb-\ov{\tr \Xb}  \right) A\right)    + \left(\frac 3 2  \nabc_3 H \right) \hot \mathfrak{X} +\left(2  H- \Hb \right) \hot \mathfrak{B} \Big)\\
 &=& -  \PF \left(\nabc_3\nabc_3A +\frac 1 2 \left(3\tr \Xb-\ov{\tr \Xb}  \right) \nabc_3A+\frac 1 2 \nabc_3\left(3\tr \Xb-\ov{\tr \Xb}  \right) A\right)\\
 &&+\tr\Xb  \PF \left(\nabc_3A +\frac 1 2 \left(3\tr \Xb-\ov{\tr \Xb}  \right) A\right)   + \left(\frac 3 2  \nabc_3 H \right) \hot \nabc_3\mathfrak{X} + \left(\frac 3 2  \nabc_3\nabc_3 H \right) \hot \mathfrak{X} \\
 &&+\left(2  H- \Hb \right) \hot \nabc_3\mathfrak{B} +\nabc_3\left(2  H- \Hb \right) \hot \mathfrak{B} 
 \eeaa
 Using \eqref{nabc-3-mathfrak-X} to express $\nabc_3 \Xfr$, and writing $\nabc_3 \Bfr=\Pfr-C_1 \Bfr$, we obtain
 \beaa
&& \nabc_3\Big(\M_2[A, \mathfrak{X}, \mathfrak{B}] \Big)\\
 &=& -  \PF \left(\nabc_3\nabc_3A +\frac 1 2 \left(\tr \Xb-\ov{\tr \Xb}  \right) \nabc_3A+ \left(-\frac 9 4 \tr \Xb^2+\frac 1 2 \ov{\tr \Xb} \tr\Xb+\frac 1 4 \ov{\tr \Xb}^2  \right) A\right)\\
 && -  3   \nabc_3 H  \c \nabc \mathfrak{F}   - \left(\frac 3 2  \nabc_3 H \c  \ov{H} \right)  \mathfrak{F} + \left(\frac 3 2  \nabc_3\nabc_3 H-\frac 3 4  \ov{\tr \Xb}  \nabc_3 H \right) \hot \mathfrak{X} \\
 &&+\left(2  H- \Hb \right) \hot \Pfr +\left( -\nabc_3\left( H+\Hb \right)-C_1\left(2  H- \Hb \right)\right)  \hot \mathfrak{B} 
 \eeaa
We therefore obtain
\beaa
&&\mathcal{P}_{C_2}\Big(\M_2[A, \mathfrak{X}, \mathfrak{B}] \Big)+\frac 1 2 \left(\tr \Xb +\ov{\tr\Xb}\right) \M_2[A, \mathfrak{X}, \mathfrak{B}]\\
&=& -  \PF \left(\nabc_3\nabc_3A + \left(C_2 + \tr \Xb \right) \nabc_3A+ \left(-\frac 32 \tr \Xb^2+ \ov{\tr \Xb} \tr\Xb +\frac 1 2 \left(3\tr \Xb-\ov{\tr \Xb}  \right) C_2 \right) A\right)\\
 && -  3   \nabc_3 H  \c \nabc \mathfrak{F}   - \left(\frac 3 2  \nabc_3 H \c  \ov{H} \right)  \mathfrak{F} + \left(\frac 3 2  \nabc_3\nabc_3 H+\frac 3 2 \left(C_2+ \frac 1 2\tr \Xb \right)  \nabc_3 H \right) \hot \mathfrak{X} \\
 &&+\left(2  H- \Hb \right) \hot \Pfr +\left( -\nabc_3\left( H+\Hb \right)+(C_2-C_1+ \frac 1 2 \left(\tr \Xb +\ov{\tr\Xb}\right))\left(2  H- \Hb \right)\right)  \hot \mathfrak{B} 
\eeaa
 We now want to relate the first line of the above to the
 relation \eqref{relation-F-A-B}. Observe that
 \beaa
&&  \nabc_3 \left(\PF  \left(\nabc_3A+\frac{1}{2}\tr\Xb A\right)\right)\\
&=& \PF  \left(\nabc_3\nabc_3A+\frac{1}{2}\tr\Xb \nabc_3A-\frac{1}{4}\tr\Xb^2 A\right)-\tr\Xb \PF  \left(\nabc_3A+\frac{1}{2}\tr\Xb A\right)\\
 &=& \PF  \left(\nabc_3\nabc_3A-\frac{1}{2}\tr\Xb \nabc_3A-\frac{3}{4}\tr\Xb^2 A\right)
 \eeaa
 The first line of the above can then be written as
 \beaa
 &&-  \PF \left(\nabc_3\nabc_3A + \left(C_2 + \tr \Xb \right) \nabc_3A+ \left(-\frac 32 \tr \Xb^2+ \ov{\tr \Xb} \tr\Xb +\frac 1 2 \left(3\tr \Xb-\ov{\tr \Xb}  \right) C_2 \right) A\right)\\
&=&-  \PF \Big(\nabc_3\nabc_3A-\frac 1 2 \tr\Xb \nabc_3 A-\frac 34 \tr \Xb^2 A \Big)-\left(C_2 +\frac 3 2  \tr \Xb \right) \PF\Big(  \nabc_3A+\frac 1 2 \tr\Xb A\Big)  \\
&&-\PF \left(-\frac 32 \tr \Xb^2+ \ov{\tr \Xb} \tr\Xb +\frac 1 2 \left(2\tr \Xb-\ov{\tr \Xb}  \right) C_2 \right) A \\
&=&-  \nabc_3 \left(\PF  \left(\nabc_3A+\frac{1}{2}\tr\Xb A\right)\right)-\left(C_2 +\frac 3 2  \tr \Xb \right) \PF\Big(  \nabc_3A+\frac 1 2 \tr\Xb A\Big)  \\
&&-\PF \left(-\frac 32 \tr \Xb^2+ \ov{\tr \Xb} \tr\Xb +\frac 1 2 \left(2\tr \Xb-\ov{\tr \Xb}  \right) C_2 \right) A 
 \eeaa
Using \eqref{relation-F-A-B} we obtain
\beaa
-  \nabc_3 \left(\PF  \left(\nabc_3A+\frac{1}{2}\tr\Xb A\right)\right)&=& -  \nabc_3 \left( \frac 1 2 \DDc \hot \mathfrak{B}+  3H  \hot  \mathfrak{B} -\left(3 \ov{P}+2\PF\ov{\PF}\right)\mathfrak{F}\right)\\
&=& -   \frac 1 2 \nabc_3\DDc \hot \mathfrak{B}-  3H  \hot  \nabc_3\mathfrak{B}-  3\nabc_3H  \hot  \mathfrak{B} \\
&&+\left(3 \ov{P}+2\PF\ov{\PF}\right)\nabc_3\mathfrak{F}+\nabc_3\left(3 \ov{P}+2\PF\ov{\PF}\right)\mathfrak{F}
\eeaa
Using now \eqref{commutator-nabc-3-F-formula} applied to $\Bfr$, we obtain
\beaa
&&-  \nabc_3 \left(\PF  \left(\nabc_3A+\frac{1}{2}\tr\Xb A\right)\right)\\
&=& -   \frac 1 2 \DDc \hot \nabc_3 \mathfrak{B} +   \frac 1 4 \tr \Xb  \DDc\hot \Bfr  -  \frac 7 2 H  \hot  \nabc_3\mathfrak{B}+\left( -  3\nabc_3H  +\frac 1 2 \tr\Xb H \right) \hot  \mathfrak{B} \\
&&+\left(3 \ov{P}+2\PF\ov{\PF}\right)\nabc_3\mathfrak{F}+\nabc_3\left(3 \ov{P}+2\PF\ov{\PF}\right)\mathfrak{F}
\eeaa
We therefore obtain, using once again \eqref{relation-F-A-B}:
\beaa
&&-  \nabc_3 \left(\PF  \left(\nabc_3A+\frac{1}{2}\tr\Xb A\right)\right)-\left(C_2 +\frac 3 2  \tr \Xb \right) \PF\Big(  \nabc_3A+\frac 1 2 \tr\Xb A\Big)\\
&=& -   \frac 1 2 \DDc \hot \nabc_3 \mathfrak{B} -\frac 1 2 \left( C_2 + \tr \Xb \right)  \DDc\hot \Bfr  -  \frac 7 2 H  \hot  \nabc_3\mathfrak{B}+\left( -  3\nabc_3H  +(-3C_2-4 \tr\Xb)  H \right) \hot  \mathfrak{B} \\
&&+\left(3 \ov{P}+2\PF\ov{\PF}\right)\nabc_3\mathfrak{F}+\Big( \nabc_3\left(3 \ov{P}+2\PF\ov{\PF}\right)+\left(C_2 +\frac 3 2  \tr \Xb \right) \left(3 \ov{P}+2\PF\ov{\PF}\right) \Big) \mathfrak{F}
\eeaa
 Writing $\nabc_3 \Bfr=\mathfrak{P}-C_1\Bfr$ and $\nabc_3 \Ffr=\mathfrak{Q}-C_2\Ffr$, we obtain
 \beaa
&&-  \nabc_3 \left(\PF  \left(\nabc_3A+\frac{1}{2}\tr\Xb A\right)\right)-\left(C_2 +\frac 3 2  \tr \Xb \right) \PF\Big(  \nabc_3A+\frac 1 2 \tr\Xb A\Big)\\
&=& -   \frac 1 2 \DDc \hot \mathfrak{P}  -  \frac 7 2 H  \hot  \mathfrak{P}+\frac 1 2 \left(C_1- C_2 - \tr \Xb \right)  \DDc\hot \Bfr \\
&&+\left( -  3\nabc_3H+\frac 1 2 \DDc C_1  +(\frac 7 2 C_1-3C_2-4 \tr\Xb)  H \right) \hot  \mathfrak{B} \\
&&+\left(3 \ov{P}+2\PF\ov{\PF}\right)\mathfrak{Q} -  2\PF \ov{\PF}\left( \tr\Xb+ \ov{\tr \Xb}\right)\mathfrak{F}
\eeaa
 where we used that  $\nabc_3\left(3 \ov{P}+2\PF\ov{\PF}\right)+\left(\frac 3 2  \tr \Xb \right) \left(3 \ov{P}+2\PF\ov{\PF}\right)=-  2\PF \ov{\PF}\left( \tr\Xb+ \ov{\tr \Xb}\right)$.

 Finally, to express the last line we recall \eqref{relation-F-B-A} and then write
\beaa
&&-\PF \left(-\frac 32 \tr \Xb^2+ \ov{\tr \Xb} \tr\Xb +\frac 1 2 \left(2\tr \Xb-\ov{\tr \Xb}  \right) C_2 \right) A\\
&=& \left(-\frac 32 \tr \Xb^2+ \ov{\tr \Xb} \tr\Xb +\frac 1 2 \left(2\tr \Xb-\ov{\tr \Xb}  \right) C_2 \right) \nabc_4 \mathfrak{F}\\
&& \left(-\frac 32 \tr \Xb^2+ \ov{\tr \Xb} \tr\Xb +\frac 1 2 \left(2\tr \Xb-\ov{\tr \Xb}  \right) C_2 \right) \left(\frac 3 2 \ov{\tr X} +\frac 1 2  \tr X\right)\mathfrak{F}\\ 
&& \left(-\frac 32 \tr \Xb^2+ \ov{\tr \Xb} \tr\Xb +\frac 1 2 \left(2\tr \Xb-\ov{\tr \Xb}  \right) C_2 \right) \frac 1 2 \DDc \hot \mathfrak{X} \\
&&  \left(-\frac 32 \tr \Xb^2+ \ov{\tr \Xb} \tr\Xb +\frac 1 2 \left(2\tr \Xb-\ov{\tr \Xb}  \right) C_2 \right)  \frac 1 2 \left( 3    H+ \Hb \right) \hot \mathfrak{X} 
\eeaa

By putting everything together we finally obtain
\beaa
&&\mathcal{P}_{C_2}\Big(\M_2[A, \mathfrak{X}, \mathfrak{B}] \Big)+\frac 1 2 \left(\tr \Xb +\ov{\tr\Xb}\right) \M_2[A, \mathfrak{X}, \mathfrak{B}]\\
&=&\left(3 \ov{P}+2\PF\ov{\PF}\right)\mathfrak{Q}-   \frac 1 2 \left( \DDc \hot \mathfrak{P}  + (3H+2\Hb)  \hot  \mathfrak{P}\right)  -  2\PF \ov{\PF}\left( \tr\Xb+ \ov{\tr \Xb}\right)\mathfrak{F} \\
&& +W_4^{\Ffr} \nabc_4 \mathfrak{F}+W_a^{\Ffr}  \c \nabc \mathfrak{F} + W_0^{\Ffr} \mathfrak{F}+ W_{a}^{\Bfr} \DDc\hot \Bfr +W_0^{\Bfr} \hot  \mathfrak{B}+W_a^{\Xfr} \DDc \hot \mathfrak{X} + W_0^{\Xfr} \hot \mathfrak{X}  
\eeaa
where
\bea
W_4^{\Ffr}&=& -\frac 32 \tr \Xb^2+ \ov{\tr \Xb} \tr\Xb +\frac 1 2 \left(2\tr \Xb-\ov{\tr \Xb}  \right) C_2 \label{eq:W4Ffr}\\
W_a^{\Ffr}&=&  -  3   \nabc_3 H \label{eq:W-a-Ffr}\\
W_0^{\Ffr}&=&  \left(-\frac 32 \tr \Xb^2+ \ov{\tr \Xb} \tr\Xb +\frac 1 2 \left(2\tr \Xb-\ov{\tr \Xb}  \right) C_2 \right) \left(\frac 3 2 \ov{\tr X} +\frac 1 2  \tr X\right)-\frac 3 2  \nabc_3 H \c  \ov{H}\\
W_{a}^{\Bfr}&=& \frac 1 2 \left(C_1- C_2 - \tr \Xb \right) \label{eq:WaBfr}\\
W_0^{\Bfr}&=&  -  3\nabc_3H+\frac 1 2 \DDc C_1  +(\frac 7 2 C_1-3C_2-4 \tr\Xb)  H \\
&&-\nabc_3\left( H+\Hb \right)+(C_2-C_1+ \frac 1 2 \left(\tr \Xb +\ov{\tr\Xb}\right))\left(2  H- \Hb \right)\\
W_a^{\Xfr}&=&  \frac 1 2\left(-\frac 32 \tr \Xb^2+ \ov{\tr \Xb} \tr\Xb +\frac 1 2 \left(2\tr \Xb-\ov{\tr \Xb}  \right) C_2 \right) \label{eq:WaXfr}\\
W_0^{\Xfr}&=&  \left(-\frac 32 \tr \Xb^2+ \ov{\tr \Xb} \tr\Xb +\frac 1 2 \left(2\tr \Xb-\ov{\tr \Xb}  \right) C_2 \right)  \frac 1 2 \left( 3    H+ \Hb \right)\\
&&+\frac 3 2  \nabc_3\nabc_3 H+\frac 3 2 \left(C_2+ \frac 1 2\tr \Xb \right)  \nabc_3 H 
\eea
  To complete the proof of the second part of the Proposition, we need to compute the terms $W_4^{\Ffr}$, $W_0^{\Ffr}$, $W_{a}^{\Bfr}$, $W_a^{\Xfr}$ and $W_0^{\Xfr}$. 
Recall that, according to \eqref{definition-C1-C2} and \eqref{definition-C1-C22}, we can write
\beaa
C_1&=&2\trchb + O\left(\frac{|a|}{r^2} \right) \qquad C_2=\trchb  + O\left(\frac{|a|}{r^2} \right)
 \eeaa
We then have
\beaa
W_4^{\Ffr}&=& -\frac 32 \tr \Xb^2+ \ov{\tr \Xb} \tr\Xb +\frac 1 2 \left(2\tr \Xb-\ov{\tr \Xb}  \right) C_2\\
&=& -\frac 32 \trchb^2+ \trchb^2 +\frac 1 2 \trchb^2+ O\left(\frac{|a|}{r^3} \right) =O\left(\frac{|a|}{r^3} \right)\\
W_{a}^{\Bfr}&=& \frac 1 2 \left(C_1- C_2 - \tr \Xb \right) =\frac 1 2 \left(2\trchb-\trchb - \trchb \right) + O\left(\frac{|a|}{r^2} \right) = O\left(\frac{|a|}{r^2} \right) 
\eeaa
  Finally observe that $W_0^{\Ffr}$, $W_a^{\Xfr}$ and $W_0^{\Xfr}$ are $O(|a|)$ because they are multiplied by $W_4^{\Ffr}=O\left(\frac{|a|}{r^3} \right)$.

\subsection{Proof of Proposition \ref{prop:potentials-real}}\label{sec:proof-potentials}

        Recall that 
        \beaa
            V_1&:=&  \tilde{V}_1+ f_1^{-1}\square(f_1)\\
        V_2&:=&  \tilde{V}_2+ f_2^{-1}\square(f_2)+3 \ov{P}+2\PF\ov{\PF}
        \eeaa
        We start by computing the real and imaginary part of $f^{-1} \square(f)$. 
        \begin{lemma} For $f=q^n \ov{q}^m$, we have
            \beaa
   \Re( f^{-1} \square f )   &=&-  \frac{(m+n)(m+n+1)}{ 4} \trch\trchb  -\frac{(n-m)^2+m + n}{ 4}   \atrch\atrchb \\
    &&  - (n+m) \rho + 2nm|\etab|^2+ (m^2+n^2+m+n) \eta \c \etab 
    \eeaa
    and 
       \beaa
   \Im( f^{-1} \square f )   &=& (m-n)\Big(-\frac 1 2(n+m+1) \trchb  \atrch + \rhod    -(m+n+1) \eta \wedge \etab \Big)
    \eeaa
    In particular,  for $f_1=(q)^{1/2} (\ov{q})^{9/2}$ and  $f_2=q \ov{q}^2$, we obtain
                \beaa
   \Re( f_1^{-1} \square f_1 )   &=&-  \frac{15}{ 2} \trch\trchb  -\frac{21}{ 4}   \atrch\atrchb   - 5 \rho + \frac 9 2 |\etab|^2+ \frac{51}{2} \eta \c \etab \\
   \Im( f_1^{-1} \square f_1 )   &=& -12 \trchb  \atrch +4 \rhod    -24 \eta \wedge \etab 
    \eeaa
    and
    \beaa
      \Re( f_2^{-1} \square f_2 )   &=&-  3 \trch\trchb  -  \atrch\atrchb   - 3 \rho + 4|\etab|^2+ 8 \eta \c \etab \\
       \Im( f_2^{-1} \square f_2 )   &=&-2 \trchb  \atrch + \rhod    -4 \eta \wedge \etab 
    \eeaa
    \end{lemma}
      \begin{proof} Recall that for a scalar
\bea\label{wave-GKS}
\begin{split}
f^{-1} \square f &=- f^{-1} e_4e_3 f -\frac 1 2 \trchb f^{-1}e_4 f -\frac 1 2 \trch f^{-1}e_3 f +f^{-1} \lap f+ 2\etab \c   f^{-1}\nab f. 
\end{split}
\eea
Using 
 \beaa
 f^{-1}  \nab_3(f)&=& \left(\frac n 2 \ov{\tr \Xb} +\frac m 2   \tr \Xb \right) , \\ 
 f^{-1}   \nab_4(f) &=& \left(\frac n 2 \tr X +\frac m 2   \ov{\tr X} \right) ,\\
  2   f^{-1} \nab f    &=&  \left(m H+n\ov{H}+n \Hb+ m  \ov{\Hb}  \right).
    \eeaa
    we have
    \beaa
    \nab_4 \nab_3 f &=& \left(\frac n 2 \nab_4\ov{\tr \Xb} +\frac m 2   \nab_4\tr \Xb \right)f + \left(\frac n 2 \ov{\tr \Xb} +\frac m 2   \tr \Xb \right)\nab_4 f\\
    &=& \left(\frac n 2 (-\frac 1 2 \ov{\tr X}\ov{ \tr \Xb} + \ov{\DD} \c \Hb+ \Hb \c \ov{\Hb} + 2 P) +\frac m 2  (-\frac 1 2 \tr X \tr \Xb + \DD \c \ov{\Hb}+ \Hb \c \ov{\Hb} + 2 \ov{P}) \right)f \\
    &&+ \left(\frac n 2 \ov{\tr \Xb} +\frac m 2   \tr \Xb \right) \left(\frac n 2 \tr X +\frac m 2   \ov{\tr X} \right) f 
    \eeaa
    which gives
    \beaa
    f^{-1} \nab_4 \nab_3 f     &=&  \frac{nm - n}{ 4} \ov{\tr X}\ov{ \tr \Xb} + \frac{nm- m}{ 4} \tr X \tr \Xb+\frac{n^2}{4} \tr X \ov{\tr\Xb}+\frac{m^2}{4} \tr\Xb \ov{\tr X} \\
    &&+ \frac n 2\ov{\DD} \c \Hb+\frac m 2 \DD \c \ov{\Hb} +\frac {n+m}{ 2} \Hb \c \ov{\Hb} + n P+ m \ov{P} 
    \eeaa
    Observe that $\Re( \ov{\tr X}\ov{ \tr \Xb})=\trch\trchb - \atrch\atrchb$, while $\Re(\tr X \ov{\tr \Xb})=\trch\trchb + \atrch\atrchb$.
    In particular 
        \beaa
\Re(    f^{-1} \nab_4 \nab_3 f   )  &=&  \frac{2nm-m - n}{ 4} (\trch\trchb - \atrch\atrchb) +\frac{n^2+m^2}{4} (\trch\trchb + \atrch\atrchb) \\
    &&+ (n+m) \div \etab  +(n+m)| \etab|^2 + (n+m) \rho\\
     &=&  \frac{(m+n)(m+n-1)}{ 4} \trch\trchb  +\frac{(n-m)^2+m + n}{ 4}   \atrch\atrchb \\
    &&+ (n+m) \div \etab  +(n+m)| \etab|^2 + (n+m) \rho
    \eeaa
          Observe that $\Im( \ov{\tr X}\ov{ \tr \Xb})=\trch\atrchb+ \trchb\atrch=0$ and $\Im(\tr X\tr \Xb)=-\trch\atrchb- \trchb\atrch=0$, while $\Im(\tr X \ov{\tr \Xb})=-2\trchb \atrch$.
    In particular 
        \beaa
  \Im(  f^{-1} \nab_4 \nab_3 f  )   &=&-\frac{n^2}{2}\trchb \atrch+\frac{m^2}{2}\trchb \atrch+ (n-m)\left( \rhod+ \curl \etab\right) \\
  &=& (n-m)\left( -\frac{1}{2}(n+m)\trchb \atrch+\rhod+ \curl \etab\right) 
    \eeaa

    Also 
    \beaa
    \lap f &=&\frac 1 2  \nab \c ( \left(m H+n\ov{H}+n \Hb+ m  \ov{\Hb}  \right) f ) \\
  &=&\frac 1 2  \nab \c \left(m H+n\ov{H}+n \Hb+ m  \ov{\Hb}\right) f+ \frac 1 4  \left(m H+n\ov{H}+n \Hb+ m  \ov{\Hb}  \right) \c  \left(m H+n\ov{H}+n \Hb+ m  \ov{\Hb}  \right) f
    \eeaa
    In particular
    \beaa
    \Re(f^{-1}  \lap f )&=& \frac 1 2  \nab \c \left((m+n) \eta+(m+n) \etab\right) + \frac 1 4 \left( (m+n) \eta+(m+n) \etab \right) \c \left( (m+n) \eta+(m+n) \etab\right) \\
    &&- \frac 1 4 ((m -n)\dual \eta +(n-m) \dual \etab) \c ((m -n)\dual \eta +(n-m) \dual \etab)\\
    &=& \frac 1 2(m+n)  \left( \div\eta+\div \etab\right)+ \frac 1 4(m+n)^2 \left(  \eta+ \etab \right) \c \left(  \eta+ \etab\right) - \frac 1 4(m -n)^2 ( \eta -  \etab) \c ( \eta - \etab)\\
    &=&(m+n)  \left( \div \etab\right) + 2nm|\etab|^2+ (m^2+n^2) \eta \c \etab
    \eeaa
since $\div \eta=\div \etab$ and $|\eta|^2=|\etab|^2$, and 
       \beaa
\Im( f^{-1}  \lap f  ) &=&\frac 1 2  \nab \c \left((m -n)\dual \eta +(n-m) \dual \etab \right) + \frac 1 2 ((m+n) \eta + (m+n) \etab )\c \left((m -n)\dual \eta +(n-m) \dual \etab \right)\\
&=&\frac 1 2 (m-n ) \Big( \curl \eta - \curl  \etab  +(m+n) ( -2\eta \c \dual \etab  ) \Big)\\
&=& (m-n ) \Big(  - \curl  \etab  -(m+n) \eta \wedge \etab   \Big)
   \eeaa
   since $\curl \eta = - \curl \etab$. 

    We therefore have
    \beaa
   \Re( f^{-1} \square f )&=&-\Re( f^{-1} e_4e_3 f) -\frac 1 2 \trchb  \Re(f^{-1}e_4 f) -\frac 1 2 \trch \Re(f^{-1}e_3 f) +\Re(f^{-1} \lap f)+ 2\etab \c  \Re( f^{-1}\nab f) \\
   &=&-  \frac{(m+n)(m+n-1)}{ 4} \trch\trchb  -\frac{(n-m)^2+m + n}{ 4}   \atrch\atrchb \\
    &&- (n+m) \div \etab  -(n+m)| \etab|^2 - (n+m) \rho\\
    && -\frac 1 2 \trchb  \frac{n+m}{2} \trch -\frac 1 2 \trch \frac{n+m}{2} \trchb +(m+n)  \left( \div \etab\right) + 2nm|\etab|^2+ (m^2+n^2) \eta \c \etab\\
    &&+ \etab \c ((m+n) \eta+(m+n) \etab) 
    \eeaa
    which finally gives
        \beaa
   \Re( f^{-1} \square f )   &=&-  \frac{(m+n)(m+n+1)}{ 4} \trch\trchb  -\frac{(n-m)^2+m + n}{ 4}   \atrch\atrchb \\
    &&  - (n+m) \rho + 2nm|\etab|^2+ (m^2+n^2+m+n) \eta \c \etab 
    \eeaa

 Also,
    \beaa
   \Im( f^{-1} \square f )&=&-\Im( f^{-1} e_4e_3 f) -\frac 1 2 \trchb  \Im(f^{-1}e_4 f) -\frac 1 2 \trch \Im(f^{-1}e_3 f) +\Im(f^{-1} \lap f)+ 2\etab \c  \Im( f^{-1}\nab f) \\
   &=& (m-n)\Big( -\frac{1}{2}(n+m)\trchb \atrch+\rhod+ \curl \etab  -\frac 1 4 \trchb  \atrch  +\frac 1 4 \trch  \atrchb  \\
   &&- \curl  \etab  -(m+n) \eta \wedge \etab  + \etab \c  (\dual \eta - \dual \etab )\Big)\\
   &=& (m-n)\Big(-\frac 1 2(n+m+1) \trchb  \atrch + \rhod    -(m+n+1) \eta \wedge \etab \Big)
    \eeaa
as stated. 
        \end{proof}
        
        We now compute $\tilde{V}_1$. Using \eqref{tilde-V1}, we have
        \beaa
\Re(\tilde{V}_1)&=&\Re( \frac{9}{2}\tr\Xb \ov{\tr X}-9 \ov{ \Hb} \c  H) +4\rhoF^2+4\dual\rhoF^2-\Re(\hat{V}_1)\\
&&+ \frac 14 \trch\trchb+\frac 1 4 \atrch\atrchb+ \rho-\rhoF^2-\dual\rhoF^2\\
&=&\frac {19}{ 4} \trch\trchb + \frac{19}{4} \atrch\atrchb+ \rho+3\rhoF^2+3\dual\rhoF^2 -18\eta \c \etab-\Re(\hat{V}_1)
\eeaa
and
\beaa
\Im(\tilde{V}_1)&=&\Im( \frac{9}{2}\tr\Xb \ov{\tr X}-9 \ov{ \Hb} \c  H)-\Im(\hat{V}_1)- \rhod+ \eta \wedge \etab  \\
&=& 9 \trchb \atrch - \rhod +19 \eta \wedge \etab-\Im(\hat{V}_1)
\eeaa
Using \eqref{V1hat}, we have
\beaa
\Re(\hat{V}_1)&=& \Re( I^{\Bfr}_{3}) +\Re(J^{\Bfr}_3) +\Re(K^{\Bfr}_3)+\Re(M^{\Bfr}_3)\\
\Im(\hat{V}_1)&=&\Im( I^{\Bfr}_{3}) +\Im(J^{\Bfr}_3) +\Im(K^{\Bfr}_3)+\Im(M^{\Bfr}_3)
\eeaa
Using \eqref{I-Bfr-3}, and writing $C_1=2\trchb+i p_1 \atrchb$, we obtain
\beaa
\Re(I^{\Bfr}_{3})&=& -2\rho -2\rhoF^2 -2\dual\rhoF^2 -2 \eta \c (\eta-2\etab)+2\nabc_4 \trchb\\
&=& -2\rho -2\rhoF^2 -2\dual\rhoF^2 -2 \eta \c (\eta-2\etab)+2(-\frac 1 2 \trch\trchb+\frac 1 2 \atrch\atrchb+ 2 \div \etab+2|\etab|^2+2\rho)\\
&=&- \trch\trchb+ \atrch\atrchb +2\rho -2\rhoF^2 -2\dual\rhoF^2+ 4 \div \etab +4 \eta \c \etab+2|\etab|^2
\eeaa
and using that $\trchb\atrch +\trch\atrchb=0$, 
\beaa
\Im(I^{\Bfr}_{3})&=&  2\rhod-2\eta \wedge \etab+p_1\nabc_4 \atrchb  \\
&=&  2\rhod-2\eta \wedge \etab+p_1\left(-\frac 1 2(\atrch \trchb+\trch\atrchb)+ 2 \curl \etab +2 \dual \rho\right) \\
&=&(2+2 p_1)\dual \rho+ 2 p_1 \curl \etab-2\eta \wedge \etab 
\eeaa
Using \eqref{J-Bfr-3}, we have
\beaa
\Re(J^\Bfr_3)&=&\Re( -  3  \ov{\tr X}\trchb +\frac 1 2 (\ov{\DDc} \c H)+ \frac 1 2 (H \c \ov{H}))=-3\trch\trchb+\div\etab+|\etab|^2\\
\Im(J^\Bfr_3)&=& \Im(-  3  \ov{\tr X}\trchb +\frac 1 2 (\ov{\DDc} \c H))= -  3\trchb \atrch   -\curl \etab
\eeaa
Using \eqref{K-Bfr-3}, we have
\beaa
\Re(K^\Bfr_3)&=& -3 \nabc_3( \trch)= \frac 3 2 \trchb\trch-\frac 3 2 \atrch\atrchb-6\rho-6\div\etab-6|\etab|^2\\
\Im(K^\Bfr_3)&=& -3 \nabc_3( \atrch)=6 \dual \rho+6 \curl \etab
\eeaa
Using \eqref{M-Bfr-3}, we have
\beaa
\Re(M^{\Bfr}_3)&=&\Re( \eta \c \left( 6(\eta+ i \dual \eta)+ \eta- i \dual \eta + 3(\etab - i \dual \etab) \right))=7 |\etab|^2 +3\eta \c \etab\\
\Im(M^{\Bfr}_3)&=&\Im( \eta \c \left( 6(\eta+ i \dual \eta)+ \eta- i \dual \eta + 3(\etab - i \dual \etab)  \right) )=-3 \eta \c \dual \etab=-3 \eta \wedge \etab
\eeaa
This gives
\beaa
\Re(\hat{V}_1)&=&-\frac 5 2  \trch\trchb-\frac 1 2  \atrch\atrchb -4\rho -2\rhoF^2 -2\dual\rhoF^2- \div \etab +7 \eta \c \etab+4|\etab|^2 \\
\Im(\hat{V}_1)&=& -  3\trchb \atrch +(8+2 p_1)\dual \rho+( 2 p_1 +5)\curl \etab-5\eta \wedge \etab  
\eeaa
and therefore
        \beaa
\Re(\tilde{V}_1)&=&\frac {29}{ 4} \trch\trchb + \frac{21}{4} \atrch\atrchb+ 5\rho+5\rhoF^2+5\dual\rhoF^2  + \div \etab -25 \eta \c \etab-4|\etab|^2\\
\Im(\tilde{V}_1)&=&12\trchb \atrch   -(9+2 p_1)\dual \rho-( 2 p_1 +5)\curl \etab+24\eta \wedge \etab  
\eeaa
Finally this gives:
\beaa
\Re( V_1)&=& \Re( \tilde{V}_1)+ \Re(f_1^{-1}\square(f_1))\\
&=& \frac {29}{ 4} \trch\trchb + \frac{21}{4} \atrch\atrchb+ 5\rho+5\rhoF^2+5\dual\rhoF^2  + \div \etab -25 \eta \c \etab-4|\etab|^2\\
&&-  \frac{15}{ 2} \trch\trchb  -\frac{21}{ 4}   \atrch\atrchb   - 5 \rho + \frac 9 2 |\etab|^2+ \frac{51}{2} \eta \c \etab\\
&=& -\frac {1}{ 4} \trch\trchb +5\rhoF^2+5\dual\rhoF^2  + \div \etab +\frac 1 2  \eta \c \etab+\frac 1 2 |\etab|^2
 \eeaa
 and
 \beaa
\Im(  V_1)&=& \Im( \tilde{V}_1)+ \Im(f_1^{-1}\square(f_1))\\
&=&12\trchb \atrch   -(9+2 p_1)\dual \rho-( 2 p_1 +5)\curl \etab+24\eta \wedge \etab  -12 \trchb  \atrch +4 \rhod    -24 \eta \wedge \etab \\
&=&  -(2 p_1+5)\dual \rho-( 2 p_1 +5)\curl \etab
\eeaa
Observe that for $p_1=-\frac 5 2 $, we obtain $\Im(V_1)=0$, as desired.

We now compute $\tilde{V_2}$. Using \eqref{tilde-V2}, we have
\beaa
\Re(\tilde{V}_2)&=& \Re( \frac 3 4 \tr \Xb  \ov{\tr X}+ \frac 1 4\ov{\tr \Xb}   \tr X-3\ov{P} +P -4\PF\ov{\PF} + \frac 3 2\ov{\DDc}\c H)-\eta \c \etab-\Re( \hat{V}_2)\\
&&-\frac 1 2  \trch\trchb- \frac 1 2 \atrch\atrchb-2\rho+2\rhoF^2+2\dual\rhoF^2\\
&=&\frac 1 2  \trch\trchb+ \frac 1 2 \atrch\atrchb-4\rho-2\rhoF^2-2\dual\rhoF^2+3\div \etab-\eta \c \etab-\Re( \hat{V}_2)
\eeaa
and
\beaa
\Im(\tilde{V}_2)&=&\Im( \frac 3 4 \tr \Xb  \ov{\tr X}+ \frac 1 4\ov{\tr \Xb}   \tr X-3\ov{P} +P + \frac 3 2\ov{\DDc}\c H) - \eta \wedge \etab- \Im(\hat{V}_2) - 2\rhod+2 \eta \wedge \etab  \\
&=&  \trchb \atrch + 2 \dual \rho -3\curl\etab + \eta \wedge \etab  - \Im(\hat{V}_2)
\eeaa
where 
\beaa
\Re(\hat{V}_2)&=&\Re( I^{\Ffr}_{3} )+\Re(J^{\Ffr}_3) +\Re(K^{\Ffr}_3)+\Re(M^{\Ffr}_3)\\
\Im(\hat{V}_2)&=&\Im( I^{\Ffr}_{3} )+\Im(J^{\Ffr}_3) +\Im(K^{\Ffr}_3)+\Im(M^{\Ffr}_3)
\eeaa
Using \eqref{I-Ffr-3}, and writing $C_2=\trchb+ i p_2\atrchb$ we obtain
\beaa
\Re(I^{\Ffr}_3)&=& -2\rho -2\rhoF^2 -2\dual\rhoF^2 -2 \eta \c (\eta-2\etab)+\nabc_4 \trchb\\
&=& -\frac 1 2 \trch\trchb+\frac 1 2 \atrch\atrchb -2\rhoF^2 -2\dual\rhoF^2+ 2 \div \etab+4\eta\c\etab\\
\Im(I^{\Ffr}_3)&=& 4\rhod-4\eta \wedge \etab+p_2\nabc_4 \atrchb \\
&=& (4+2 p_2)\dual \rho+ 2 p_2 \curl \etab-4\eta \wedge \etab  
\eeaa
Using \eqref{J-Ffr-3}, we have
\beaa
\Re(J^{\Ffr}_3)&=&\Re(- \trchb  \left(\frac 3 2 \ov{\tr X} +\frac 1 2  \tr X\right)+\frac 1 2(\DDc \c \ov{H})+ \frac 1 2 (H \c \ov{H}))=-2\trch\trchb+\div\etab+|\etab|^2\\
\Im(J^{\Ffr}_3)&=&\Im(- \frac 1 2   \left(\tr \Xb +\ov{\tr\Xb}\right)  \left(\frac 3 2 \ov{\tr X} +\frac 1 2  \tr X\right)+\frac 1 2(\DDc \c \ov{H}))=- \trchb\atrch +\curl \etab
\eeaa
Using \eqref{K-Ffr-3}, we have
\beaa
\Re(K^{\Ffr}_3)&=& -2\nabc_3\left(\trch \right)=  \trchb\trch- \atrch\atrchb-4\rho-4\div\etab-4|\etab|^2\\
\Im(K^{\Ffr}_3)&=&-\nabc_3\atrch=2 \dual \rho+2 \curl \etab 
\eeaa
Using \eqref{M-Ffr-3}, we have
\beaa
\Re(M^{\Ffr}_3)&=&\Re( \eta \c \left(4   (\eta+i \dual \eta)+ (\eta- i \dual \eta)+  \etab + i \dual \etab \right))=5|\etab|^2+\eta\c\etab\\
\Im(M^{\Ffr}_3)&=&\Im( \eta \c \left(4   (\eta+i \dual \eta)+ (\eta- i \dual \eta)+  \etab + i \dual \etab \right))=\eta \wedge \etab
\eeaa
This gives
\beaa
\Re(\hat{V}_2)&=&-\frac 3 2 \trch\trchb-\frac 1 2 \atrch\atrchb-4\rho -2\rhoF^2 -2\dual\rhoF^2- 2 \div \etab+5\eta\c\etab+\div\etab+2|\etab|^2\\
\Im(\hat{V}_2)&=&- \trchb\atrch+(2 p_2+6)\dual \rho+( 2 p_2+3) \curl \etab -3\eta \wedge \etab 
\eeaa
and therefore
\beaa
\Re(\tilde{V}_2)&=&2 \trch\trchb+ \atrch\atrchb+4\div \etab-6\eta \c \etab-2|\etab|^2\\
\Im(\tilde{V}_2)&=&2  \trchb \atrch    -(2 p_2+4)\dual \rho-( 2 p_2+6) \curl \etab +4\eta \wedge \etab 
\eeaa
Finally this gives:
\beaa
\Re(V_2)&=& \Re( \tilde{V}_2)+ \Re(f_2^{-1}\square(f_2))+\Re(3 \ov{P}+2\PF\ov{\PF})
\\
&=&2 \trch\trchb+ \atrch\atrchb+4\div \etab-6\eta \c \etab-2|\etab|^2\\
&&-  3 \trch\trchb  -  \atrch\atrchb   - 3 \rho + 4|\etab|^2+ 8 \eta \c \etab+3\rho+2\rhoF^2+2\dual\rhoF^2\\
&=&- \trch\trchb+2\rhoF^2+2\dual\rhoF^2+4\div \etab+2\eta \c \etab+2|\etab|^2
\eeaa
and
\beaa
\Im(V_2)&=& \Im( \tilde{V}_2)+ \Im(f_2^{-1}\square(f_2))+\Im(3 \ov{P})\\
&=&2  \trchb \atrch    -(2 p_2+4)\dual \rho-( 2 p_2+6) \curl \etab +4\eta \wedge \etab -2 \trchb  \atrch + \rhod    -4 \eta \wedge \etab-3\dual\rho\\
&=&  -(2 p_2+6)\dual \rho-( 2 p_2+6) \curl \etab
\eeaa
Observe that for $p_2=-3$, we obtain $\Im(V_2)=0$, as desired. This completes the proof of the proposition.

\subsection{Proof of Lemma \ref{lemma:lot-terms}}\label{sec:proof-lemma-lot-terms}
From \eqref{eq:W4Ffr} and $C_2=\trchb -3 i \atrchb$ as in \eqref{definition-C1-C22-final}, we compute
\beaa
\Im(W_4^{\Ffr})&=&\Im( -\frac 32 \tr \Xb^2+ \ov{\tr \Xb} \tr\Xb  +\frac 1 2 \left(2\tr \Xb-\ov{\tr \Xb}  \right) C_2)\\
&=& 3 \trchb \atrchb + \Im(\frac 1 2 \left(\trchb - 3 i \atrchb  \right) (\trchb - 3 i \atrchb))=0
\eeaa
and similarly, from \eqref{eq:WaXfr}
\beaa
\Im(W_a^{\Xfr})&=&  \frac 1 2\Im\left(-\frac 32 \tr \Xb^2+ \ov{\tr \Xb} \tr\Xb +\frac 1 2 \left(2\tr \Xb-\ov{\tr \Xb}  \right) C_2 \right) =0
\eeaa

From \eqref{eq:WaBfr}, we obtain
\beaa
 W_{a}^{\Bfr}&=& \frac 1 2 \left(C_1- C_2 - \tr \Xb \right) \\
 &=& \frac 1 2 \left(2\trchb -\frac 5 2  i  \atrchb- \trchb +3 i \atrchb - \trchb+i \atrchb \right) =  \frac 3 4  i  \atrchb 
 \eeaa
 
 From \eqref{eq:Y-Ffr-a}
 \beaa
 Y^{\Ffr}_{a}&=& 2C_1-2C_2+\tr \Xb-3 \ov{\tr\Xb} \\
 &=& 2(2\trchb -\frac 5 2  i  \atrchb)-2(\trchb -3 i \atrchb)+\trchb - i \atrchb -3(\trchb + i \atrchb )= -3  i \atrchb 
 \eeaa
 
 We now compute the imaginary parts of $Z^{\Bfr}_{a}$ and $W_a^{\Ffr}  -Z^{\Ffr}_{a}$ . From \eqref{eq:Z-Bfr-a} and \eqref{eq:Z-Ffr-a}, we have
 \beaa
\Im( Z^{\Bfr}_{a})&=&\Im( I^{\Bfr}_a) + \Im(J^{\Bfr}_a) + \Im(L^{\Bfr}_a)+\Im(M_a^{\Bfr})-2\etab \c \Im(C_1)\\
\Im( Z^{\Ffr}_{a})&=&\Im( I^{\Ffr}_a) + \Im(J^{\Ffr}_a) + \Im(L^{\Ffr}_a)+\Im(M_a^{\Ffr})-2\etab \c \Im(C_2 )
 \eeaa
 Using \eqref{eq:I-Bfr-a} and \eqref{eq:I-Ffr-a}, we have
 \beaa
\Im( I^{\Bfr}_a)=\Im( I^{\Ffr}_a)&=&\Im \big( -  2\nabc_3(\eta-\etab)+\trchb (\eta-\etab)-\atrchb \dual (\eta-\etab) \big)=0
 \eeaa
 Using \eqref{eq:J-Bfr-a} and \eqref{eq:J-Ffr-a}, we write
 \beaa
J^{\Bfr}_a&=&- \frac 1 2(\tr\Xb+\ov{\tr\Xb})\Big(2(\eta-\etab)  -\left( 6H+ \ov{H}+ 3  \ov{ \Hb}  \right)\Big) \\
&& - 2\nabc  C_1  -  \tr \Xb  H - \frac 1 2 (\ov{\tr\Xb}-\tr \Xb)\ov{\Hb} +\frac 1 2 (\ov{\tr\Xb}-\tr\Xb )\, \ov{H}
 \eeaa
 and
 \beaa
 J^{\Ffr}_a&=&- \frac 1 2(\tr\Xb+\ov{\tr\Xb})\Big(2(\eta-\etab)  - \left(4   H+ \ov{H}+  \Hb \right)\Big) \\
 &&- 2\nabc  C_2+ \frac 1 2 \ov{\tr\Xb}\,  \ov{H}-\frac 1 2\left( \tr\Xb - \ov{\tr\Xb}\right) \Hb  -\frac 1 2  (2 \tr \Xb+\ov{\tr\Xb} ) H
 \eeaa
 By writing $\nabc C_1=\nab C_1-\ze C_1$ and $\nabc C_2=\nab C_2-\ze C_2$, we obtain
  \beaa
\Im(J^{\Bfr}_a)&=&\trchb \Im \left( 6H+ \ov{H}+ 3  \ov{ \Hb}  \right)  - 2\nab \Im(C_1)+2\ze \Im(C_1)\\
&&  -\atrchb \etab  +\atrchb \eta  -(\trchb \dual \eta - \atrchb \eta)\\
&=&\trchb \left( 5\dual \eta - 3 \dual \etab \right)  - 2\nab \Im(C_1)+2\ze \Im(C_1)  -\atrchb \etab  +2\atrchb \eta  -\trchb \dual \eta \\
&=&  - 2\nab \Im(C_1)+2\ze \Im(C_1)  -\atrchb \etab  +2\atrchb \eta  +4\trchb \dual \eta-3 \trchb  \dual \etab  
 \eeaa
 and
 \beaa
\Im( J^{\Ffr}_a)&=& \trchb \Im \left(4   H+ \ov{H}+  \Hb \right) - 2\nab \Im(C_2)+2\ze \Im(C_2)\\
&&-\trchb \dual \eta +\atrchb \etab    -(\trchb \dual \eta - \atrchb \eta)\\
&=& \trchb  \left(3\dual \eta + \dual \etab \right) - 2\nab \Im(C_2)+2\ze \Im(C_2)-2\trchb \dual \eta +\atrchb \etab   + \atrchb \eta\\
&=&  - 2\nab \Im(C_2)+2\ze \Im(C_2) +\atrchb \etab   + \atrchb \eta +\trchb \dual \eta+\trchb  \dual \etab 
 \eeaa 
 Using \eqref{eq:L-Bfr-a} and \eqref{eq:L-Ffr-a}, we have
 \beaa
\Im( L^{\Bfr}_a)&=&\Im(- \left( 3\tr\Xb +\ov{\tr\Xb}\right) (\eta-\etab) )=2\atrchb (\eta-\etab)\\
\Im( L^{\Ffr}_a)&=&\Im(-  \left(\tr \Xb+\ov{\tr \Xb}  \right) (\eta-\etab)  )=0
 \eeaa
 Using \eqref{eq:M-a-Bfr} and \eqref{eq:M-a-Ffr}, we have
 \beaa
\Im( M_a^{\Bfr})&=& \nabc_3 \Im\left( 6H+ \ov{H}+ 3  \ov{ \Hb}  \right)-\frac  1 2   \trchb\, \Im\left( 6H+ \ov{H}+ 3  \ov{ \Hb}  \right)+\frac 1 2 \atrchb\, \dual\Im \left( 6H+ \ov{H}+ 3  \ov{ \Hb}  \right)  \\
&=& \nab_3 \left( 5\dual \eta - 3 \dual \etab \right)-\frac  1 2   \trchb\, \left( 5\dual \eta - 3 \dual \etab \right)+\frac 1 2 \atrchb\,  \left( -5 \eta + 3  \etab  \right)  
 \eeaa
and 
\beaa
\Im(M_a^{\Ffr})&=& \nabc_3 \Im\left(4   H+ \ov{H}+  \Hb \right)-\frac  1 2   \trchb\, \Im \left(4   H+ \ov{H}+  \Hb \right)+\frac 1 2 \atrchb\, \dual \Im \left(4   H+ \ov{H}+  \Hb \right) \\
&=& \nab_3 \left(3\dual \eta + \dual \etab \right)-\frac  1 2   \trchb\,  \left(3\dual \eta + \dual \etab \right)+\frac 1 2 \atrchb\,  \left(-3 \eta - \etab \right) 
\eeaa
We therefore obtain
\beaa
\Im( Z^{\Bfr}_{a})&=&   - 2\nab \Im(C_1)+2\ze \Im(C_1)  -\atrchb \etab  +2\atrchb \eta  +4\trchb \dual \eta-3 \trchb  \dual \etab  \\
&&+2\atrchb (\eta-\etab)+\nab_3 \left( 5\dual \eta - 3 \dual \etab \right)-\frac  1 2   \trchb\, \left( 5\dual \eta - 3 \dual \etab \right)+\frac 1 2 \atrchb\,  \left( -5 \eta + 3  \etab  \right)  -2\etab \c \Im(C_1)\\
&=&   - 2\nab \Im(C_1)+2(\ze-\etab) \Im(C_1)    +\frac 3 2 \atrchb( \eta -\etab) +\frac 3 2 \trchb (\dual \eta-  \dual \etab ) +\nab_3 \left( 5\dual \eta - 3 \dual \etab \right)
\eeaa
and, also using \eqref{eq:W-a-Ffr}
\beaa
\Im( Z^{\Ffr}_{a}-W_a^{\Ffr})&=&  - 2\nab \Im(C_2)+2\ze \Im(C_2) +\atrchb \etab   + \atrchb \eta +\trchb \dual \eta+\trchb  \dual \etab \\
&& +\nab_3 \left(3\dual \eta + \dual \etab \right)-\frac  1 2   \trchb\,  \left(3\dual \eta + \dual \etab \right)+\frac 1 2 \atrchb\,  \left(-3 \eta - \etab \right)-2\etab \c \Im(C_2 )+  3   \nab_3 \dual \eta\\
&=&  - 2\nab \Im(C_2)+2(\ze-\etab) \Im(C_2) +\frac 1 2 \atrchb (\etab-\eta)+\frac 1 2 \trchb ( \dual \etab -\dual \eta) +\nab_3 \left(6\dual \eta + \dual \etab \right)
\eeaa

We now evaluate in the outgoing frame. Since $\ze=-\etab$, and by writing $\Im(C_1)=p_1\atrchb$,  and $\Im(C_1)=p_2\atrchb$ we have
\beaa
\Im( Z^{\Bfr}_{a})&=&   - 2p_1\nab\atrchb -4p_1 \atrchb   \etab  +\frac 3 2 \atrchb( \eta -\etab) +\frac 3 2 \trchb (\dual \eta-  \dual \etab ) +\nab_3 \left( 5\dual \eta - 3 \dual \etab \right)
\eeaa
and
\beaa
\Im( Z^{\Ffr}_{a}-W_a^{\Ffr})&=&  - 2p_2\nab \atrchb-4p_2\atrchb \etab+\frac 1 2 \atrchb (\etab-\eta)+\frac 1 2 \trchb ( \dual \etab -\dual \eta) +\nab_3 \left(6\dual \eta + \dual \etab \right)
\eeaa
We start by evaluating at $2$. Then since $\eta_2=-\etab_2$ and $\dual \eta_2=\dual\etab_2$, we have
\beaa
\Im( Z^{\Bfr}_{2})&=& -4p_1 \atrchb   \etab_2  - 3  \atrchb \etab_2+\nab_3 \left( 5\dual \eta - 3 \dual \etab \right)_2
\eeaa
Also using that $\nabc_3 \dual \eta_2=-2\atrchb \etab_2$ and $(\nabc_3 \dual \etab)_2 =-\atrchb\etab_2$, see \cite{GKS}, we have
\beaa
\Im( Z^{\Bfr}_{2})&=& -4p_1 \atrchb   \etab_2  - 3  \atrchb \etab_2-10\atrchb \etab_2+3\atrchb\etab_2\\
&=& -2(2p_1+5) \atrchb   \etab_2 
\eeaa
which indeed vanish for $p_1=-\frac 5 2$. We now also evaluate at $1$. We have, using that $\nab_1(\atrchb)= - 3   \atrchb  \etab_1- \trchb  \dual \etab_1$ and $\eta_1=\etab_1$ and $\dual \eta_1=-\dual\etab_1$, 
\beaa
\Im( Z^{\Bfr}_{1})&=&   - 2p_1(- 3   \atrchb  \etab_1- \trchb  \dual \etab_1) -4p_1 \atrchb   \etab_1 - 3  \trchb   \dual \etab_1  +\nab_3 \left( 5\dual \eta - 3 \dual \etab \right)_1\\
&=&   2p_1  \atrchb  \etab_1+(2p_1-3) \trchb  \dual \etab_1  +\nab_3 \left( 5\dual \eta - 3 \dual \etab \right)_1
\eeaa
Also using that $\nabc_3 \dual \eta_1=\trchb\dual \etab_1+ \atrchb \etab_1$ and $(\nabc_3 \dual \etab)_1= -\trchb\dual \etab_1$, we have
\beaa
\Im( Z^{\Bfr}_{1})&=&   2p_1  \atrchb  \etab_1+(2p_1-3) \trchb  \dual \etab_1 +5\trchb\dual \etab_1+ 5\atrchb \etab_1 +3\trchb\dual \etab_1\\
&=&  ( 2p_1 +5)( \atrchb  \etab_1+\trchb  \dual \etab_1)
\eeaa
which vanishes for $p_1=-\frac 5 2$. 

Similarly, we have
\beaa
\Im( Z^{\Ffr}_{2}-W_2^{\Ffr})&=& -4p_2\atrchb \etab_2+ \atrchb \etab_2 +\nab_3 \left(6\dual \eta + \dual \etab \right)_2\\
&=& -4p_2\atrchb \etab_2+ \atrchb \etab_2 -12\atrchb \etab_2-\atrchb\etab_2\\
&=& -(4p_2+12)\atrchb \etab_2
\eeaa
and
\beaa
\Im( Z^{\Ffr}_{1}-W_1^{\Ffr})&=&  - 2p_2\nab_1 \atrchb-4p_2\atrchb \etab_1+ \trchb  \dual \etab_1 +\nab_3 \left(6\dual \eta + \dual \etab \right)_1\\
&=& 2p_2   \atrchb  \etab_1+(2p_2+1)\trchb  \dual \etab_1 +\nab_3 \left(6\dual \eta + \dual \etab \right)_1\\
&=& 2p_2   \atrchb  \etab_1+(2p_2+1)\trchb  \dual \etab_1 +6\trchb\dual \etab_1+6 \atrchb \etab_1 -\trchb\dual \etab_1\\
&=& (2p_2 +6) ( \atrchb  \etab_1+\trchb  \dual \etab_1 )
\eeaa
which vanishes for $p_2=-3$. This completes the proof of the Lemma.

\end{document}